\documentclass[11pt]{article}

\usepackage{empheq}
\usepackage{amsmath}
\usepackage{amsthm}
\usepackage{amssymb}
\usepackage{mathtools}
\usepackage{thmtools}
\usepackage[cal=euler,scr=boondoxupr,bb=bboldxLight]{mathalpha}
\usepackage{upgreek}
\usepackage[T1]{fontenc}
\usepackage{lmodern}
\usepackage{mleftright}
\usepackage[sf,bf]{titlesec}
\usepackage{bm}
\usepackage{euscript}
\usepackage{url}            %
\usepackage{booktabs}       %
\usepackage{amsfonts}       %
\usepackage{nicefrac}       %
\usepackage{microtype}
\usepackage{color}
\usepackage{tabu}
\usepackage{multirow}
\usepackage{float}
\usepackage{booktabs}
\usepackage{tablefootnote}
\usepackage{algorithm}
\usepackage{algorithmic}
\usepackage{comment}
\usepackage{graphicx}
\usepackage{subcaption}
\usepackage{framed}
\usepackage{enumitem}
\usepackage{textcomp}
\usepackage{setspace}
\usepackage{hyperref}
\usepackage[capitalize,noabbrev]{cleveref}
\usepackage{latexsym}
\usepackage{tcolorbox}
\usepackage[numbers,compress]{natbib}
\usepackage{alltt}
\usepackage{stmaryrd}
\usepackage{xspace}
\usepackage{xfrac}    %
\usepackage{tikz}
\usepackage{tikz-qtree}
\usepackage{todonotes}
\usepackage{inconsolata} %
\usepackage{fontawesome5}

\usetikzlibrary{calc}

\usepackage[margin=1in]{geometry}
\usepackage{fancyhdr}
\fancypagestyle{plain}{\fancyhf{} 
	
	\fancyfoot[C]{\textsf{\thepage}}
}
\definecolor{mydarkblue}{rgb}{0,0.08,0.45}
\hypersetup{ %
	colorlinks=true,
	linkcolor=mydarkblue,
	citecolor=mydarkblue,
	filecolor=mydarkblue,
	urlcolor=mydarkblue}
    
\hypersetup{
    pdftitle={Symmetry-Compatible Principle for Optimizer Design: Embeddings, LM Heads, SwiGLU MLPs, and MoE Routers},
    pdfauthor={Tim Tsz-Kit Lau, Weijie Su},
}

\newcommand{\calC}{\mathcal{C}}

\newcommand{\calE}{\mathcal{E}}

\newcommand{\calG}{\mathcal{G}}

\newcommand{\calL}{\mathcal{L}}

\newcommand{\calS}{\mathcal{S}}
\newcommand{\calT}{\mathcal{T}}
\newcommand{\calU}{\mathcal{U}}

\newcommand{\scrO}{\mathscr{O}}

\newcommand{\scrU}{\mathscr{U}}

\newcommand{\Prob}{\mathbb{P}}

\newcommand{\One}{\bm{1}}

\newcommand{\RR}{\mathbb{R}}

\newcommand{\OO}{\mathbb{O}}
\newcommand{\Rp}{\RR_+}

\newcommand{\bbS}{\mathbb{S}}

\newcommand{\NN}{\mathbb{N}}

\DeclareMathOperator{\prox}{prox}

\DeclareMathOperator{\proj}{proj}

\DeclareMathOperator*{\argmin}{argmin}

\DeclareMathOperator*{\Diag}{Diag}

\newcommand{\diag}{\mathrm{diag}}
\newcommand{\sgn}{\mathrm{sgn}}
\DeclareMathOperator*{\dom}{dom}

\newcommand{\softmax}{\mathrm{softmax}}

\newcommand{\tf}{\widetilde{f}}

\newcommand{\tz}{\widetilde{z}}
\newcommand{\tG}{\widetilde{G}}

\newcommand{\tW}{\widetilde{W}}

\newcommand{\summ}{\sum_{i=1}^m}
\newcommand{\sumn}{\sum_{i=1}^n}

\newcommand{\tU}{\widetilde{U}}
\newcommand{\tV}{\widetilde{V}}
\newcommand{\tM}{\widetilde{M}}
\newcommand{\tN}{\widetilde{N}}

\newcommand{\oz}{\overline{z}}

\newcommand{\tSigma}{\widetilde{\Sigma}}

\newcommand{\tr}{\mathrm{tr}}

\renewcommand{\vec}{\mathrm{vec}}

\newcommand{\sfA}{\mathsf{A}}

\newcommand{\sfL}{\mathsf{L}}

\newcommand{\sfR}{\mathsf{R}}

\newcommand{\oRR}{\overline{\RR}}

\newcommand{\midd}{\,|\kern-0.25ex|\,}

\newcommand{\set}[1]{\llbracket #1\rrbracket}

\newcommand{\dotp}[2]{\langle #1, #2\rangle}
\newcommand{\dotpF}[2]{\left\langle #1, #2\right\rangle_{\rm F}}

\newcommand{\half}{\sfrac12}

\DeclareMathAlphabet\rsfscr{U}{rsfso}{m}{n}

\makeatletter
\DeclareFontFamily{OMX}{MnSymbolE}{}
\DeclareSymbolFont{MnLargeSymbols}{OMX}{MnSymbolE}{m}{n}
\SetSymbolFont{MnLargeSymbols}{bold}{OMX}{MnSymbolE}{b}{n}
\DeclareFontShape{OMX}{MnSymbolE}{m}{n}{
	<-6>  MnSymbolE5
	<6-7>  MnSymbolE6
	<7-8>  MnSymbolE7
	<8-9>  MnSymbolE8
	<9-10> MnSymbolE9
	<10-12> MnSymbolE10
	<12->   MnSymbolE12
}{}
\DeclareFontShape{OMX}{MnSymbolE}{b}{n}{
	<-6>  MnSymbolE-Bold5
	<6-7>  MnSymbolE-Bold6
	<7-8>  MnSymbolE-Bold7
	<8-9>  MnSymbolE-Bold8
	<9-10> MnSymbolE-Bold9
	<10-12> MnSymbolE-Bold10
	<12->   MnSymbolE-Bold12
}{}

\let\llangle\@undefined
\let\rrangle\@undefined
\DeclareMathDelimiter{\llangle}{\mathopen}%
{MnLargeSymbols}{'164}{MnLargeSymbols}{'164}
\DeclareMathDelimiter{\rrangle}{\mathclose}%
{MnLargeSymbols}{'171}{MnLargeSymbols}{'171}
\makeatother

\makeatletter
\@ifundefined{newcounteralias}{}{%
  \renewcommand\thmt@autorefsetup{%
    \@xa\def
      \csname\thmt@envname autorefname\@xa\endcsname
      \@xa{\thmt@thmname}%
  }%
}
\makeatother

\newtheoremstyle{sftheorem}
{\topsep}
{\topsep}
{\itshape}
{}
{\bfseries\sffamily}
{.}
{.5em}
{\thmname{#1}\thmnumber{ #2}\thmnote{ (#3)}}

\newtheoremstyle{sfdefinition}
{\topsep}
{\topsep}
{\normalfont}
{}
{\bfseries\sffamily}
{.}
{.5em}
{}

\theoremstyle{sftheorem}
\newtheorem{theorem}{Theorem}[section]
\newtheorem{proposition}[theorem]{Proposition}
\newtheorem{lemma}[theorem]{Lemma}
\newtheorem{corollary}[theorem]{Corollary}

\theoremstyle{sfdefinition}
\newtheorem{definition}{Definition}[section]
\newtheorem{assumption}{Assumption}[section]
\newtheorem{example}{Example}[section]

\theoremstyle{remark}
\newtheorem{remark}{Remark}[section]

\crefname{assumption}{Assumption}{Assumptions}
\Crefname{assumption}{Assumption}{Assumptions}
\crefname{problem}{Problem}{Problems}
\Crefname{problem}{Problem}{Problems}
\crefname{example}{Example}{Examples}
\Crefname{example}{Example}{Examples}

\let\le\leqslant
\let\ge\geqslant

\let\succeq\succcurlyeq

\let\bar\overline

\newcommand{\lrangle}[1]{\left\llangle #1 \right\rrangle}
\renewcommand{\dotpF}[2]{\lrangle{#1, #2}_{\rm F}}
\newcommand{\norm}[1]{\left\lVert#1\right\rVert}
\newcommand{\euclidnorm}[1]{\left\lVert#1\right\rVert_2}

\newcommand{\vecnorm}[2]{\left\lVert #1 \right\rVert_{{#2}}}
\newcommand{\matsnorm}[2]{\lvert\kern-0.25ex\lvert\kern-0.25ex\lvert #1 \rvert\kern-0.25ex\rvert\kern-0.25ex\rvert_{#2}}

\newcommand{\fronorm}[1]{\matsnorm{#1}{\mathrm{F}}}

\newcommand{\onenorm}[1]{\vecnorm{#1}{1}}

\newcommand{\nucnorm}[1]{\matsnorm{#1}{\mathrm{nuc}}}

\newcommand{\specnorm}[1]{\matsnorm{#1}{\mathrm{S}}}

\renewcommand{\left}{\mleft}
\renewcommand{\right}{\mright}

\newcommand{\PL}{P\L{}\xspace}
\newcommand{\KLo}{K\L{}\xspace}

\newcommand{\rank}{\mathrm{rank}}
\newcommand{\srank}{\mathrm{srank}}

\newcommand{\Adam}{\textsc{Adam}\xspace}
\newcommand{\AdamW}{\textsc{AdamW}\xspace}

\newcommand{\RMSprop}{\textsc{RMSprop}\xspace}

\newcommand{\AdaGrad}{\textsc{AdaGrad}\xspace}

\newcommand{\Adafactor}{\textsc{Adafactor}\xspace}

\newcommand{\Muon}{\textsc{Muon}\xspace}
\newcommand{\Shampoo}{\textsc{Shampoo}\xspace}

\newcommand{\signSGD}{\textsc{signSGD}\xspace}

\newcommand{\StableAdamW}{\textsc{StableAdamW}\xspace}

\newcommand{\Scion}{\textsc{Scion}\xspace}
\newcommand{\PolarGrad}{\textsc{PolarGrad}\xspace}

\newcommand{\LeftPolarGradM}{\textsc{LeftPolarGradM}\xspace}
\newcommand{\RightPolarGradM}{\textsc{RightPolarGradM}\xspace}
\newcommand{\RowNormM}{\textsc{RowNormM}\xspace}
\newcommand{\HybridPolarGradM}{\textsc{HybridPolarGradM}\xspace}

\newcommand{\Dion}{\textsc{Dion}\xspace}
\newcommand{\Gluon}{\textsc{Gluon}\xspace}

\newcommand{\ProxPolarGrad}{\textsc{ProxPolarGrad}\xspace}

\newcommand{\Disco}{\textsc{Disco}\xspace}

\newcommand{\RMSNorm}{\textsc{RMSNorm}\xspace}
\newcommand{\MoE}{\textsc{MoE}\xspace}

\newcommand{\NorMuon}{\textsc{NorMuon}\xspace}
\newcommand{\OLMoE}{\textsc{OLMoE}\xspace}

\newcommand{\AlgoPerf}{\textsc{AlgoPerf}\xspace}

\newcommand{\RMS}{\mathrm{RMS}}

\newcommand{\msgn}{\mathrm{msgn}}

\newcommand{\polar}{\mathrm{polar}}
\newcommand{\sfp}{\mathsf{p}}

\newcommand{\rmout}{\mathrm{out}}

\newcommand{\GL}{\mathrm{GL}}
\newcommand{\reshape}{\mathrm{reshape}}
\newcommand{\LPRO}{\mathsf{LPRO}}
\newcommand{\hyb}{\mathsf{hyb}}
\newcommand{\nuc}{\mathsf{nuc}}
\newcommand{\row}{\mathsf{row}}
\newcommand{\router}{\mathsf{router}}

\begin{document}
\title{\sffamily\bfseries 
Symmetry-Compatible Principle for Optimizer Design: \\Embeddings, LM Heads, SwiGLU MLPs, and \MoE Routers
}
\author{
    Tim Tsz-Kit Lau%
    \thanks{University of Pennsylvania, Philadelphia, PA 19104, USA. Emails: 
    \href{mailto:timlautk@gmail.com}{\texttt{timlautk@gmail.com}}, 
    \href{mailto:suw@wharton.upenn.edu}{\texttt{suw@wharton.upenn.edu}}. 
    }
    \and 
    Weijie Su%
    \footnotemark[1]
}
    
\maketitle

\begin{abstract}
A striking geometric disparity has long persisted in the practice of deep learning. While modern neural network architectures naturally exhibit rich symmetry and equivariance properties, popular optimization methods, such as \Adam and its variants, operate inherently coordinate-wise, rendering them unable to respect the equivariance structures of the parameter space. In this paper, we address this disparity by introducing a \emph{symmetry-compatible principle} for optimizer design. Specifically, we argue that the gradient update rule should be equivariant under the symmetry group acting on the corresponding weight block, and should remove symmetry-redundant directions when the parameterization has quotient structure. Following this principle, we first provide a unified perspective on bi-orthogonally equivariant updates for general matrix layers, as employed by stochastic spectral descent, \Muon, \Scion, and polar gradient methods. More importantly, by moving from orthogonal groups to permutation and shared-shift symmetries, we derive new classes of symmetry-compatible optimizers tailored to parameter blocks whose symmetries differ from those of ordinary matrix layers: for embedding and LM head matrices, left-permutation and right-orthogonal equivariance leads to one-sided spectral, row-norm, and hybrid row-norm/spectral updates, with projected variants for LM heads; for SwiGLU MLP projections, intermediate-neuron permutation symmetry motivates row-aware and column-aware variants; and for \MoE routers, expert-permutation symmetry together with shared-logit-shift invariance gives rise to projected centered row-norm and left-spectral updates. These constructions yield an end-to-end layerwise optimizer stack in which each major matrix-valued parameter class is assigned an update whose equivariance matches its symmetry group. We corroborate this optimizer design principle through extensive pre-training experiments on dense and sparse \MoE language models, including Qwen3-0.6B-style, Gemma 3 1B-style, \OLMoE-1B-7B-style, and downsized gpt-oss architectures. Across these experiments, symmetry-compatible update rules consistently improve final validation loss, reduce expert load imbalance in sparse \MoE models, and in several cases control final vocabulary-logit growth, improve router stability, and overall training stability over the corresponding \AdamW updates. 
\begin{center}
\faGithub\, \url{https://github.com/timlautk/equivariant_optimizers}
\end{center}
\end{abstract}

\section{Introduction}
\label{sec:intro}
The most widely used optimizers in deep learning, such as \Adam \citep{kingma2015}, \Adafactor \citep{shazeer2018adafactor}, \RMSprop \citep{tieleman2012}, \AdaGrad \citep{duchi2011adagrad,mcmahan2010adaptive}, and their variants, all belong to the broad family of \emph{coordinate-wise adaptive gradient methods}. These methods treat model parameters as a single long concatenated vector and update each coordinate independently. Despite their empirical success, this design implicitly assumes that every entry of a weight matrix is an independent coordinate in a high-dimensional vector space. This assumption is rarely questioned, yet it strongly shapes the training dynamics of modern neural networks. In particular, such a geometry-blind treatment ignores the rich matrix structure of neural network parameters and fails to distinguish between the geometries of different layer types, such as embeddings, LM heads, dense linear layers, attention projections, SwiGLU MLP projections, and \MoE routers. 

At the same time, our theoretical understanding of optimizer behavior remains limited across the two major families most relevant to modern large-scale training: \emph{coordinate-wise adaptive gradient optimizers} and \emph{spectral optimizers}. In language model pre-training in particular, comparisons between these optimizer families are still largely empirical, relying on large-scale benchmarking exercises \citep{wen2025fantastic,semenov2025benchmarking} and speedrunning \citep{modded_nanogpt_2024}, with relatively little analysis of their different geometric behavior and training dynamics. Hyperparameter transfer rules \citep{yang2021tuning} and scaling-law prescriptions \citep{kaplan2020scaling,hoffmann2022training}, for example, are often applied across optimizers, even though their original development was tied primarily to coordinate-wise adaptive methods, particularly \AdamW \citep{loshchilov2019decoupled}. Another notable benchmarking effort is \AlgoPerf: \textsc{Training Algorithms} \citep{dahl2023benchmarking,kasimbeg2025accelerating}, which evaluates training speedups obtained solely from changes to the training algorithm and aims to provide a more comprehensive comparison of optimizers. However, \AlgoPerf does not include a language modeling workload, and its workloads are far smaller than the language models considered in modern pre-training. Such benchmarking practices implicitly assume that different optimizer families are directly comparable and share similar training phenomena, which need not be the case. 

The central thesis of this paper is that optimizer design for modern neural networks should be layerwise and symmetry-compatible. Rather than applying a single coordinate-wise optimizer to all parameters, we propose a layerwise symmetry-compatible principle: each major matrix-valued parameter class should be updated by an optimizer whose equivariance matches the symmetry of that parameter class. This leads to a broad family of \emph{equivariant optimizers}, whose update laws are matched to the symmetry groups of the parameter blocks on which they act. 

\Cref{fig:conceptual_comparison} summarizes this shift. The coordinate-wise view treats matrix-valued parameters as vectorized collections of independent coordinates, leading to updates that can discard spectral structure and break natural equivariances. In contrast, the symmetry-aware matrix view starts from the layerwise geometry of each parameter class and derives optimizer updates whose equivariance matches that geometry.

\begin{figure}[h!]
	\centering
	\resizebox{\textwidth}{!}{
		\begin{tikzpicture}[
			x=1cm, y=1cm,
			font=\small,
			box/.style={
				draw,
				rounded corners=2mm,
				thick,
				align=center,
				minimum height=1.25cm,
				inner sep=7pt
			},
			panel/.style={
				draw,
				rounded corners=2mm,
				thick,
				inner sep=12pt
			},
			arr/.style={->, thick}
			]
			
			\node[panel, minimum width=6.5cm, minimum height=9.8cm, anchor=north west] (L) at (-7.5,0) {};
			\node[anchor=north] at ($(L.north)+(0,-0.45)$) {\textbf{Coordinate-wise view}};
			
			\node[box, fill=red!8, text width=4.9cm] (vec) at (-4.25,-1.75)
			{Parameters treated\\ as a long vector};
			
			\node[box, fill=red!8, text width=4.9cm] (entry) at (-4.25,-4.00)
			{Entrywise adaptive updates\\
				(\Adam, \AdaGrad, \Adafactor, \RMSprop)};
			
			\node[box, fill=red!8, text width=4.9cm] (bad1) at (-4.25,-6.35)
			{Break orthogonal equivariance\\
				for matrix layers};
			
			\node[box, fill=red!8, text width=4.9cm] (bad2) at (-4.25,-8.55)
			{Discard spectral structure \\
				and induce \\mismatched geometry};
			
			\draw[arr] (vec) -- (entry);
			\draw[arr] (entry) -- (bad1);
			\draw[arr] (bad1) -- (bad2);
			
			\node[panel, minimum width=6.5cm, minimum height=9.8cm, anchor=north west] (R) at (1.0,0) {};
			\node[anchor=north] at ($(R.north)+(0,-0.45)$) {\textbf{Symmetry-aware matrix view}};
			
			\node[box, fill=blue!8, text width=5.15cm] (sym) at (4.25,-1.75)
			{Matrix parameters have
				layerwise symmetry and geometry};
			
			\node[box, fill=blue!8, text width=5.15cm] (eqv) at (4.25,-4.00)
			{Update maps should match\\
				the symmetry of each parameter class};
			
			\node[box, fill=blue!8, text width=5.15cm] (classes) at (4.25,-6.35)
			{Spectral, one-sided spectral,\\
				row-norm, and hybrid optimizers};
			
			\node[box, fill=blue!8, text width=5.15cm] (design) at (4.25,-8.55)
			{Architecture--optimizer co-design\\
				for linear layers, SwiGLU MLPs, \\
				embeds., heads and \MoE routers};
			
			\draw[arr] (sym) -- (eqv);
			\draw[arr] (eqv) -- (classes);
			\draw[arr] (classes) -- (design);
			
			\draw[very thick, ->] (-0.25,-4.95) -- (0.35,-4.95);
			\node[align=center] at (0.0,-4.10) {\small rethink\\ optimizer\\ geometry};
			
		\end{tikzpicture}
	}
	\caption{Two perspectives on deep learning optimization. Left:
		coordinate-wise adaptive methods treat matrix parameters as vectors and
		ignore matrix geometry. Right: the symmetry- and equivariance-based
		viewpoint developed in this paper leads to a family of equivariant,
		layer-specific optimizer classes and architecture--optimizer co-design.}
	\label{fig:conceptual_comparison}
\end{figure}
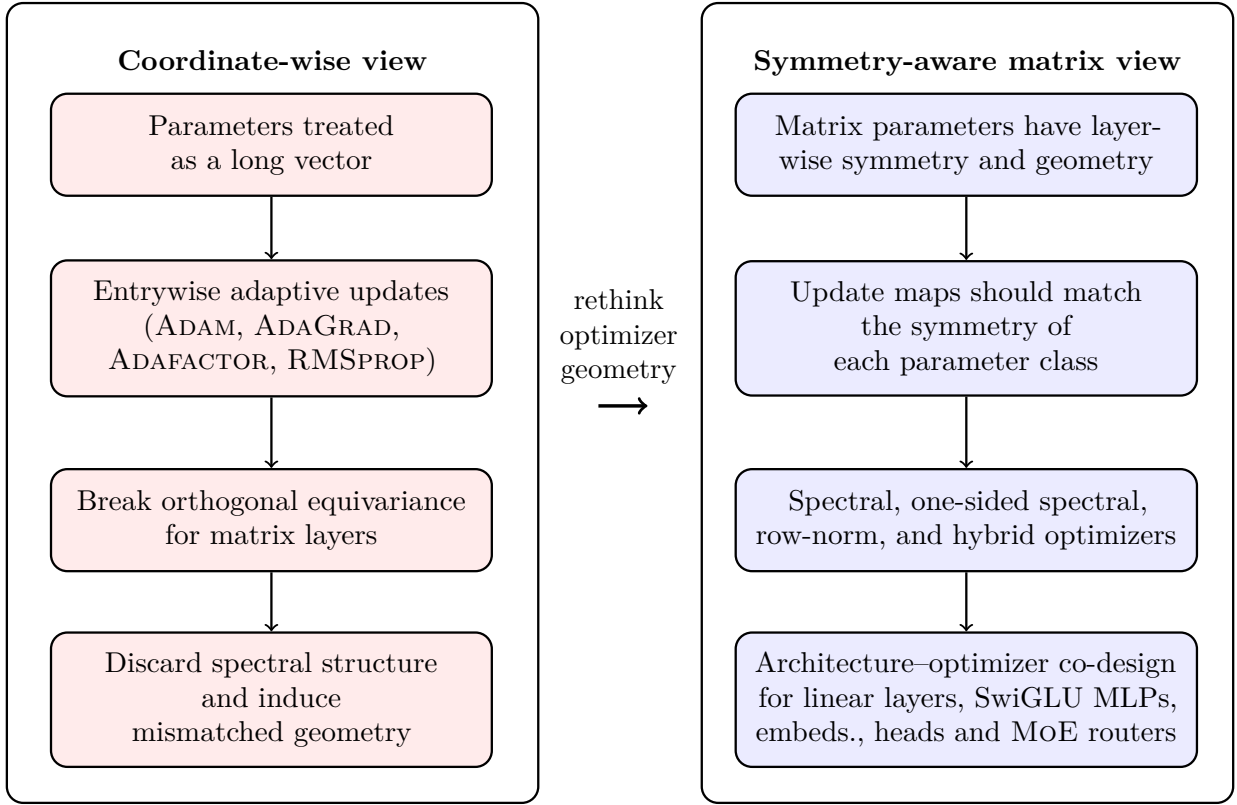

\paragraph{Contributions.}
Our work makes the following contributions.

\begin{enumerate}[leftmargin=*]
	\item \textbf{A symmetry-compatible principle for matrix-gradient optimizer design.}
	We argue that optimizer updates should respect the symmetry structure of the parameter block on which they act. Popular coordinate-wise adaptive optimizers such as \Adam, \AdamW, and \RMSprop are often \emph{geometrically mismatched} for matrix-valued parameters, since their updates generally fail to respect the natural equivariance and quotient structures of modern neural network layers. Fully-connected layers, attention projections, embedding and LM head matrices, SwiGLU MLP projections, and \MoE router matrices all possess nontrivial row, column, permutation, spectral, or shared-shift geometries. Our central message is that neural network weight matrices live in layer-dependent geometries that coordinate-wise adaptive methods do not explicitly capture.

	\item \textbf{A unifying equivariance view of spectral optimizers.}
	We show that bi-orthogonal equivariance naturally leads to the class of \emph{spectral optimizers}. This class includes or provides a unifying interpretation of stochastic spectral descent (SSD) \citep{carlson2015stochasticRBM}, \Muon \citep{jordan2024muon}, \Scion \citep{pethick2025training}, and polar gradient methods (\PolarGrad) \citep{lau2025polargrad}. These methods compute, exactly or approximately, the orthogonal polar factor of an update direction $D$, such as a gradient $G$ or momentum $M$:
	\[
   	D = U\Sigma V^\top \quad\Rightarrow\quad U_{\sfp}\coloneqq \polar(D)=UV^\top .
   	\]
	Such updates are bi-orthogonally equivariant, preserve the singular-vector geometry of the update direction, and arise naturally from matrix geometry. This perspective also gives a symmetry-based interpretation of the spectral-norm steepest descent principle underlying \Muon \citep{bernstein2024old,bernstein2024modular,jordan2024muon}: because the spectral norm is unitarily invariant, the corresponding polar update is naturally bi-orthogonally equivariant.

    \item \textbf{Layerwise equivariant optimizers for architecture--optimizer co-design.}
    Beyond full spectral optimizers for ordinary matrix layers, we derive symmetry-compatible optimizer classes for layers whose symmetries differ from those of standard linear maps. For embeddings and LM heads, left-permutation/right-orthogonal equivariance leads to one-sided spectral, row-norm, and hybrid row-norm/spectral updates; for LM heads, softmax shared-logit-shift invariance further requires projected updates on the horizontal quotient space. For SwiGLU MLP projections, intermediate-neuron permutation symmetry motivates row-aware updates for gate and up projections and column-aware updates for down projections. For \MoE routers, expert-permutation symmetry together with shared-logit-shift invariance leads to projected centered row-norm, left-spectral, and hybrid updates. The corresponding practical momentum variants are denoted \RightPolarGradM, \LeftPolarGradM, \RowNormM, and \HybridPolarGradM. These constructions instantiate an architecture--optimizer co-design principle based on layerwise equivariance and quotient geometry.
    
    \item \textbf{End-to-end pre-training evidence.}
    We evaluate the proposed layerwise optimizer assignments in dense and sparse \MoE language model pre-training experiments (\Cref{sec:expt}). To the best of our knowledge, these experiments provide the first end-to-end pre-training optimizer stack in which all major matrix-valued parameter classes in language models are assigned updates according to their layerwise symmetry. Replacing \AdamW on large vocabulary-indexed matrices with row-norm or hybrid symmetry-compatible updates consistently improves final validation loss. The gains are modest but visible for the smaller Qwen3-0.6B-style dense model, become more pronounced for the larger Gemma 3 1B-style model, and persist in sparse \MoE experiments based on \OLMoE-1B-7B and downsized gpt-oss (\Cref{fig:gpt-oss_val}). In dense models, hybrid row-norm/spectral updates for SwiGLU MLP projections further improve validation loss, and projected symmetry-compatible LM head updates reduce final vocabulary-logit growth in Gemma 3 1B-style pre-training. In the \MoE setting, symmetry-compatible router updates improve over coordinate-wise router updates, reduce expert load imbalance, improve router stability as measured by router z-loss, and can reduce training loss spikes. 
\end{enumerate}

As a representative example, \Cref{fig:gpt-oss_val} shows the effect of symmetry-compatible assignments in a sparse \MoE pre-training experiment.
\begin{figure}[h!]
    \centering
    \includegraphics[width=\textwidth]{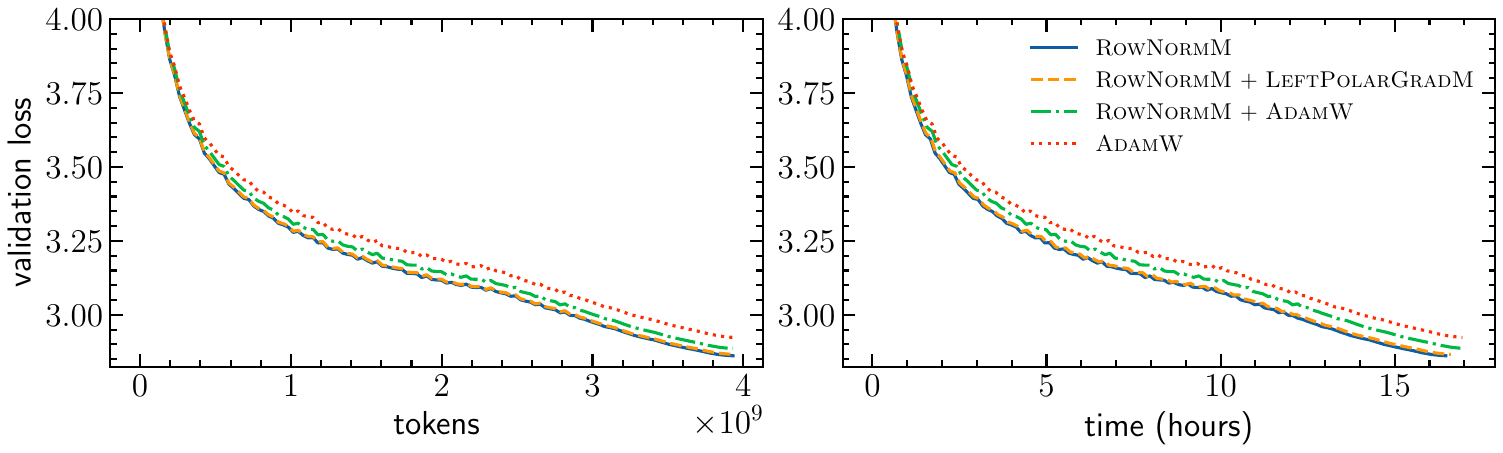}
    \caption{Validation losses for downsized gpt-oss pre-training. The configurations differ in the optimizers for the embedding, LM head, and router matrices; see \Cref{subsec:gpt-oss} for details. Configurations (i) and (ii) use symmetry-compatible optimizers derived from the layerwise equivariance principle, while configuration (iii) replaces the router update by \AdamW and configuration (iv) uses \AdamW for the embedding, LM head, and router matrices.}
    \label{fig:gpt-oss_val}
\end{figure}

\paragraph{Scope and limitations.}
Our goal is not to claim that equivariant optimizers dominate coordinate-wise adaptive methods in all regimes. Rather, we develop a layerwise equivariance principle for matrix-valued parameters and show that it leads to practical optimizer assignments that are competitive and often beneficial in representative pre-training settings. The empirical results should be viewed as evidence for the usefulness of the principle, not as an exhaustive large-scale optimizer benchmark. 

\paragraph{Organization.}
We first introduce notation and closely related work in \Cref{sec:prelim}. In \Cref{sec:general}, we develop the layerwise symmetry-compatible principle, beginning from a linear-operator view of matrix parameters and the resulting coordinate-free equivariance requirements. We then derive symmetry-compatible optimizer classes for embeddings, LM heads, SwiGLU MLP projections, and \MoE routers, including one-sided spectral, row-norm, projected, and hybrid variants. In \Cref{subsec:spectral_optim}, we characterize the spectral optimizer class arising from bi-orthogonal equivariance for ordinary matrix layers. We present practical momentum implementations in \Cref{subsec:practical} and dense and \MoE language model pre-training experiments in \Cref{sec:expt}. We conclude with a discussion of broader implications and future directions in \Cref{sec:discussion}.

\section{Preliminaries and Related Work}
\label{sec:prelim}
In this section, we introduce necessary notation and related work for self-containedness. For an extended overview of related work, we refer the readers to \Cref{sec:further_related}.

\paragraph{Notation.} 
For any real-valued square matrix $S\in\RR^{d\times d}$, $\diag(S)\in\RR^d$ denotes the vector of its diagonal entries, $\Diag(S)\in\RR^{d\times d}$ the diagonal matrix with diagonal entries equal to those of $S$, and $\tr(S)$ is its trace. 
For any $x\in\RR^d$, $\Diag(x)\in\RR^{d\times d}$ is the diagonal matrix with diagonal entries equal to the entries of $x$.
We define $m\wedge n\coloneqq\min\{m, n\}$. 
For any $m\times n$ real-valued matrices $A \coloneqq (a_{i,j})_{1\le i\le m, 1\le j\le n}$ and $B\coloneqq (b_{i,j})_{1\le i\le m, 1\le j\le n}$, we denote the Frobenius inner product of $A$ and $B$ by $\dotpF{A}{B} \coloneqq \tr(A^\top B) = \sum_{i,j} a_{i,j}b_{i,j}$. 
For a matrix $A\in\RR^{m\times n}$, we denote 
its Frobenius norm by $\fronorm{A} \coloneqq \sqrt{\dotpF{A}{A}}$, its spectral norm by $\specnorm{A} \coloneqq \sup_{x\in\RR^n, x\ne0}\{\|Ax\|_2/\|x\|_2\}$, its nuclear norm by $\nucnorm{A} \coloneqq \sum_{i=1}^{m\wedge n} \sigma_i(A)$, where $\sigma(A) = (\sigma_1(A), \ldots, \sigma_{m\wedge n}(A))^\top$ is the vector of nonincreasing ordered singular values of $A$, and its max norm by $\matsnorm{A}{\max} \coloneqq \max_{1\le i\le m, 1\le j\le n}|a_{i,j}|$. 
The Schatten $p$-norm of $A$ is denoted by $\matsnorm{A}{p} \coloneqq\norm{\sigma(A)}_p$. 
The Hadamard product of $A\in\RR^{m\times n}$ and $B\in\RR^{m\times n}$ is denoted by $A\odot B \coloneqq (a_{i,j}b_{i,j})_{1\le i\le m, 1\le j\le n}$. 
For the the matrix $A\in\RR^{m\times n}$, we denote by $\vec(A)\in\RR^{mn}$ its vectorization by rows. Conversely, for $x\in\RR^{mn}$, we write $\reshape(x,m,n)\in\RR^{m\times n}$ for the inverse operation, so that $\reshape(\vec(A),m,n)=A$ for all $A\in\RR^{m\times n}$.
Let $\bbS^d\coloneqq \{A\in\RR^{d\times d} : A = A^\top\}$ denote the space of real symmetric matrices in $\RR^{d\times d}$, $\bbS^d_{+} \coloneqq \{A\in\bbS^d : A\succeq 0\}$ the set of symmetric positive semidefinite matrices, and $\bbS^d_{++} \coloneqq \{A\in\bbS^d : A\succ 0\}$ the set of symmetric positive definite matrices, where $\succeq$ and $\succ$ denote L\"{o}wner orders. 
Let $\OO^d \coloneqq \{A\in\RR^{d\times d} : A^\top A = AA^\top=I_d\}$ denote the set of $d\times d$ orthogonal matrices, where $I_d\in\RR^{d\times d}$ is the $d\times d$ identity matrix. 
Let $\Prob^d \coloneqq \{P\in\{0,1\}^{d\times d} : P\One_d=\One_d,\; P^\top \One_d=\One_d\}$ denote the set of $d\times d$ permutation matrices, where $\One_d$ is the all-ones vector in $\RR^d$. 
Let $\calE$ be a Euclidean space endowed with an inner product $\dotp{\cdot}{\cdot}$ and the induced norm $\|\cdot\|$. 
The domain of a function $f\colon\calE\to\oRR\coloneqq \RR\cup\{\pm\infty\}$ is $\dom f \coloneqq \{x\in\calE : f(x)<\infty\}$. A function $f\colon\calE\to\oRR$ is said to be \emph{proper} if it has a nonempty domain. 
The (convex) indicator function $\iota_\calC(x)$ of a nonempty closed convex set $\calC$ at $x$ equals $0$ if $x\in\calC$ and $+\infty$ otherwise.  
The Euclidean projection of $x$ onto a nonempty closed convex set $\calC$ is denoted by $\proj_{\calC}(x)$. 
$\NN$ denotes the set of nonnegative integers and $\NN^* \coloneqq \NN\setminus\{0\}$ denotes the set of positive integers. 
For a function $f\colon\calE\to\oRR$, we use $\argmin f$ to denote the unique minimizer of $f$.

\subsection{Matrix-Gradient Optimizers}
The recent release of \Muon \citep{jordan2024muon}, together with its strong empirical performance in the \texttt{modded-nanogpt} speedrun \citep{modded_nanogpt_2024}, has renewed interest in matrix-gradient optimizers for deep learning. This has led to a rapidly growing line of work on geometry-aware and matrix-structured optimization methods \citep{pethick2025training,lau2025polargrad,crawshaw2025exploration,veprikov2025preconditioned,chen2025muon,kravatskiy2025ky,frans2025stable,huang2025limuon,si2025adamuon,zhang2025adagrad,xu2026fismo,jiang2026adaptive,zhang2026mousse,gong2026aro,qi2026delving,xu2026width,du2026newton}. Conceptually, \Muon is closely related to stochastic spectral descent (SSD) \citep{carlson2015stochasticRBM,carlson2015preconditioned,carlson2016stochastic}, since both methods can be derived from steepest descent with respect to the spectral norm. We emphasize that this spectral-norm steepest descent perspective is already closely aligned with the equivariance view developed here: because the spectral norm is unitarily invariant, its steepest descent direction is the orthogonal polar factor, and the resulting \Muon update is implicitly bi-orthogonally equivariant. Our contribution is to make this equivariance principle explicit, to place \Muon and related methods inside a broader class of spectral optimizers, and to extend the same symmetry-based design logic to layers whose symmetries are not fully bi-orthogonal, such as embeddings, LM heads, SwiGLU MLP projections, and \MoE routers. 

On the theoretical side, local one-step analyses of simplified \Muon-type updates have been developed in \citep{su2025isotropic,davis2025spectral,gonon2026insights}, while several recent works study convergence rates and optimization guarantees under different assumptions \citep{li2025muon,lau2025polargrad,chang2025convergence,kovalev2025non,shen2025convergence,kim2026convergence,ma2026preconditioning}. Our work is aligned with this broad effort, but differs in emphasis: rather than viewing matrix-gradient optimizers primarily as normalization or preconditioning heuristics, we derive them from symmetry and equivariance principles for matrix-valued neural network parameters. 

A separate but related line of work develops matrix-gradient optimizers from second-order or preconditioning perspectives. These include Kronecker-factored or layerwise preconditioners such as K-FAC \citep{martens2015optimizing,eschenhagen2023kronecker}, \Shampoo \citep{gupta2018shampoo,anil2020scalable,shi2023distributed,eschenhagen2026clarifying}, BFGS and L-BFGS-type methods \citep{goldfarb2020practical}, SOAP \citep{vyas2024soap}, KL-Shampoo and KL-SOAP \citep{lin2025understanding}, and learned or adaptive preconditioners such as preconditioned SGD (PSGD) \citep{li2017preconditioned,pooladzandi2024curvature}. These methods typically approximate curvature or preconditioning structure, whereas spectral and polar updates can also be understood as enforcing equivariance properties of the update map itself. This distinction is important for our framework, since the appropriate optimizer geometry depends on the symmetry group of the layer, not only on curvature approximation. 

Other related directions include imposing constraints directly on the weights, such as Stiefel-manifold interpretations and manifold-constrained optimizers \citep{bernstein2025manifolds,buchanan2025mmuonadmm,xie2026controlled,yang2026manifold,gu2026mano}, as well as variance reduction and low-rank gradient projection methods such as MARS-M and GaLore \citep{liu2025mars,yuan2024mars,zhao2024galore,su2025galore}. These methods address complementary aspects of matrix-gradient optimization, including weight constraints, variance control, and computational efficiency. We refer readers to the recent review \citep{pethick2025training_review} for a broader overview of geometry-aware optimization methods in deep learning.

\subsection{Matrix Optimization Problems, L\"owner Operators, and Spectral Operators}
\label{subsec:related_mop_spectral}

Matrix optimization problems have long been studied as a distinct class of optimization problems because matrices carry algebraic and geometric structures, such as eigenvalues, singular values, ranks, invariant subspaces, and unitary symmetries, that are obscured by vectorization \citep{ding2018spectral,ding2020spectral}. The foundations for convex and unitarily invariant matrix functions, eigenvalue optimization, and spectral optimization were developed in convex matrix analysis and variational analysis \citep{lewis1995convex,lewis1996group,lewis1996eigenvalue,lewis2003mathematics}.

Our framework is also closely related to spectral functions and spectral operators \citep{horn1994topics,bhatia2013matrix,higham2008functions,sun2008lowner,chen2021spectral}. For rectangular matrices, such operators act on singular values while preserving singular vectors, $G = U \Diag(\sigma(G)) V^\top \mapsto \calT(G) = U \Diag(\psi(\sigma(G))) V^\top$. This is the same operator-theoretic structure underlying spectral matrix-gradient optimizers such as stochastic spectral descent, \Muon, \Scion, and polar gradient methods.

\subsection{Symmetry and Equivariance in Deep Learning}
\label{subsec:related_symmetry_equivariance}

There is a long line of work recognizing symmetry and equivariance as organizing principles in neural networks, both for understanding optimization, generalization, and representation learning \citep{ng2004feature,he2016deep,schubert2019circular,li2021why,abbe2022non,zhao2022symmetry,zhao2024improving,putterman2025gl,zhao2026symmetry}, and for designing equivariant architectures \citep{lim2023equivariant,bao2019equivariant,kondor2025principles}. Our work is complementary: rather than imposing equivariance on the architecture or studying equivariance of existing training dynamics, we impose equivariance on the optimizer update map acting on parameter tensors. Thus, our viewpoint extends the equivariance principle from architecture design to optimizer design, where the relevant symmetry is the internal geometry of the parameter space rather than only the symmetry of the input or output domain.

\section{Equivariant Optimizers from Layerwise Symmetry}
\label{sec:general}
Modern deep learning architectures contain matrix-valued parameters with different symmetry structures. The common principle is that a parameter matrix does not always represent an arbitrary array of coordinates, but often represents a linear map between two structured spaces. If the coordinates of these spaces are changed, the parameter and its gradient transform accordingly, and a geometry-compatible optimizer should transform in the same way.

We first state this principle in a general form. Let $W\in\RR^{m\times n}$ represent a linear map from an input space to an output space. Suppose the output and input coordinates are transformed by invertible matrices $P\in\GL(m)$ and $Q\in\GL(n)$. Then the same linear map is represented by $\tW = P W Q^{-1}$.  If $\tf(\tW)\coloneqq f(P^{-1}\tW Q)$, then standard matrix calculus gives
\[
    \nabla_{\tW}\tf(\tW) = P^{-\top}\,\nabla_W f(W)Q^\top.
\]
Thus, under a general change of coordinates, the gradient transforms contravariantly with respect to the output coordinates and covariantly with respect to the input coordinates. 

In this work, we study the equivariance of the update map $\scrU\colon\RR^{m\times n}\to\RR^{m\times n}$ in matrix-optimizer iterations
\[
    (\forall k\in\NN)\qquad W_{k+1}=W_k-\gamma_k\scrU(D_k), 
\]
where $D_k$ is an update direction, such as a gradient or momentum. The relevant requirement is not necessarily that the layerwise loss function be invariant under arbitrary transformations, but that the optimizer update transform consistently with the representation of its input direction. Thus, once a layer symmetry specifies a transformation law $D_k\mapsto g\cdot D_k$, we require
\[
    \scrU(g\cdot D_k)=g\cdot \scrU(D_k).
\]
When $D_k$ transforms equivariantly, the update $\scrU(D_k)$ therefore transforms equivariantly as well.

In this paper, however, we do not require equivariance under all invertible changes of coordinates. The relevant symmetry group depends on the layer. For ordinary linear and attention matrices, the natural coordinate changes are orthonormal changes of basis, so $P\in\OO^m$ and $Q\in\OO^n$. In this case $P^{-\top}=P$ and $Q^{-1}=Q^\top$, and both the parameter and gradient transform as $W\mapsto PWQ^\top$ and $G\mapsto PGQ^\top$. 
This leads to the bi-orthogonal equivariance condition
\[
    \scrU(PGQ^\top)=P\,\scrU(G)\,Q^\top .
\]
For embedding and LM head matrices $W\in\RR^{v\times d}$, the row axis indexes vocabulary items, so the admissible left action is not a general orthogonal rotation but a permutation $P\in\Prob^v$, while the hidden feature axis still admits right orthogonal transformations. For \MoE routers, the row axis indexes experts and additionally has a shared-logit-shift invariance. For SwiGLU MLP projections, the relevant symmetry is permutation of intermediate neurons, which acts on the rows of the gate and up projections and on the columns of the down projection.

This gives a layerwise equivariance principle: the optimizer update map should commute with the symmetry group of the parameter block on which it acts. Full bi-orthogonal equivariance leads to spectral optimizers for ordinary matrix layers; left-permutation/right-orthogonal equivariance leads to row-aware and right-spectral optimizers for embeddings and LM heads; intermediate-neuron permutation symmetry leads to row- and column-aware updates for SwiGLU MLP projections; and expert-permutation plus shared-shift symmetry leads to centered row-aware or left-spectral updates for \MoE routers.

\subsection{A General Symmetry-Induced Optimizer Geometry}
\label{subsec:general_symmetry_geometry}
Let $W\in\RR^{m\times n}$ be a layer parameter and let $f\colon\RR^{m\times n}\to\RR$ be the corresponding layerwise loss. Suppose a group $\calG$ acts on the parameter space by transformations $W\mapsto g\cdot W$. In the matrix settings considered below, this action is typically of the form $g\cdot W = P W Q^{-1}$, or, after restricting to orthogonal or permutation symmetries, $g\cdot W = P W Q^\top$. We say that the parameterization admits the symmetry group $\calG$ if $f(g\cdot W)=f(W)$ for all $g\in\calG$. The corresponding optimizer update map should satisfy
\[
    (\forall g\in\calG)\qquad \scrU(g\cdot G)=g\cdot \scrU(G), 
\]
where $G$ is an update direction, such as a gradient or momentum, expressed in the corresponding transformed coordinates. This condition ensures that the optimizer does not depend on arbitrary choices of representation that are invisible to the model.

\subsection{Bi-Orthogonal Equivariance for Ordinary Matrix Layers}
\label{subsec:biorthogonal_equivariance}
The general reparameterization $W\mapsto PWQ^{-1}$ specializes to $W\mapsto PWQ^\top$ when the admissible coordinate changes are orthogonal. This is the natural case for ordinary linear layers and attention projection matrices, where both input and output coordinates represent continuous feature bases. Therefore an update map for such layers should satisfy
\[
    (\forall P\in\OO^m)(\forall Q\in\OO^n)\qquad\scrU(PGQ^\top)=P\,\scrU(G)Q^\top.
\]
This is exactly \emph{bi-orthogonal equivariance}. 

\begin{definition}[Bi-orthogonal equivariance]
\label{def:orthogonal_equivariance}
Let $\scrU\colon\RR^{m\times n}\to\RR^{m\times n}$ be a matrix-valued map. We say that $\scrU$ is \emph{bi-orthogonally equivariant} if, for all $G\in\RR^{m\times n}$ and all $P\in\OO^m$, $Q\in\OO^n$,
\[
    \scrU(PGQ^\top)=P\,\scrU(G)Q^\top.
\]
\end{definition}
Thus, bi-orthogonal equivariance is exactly the requirement that if $W$ and its gradient are transformed as $W\mapsto PWQ^\top$ and $G\mapsto PGQ^\top$, then the optimizer update transforms in the same way. More generally, if an update rule has the form $\Delta W=\scrU(W,G)$, one may require
\[
    \scrU(PWQ^\top,PGQ^\top)=P\,\scrU(W,G)Q^\top.
\]
In this paper, we focus on update maps of the form $\scrU(G)$ or $\scrU(D)$, where $D$ is a gradient-derived direction such as momentum. Terms that depend explicitly on $W$, such as decoupled weight decay or scalar step-size scaling, can typically be handled separately and preserve the same equivariance, so we suppress the dependence on $W$ for simplicity. 

For ordinary matrix layers, bi-orthogonal equivariance motivates polar and spectral update directions. In particular, the orthogonal polar factor  \citep{autonne1902sur,higham1986computing} satisfies 
\begin{equation}\label{eqn:polar}
    \polar(PGQ^\top)=P\,\polar(G)Q^\top,
\end{equation}
and hence \Muon-style and polar-gradient updates are symmetry-compatible for ordinary matrix layers. We defer the full characterization of bi-orthogonally equivariant maps as spectral operators to \Cref{subsec:spectral_optim}. Standard momentum constructions such as EMA, Polyak, and Nesterov momentum preserve the same equivariance because their buffers are linear combinations of past gradients. We state this fact for Nesterov momentum in the following proposition.

\begin{proposition}[Nesterov momentum is bi-orthogonally equivariant]
\label{prop:momentum}
Let $(W_k)_{k\in\mathbb{N}}\subset\RR^{m\times n}$ be a parameter sequence, and define $G_k \coloneqq \nabla_W f(W_k)$ for $k\in\NN$. Fix orthogonal matrices $P\in\OO^m$ and $Q\in\OO^n$, and define the transformed parameter sequence $\tW_k \coloneqq P W_k Q^\top$ for $k\in\NN$, with transformed gradients $\tG_k \coloneqq \nabla_W f(\tW_k)$ for $k\in\NN$. Let the momentum buffer be $M_k = \beta M_{k-1} + G_k$ with $M_{-1}=0$, and define the update direction by $N_k \coloneqq G_k + \beta M_k$. For the transformed sequence, let $\tM_k = \beta \tM_{k-1} + \tG_k$ with $\tM_{-1}=0$, and $\tN_k \coloneqq \tG_k+ \beta \tM_k$. Then we have $\tM_k = PM_kQ^\top$, $\tN_k = PN_kQ^\top$, and $\polar(\tN_k)=P\,\polar(N_k)Q^\top$.
\end{proposition}
All proofs are given in \Cref{sec:proofs}. 

\begin{remark}[Momentum-first or polar-first \PolarGrad?]
\citet{lau2025polargrad} studied two variants of \PolarGrad---\emph{momentum-first}, which uses the orthogonal polar factor of the EMA momentum of the gradient as the update direction, and \emph{polar-first}, which uses the EMA momentum of the orthogonal polar factor of the gradient as the update direction. They resemble previous orthogonalized gradient optimizers in a similar spirit. \citet{tuddenham2022orthogonalising} adopt a polar-first update, while \citet{jordan2024muon} adopt a momentum-first update, which outperforms the polar-first one in practice. Indeed, \Cref{prop:momentum} provides an intuitive explanation for why momentum-first \PolarGrad is preferred to its polar-first counterpart. 
Only momentum-first updates are directly covered by the bi-orthogonal equivariance argument above: the gradient and its EMA momentum transform equivariantly, and the polar map is then applied to this equivariant momentum direction. In contrast, $\polar(\beta M_{k-1} + (1-\beta)G_k) \ne \beta\polar(M_{k-1}) + (1-\beta)\polar(G_k)$ in general because $\polar(\cdot)$ is nonlinear. Thus, the polar-first update is generally not obtained by applying a spectral map to an equivariant momentum direction. In the momentum-first update, the polar step instead extracts a ``best orthogonal direction'' from the smoothed matrix signal via EMA momentum, which tends to be less noisy and better behaved.
\end{remark}

\subsection{Optimizers for Embeddings and LM Heads via Left-Permutation Right-Orthogonal Equivariance}
\label{subsec:lpro}
For vocabulary-indexed matrices such as input embeddings and untied LM heads, rows index vocabulary items and columns correspond to hidden features. This motivates left-permutation and right-orthogonal equivariant update maps. For an update direction $D\in\RR^{v\times d}$, representative examples include row-norm updates,
right-spectral updates, and hybrid row-norm/right-spectral updates. In the embedding case, these take the form
\[
    \scrU_{\row}(D) = \Diag(\eta(\|D_{1:}\|_2),\dots,\eta(\|D_{v:}\|_2))D,
\]
\[
    \scrU_{\sfR}(D) = D(D^\top D+\varepsilon I)^{-\half},
\]
and
\[
    \scrU_{\hyb}(D) = \scrU_{\sfR}(\scrU_{\row}(D)) \quad\text{or}\quad \scrU_{\row}(\scrU_{\sfR}(D)).
\]
The row-norm update acts locally on vocabulary rows, the right-spectral update acts globally through the hidden-feature Gram matrix, and the hybrid update combines the two.

For untied LM heads, there is an additional quotient geometry. Since softmax probabilities are invariant under adding the same scalar to all vocabulary logits, adding a shared row vector $\One_v a^\top$ to the LM head changes $W_{\rmout}h$ only by a shared logit shift and therefore leaves the output distribution unchanged. Thus, the shared-row direction is vertical, and the intrinsic LM head update should be horizontal. Let
\[
    \Pi_v^\perp \coloneqq I_v-\frac{1}{v}\One_v\One_v^\top .
\]
For LM head updates, we therefore apply the corresponding row-norm, right-spectral, or hybrid maps to the centered direction  $D_c=\Pi_v^\perp D$, and project the final update back to the horizontal subspace when rowwise nonlinear operations are used. For example, the horizontal LM head row-norm update is
\[
    \scrU_{\row}^{\rm LM}(D)
    =
    \Pi_v^\perp
    \Diag(\eta(\|D_{c,1:}\|_2),\dots,\eta(\|D_{c,v:}\|_2))D_c,
    \qquad
    D_c=\Pi_v^\perp D .
\]
The final projection is not redundant: although $\mathbf 1_v^\top D_c=0$, the rowwise-rescaled matrix
\[
    \Diag(\eta(\|D_{c,1:}\|_2),\dots,\eta(\|D_{c,v:}\|_2))D_c
\]
need not be row-centered, since the scaling factors vary across rows. Similarly, for hybrid row-norm/right-spectral LM head updates, we reproject after rowwise nonlinear stages and, in practice, also project the final update to ensure that the actual parameter update remains horizontal.

We now explain why these updates are natural from the symmetry of embeddings and LM heads. In empirical uses of \Muon \citep{jordan2024muon}, it is often recommended that \AdamW be used for embedding and LM head matrices. For embeddings, this recommendation is motivated in part by modular norm theory \citep{large2024scalable}; for LM heads, it appears to be driven more by empirical considerations. Relatedly, \Scion \citep{pethick2025training} derives embedding updates from induced operator norms in its linear minimization oracle framework. These approaches depend on a particular choice of norm. Here we instead derive optimizer classes directly from the symmetry of the parameterization.

Let $v\in\NN^*$ denote the vocabulary size and $d\in\NN^*$ the embedding dimension, typically with $v\gg d$. Consider an input embedding matrix $E\in\RR^{v\times d}$ and an untied LM head matrix $W_{\rm out}\in\RR^{v\times d}$. In both cases, rows index vocabulary items, while columns correspond to hidden features. Thus, the row axis admits permutation symmetry, whereas the hidden feature axis admits right orthogonal symmetry. The natural equivariance condition for an update map is therefore
\begin{equation}\label{eqn:left_perm_right_ortho}
    \scrU_{\LPRO}(PDR^\top)=P\,\scrU_{\LPRO}(D)R^\top,
\end{equation}
for all $D\in\RR^{v\times d}$, permutation matrices $P\in\Prob^v$, and orthogonal matrices $R\in\OO^d$. We call such maps \emph{left-permutation right-orthogonal} (LPRO) equivariant. For untied LM heads, the same principle is applied on the horizontal quotient space defined by $\Pi_v^\perp$.

\begin{definition}[Left-permutation and right-orthogonal equivariant maps]
\label{def:lpro_equivariant}
A map $\scrU_{\LPRO}\colon\RR^{v\times d}\to\RR^{v\times d}$ is said to be \emph{left-permutation and right-orthogonal equivariant} if \eqref{eqn:left_perm_right_ortho} holds for all $D\in\RR^{v\times d}$, $P\in\Prob^v$, and $R\in\OO^d$. We denote the set of such maps by $\calU_{\LPRO}^{v\times d}$.
\end{definition}

\subsubsection{Right-Spectral Optimizers}
A first natural subclass of LPRO-equivariant maps is given by right-spectral updates, 
\begin{equation}\label{eqn:right_spectral}
    \scrU_{\sfR}(D)=D\,\Phi(D^\top D),
\end{equation}
where $\Phi\colon\bbS_+^d\to\RR^{d\times d}$ is an orthogonally equivariant spectral operator. Equivalently, if $D^\top D=V\Diag(\lambda(D^\top D))V^\top$, then
\[
    \Phi(D^\top D) = V\Diag(\psi(\lambda(D^\top D)))V^\top
\]
for some absolutely symmetric map $\psi\colon\Rp^d\to\RR^d$.

\begin{theorem}[Right-spectral updates are LPRO-equivariant]
\label{thm:right_spectral}
Let $\scrU_{\sfR}$ be of the form \eqref{eqn:right_spectral}, where $\Phi(RXR^\top)=R\Phi(X)R^\top$ for all $X\in\bbS_+^d$ and $R\in\OO^d$. Then
\[
    \scrU_{\sfR}(PDR^\top)=P\,\scrU_{\sfR}(D)R^\top
\]
for all $D\in\RR^{v\times d}$, $P\in\Prob^v$, and $R\in\OO^d$. We denote the set of right-spectral maps by $\calU_{\sfR}^{v\times d}$.
\end{theorem}

The choice $\Phi(X)=(X+\varepsilon I)^{-\half}$ yields the damped right-polar update
\[
    \scrU_{\sfR}(D) = D(D^\top D+\varepsilon I)^{-\half}.
\]
When $\varepsilon=0$ and $D$ has full column rank, this is the orthogonal polar factor of $D$. Thus, right-spectral updates are the one-sided analogue of spectral or polar-gradient updates, but they require only the smaller right Gram matrix $D^\top D\in\RR^{d\times d}$.

For LM heads, right-spectral updates preserve horizontality once the input direction is centered. Indeed, if $\mathbf 1_v^\top D_c=0$, then
\[
    \mathbf 1_v^\top D_c\Phi(D_c^\top D_c)=0 .
\]
Thus, in exact arithmetic, the final projection is not necessary for purely right-spectral maps. We nevertheless use the quotient-space form
\[
    \scrU_{\sfR}^{\rm LM}(D)
    =
    \Pi_v^\perp D_c(D_c^\top D_c+\varepsilon I)^{-\half},
    \qquad
    D_c=\Pi_v^\perp D,
\]
as a numerically robust horizontal update, especially when the inverse square root or polar factor is computed by an inexact oracle.

This computational distinction is important for embeddings and LM heads, where $v\gg d$. Although coordinate-wise adaptive methods such as \Adam are often used for these matrices in practice, they are not LPRO-equivariant. The reason for their empirical use may be computational rather than geometric: accurately approximating polar factors of tall-skinny, ill-conditioned vocabulary matrices can be challenging with simple Newton--Schulz iterations. More robust polar decomposition algorithms, such as QDWH \citep{nakatsukasa2010optimizing} and ZOLO-PD \citep{nakatsukasa2016computing}, can compute such updates more accurately \citep{lau2025polargrad}.

Right-spectral normalization is also natural statistically. In a mini-batch of size $b$, gradients of embedding and LM head matrices have rank at most $O(b)$, since they factor through token occurrences and hidden features. Thus, even when the vocabulary dimension is very large, the gradient often lies in a low-dimensional feature subspace. Right-spectral updates act on this intrinsic singular geometry through $D^\top D$, rather than on individual coordinates.

\subsubsection{Row-Norm and Hybrid LPRO-Equivariant Optimizers}
\label{subsubsec:general_LPRO}
Right-spectral maps do not exhaust all LPRO-equivariant updates. Restricting the left symmetry from the full orthogonal group to the permutation group allows update maps that depend on individual row norms. In particular, row-norm maps of the form 
\begin{equation}\label{eqn:row_norm}
    \scrU_{\row}(D) = \Diag(\eta(\|D_{1:}\|_2),\dots,\eta(\|D_{v:}\|_2))D
\end{equation}
are LPRO-equivariant for any scalar function $\eta\colon\Rp\to\RR$, because left multiplication by a permutation matrix permutes the row norms and right orthogonal transformations preserve them. We denote the set of such maps by $\calU_{\row}^{v\times d}$.

For LM heads, the horizontal version is
\begin{equation}\label{eqn:lm_head_horizontal_row_norm}
    \scrU_{\row}^{\rm LM}(D)
    =
    \Pi_v^\perp
    \Diag(\eta(\|D_{c,1:}\|_2),\dots,\eta(\|D_{c,v:}\|_2))D_c,
    \qquad
    D_c=\Pi_v^\perp D .
\end{equation}
This map is permutation equivariant and right-orthogonal equivariant on the horizontal quotient space, and it satisfies
\[
    \One_v^\top \scrU_{\row}^{\rm LM}(D)=0 .
\]
The final projection in \eqref{eqn:lm_head_horizontal_row_norm} is essential for row-norm updates, because row-dependent rescaling does not preserve the zero-row-mean constraint.

Thus, right-spectral maps form a global Gram-based subclass, while row-norm maps form a local row-adaptive subclass:
\[
    \calU_{\sfR}^{v\times d}\subset\calU_{\LPRO}^{v\times d},
    \qquad
    \calU_{\row}^{v\times d}\subset\calU_{\LPRO}^{v\times d}.
\]
Hybrid LPRO maps are obtained by composing these two types of maps. For LM heads, hybrid maps are obtained by the same compositions, but with horizontal projections inserted after rowwise nonlinear stages and on the final update. For example, a row-then-right-spectral horizontal hybrid update may be written as
\[
    D_0=\Pi_v^\perp D,\qquad
    D_1=\Pi_v^\perp
    \Diag(\eta(\|D_{0,1:}\|_2),\dots,\eta(\|D_{0,v:}\|_2))D_0,
\]
\[
    \scrU_{\hyb}^{\rm LM}(D)
    =
    \Pi_v^\perp D_1(D_1^\top D_1+\varepsilon I)^{-\half}.
\]

\begin{proposition}[Closure under composition]
\label{prop:closure_composition}
If $\scrU_1,\scrU_2\colon\RR^{v\times d}\to\RR^{v\times d}$ are both left-permutation and right-orthogonal equivariant, then $\scrU_2\circ \scrU_1$ is also left-permutation and right-orthogonal equivariant. 
\end{proposition}

\begin{definition}[Hybrid LPRO-equivariant maps]
A map $\scrU$ is called a \emph{hybrid LPRO-equivariant map} if it is a finite composition of right-spectral maps and row-norm maps. We denote the set of such maps by $\calU_{\hyb}^{v\times d}$. 
\end{definition}

Accordingly, LPRO-compatible optimizers for embeddings are obtained by applying maps in $\calU_{\sfR}^{v\times d}$, $\calU_{\row}^{v\times d}$, or $\calU_{\hyb}^{v\times d}$ to update directions that transform in the same way as the gradient, such as momentum buffers formed from past gradients. For untied LM heads, the corresponding optimizer should additionally be horizontal with respect to the shared-logit-shift quotient, which is enforced by the projections $\Pi_v^\perp$ above.

\begin{remark}
Unlike the bi-orthogonal case, LPRO equivariance does not characterize a single spectral form. Right-spectral updates are a canonical global subclass, but row-norm and hybrid maps are also symmetry-compatible. This distinction is important: embeddings and LM heads have a discrete vocabulary symmetry on the left, not a full continuous orthogonal symmetry, so row-aware operations are allowed by the layer geometry. For LM heads, softmax shift invariance imposes an additional horizontal constraint; the actual update should be projected back to the quotient space whenever rowwise nonlinear operations are used.
\end{remark}

\paragraph{Examples.}
Several recent optimizers can be interpreted through this framework. SCALE \citep{glentis2025minimalist} applies column normalization to the EMA momentum for LM heads; under our row-vocabulary convention, this corresponds to a row-norm-based update. Other row- or column-norm-based optimizers include RMNP \citep{deng2026rmnp}, REG \citep{liu2025reg} and \textsc{Nora} \citep{yuan2026nora}. Finally, \NorMuon \citep{li2025normuon}, \textsc{Muon}+ \citep{zhang2026muon+}, and \textsc{MuonEq} \citep{chang2026muoneq}, which apply row-wise and/or column-wise normalization to the orthogonal polar factor of the EMA momentum, can be viewed as hybrid spectral/row-norm optimizers.

\subsection{Optimizers for SwiGLU MLP Projections}
\label{subsec:swiglu}
We next consider SwiGLU MLP projection matrices \citep{dauphin2017language,shazeer2020glu}. Unlike ordinary linear and attention projection matrices, SwiGLU projections do not possess full bi-orthogonal symmetry. Instead, their natural symmetry is the permutation symmetry of intermediate neurons. This suggests that the gate and up projections should use left-permutation/right-orthogonal equivariant updates, while the down projection should use the corresponding transposed geometry.
Concretely, for $W_{\mathrm{gate}},W_{\mathrm{up}}\in\RR^{d_{\mathrm{ff}}\times d_{\mathrm{model}}}$ and $W_{\mathrm{down}}\in\RR^{d_{\mathrm{model}}\times d_{\mathrm{ff}}}$, we apply the LPRO-compatible optimizer classes of \Cref{subsec:lpro} to $W_{\mathrm{gate}}$, $W_{\mathrm{up}}$, and $W_{\mathrm{down}}^\top$.

This viewpoint is closely related to Aurora \citep{aurora2026leverage}, an optimizer designed for the up and gate projections in SwiGLU MLPs. Aurora alternates between row-normalization and polar steps on the momentum, which is similar in spirit to the hybrid row-norm/right-spectral optimizers developed above. Our contribution here is to derive this type of update from the intermediate-neuron permutation symmetry of the SwiGLU block.

\begin{proposition}[Intermediate-neuron permutation symmetry of SwiGLU MLPs]
\label{prop:swiglu_permutation_symmetry}
Let us consider the SwiGLU MLP block defined by 
\[
\mathrm{SwiGLU}(x;W_{\mathrm{gate}},W_{\mathrm{up}},W_{\mathrm{down}}) \coloneqq W_{\mathrm{down}} \left(\sigma(W_{\mathrm{gate}}x)\odot(W_{\mathrm{up}}x) \right),
\]
where $W_{\mathrm{gate}},W_{\mathrm{up}}\in\RR^{d_{\mathrm{ff}}\times d_{\mathrm{model}}}$, $W_{\mathrm{down}}\in\RR^{d_{\mathrm{model}}\times d_{\mathrm{ff}}}$, $\sigma$ is applied coordinatewise, and $\odot$ denotes the Hadamard product. Let $P\in\Prob^{d_{\mathrm{ff}}}$ be a permutation matrix, and define $\tW_{\mathrm{gate}} \coloneqq PW_{\mathrm{gate}}$, $\tW_{\mathrm{up}} \coloneqq PW_{\mathrm{up}}$, and $\tW_{\mathrm{down}} \coloneqq W_{\mathrm{down}}P^\top$.
Then, for every input $x\in\RR^{d_{\mathrm{model}}}$,
\[
\mathrm{SwiGLU}(x; \tW_{\mathrm{gate}}, \tW_{\mathrm{up}}, \tW_{\mathrm{down}}) = \mathrm{SwiGLU}(x;W_{\mathrm{gate}},W_{\mathrm{up}},W_{\mathrm{down}}).
\]
\end{proposition}

\begin{remark}[Failure of full left-orthogonal symmetry]
The permutation matrix $P$ in \Cref{prop:swiglu_permutation_symmetry} cannot generally be replaced by an arbitrary orthogonal matrix $Q\in\OO^{d_{\mathrm{ff}}}$. For an elementwise nonlinearity $\sigma$, one typically has $\sigma(Qz)\neq Q\sigma(z)$, and similarly $(Qa)\odot(Qb)\neq Q(a\odot b)$ for a general orthogonal matrix $Q$. Thus, the SwiGLU intermediate dimension has a discrete permutation symmetry, not a full left-orthogonal symmetry.
\end{remark}

\Cref{prop:swiglu_permutation_symmetry} implies that optimizers for $W_{\mathrm{gate}}$ and $W_{\mathrm{up}}$ should commute with permutations of their rows and orthogonal changes of basis in the model dimension. Thus, the same LPRO-compatible row-norm, right-spectral, and hybrid row-norm/right-spectral updates developed for embeddings and LM heads apply directly to these matrices. For the down projection, the intermediate-neuron axis appears as the column dimension, so the same principle applies to $W_{\mathrm{down}}^\top$. Equivalently, one may view the down-projection update as a right-permutation/left-orthogonal analogue of the LPRO updates.

More generally, this intermediate-neuron permutation symmetry is not specific to SwiGLU. Any feed-forward block whose hidden nonlinearity is applied coordinatewise admits a permutation symmetry of the hidden units: simultaneously permuting the rows of the input projection and the corresponding columns of the output projection leaves the represented function unchanged. For gated MLPs such as GLU, GeGLU, ReGLU, and SwiGLU, the same symmetry holds provided the gate and value projections are permuted together, along with the corresponding columns of the down projection.

\subsection{Optimizers for \MoE Routers}
\label{subsec:moe_router}

We now consider the router matrix in a mixture-of-experts (\MoE) model. Unlike ordinary linear layers, embeddings, and LM heads, the router has an additional symmetry: experts are interchangeable, and the softmax is invariant under adding a shared scalar offset to all expert logits.

Let $W\in\RR^{e\times d}$ denote the router matrix, where $e\in\NN^*$ is the number of experts and $d\in\NN^*$ is the hidden dimension. The routing distribution is $p(x;W)=\softmax(Wx)$ for $x\in\RR^d$. Since permuting the rows of $W$ only relabels the experts, and since
$\softmax(z+c\One_e)=\softmax(z)$ for all $z\in\RR^e$ and $c\in\RR$, the router parameters admit the symmetry
\[
(\forall P\in\Prob^e,\forall a\in\RR^d)\qquad W\mapsto PW+\One_e a^\top .
\]
Indeed,
\[
(\forall x\in\RR^d)\qquad (PW+\One_e a^\top)x=P(Wx)+(a^\top x)\One_e,
\]
so the router logits are unchanged up to expert relabeling and a shared logit shift.

This symmetry suggests that router updates should be defined on the centered expert geometry. Let
\[
    \Pi_\perp\coloneqq I_e-\frac1e\One_e\One_e^\top
\]
be the orthogonal projector onto $\One_e^\perp$, and define $D_c\coloneqq \Pi_\perp D$. 

The centered direction $D_c$ removes the shared-row component of $D$ and captures the intrinsic variation across experts. This motivates router update maps built from the centered direction, for example through the centered left Gram matrix
\[
    D_cD_c^\top=\Pi_\perp DD^\top\Pi_\perp .
\]

\begin{definition}[Router-compatible update maps]
\label{def:router_equivariance}
A map $\scrU\colon\RR^{e\times d}\to\RR^{e\times d}$ is called \emph{router-compatible} if it is expert-permutation equivariant, shared-row-shift invariant, and horizontal, namely,
\[
    \scrU(PD+\One_e a^\top)=P\,\scrU(D),
    \qquad
    \One_e^\top \scrU(D)=0,
\]
for all $D\in\RR^{e\times d}$, all permutation matrices $P\in\Prob^e$, and all $a\in\RR^d$.
\end{definition}

The horizontality condition reflects the quotient geometry induced by softmax shift invariance. An update component in the shared-row direction $\One_e a^\top$ changes all expert logits by the same scalar and therefore does not affect the router probabilities. A geometry-compatible router update should therefore lie in the horizontal subspace
\[
    \{U\in\RR^{e\times d}:\One_e^\top U=0\}.
\]

We consider two basic router-compatible update families. The first is a left-spectral update in the centered expert geometry:
\[
    \scrU_{\sfL}^{\router}(D)=\Psi(D_cD_c^\top)D_c,
    \qquad
    D_c=\Pi_\perp D,
\]
where $\Psi\colon\bbS_+^e\to\RR^{e\times e}$ is permutation equivariant and preserves the centered expert subspace. For example, the damped centered left-polar update corresponds to
\[
    \Psi(X)=(X+\varepsilon \Pi_\perp)^{-\half}
\]
on $\One_e^\perp$, or equivalently to applying the inverse square root on the nonzero centered expert subspace. Since left multiplication by a matrix acting on the centered expert space preserves $\One_e^\perp$, this update is horizontal:
\[
    \One_e^\top \scrU_{\sfL}^{\router}(D)=0 .
\]

The second family is a centered row-norm update. Importantly, for row-norm maps it is not sufficient to center the input direction once. Although $D_c=\Pi_\perp D$ satisfies $\One_e^\top D_c=0$, the rowwise-rescaled matrix
\[
    \Diag(\eta(\|D_{c,1:}\|_2),\dots,\eta(\|D_{c,e:}\|_2))D_c
\]
need not remain centered, because the scaling factors vary across expert rows. Therefore the router-compatible row-norm update is
\[
    \scrU_{\row}^{\router}(D)
    =
    \Pi_\perp
    \Diag(\eta(\|D_{c,1:}\|_2),\dots,\eta(\|D_{c,e:}\|_2))D_c,
    \qquad
    D_c=\Pi_\perp D,
\]
where $\eta\colon\Rp\to\RR$ is applied pointwise to the centered expert-row norms. The final projection is essential: it ensures that the actual parameter update, not only the input to the row-norm map, lies in the horizontal router subspace.

\begin{proposition}[Router-compatible update families]
\label{prop:router_compatible_updates}
Assume that $\Psi$ is permutation equivariant and maps the centered expert subspace to itself. Then the centered left-spectral update
\[
    \scrU_{\sfL}^{\router}(D)=\Psi(D_cD_c^\top)D_c,
    \qquad D_c=\Pi_\perp D,
\]
is router-compatible. Moreover, for any scalar function $\eta\colon\Rp\to\RR$, the projected centered row-norm update
\[
    \scrU_{\row}^{\router}(D)
    =
    \Pi_\perp
    \Diag(\eta(\|D_{c,1:}\|_2),\dots,\eta(\|D_{c,e:}\|_2))D_c,
    \qquad D_c=\Pi_\perp D,
\]
is router-compatible. In particular, for all $D\in\RR^{e\times d}$, $P\in\Prob^e$, and $a\in\RR^d$,
\[
    \scrU_{\sfL}^{\router}(PD+\One_e a^\top)=P\,\scrU_{\sfL}^{\router}(D),
    \qquad
    \scrU_{\row}^{\router}(PD+\One_e a^\top)=P\,\scrU_{\row}^{\router}(D),
\]
and both updates satisfy
\[
    \One_e^\top \scrU_{\sfL}^{\router}(D)=0,
    \qquad
    \One_e^\top \scrU_{\row}^{\router}(D)=0 .
\]
\end{proposition}

The converse is false in general: router compatibility does not force an update map to be left-spectral. Projected centered row-norm updates are already router-compatible, but they depend on individual centered expert-row norms rather than only on the centered Gram matrix $D_cD_c^\top$. Thus, left-spectral and row-norm updates should be viewed as two natural subclasses of router-compatible maps.

Hybrid router maps are obtained by composing router-compatible maps and projecting the final update back to the horizontal subspace. If $\scrU_1$ and $\scrU_2$ are router-compatible, then the horizontally projected composition
\[
    \Pi_\perp(\scrU_2\circ\scrU_1)
\]
is also router-compatible. Therefore, finite compositions of centered left-spectral and projected centered row-norm updates remain router-compatible. A representative row-after-left-spectral hybrid router update is
\[
    Z=\Psi(D_cD_c^\top)D_c,\qquad D_c=\Pi_\perp D,
\]
\[
    \scrU_{\hyb}^{\router}(D)
    =
    \Pi_\perp
    \Diag(\eta(\|Z_{1:}\|_2),\dots,\eta(\|Z_{e:}\|_2))Z .
\]
Such hybrid router optimizers combine the global expert-mixing geometry of left-spectral updates with the local expert-wise normalization of row-norm updates, while preserving expert-permutation equivariance, shared-row-shift invariance, and horizontality of the actual update.

\begin{definition}[Router-compatible optimizers]
A matrix optimizer for an \MoE router is called \emph{router-compatible} if its update rule has the form
\[
    (\forall k\in\NN)\qquad W_{k+1}=W_k-\gamma_k\scrU(D_k),
\]
where $\scrU$ is router-compatible and $D_k\in\RR^{e\times d}$ is an update direction that transforms under expert permutations and shared row shifts in the same way as the gradient.
\end{definition}

\begin{remark}[Why the final projection matters]
The final projection in the row-norm and hybrid router updates is not a numerical detail. Centering the direction $D_c=\Pi_\perp D$ removes the shared-row component before computing row norms, but row-dependent rescaling generally reintroduces a shared-row component. The final projection $\Pi_\perp$ removes this component and ensures that the update is horizontal. This is the router analogue of the horizontal update condition for untied LM heads under softmax shared-logit-shift invariance.
\end{remark}

\begin{remark}
Left- and right-spectral optimizers bear some formal resemblance to one-sided \Shampoo \citep{xie2025structured,li2026convergence} and ASGO \citep{an2025asgo}. A key difference is that our framework applies the spectral transformation directly to the current update direction, or to a symmetry-compatible momentum direction, rather than maintaining moving averages of Gram matrices. More importantly, our design principle identifies which layer types such one-sided updates are appropriate for, according to the symmetry group of the corresponding parameter block.
\end{remark}

\begin{remark}[Connection to non-Euclidean norm-based steepest descent]
This symmetry-based construction differs from non-Euclidean steepest descent approaches based on choosing a matrix norm \citep{bernstein2024old,veprikov2025preconditioned,xu2026width}. Row-wise or column-wise normalizations generally preserve only one-sided or permutation symmetries, rather than full bi-orthogonal symmetry. Thus, while such geometries may be useful in layer-specific settings, they should be distinguished from fully spectral constructions for ordinary matrix layers.
\end{remark}

\begin{table*}[t]
\centering
\caption{
    \textbf{Optimizer classes across LLM layers induced by symmetry.}
    Each matrix parameter has a natural symmetry group, which determines the corresponding symmetry-compatible optimizer class. For LM heads and \MoE routers, softmax shared-logit-shift invariance further induces quotient geometries; projected updates remove the shared-row direction so that the actual update acts on the non-redundant degrees of freedom.
}
\label{tab:optimizer-classes}
\small
\begin{tabular}{lll}
    \toprule
    \textbf{Layer} &
    \textbf{Symmetry group} &
    \textbf{Optimizer classes} \\
    \midrule
    Linear / MLP / attention weights
        & $\OO^{d_\mathrm{out}}\times \OO^{d_\mathrm{in}}$
        & full spectral \\
    Embedding
        & $\Prob^v \times \OO^d$
        & LPRO: row-norm / right-spectral / hybrid \\
    \multirow{2}{*}{LM head}
        & $\Prob^v \times \OO^d$
        & LPRO: row-norm / right-spectral / hybrid \\
        & $\RR^{v\times d}/(\One_v\RR^{1\times d})$
        & projected updates \\
    SwiGLU MLP $(W_{\mathrm{gate}},W_{\mathrm{up}},W_{\mathrm{down}}^\top)$
        & $\Prob^{d_{\mathrm{ff}}}\times \OO^d$
        & LPRO: row-norm / right-spectral / hybrid \\
    \multirow{2}{*}{\MoE router}
        & $\Prob^e \times (\One_e\RR^{1\times d})$
        & centered row-norm / left-spectral / hybrid \\
        & $\RR^{e\times d}/(\One_e\RR^{1\times d})$
        & projected updates \\
    \bottomrule
\end{tabular}
\end{table*}

\subsection{Symmetry-to-Optimizer Principle and Architecture--Optimizer Co-Design}
\label{subsec:symmetry_to_optimizer}
The preceding developments suggest a unifying principle for optimizer design in modern deep learning architectures: the optimizer geometry should be determined by the symmetry structure of the underlying parameterization. This does not require the full layerwise loss to be globally invariant under every group action considered. Rather, it is enough that the gradient, or more generally the chosen update direction, transforms equivariantly under the relevant representation of the parameter block. The optimizer update should then transform in the same way, and should lie in the corresponding quotient or horizontal subspace whenever the parameterization contains symmetry-redundant directions.

Accordingly, whenever the update direction associated with a layer parameter $W$ transforms under a symmetry group $\calG$, a natural optimizer should use an update map $\scrU$ that is equivariant under the same induced action. For matrix-valued parameters, this principle leads to several canonical optimizer geometries: full spectral updates under bi-orthogonal symmetry; right-spectral updates under right-orthogonal symmetry; left-spectral updates under left-orthogonal symmetry; row-norm and hybrid updates under left-permutation/right-orthogonal symmetry; and projected centered left-spectral or row-norm updates under the expert-permutation and shared-row-shift quotient geometry of \MoE routers. For LM heads, the same softmax shared-logit-shift invariance induces a horizontal quotient constraint on the vocabulary rows. Thus, the symmetry structure of the layer determines not only the appropriate optimizer class, but also whether the final update should be projected to remove redundant directions.

This principle also gives a practical recipe for architecture--optimizer co-design. Given a new architecture block, one should:
\begin{enumerate}
    \item identify the symmetry group of the parameterization;
    \item determine whether the symmetry acts on the left, on the right, on both sides, or only after quotienting out symmetry-redundant directions;
    \item choose the matching symmetry-compatible optimizer class;
    \item project the update onto the horizontal subspace when the parameterization has redundant shift directions; and
    \item use the smallest Gram matrix or invariant statistic compatible with that symmetry.
\end{enumerate}
In \Cref{subsec:practical}, we instantiate this recipe by extending momentum \PolarGrad \citep{lau2025polargrad} to one-sided and hybrid variants, including \LeftPolarGradM, \RightPolarGradM, \RowNormM, and \HybridPolarGradM.

\begin{remark}[Projected and proximal extensions]
The symmetry-compatible optimizer classes above extend naturally to projected and proximal updates. If $\calT(D_k)$ is $\calG$-equivariant and the constraint set $\calC$ is $\calG$-invariant, then the projected update
\[
    W_{k+1}=\proj_{\calC}\bigl(W_k-\gamma_k\calT(D_k)\bigr)
\]
is symmetry-compatible, since the Euclidean projection onto a $\calG$-invariant set is $\calG$-equivariant whenever it is uniquely defined. Similarly, if $h$ is a $\calG$-invariant regularizer, then
\[
    W_{k+1}=\prox_{\gamma_k h}\bigl(W_k-\gamma_k\calT(D_k)\bigr)
\]
preserves the same symmetry whenever the proximal map is uniquely defined. In the bi-orthogonal setting, this includes unitarily invariant constraints and regularizers, such as spectral-norm, nuclear-norm, rank, and Schatten-$p$ constraints or penalties. We leave a systematic study of such variants, including \ProxPolarGrad, to future work.
\end{remark}

\subsection{Practical Optimizers for Embeddings, LM Heads, SwiGLU MLP Projections, and \MoE Routers}
\label{subsec:practical}
The preceding subsections introduced three main classes of symmetry-compatible optimizers: one-sided spectral optimizers, row-norm-based optimizers, and hybrid variants obtained by composing row-wise normalization with one-sided spectral updates. We now describe their practical momentum implementations. The main computational issue is the efficient and stable approximation of matrix inverse square roots, or equivalently orthogonal polar factors, which we compute using GPU-friendly numerical linear algebra routines. In addition, for untied LM heads and \MoE routers, the practical updates must respect the horizontal quotient geometry induced by softmax shared-logit-shift invariance. Thus, when row-wise nonlinear operations are used for these blocks, we project the final update back to the corresponding centered subspace.

\subsubsection{One-Sided Spectral Optimizers}

For embedding, LM head, and SwiGLU MLP projection matrices, the relevant matrices are often tall-skinny. In this regime, right-spectral updates are attractive because they only require the inverse square root of the smaller right Gram matrix
\[
C_k \coloneqq G_k^\top G_k
\qquad\text{or, with momentum,}\qquad
C_k \coloneqq M_k^\top M_k .
\]
For an untied LM head, we use the centered momentum direction
\[
    \Pi_v^\perp \coloneqq I_v-\frac{1}{v}\One_v\One_v^\top,
    \qquad
    M_{k,c}\coloneqq \Pi_v^\perp M_k,
\]
so that the update removes the shared vocabulary-logit-shift direction. The corresponding right Gram matrix is
\[
    C_k \coloneqq M_{k,c}^\top M_{k,c}.
\]
For \MoE routers, the corresponding left-spectral update is applied in the centered expert geometry. Writing
\[
    \Pi_\perp \coloneqq I_e-\frac{1}{e}\One_e\One_e^\top,
    \qquad
    M_{k,c}\coloneqq \Pi_\perp M_k,
\]
the relevant Gram matrix is
\[
    C_k \coloneqq M_{k,c}M_{k,c}^\top .
\]

To compute the inverse square roots in these updates, we use Newton--Schulz iterations with the polynomial coefficients of \textsc{Polar Express} \citep{amsel2025polar}. This is motivated by the connection between polynomial iterations for polar decomposition and inverse-square-root computation \citep{higham1997stable}. For numerical stability, the Gram inverse-square-root implementation is performed in \texttt{float32}. Other fast inverse-square-root routines, such as PRISM \citep{yang2026prism}, could be used in the same role. For \RightPolarGradM, we also provide a Gram Newton--Schulz implementation \citep{zhang2026gram}, which directly approximates
\[
    M_k(M_k^\top M_k)^{-\half}
\]
and supports more inner iterations in lower-precision formats such as \texttt{bfloat16} or \texttt{float16}. The resulting one-sided momentum polar-gradient algorithms are summarized in \Cref{alg:one-sided_polargrad}.

\begin{algorithm}[h!]
	\caption{\LeftPolarGradM and \RightPolarGradM}
	\label{alg:one-sided_polargrad}
    	\begin{algorithmic}[1]
    		\REQUIRE $W_0\in\RR^{m\times n}$, $M_{-1}=0$, learning rates $\{\gamma_k\}_{k\ge 0}$, momentum $\beta\in[0,1)$, scaling exponent $\alpha\in[0,1]$, damping $\varepsilon>0$, weight decay $\lambda\ge0$
    		\FOR{$k=0, \ldots, K-1$}
    			\STATE $G_k = \nabla_W f(W_k)$
    			\STATE $M_k = \beta M_{k-1} + (1-\beta)G_k$
    			\IF{\LeftPolarGradM}
    				\STATE $C_k = M_k M_k^\top$
    				\STATE $L_k = (C_k + \varepsilon I)^{-\half}$ via \textsc{Polar Express} or Newton--Schulz iteration
    				\STATE $\nu_k = \max\{\tr(C_k L_k), \varepsilon\}$
    				\STATE $W_{k+1} = (1-\gamma_k\lambda)W_k - \gamma_k \nu_k^\alpha L_k M_k$
    			\ELSIF{\RightPolarGradM}
    				\STATE $C_k = M_k^\top M_k$
    				\STATE $R_k = (C_k + \varepsilon I)^{-\half}$ via \textsc{Polar Express} or Newton--Schulz iteration
    				\STATE $\nu_k = \max\{\tr(C_k R_k), \varepsilon\}$
    				\STATE $W_{k+1} = (1-\gamma_k\lambda)W_k - \gamma_k \nu_k^\alpha M_k R_k$
    			\ENDIF
    		\ENDFOR
    		\ENSURE $W_K$
	\end{algorithmic}
\end{algorithm}

At the level of exact polar decomposition, \LeftPolarGradM and \RightPolarGradM compute the same polar direction whenever both sides are well-defined. Their distinction is therefore computational and architectural: they differ in which Gram matrix is formed, which inverse square root is computed, and which layer symmetry they are intended to respect. For LM heads and \MoE routers, the one-sided spectral updates are applied to centered momentum directions when the corresponding softmax shared-logit-shift quotient geometry is present.

\subsubsection{Row-Norm-Based and Hybrid Variants}

For embedding, LM head, and SwiGLU MLP projection matrices, row-norm-based updates provide a cheaper LPRO-compatible alternative. Given a momentum direction $M_k$, define
\[
D_\eta(M_k)
\coloneqq
\Diag(\eta(\|M_{k,1:}\|_2),\dots,\eta(\|M_{k,m:}\|_2)).
\]
For embeddings and SwiGLU gate/up projections, a row-norm update takes the form
\[
    \calT_{\row}(M_k)=D_\eta(M_k)M_k .
\]
Here $\eta$ may be chosen as a bounded row-scaling rule or as a smoothed normalization rule such as $\eta(t)=1/(t+\varepsilon)$.

For untied LM heads, the update must also remove the shared vocabulary-logit-shift direction. Let
\[
    \Pi_v^\perp \coloneqq I_v-\frac1v\One_v\One_v^\top,
    \qquad
    M_{k,c}\coloneqq \Pi_v^\perp M_k .
\]
The horizontal LM head row-norm update is
\[
    \calT_{\row}^{\rm LM}(M_k)
    =
    \Pi_v^\perp D_\eta(M_{k,c})M_{k,c}.
\]
The final projection is essential for row-norm updates: even though $\One_v^\top M_{k,c}=0$, row-dependent rescaling by $D_\eta(M_{k,c})$ need not preserve the zero-row-mean constraint.

Hybrid variants combine row-wise scaling with a one-sided spectral step. For embeddings and SwiGLU gate/up projections, there are two natural orders:
\[
    \text{right-spectral/row-norm:}
    \qquad
    Z_k=M_k(M_k^\top M_k+\varepsilon I)^{-\half},
    \qquad
    \calT_{\hyb}(M_k)=D_\eta(Z_k)Z_k,
\]
and
\[
    \text{row-norm/right-spectral:}
    \qquad
    Z_k=D_\eta(M_k)M_k,
    \qquad
    \calT_{\hyb}(M_k)=Z_k(Z_k^\top Z_k+\varepsilon I)^{-\half}.
\]
Both remain left-permutation and right-orthogonal equivariant. For LM heads, the same constructions are applied to the centered momentum $M_{k,c}=\Pi_v^\perp M_k$, with horizontal projections inserted after rowwise nonlinear stages and on the final update. For example, the row-norm/right-spectral order becomes
\[
    Z_k=\Pi_v^\perp D_\eta(M_{k,c})M_{k,c},
    \qquad
    \calT_{\hyb}^{\rm LM}(M_k)
    =
    \Pi_v^\perp Z_k(Z_k^\top Z_k+\varepsilon I)^{-\half}.
\]

For \MoE routers, the same principle applies in the centered expert geometry. Writing
\[
    \Pi_\perp \coloneqq I_e-\frac1e\One_e\One_e^\top,
    \qquad
    M_{k,c}=\Pi_\perp M_k,
\]
the router row-norm update is
\[
    \calT_{\row}^{\router}(M_k)
    =
    \Pi_\perp D_\eta(M_{k,c})M_{k,c}.
\]
Hybrid router updates combine a centered left-spectral step with projected row-wise normalization across experts. We refer to these practical hybrid variants collectively as \HybridPolarGradM. The algorithms for \RowNormM and \HybridPolarGradM are summarized in \Cref{alg:row_hybrid_practical}.

\begin{algorithm}[h!]
\caption{\RowNormM and \HybridPolarGradM with Optional Horizontal Projection}
\label{alg:row_hybrid_practical}
\begin{algorithmic}[1]
    \REQUIRE $W_0$, $M_{-1}=0$, learning rates $\{\gamma_k\}_{k\ge0}$, momentum $\beta\in[0,1)$, scaling exponent $\alpha\in[0,1]$, row-scaling rule $\eta$, damping $\varepsilon>0$, weight decay $\lambda\ge0$
    \FOR{$k=0,\ldots,K-1$}
        \STATE $G_k=\nabla_W f(W_k)$
        \STATE $M_k=\beta M_{k-1}+(1-\beta)G_k$

        \STATE Set the working direction $\bar M_k$ and final projector $\Pi$ by layer type:
        \[
        (\bar M_k,\Pi)=
        \begin{cases}
        (M_k,I), & \text{embedding or SwiGLU gate/up},\\
        (\Pi_v^\perp M_k,\Pi_v^\perp), & \text{LM head},\\
        (\Pi_e^\perp M_k,\Pi_e^\perp), & \MoE\text{ router},
        \end{cases}
        \]
        where
        \[
            \Pi_v^\perp=I_v-\frac1v\One_v\One_v^\top,
            \qquad
            \Pi_e^\perp=I_e-\frac1e\One_e\One_e^\top .
        \]

        \IF{\RowNormM}
            \STATE $D_k=\Diag(\eta(\|\bar M_{k,1:}\|_2),\dots,\eta(\|\bar M_{k,m:}\|_2))$
            \STATE $U_k=\Pi D_k\bar M_k$

        \ELSIF{\HybridPolarGradM, right-spectral/row-norm order}
            \STATE $C_k=\bar M_k^\top\bar M_k$
            \STATE $R_k=(C_k+\varepsilon I)^{-\half}$ via \textsc{Polar Express} or Newton--Schulz iteration
            \STATE $\nu_k=\max\{\tr(C_kR_k),\varepsilon\}$
            \STATE $Z_k=\Pi \nu_k^\alpha \bar M_kR_k$
            \STATE $D_k=\Diag(\eta(\|Z_{k,1:}\|_2),\dots,\eta(\|Z_{k,m:}\|_2))$
            \STATE $U_k=\Pi D_kZ_k$

        \ELSIF{\HybridPolarGradM, row-norm/right-spectral order}
            \STATE $D_k=\Diag(\eta(\|\bar M_{k,1:}\|_2),\dots,\eta(\|\bar M_{k,m:}\|_2))$
            \STATE $Z_k=\Pi D_k\bar M_k$
            \STATE $C_k=Z_k^\top Z_k$
            \STATE $R_k=(C_k+\varepsilon I)^{-\half}$ via \textsc{Polar Express} or Newton--Schulz iteration
            \STATE $\nu_k=\max\{\tr(C_kR_k),\varepsilon\}$
            \STATE $U_k=\Pi\nu_k^\alpha Z_kR_k$

        \ELSIF{\MoE router and \HybridPolarGradM, left-spectral/row-norm order}
            \STATE $C_k=\bar M_k\bar M_k^\top$
            \STATE $L_k=(C_k+\varepsilon I)^{-\half}$ via \textsc{Polar Express} or Newton--Schulz iteration
            \STATE $\nu_k=\max\{\tr(C_kL_k),\varepsilon\}$
            \STATE $Z_k=\Pi\nu_k^\alpha L_k\bar M_k$
            \STATE $D_k=\Diag(\eta(\|Z_{k,1:}\|_2),\dots,\eta(\|Z_{k,e:}\|_2))$
            \STATE $U_k=\Pi D_kZ_k$
        \ENDIF

        \STATE $W_{k+1}=(1-\gamma_k\lambda)W_k-\gamma_k U_k$
    \ENDFOR
    \ENSURE $W_K$
\end{algorithmic}
\end{algorithm}

For down projections in SwiGLU MLPs, we apply the same procedure to $W_{\mathrm{down}}^\top$ and transpose the resulting update back. For routers, the hybrid update uses the left-spectral branch; for embeddings, LM heads, and SwiGLU gate/up projections, it uses the right-spectral branches.

The practical distinction among these variants is both geometric and computational. \RightPolarGradM preserves the right-orthogonal geometry through a Gram inverse square root, while \RowNormM is a purely row-adaptive LPRO-compatible alternative. \HybridPolarGradM interpolates between the two by composing row normalization and one-sided spectral normalization. For LM heads and \MoE routers, softmax shared-logit-shift invariance imposes an additional quotient geometry; therefore, row-norm and hybrid updates are projected back to the horizontal subspace so that the actual update removes the shared-row direction. For \MoE routers, the same alternatives appear in the centered expert geometry: one may use a projected centered row-norm update, a centered left-spectral update, or a projected hybrid left-spectral/row-norm update.

\begin{remark}[Projection and proximal extensions]
The practical optimizer families described above can be combined with projection or proximal steps. For example, given a regularizer $h$ or a feasible set $\calC$, one may consider
\[
    W_{k+1}
    =
    \prox_{\gamma_k h}\left(W_k-\gamma_k\calT(D_k)\right)
    \qquad\text{or}\qquad
    W_{k+1}
    =
    \proj_{\calC}\left(W_k-\gamma_k\calT(D_k)\right),
\]
where $\calT(D_k)$ is any symmetry-compatible update direction from the preceding constructions. If $h$ or $\calC$ is invariant under the same layer symmetry, then the resulting proximal or projected update preserves the same equivariance. Thus, the geometry-aware optimizer classes introduced here are compatible with standard regularization and constrained-optimization extensions.
\end{remark}

\subsection{Spectral Optimizers as the Bi-Orthogonally Equivariant Class}
\label{subsec:spectral_optim}
We have used bi-orthogonal equivariance as the symmetry principle for ordinary matrix layers. We now record the corresponding structural characterization: direction-wise update maps satisfying this equivariance are precisely spectral operators. This explains why SSD \citep{carlson2016stochastic}, \Muon \citep{jordan2024muon}, \Scion \citep{pethick2025training}, \PolarGrad \citep{lau2025polargrad}, and related spectral updates form the canonical optimizer class for ordinary matrix layers. The preceding sections show how different symmetry groups lead instead to different optimizer classes for embeddings, LM heads, SwiGLU MLP projections, and \MoE routers.

We first recall the relevant matrix-analytic notions. A proper function $\varphi\colon\RR^{m\times n}\to\oRR$ is called \emph{spectral} if there exists a proper function $\psi\colon\RR^r\to\oRR$, where $r=\min\{m,n\}$, such that $\varphi=\psi\circ\sigma$. Its matrix-valued analogue is a \emph{spectral operator}: a map $\scrU\colon\RR^{m\times n}\to\RR^{m\times n}$ is spectral if there exists an absolutely symmetric map $\psi\colon\RR^r\to\RR^r$ such that, for every singular value decomposition $D=U\Diag(\sigma(D))V^\top$,
\[
    \scrU(D)=U\Diag(\psi(\sigma(D)))V^\top.
\]
Thus, spectral operators preserve the singular-vector geometry of their input and act only on singular values.

We next recall the standard gradient formula for spectral functions; see, e.g., \citet{lewis1995convex}. 

\begin{theorem}[Gradient formula for spectral functions]
	\label{thm:grad_spectral}
	Let $f\colon \RR^r\to\oRR$ be convex and absolutely symmetric. Then the corresponding spectral function $f\circ\sigma$ is differentiable at $W\in\RR^{m\times n}$ if and only if $f$ is differentiable at $\sigma(W)$. In this case, if $W=U\Diag(\sigma(W))V^\top$ is a singular value decomposition of $W$, then
    \[
        \nabla (f\circ\sigma)(W)=U\Diag(\nabla f(\sigma(W)))V^\top.
    \]
\end{theorem}

This formula shows that spectral scalar functions act through singular values while preserving singular directions. The same structure appears when one requires an update map to commute with arbitrary left and right orthogonal changes of coordinates. We now state the corresponding characterization.

\begin{theorem}[Characterization of bi-orthogonally equivariant update maps]
	\label{thm:orth_invar}
    A continuous matrix-valued map $\scrU\colon\RR^{m\times n}\to\RR^{m\times n}$ satisfies
    \[
        \scrU(PDQ^\top)=P\,\scrU(D)Q^\top
    \]
    for all $D\in\RR^{m\times n}$, $P\in\OO^m$, and $Q\in\OO^n$ if and only if it is a spectral operator. Equivalently, $\scrU$ is bi-orthogonally equivariant if and only if, for every singular value decomposition $D=U\Diag(\sigma(D))V^\top$, there exists an absolutely symmetric map $\psi\colon\RR^r\to\RR^r$ such that
    \[
        \scrU(D)=U\Diag(\psi(\sigma(D)))V^\top,
        \qquad r=\min\{m,n\}.
    \]
\end{theorem}

We denote the set of bi-orthogonally equivariant matrix maps from $\RR^{m\times n}$ to $\RR^{m\times n}$ by $\mathcal U_{\OO}^{m\times n}$. \Cref{thm:orth_invar} shows that bi-orthogonal equivariance is not merely a desirable property: it is a complete structural characterization of direction-wise matrix update maps. Any such update must act through the singular values of the update direction and preserve its singular vectors.

\begin{remark}
\Cref{thm:orth_invar} concerns equivariance of the optimizer update map with respect to transformations of the matrix gradient or update direction. It does not assume that the layerwise loss itself satisfies $f(PWQ^\top)=f(W)$ for every individual layer parameter $W$. The claim is instead that, whenever a matrix gradient or momentum direction is represented in a rotated basis, a symmetry-compatible optimizer should transform its update in the same way.
\end{remark}

Accordingly, we call a matrix optimizer \emph{spectral} if its update rule is bi-orthogonally equivariant. Concretely, its iterates satisfy
\[
    W_{k+1}=W_k-\gamma_k\scrU(D_k),
\]
where $\gamma_k>0$, $\scrU\in\calU_{\OO}^{m\times n}$, and $D_k$ is an update direction that itself transforms bi-orthogonally, such as the gradient or a symmetry-compatible momentum direction. By \Cref{thm:orth_invar}, every such direction-wise spectral optimizer is determined by a singular-value transformation $\psi$.

This characterization applies to maps acting on a single update direction. It does not exclude broader stateful matrix optimizers such as \Shampoo, whose auxiliary states may themselves evolve equivariantly. Spectral optimizers should therefore be understood as the bi-orthogonally equivariant class of memoryless, or direction-wise, matrix update maps; stateful equivariant optimizers form a larger class.

\subsubsection{Examples of Spectral and Equivariant Matrix Optimizers}
\label{subsubsec:spectral_examples}
The simplest spectral optimizer is vanilla gradient descent, for which the spectral operator is the identity map. Its Polyak, Nesterov, and EMA-momentum variants remain spectral because bi-orthogonally equivariant momentum constructions composed with spectral operators preserve spectrality. Other examples include stochastic spectral descent \citep{carlson2015stochasticRBM,carlson2016stochastic}, \Muon \citep{jordan2024muon}, \Scion \citep{pethick2025training}, and \PolarGrad \citep{lau2025polargrad}. Power-type spectral maps $\psi(t)=t^p$, including $p\in\{1/2,1/4\}$, are also studied in \citep{qi2026delving}.

It is useful to distinguish spectral operators acting on the current update direction from history-dependent optimizers whose states evolve equivariantly. For example, \Shampoo \citep{gupta2018shampoo,anil2020scalable} maintains left and right preconditioners based on moving averages of $G_kG_k^\top$ and $G_k^\top G_k$, and applies
\[
    L_k^{-\sfrac14}G_kR_k^{-\sfrac14}.
\]
This update is bi-orthogonally equivariant when the state variables $(L_k,R_k)$ transform accordingly, but it is not generally a spectral operator of the current gradient alone, since the preconditioners need not share the singular-vector basis of $G_k$. In special aligned cases, it reduces to a singular-value transformation and hence becomes spectral. One-sided \Shampoo \citep{xie2025structured,li2026convergence} and ASGO \citep{an2025asgo} admit a similar interpretation.

SOAP \citep{vyas2024soap} is also geometry-aware but generally not spectral in the strict direction-wise sense. It rotates gradients into learned eigenspaces and applies coordinate-wise adaptive scaling in those evolving coordinates. Thus, the update depends on history-dependent bases and coordinate-wise statistics, not only on the singular values of the current gradient. It becomes spectral only in special cases where the learned eigenspaces align with the singular-vector basis and the adaptive scaling acts symmetrically across singular directions.

Normalization preserves spectrality when the normalizing scalar is unitarily invariant. If $\scrU$ is spectral and $\alpha\colon\RR^{m\times n}\to\Rp$ is a positive spectral scalar function, then $D\mapsto \scrU(D)/\alpha(D)$ is again spectral. In particular, normalization by a unitarily invariant matrix norm preserves spectrality. By contrast, row-wise or column-wise normalization generally breaks full bi-orthogonal equivariance because it depends on a preferred coordinate system. Such normalizations may still be appropriate for layers with one-sided or permutation symmetries, but they are not spectral operators for ordinary matrix layers.

More broadly, steepest descent for matrix-valued parameters is compatible with bi-orthogonal symmetry when the underlying norm is unitarily invariant. The unit ball of a unitarily invariant norm is invariant under $D\mapsto PDQ^\top$, so the associated steepest descent direction transforms equivariantly under left and right orthogonal changes of coordinates. Non-unitarily invariant norms, such as general induced $\ell_p\to\ell_q$ operator norms with $(p,q)\ne(2,2)$, generally fail this property and impose a coordinate-dependent geometry on the update.

\section{Numerical Experiments}
\label{sec:expt}
Our experiments test the layerwise equivariance principle in full language model pre-training. Rather than changing a single optimizer in isolation, we instantiate a symmetry-compatible optimizer stack in which each major matrix-valued parameter class is assigned an update whose equivariance matches its layerwise symmetry: attention matrices use spectral or head-wise spectral updates; embeddings and LM heads use row-norm or hybrid updates; dense and expert SwiGLU MLP projections use right-spectral, row-aware, column-aware, or hybrid updates along the appropriate intermediate-neuron axis; \MoE routers use centered row-aware or left-spectral updates; and scalar or vector parameters use standard coordinate-wise optimizers.

We validate the proposed equivariant optimizer classes on four open-weight dense and \MoE language model architectures spanning different vocabulary sizes, hidden sizes, embedding and LM head matrix dimensions, and numbers of trainable parameters. Our goal is to test the practical implications of the symmetry- and geometry-based design principles developed above across architectures with distinct matrix-parameter geometries. Because our focus is optimizer behavior rather than scaling-law-optimal pre-training, we do not train the models for the large token budgets typically prescribed by scaling laws. We pre-train all models on a 10B-token subset of FineWeb-Edu \citep{lozhkov2024fineweb-edu} with context length 1024. These settings allow us to examine optimizer behavior in controlled yet nontrivial pre-training regimes.

Unless otherwise specified, we use \Muon \citep{jordan2024muon} with \textsc{Polar Express} coefficients \citep{amsel2025polar} for hidden and attention matrices, and \AdamW \citep{loshchilov2019decoupled} for scalar and vector parameters. For consistency with the intended layerwise geometry, fused attention weights are handled by applying \textsc{Polar Express} to the momentum of each attention head or fused component separately, in a similar spirit to \textsc{Muon Split} \citep{glm5team2026glm5}. Likewise, for \OLMoE-1B-7B and downsized gpt-oss, we treat fused expert SwiGLU projection tensors according to their intermediate-neuron geometry. The fused expert gate/up tensors are reshaped or interpreted so that gate/up channels associated with intermediate neurons receive row-aware or hybrid updates along the correct axis, while expert down projections are updated along the corresponding intermediate-neuron axis. This avoids applying a geometry-aware optimizer to an artifact of tensor storage rather than to the functional symmetry of the layer. We use untied input embeddings and output LM head weights, which allows us to assign different optimizers to these two large vocabulary-indexed matrices. Across all four experiments, we set the weight decay coefficient to $\lambda=0$ for input embedding and LM head weights, as well as \MoE router weights where present. Full experimental details are given in \Cref{sec:details_expt}. Code is available at \url{https://github.com/timlautk/equivariant_optimizers}.

\subsection{Qwen3-0.6B-Style Pre-Training}
\label{subsec:qwen3}
We first pre-train a Qwen3-0.6B-style dense language model \citep{qwen3technicalreport} incorporating several recent architectural innovations, including Grouped Query Attention (GQA) \citep{ainslie2023gqa}, the SwiGLU activation function \citep{dauphin2017language,shazeer2020glu}, Rotary Positional Embeddings (RoPE) \citep{su2024roformer}, pre-norm normalization \citep{xiong2020layer} with \RMSNorm \citep{zhang2019root,jiang2023pre}, and QK-Norm \citep{dehghani2023scaling}, without QKV bias terms. The model uses a vocabulary size of 151,936 and hidden dimension 1024. Since untying the embeddings increases the number of trainable parameters, we reduce the number of hidden layers from 28 to 20, resulting in a total of 625,784,832 trainable parameters.

We compare three optimizer assignments for the vocabulary-indexed matrices, namely the input embedding and LM head matrices: (i) \RowNormM, (ii) \HybridPolarGradM with row-norm/right-spectral, (iii) \AdamW. For SwiGLU MLP projections, we compare (a) right-spectral updates, equivalently \Muon-style updates, with (b) hybrid row-norm/right-spectral updates applied along the appropriate intermediate-neuron axis: row-aware for gate and up projections, and column-aware for down projections.

Although different LPRO-compatible optimizers could in principle be assigned to the embedding and LM head matrices, we use the same optimizer for both layers in each configuration for simplicity. 

\begin{figure}[h!]
    \centering
    \begin{subfigure}[t]{\textwidth}
        \centering
        \includegraphics[width=\textwidth]{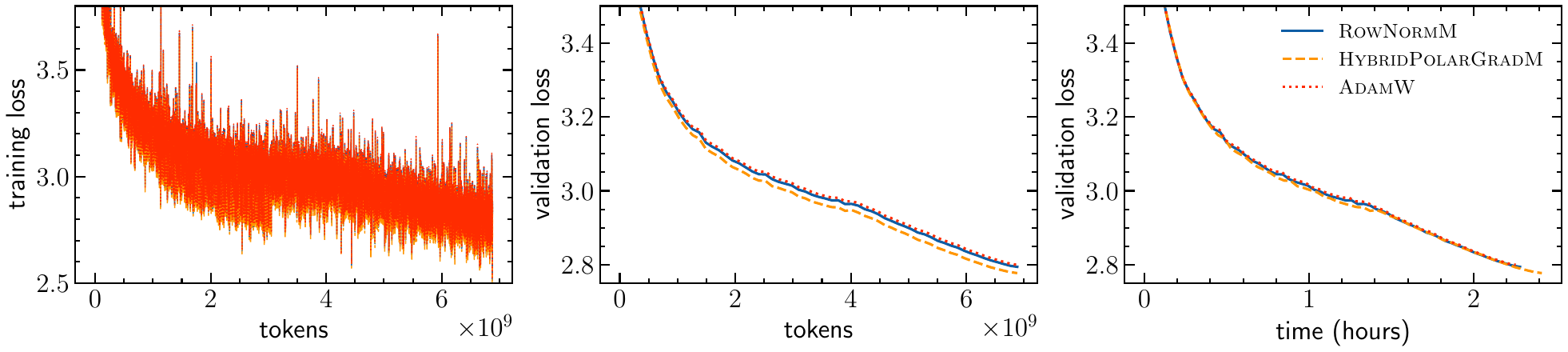}
        \caption{SwiGLU MLP projection matrices use \Muon, equivalently \RightPolarGradM with $\alpha=0$.}
        \label{fig:qwen3_muon_mlp}
    \end{subfigure}%
    \vspace*{2.5mm}
    \begin{subfigure}[t]{\textwidth}
        \centering
        \includegraphics[width=\textwidth]{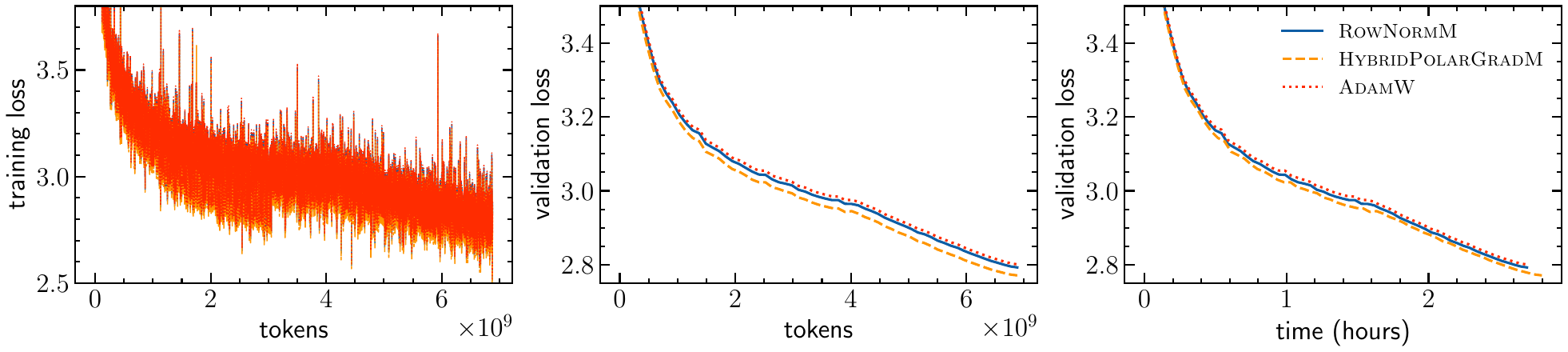}
        \caption{SwiGLU MLP projection matrices use \HybridPolarGradM with a row-norm/right-spectral composition.}
        \label{fig:qwen3_hybrid_mlp}
    \end{subfigure}
    \caption{Training and validation losses for Qwen3-0.6B-style pre-training. In each subfigure, the three configurations differ only in the optimizer assigned to the input embedding and LM head matrices: \RowNormM, \HybridPolarGradM, or \AdamW. }
    \label{fig:qwen3}
\end{figure}

The final validation losses for configurations (i)--(iii) are 2.7943, 2.7774, and 2.7996 in \Cref{fig:qwen3_muon_mlp}, and 2.7927, 2.7710, and 2.8009 in \Cref{fig:qwen3_hybrid_mlp}, respectively. As shown in \Cref{fig:qwen3}, in both settings, configuration (iii) makes comparable initial progress to configuration (i) and has lower validation loss at the earlier stage of training, but is subsequently overtaken by both \RowNormM and \HybridPolarGradM. The final gap between \HybridPolarGradM and \AdamW is larger than that between \RowNormM and \AdamW, and the validation loss favors both symmetry-compatible updates.

Across \Cref{fig:qwen3_muon_mlp} and \Cref{fig:qwen3_hybrid_mlp}, using \HybridPolarGradM for the SwiGLU MLP projection matrices reduces the final validation loss relative to \Muon for the \RowNormM and \HybridPolarGradM embedding/LM head configurations, from 2.7943 to 2.7927 and from 2.7774 to 2.7710, respectively. For the \AdamW embedding/LM head configuration, the final validation loss increases slightly from 2.7996 to 2.8009. The improvement is largest when \HybridPolarGradM is also used for the input embedding and LM head matrices, suggesting a possible complementary effect between symmetry-compatible updates on vocabulary-indexed matrices and row-aware/right-spectral updates on SwiGLU MLP projections. \HybridPolarGradM remains the best-performing assignment for the input embedding and LM head matrices in both settings. Thus, the comparison between (a) and (b) suggests that applying row-norm/right-spectral hybrid updates to SwiGLU MLP projections can further improve validation loss, while the relative ranking among the three vocabulary-indexed optimizer assignments remains stable. This improvement is plausibly due to the tall-skinny geometry of the SwiGLU gate and up projections. Since $d_{\mathrm{model}}=1024$ and $d_{\mathrm{ff}}=3072$, the gate and up projections have many more rows than columns, with rows corresponding to intermediate neurons. Row-scale imbalance can therefore be important. While \Muon captures the right-spectral geometry, \HybridPolarGradM additionally normalizes intermediate-neuron rows before the spectral step, which may yield a better update geometry in this regime.

In terms of wall-clock training time, \HybridPolarGradM incurs additional overhead due to the inner Gram Newton--Schulz iterations used to approximate the right-spectral component. Consequently, configurations (ii) and (iii) require a similar amount of training time to reach comparable validation loss values, despite \HybridPolarGradM achieving a slightly lower final loss. Configurations (i) and (ii) both follow symmetry-compatible geometries for the embedding and LM head matrices, whereas configuration (iii) applies coordinate-wise \AdamW updates to these matrix-valued parameters, thereby introducing a geometry mismatch. Overall, these results are consistent with our symmetry-aware matrix view of optimizer design: even when applied only to the vocabulary-indexed matrices, geometry-compatible updates can improve the optimization trajectory and final validation loss. 

\subsection{Gemma 3 1B-Style Pre-Training}
\label{subsec:gemma3}
We next pre-train a Gemma 3 1B-style dense language model \citep{gemmateam2025gemma3technicalreport}. Compared with the Qwen3-0.6B-style experiment, this model has both a larger vocabulary size of 262,144, and a larger hidden dimension of 1152. Consequently, the input embedding and LM head matrices are substantially larger, and their matrix gradients may be more anisotropic or ill-conditioned. This setting therefore provides a more stringent test of geometry-aware optimizers for vocabulary-indexed matrix parameters. For right-spectral updates, such as those used inside \RightPolarGradM and \HybridPolarGradM, the larger hidden dimension also makes the Gram matrix inverse-square-root computation more demanding, often requiring more accurate or additional Newton--Schulz iterations. To keep the total model size close to one billion trainable parameters after using untied embeddings, we reduce the number of hidden layers from 26 to 18, resulting in 1,087,138,944 trainable parameters. We use the same three optimizer assignments for the input embedding and LM head matrices as in the Qwen3-0.6B-style experiment, and compare two assignments for the SwiGLU MLP projection matrices: (a) \Muon and (b) \HybridPolarGradM.

\begin{figure}[h!]
    \centering
    \begin{subfigure}[t]{\textwidth}
        \centering
        \includegraphics[width=\textwidth]{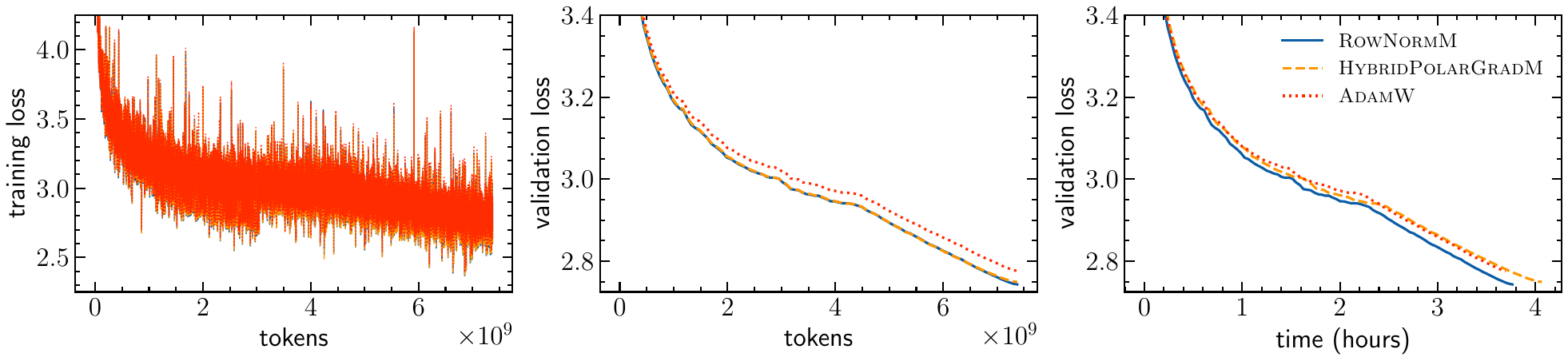}
        \caption{SwiGLU MLP projection matrices use \Muon, equivalently \RightPolarGradM with $\alpha=0$.}
        \label{fig:gemma3_muon_mlp}
    \end{subfigure}%
    \vspace*{2.5mm}
    \begin{subfigure}[t]{\textwidth}
        \centering
        \includegraphics[width=\textwidth]{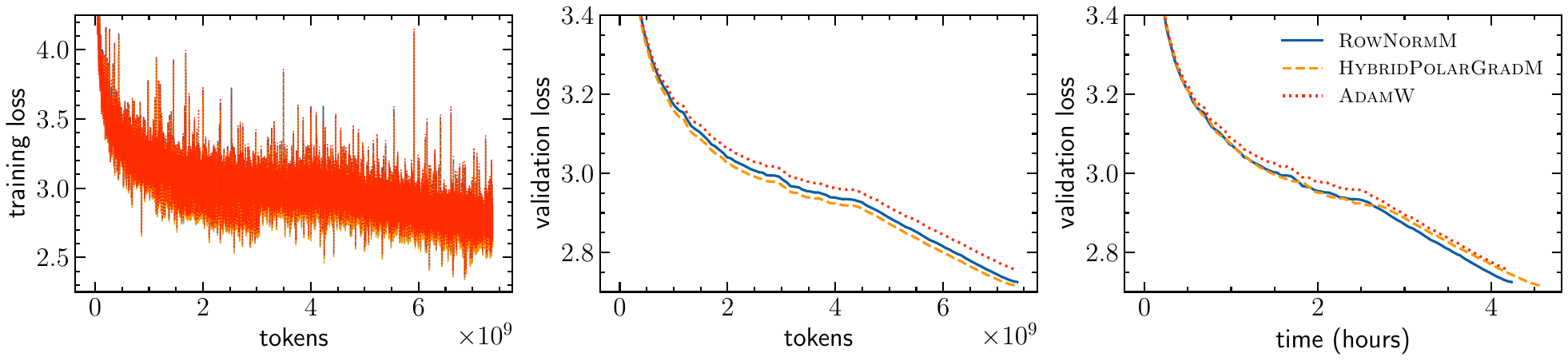}
        \caption{SwiGLU MLP projection matrices use \HybridPolarGradM with a row-norm/right-spectral composition.}
        \label{fig:gemma3_hybrid_mlp}
    \end{subfigure}
    \caption{Training and validation losses for Gemma 3 1B-style pre-training. In each subfigure, the three configurations differ only in the optimizer assigned to the input embedding and LM head matrices: \RowNormM, \HybridPolarGradM, or \AdamW. }
    \label{fig:gemma3}
\end{figure}

For configurations (i)--(iii), the final validation losses are 2.7434, 2.7492, and 2.7761 in \Cref{fig:gemma3_muon_mlp}, compared with 2.7251, 2.7160, and 2.7562 in \Cref{fig:gemma3_hybrid_mlp}, respectively. The best embedding/LM head assignment is \RowNormM when the MLP projections use \Muon, and \HybridPolarGradM when they use hybrid updates. Both assignments yield substantially lower final validation losses than \AdamW under either MLP optimizer. These findings are consistent with the hypothesis that geometry-compatible updates are particularly beneficial for large vocabulary-indexed matrices whose gradients may be anisotropic or poorly conditioned.

The comparison between \Cref{fig:gemma3_muon_mlp} and \Cref{fig:gemma3_hybrid_mlp} also shows that replacing \Muon with \HybridPolarGradM for the SwiGLU MLP projections benefits all three embedding/LM head configurations. The final validation losses fall from 2.7434 to 2.7251 with \RowNormM, from 2.7492 to 2.7160 with \HybridPolarGradM, and from 2.7761 to 2.7562 with \AdamW. A possible explanation lies in the pronounced aspect ratio of the gate and up projection matrices. In Gemma 3 1B-style models, $d_{\mathrm{model}}=1152 \ll d_{\mathrm{ff}}=6912$, giving these matrices six times as many rows as columns, with each row associated with an intermediate neuron. In addition to the right-spectral transformation used by \Muon, \HybridPolarGradM applies row normalization before the spectral step, potentially mitigating the effects of unequal gradient scales across intermediate neurons.

At the same time, \HybridPolarGradM incurs higher computational overhead than \RowNormM because it requires approximating matrix inverse square roots through inner Gram Newton--Schulz iterations. In contrast, \RowNormM performs only row-wise normalization and therefore avoids spectral computations altogether. Thus, the Gemma 3 1B-style experiment highlights a practical tradeoff: the hybrid row-norm/right-spectral update achieves the best validation loss, while the row-norm-only update provides a computationally cheaper symmetry-compatible alternative that still improves markedly over the coordinate-wise \AdamW baseline. We also provide additional experimental results for a base learning rate sweep and two extra random seeds in \Cref{subsec:gemma3_add_expt}. Finally, \Cref{subsubsec:gemma_logit_softcapping} shows that projected symmetry-compatible LM head updates reduce final vocabulary-logit growth in Gemma 3 1B-style pre-training, providing an optimizer-side mechanism that can reduce the reliance on final logit soft-capping.

\subsection{\OLMoE-1B-7B-Style Pre-Training}
\label{subsec:olmoe}
In addition to dense language models, we also pre-train a sparse Mixture-of-Experts (\MoE) model, a widely used architecture in recent open-weight language models
\citep{jiang2024mixtral,liu2024deepseek,agarwal2025gpt,qwen3.5,glm5team2026glm5,stepfun2026step,gemma4,deepseekai2026deepseekv4}. We use AllenAI's \OLMoE-1B-7B \citep{muennighoff2024olmoe}, which provides a comprehensive training recipe together with open-source data, code, and training logs. The model has vocabulary size 50,304 and hidden dimension 2048, making the embedding and LM head matrices considerably large. Relative to the original pre-training setup, we remove the auxiliary load-balancing loss \citep{shazeer2017outrageously} and the router z-loss \citep{zoph2022st} in order to reduce confounding effects from auxiliary objectives and isolate the effect of optimizer geometry. We also reduce the number of hidden layers from 16 to 12 and the number of experts from 64 to 32, yielding a total of 2,824,177,664 trainable parameters.

For matrix-valued parameters other than the hidden and attention matrices, we compare four optimizer assignments:
\begin{enumerate}[label=(\roman*)]
\itemsep0em 
\item \RowNormM for embeddings, LM head, and routers,
\item \RowNormM for embeddings and LM head, and \LeftPolarGradM for routers, 
\item \RowNormM for embeddings and LM head, and \AdamW for routers,
\item \AdamW for embeddings, LM head, and routers.
\end{enumerate}
We choose \RowNormM for the embedding and LM head matrices because it adds minimal computational overhead relative to \AdamW while preserving a symmetry-compatible geometry for vocabulary-indexed matrices. By contrast, \RightPolarGradM and \HybridPolarGradM require numerical polar decomposition, which is computationally more demanding and may require higher numerical precision for large vocabulary matrices. We leave broader ablations over these alternatives to future work. 

We also expect the relative behavior of \Muon and \HybridPolarGradM on SwiGLU MLP projections to depend on matrix aspect ratio. When $d_{\mathrm{ff}}\gg d_{\mathrm{model}}$, as in the dense Qwen3-0.6B-style and Gemma 3 1B-style models, the gate and up projections are tall-skinny and row-scale imbalance across intermediate
neurons can be substantial. In this regime, the row-normalization step in \HybridPolarGradM can be beneficial. In the \MoE experiments below, however, $d_{\mathrm{ff}}$ and $d_{\mathrm{model}}$ are much closer, so pure \Muon-style right-spectral updates may already capture much of the relevant matrix geometry. For this reason, in the \MoE experiments we use \Muon-style right-spectral updates for the SwiGLU MLP projection tensors and focus our ablations on the vocabulary-indexed matrices and \MoE routers.

\begin{figure}[h!]
    \centering
    \includegraphics[width=\textwidth]{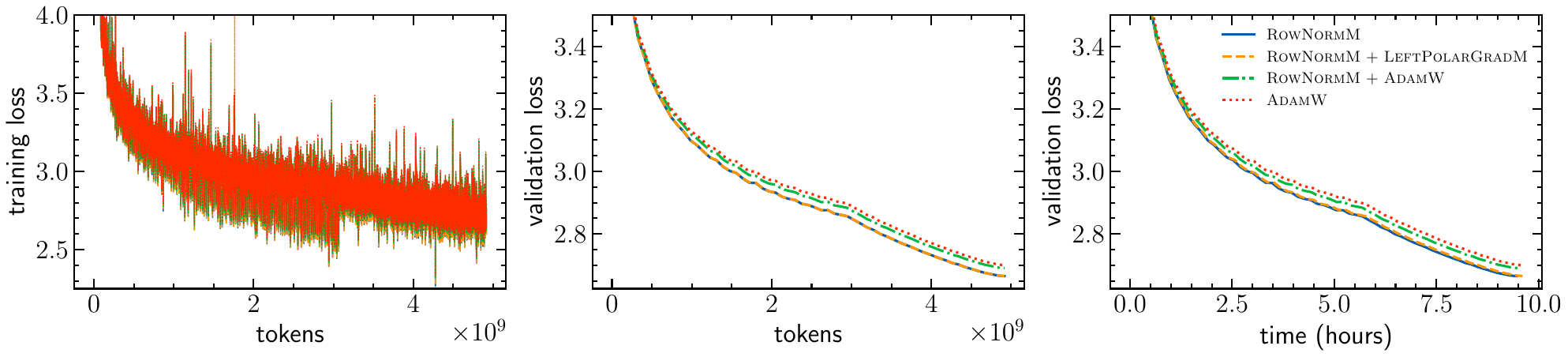}
    \caption{Training and validation losses for \OLMoE-1B-7B-style pre-training. The configurations differ in the optimizers assigned to the embedding, LM head, and router matrices.}
    \label{fig:olmoe}
\end{figure}

The final validation losses for configurations (i)--(iv) are 2.6654, 2.6657, 2.691, and 2.7003 respectively. As shown in \Cref{fig:olmoe}, configuration (iv), which uses \AdamW for the embedding, LM head, and router matrices, makes faster initial progress over roughly the first 500 steps, but is eventually overtaken by configurations (i)--(iii). This reversal occurs before the onset of learning rate decay, which linearly decreases to zero for the last 40\% of the training tokens. The validation loss gaps continue to widen during the decay phase.

Configurations (i) and (ii) use symmetry-compatible updates for all three special matrix classes considered here: embeddings, LM head, and routers. Configuration (iii) retains symmetry-compatible updates for the embedding and LM head matrices, but introduces a geometry mismatch in the router updates by using coordinate-wise \AdamW. Configuration (iv) applies \AdamW to all three classes and therefore departs most strongly from the proposed symmetry-compatible optimizer design. Empirically, both (i) and (ii) outperform (iii), while (iv) performs worst overall. These results are consistent with our theoretical perspective that respecting parameter symmetry and matrix geometry can matter for optimizer design, especially in large sparse architectures where router dynamics play a central role.

The performance gap between configurations (i) and (ii) is relatively small, suggesting that the row-norm and left-spectral router updates behave similarly in this setting. This small gap may reflect suboptimal learning-rate tuning for \RowNormM, the effect of inexact polar oracles, or that left-spectral normalization is able to further capture the relevant router geometry in this experiment. Finally, we observe that configuration (iv) might exhibit slightly more pronounced training-loss spikes at around 2.1B seen training tokens than the other configurations, despite using \Muon for the hidden and attention matrices. This suggests that geometry-matched optimizer choices for embeddings, LM heads, and routers may also improve training stability in practice.

\subsection{Downsized gpt-oss Pre-Training}
\label{subsec:gpt-oss}
We finally pre-train a downsized variant of gpt-oss-20b \citep{agarwal2025gpt}. This architecture differs from \OLMoE in several important ways. For example, it uses QKV bias terms in its GQA modules and includes bias vectors in its \MoE router networks. The model has vocabulary size 201,088, making the embedding and LM head matrices substantially larger than those in \OLMoE. To obtain a tractable experimental variant, we downsize gpt-oss-20b by reducing the number of hidden layers from 24 to 12, the hidden and intermediate dimensions from 2880 to 2048, and the number of experts from 32 to 16. This yields a total of 3,467,779,008 trainable parameters. We use the same loss function and the same four optimizer assignments as in the \OLMoE-1B-7B experiment.

\begin{figure}[h!]
    \centering
    \includegraphics[width=\textwidth]{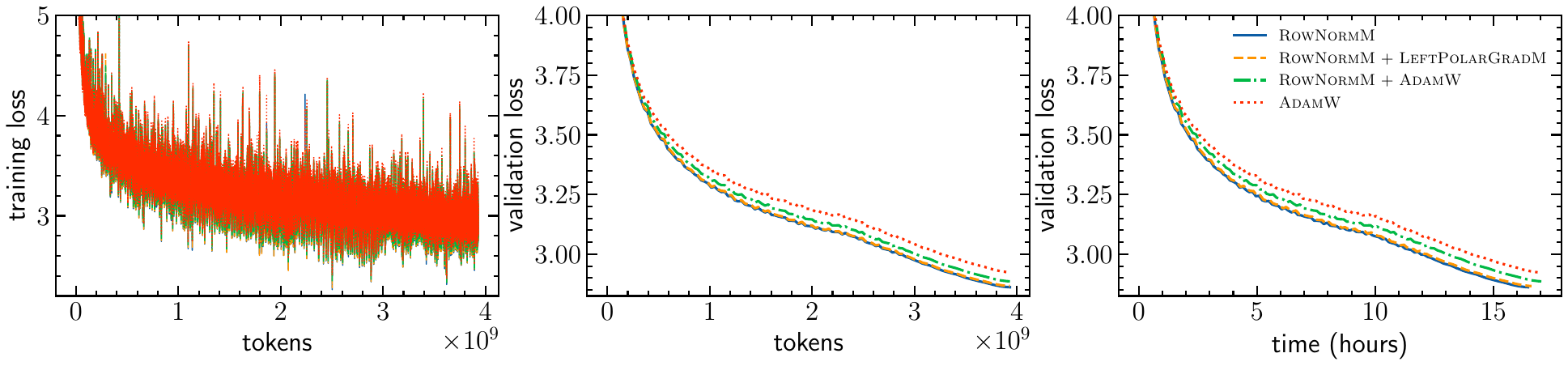}
    \caption{Training and validation losses for downsized gpt-oss pre-training. The configurations differ in the optimizers assigned to the embedding, LM head, and router matrices.}
    \label{fig:gpt-oss}
\end{figure}

The final validation losses for configurations (i)--(iv) are 2.8617, 2.8665, 2.8862, and 2.9238, respectively. As in the \OLMoE experiment, the fully coordinate-wise baseline in configuration (iv), which uses \AdamW for the embedding, LM head, and router matrices, obtains the worst final validation loss. The three configurations using \RowNormM for the embedding and LM head matrices all substantially improve over this baseline, suggesting that the benefit of symmetry-compatible updates for vocabulary-indexed matrices persists in a distinct sparse \MoE architecture. 

Among configurations (i)--(iii), the differences are smaller. As shown in \Cref{fig:gpt-oss}, configuration (i), which uses \RowNormM for embeddings, LM head, and routers, achieves the lowest final validation loss, followed closely by configuration (ii), which uses \LeftPolarGradM for routers. Configuration (iii), which keeps \RowNormM for embeddings and LM head but uses \AdamW for routers, is slightly worse than both geometry-compatible router variants. This ordering is consistent with the view that router geometry can matter, although the smaller gap between configurations (i)--(iii) suggests that the dominant improvement in this setting comes from replacing \AdamW on the large vocabulary-indexed matrices. 

Additional router diagnostics in \Cref{subsubsec:gpt_oss_router_diagnostics} show that symmetry-compatible router updates reduce load imbalance and router z-loss relative to the \AdamW router baseline, both with and without auxiliary router losses. 

Overall, the downsized gpt-oss experiment provides an additional architecture check for our optimizer design principle. Despite architectural differences from \OLMoE, including QKV biases and router bias terms, the same qualitative pattern holds: geometry-compatible optimizer assignments for special matrix-valued parameters improve the final validation loss relative to using coordinate-wise \AdamW for those parameters.

\subsection{Cross-Model Comparison}
\label{subsec:cross_model}

We compare the results across
\Cref{subsec:qwen3,subsec:gemma3,subsec:olmoe,subsec:gpt-oss}. These experiments span dense and sparse language models with increasing numbers of trainable parameters, from the Qwen3-0.6B-style dense model to the downsized gpt-oss sparse \MoE model. Across these models, the vocabulary sizes differ substantially, while the hidden dimensions are relatively close. Consequently, the embedding and LM head matrices have comparable column dimensions, determined by $d_{\mathrm{model}}$, but very different numbers of rows, determined by the vocabulary size $v$. Thus, as $v$ grows, these vocabulary-indexed matrices become increasingly tall. This changes both their computational cost and the conditioning properties of their matrix gradients, and makes them a natural testbed for row-aware and one-sided spectral optimizer design. 

The same row-versus-column distinction is also important for SwiGLU MLP projection matrices. For a dense SwiGLU block, the gate and up projections have shape $d_{\mathrm{ff}}\times d_{\mathrm{model}}$, so their rows correspond to intermediate neurons, whereas the down projection has shape $d_{\mathrm{model}}\times d_{\mathrm{ff}}$, so the same intermediate-neuron geometry appears along the column axis. In the dense Qwen3-0.6B-style and Gemma 3 1B-style models, $d_{\mathrm{ff}}$ is substantially larger than $d_{\mathrm{model}}$, placing the gate and up projections in a tall-skinny regime. This is precisely the setting where row-normalization before a right-spectral step can be useful: \HybridPolarGradM can correct row-scale imbalance across intermediate neurons while retaining spectral geometry. In the \MoE experiments, by contrast, the hidden and intermediate dimensions are closer in our downsized settings, so a pure \Muon-style right-spectral update may already capture much of the relevant matrix geometry for expert SwiGLU projection tensors. 

Across all experiments, replacing \AdamW on the large vocabulary-indexed matrices with symmetry-compatible optimizers consistently improves final validation loss. The gains are modest but visible for the smaller dense model, become more pronounced in the larger Gemma 3 1B-style experiment, and persist in both sparse \MoE experiments. This trend is consistent with our matrix-geometry perspective: as embedding and LM head matrices grow in their row dimension, coordinate-wise updates increasingly operate in a parameterization-dependent augmented space, whereas row-norm and spectral updates preserve the relevant vocabulary-indexed matrix geometry. 

The comparison should not be interpreted as a scaling law, since the models differ in architecture, vocabulary size, training length, and sparsity structure. Nevertheless, the consistent ordering across dense and sparse \MoE models provides evidence that the benefit of symmetry-compatible optimizer assignments is not restricted to a single model family or architecture. In particular, the results suggest that large vocabulary-indexed matrices are a robust setting in which geometry-compatible updates can improve optimization. The sparse \MoE experiments further show that router matrices and expert SwiGLU projection tensors provide additional layer types where symmetry-aware optimizer design can matter. 

Taken together, these experiments support the view that symmetry-compatible optimizer design is most naturally applied as a layerwise optimizer stack, rather than as a single global replacement for \AdamW.

\section{Discussion and Outlook}
\label{sec:discussion}
This work suggests a different view of deep learning optimization: optimizer design should be layerwise, geometry-aware, and symmetry-compatible. Popular coordinate-wise adaptive methods such as \Adam and \AdamW remain strong default optimizers because of their robustness, efficiency, and practical momentum. However, when applied indiscriminately to matrix-valued parameters, they treat matrices and tensors as collections of independent coordinates and therefore ignore the intrinsic geometry of the parameter blocks they update. This mismatch is especially apparent in modern architectures, where different modules play different algebraic and semantic roles. 

Our main contribution is a symmetry-compatible equivariance principle for designing optimizers for matrix-valued neural network parameters. For ordinary matrix layers, this principle recovers spectral optimizers as natural bi-orthogonally equivariant update maps. For embedding and LM head matrices, it leads to left-permutation/right-orthogonal equivariant updates. For SwiGLU MLP projections, it motivates row- and column-aware updates aligned with intermediate-neuron permutation geometry. For \MoE router weights, it yields expert-permutation equivariant and shared-logit-shift invariant updates. Together, these examples support an architecture--optimizer co-design principle: different parameter classes should be updated by optimizers whose equivariance matches their layerwise symmetry. 

This perspective connects recent progress in matrix-gradient optimization to a broader transition from generic coordinate-wise methods toward \emph{module-aware} and \emph{geometry-aware} optimization. Layerwise training itself is not new: classical examples include LARS and LAMB \citep{you2017large,you2020large}, layerwise hyperparameter prescriptions such as those arising in $\mu$P \citep{yang2021tuning}, and block-coordinate views of neural network training \citep{zeng2019global,lau2018proximal}. What is emerging, however, is a more refined class of layerwise optimizers that account for the geometry of each parameter block. Recent methods such as \Shampoo \citep{gupta2018shampoo}, \Muon \citep{jordan2024muon}, SOAP \citep{vyas2024soap}, \Scion \citep{pethick2025training}, \Gluon \citep{riabinin2025gluon}, and \PolarGrad \citep{lau2025polargrad} can be viewed as part of this trend. 

The case for geometry-aware optimizer design becomes stronger as foundation models become more heterogeneous. Large language models \citep{vaswani2017attention,radford2018improving,radford2019language,brown2020language}, vision transformers \citep{dosovitskiy2021an}, multimodal models \citep{radford2021learning,bordes2024introduction}, diffusion language models \citep{lou2024discrete,gong2025scaling,nie2025large}, \textsc{MoE}s \citep{shazeer2017outrageously,lepikhin2021gshard}, and state space models \citep{gu2024mamba,dao2024transformers,lahoti2026mamba} all contain parameter blocks with distinct natural symmetries. From this viewpoint, it is increasingly unnatural to optimize all such layers with a single coordinate-wise rule. Instead, optimizer design for modern deep learning systems should be treated as an architecture-aware problem, rather than as the selection of one universal update rule. 

This viewpoint also offers a possible interpretation of some training stabilization practices. As model scale increases, coordinate-wise adaptive optimizers are rarely used in isolation; they are typically accompanied by stabilization tricks, modified variants such as \StableAdamW \citep{wortsman2023stable}, and many recipe-level heuristics \citep{hu2025yulan}. While such techniques are practically important, some of them may be understood as partial corrections for a mismatch between optimizer geometry and model geometry. Geometry-aware optimizers provide a complementary path by encoding more appropriate invariance, normalization, and scaling structure directly into the update rule.

Several challenges remain. First, geometry-aware optimizers depend on efficient numerical linear algebra. The practical success of \Muon and related methods was enabled in part by making polar decomposition or matrix orthogonalization efficient at scale, especially through Newton--Schulz iterations. Further progress will likely depend on fast, stable, and
GPU-friendly matrix decomposition routines, including inexact polar oracles based on Newton--Schulz iteration \citep{kim2026convergence}, QDWH \citep{lau2025polargrad}, \textsc{Polar Express} \citep{amsel2025polar}, CANS \citep{grishina2025accelerating}, PRISM \citep{yang2026prism}, Turbo-\Muon \citep{boissin2025turbo}, Flash-\Muon \citep{lin2025flash}, and Gram Newton--Schulz \citep{zhang2026gram}. Second, non-elementwise optimizers raise new distributed-systems challenges, including communication costs, synchronization, memory sharding, and compatibility with tensor, pipeline, sequence, and data parallelism. Recent work on Distributed \Muon \citep{liu2025muon,essentialai2025layer,primeintellectteam2025intellect3}, \Dion \citep{ahn2025dion,ahn2025dion2}, \Disco \citep{filatov2025optimal}, and Parallel \Muon \citep{lim2025motif} suggests that these challenges can be addressed in practice.

Encouragingly, matrix-aware optimizers have already begun to appear in industry-scale model training, including work by Moonshot AI \citep{liu2025muon,kimi2025kimik2}, Essential AI \citep{essentialai2025practical_full}, Prime Intellect \citep{primeintellectteam2025intellect3}, Zhipu AI \citep{glm2025glm,glm5team2026glm5}, Zyphra \citep{anthony2025training,washbourne2026zaya18b}, Motif Technologies \citep{lim2025motif}, Arcee AI \citep{singh2026arcee}, StepFun \citep{stepfun2026step}, and DeepSeek-AI \citep{deepseekai2026deepseekv4}. These developments suggest that the question is no longer whether matrix-aware optimizers can be scaled, but how far they can be pushed once algorithmic geometry, numerical linear algebra, and distributed systems are designed in concert. 

Recent work by \citet{du2026uncovering} provides a complementary perspective, showing through a layer-peeled optimization model that symmetries in next-token distributions can transfer to learned LLM weights, logits, and context embeddings. The Newton--\Muon optimizer \citep{du2026newton} provides a complementary example of symmetry-aware matrix-gradient optimization. By deriving a \Muon-like update from a quadratic surrogate involving the layer input matrix, Newton--\Muon shows that the polar update can be combined with right preconditioning by the input second moment. In this sense, it extends the geometry of \Muon from a purely weight-gradient update to one that also reflects data geometry. 

Overall, our results suggest that much of modern deep learning still relies on optimizer updates whose geometry is mismatched to the layers they train. This mismatch may affect robustness, stability, scalability, and interpretability, since coordinate-wise adaptivity is sensitive to arbitrary parameterizations and can be viewed as operating in a pathological diagonal lifting of the original matrix space. By contrast, symmetry-compatible updates offer the prospect of more stable, parameterization-consistent, and theoretically grounded training procedures. Introducing symmetry and equivariance into optimizer design therefore opens a path toward a more principled science of large-scale model pre-training.

\subsection*{Acknowledgments}
This work was supported by computational resources from Prime Intellect.

\newpage    
\addcontentsline{toc}{section}{\protect\textbf{References}}
\bibliographystyle{custom}
{\small \bibliography{ref}}

@inproceedings{kingma2015,
	title = {Adam: a method for stochastic optimization},
	author = {Kingma, Diederik P. and Ba, Jimmy Lei},
	booktitle={International Conference on Learning Representations (ICLR)},
	year={2015},
}

@article{duchi2011adagrad,
	author = {Duchi, John and Hazan, Elad and Singer, Yoram},
	title = {Adaptive Subgradient Methods for Online Learning and Stochastic Optimization},
	journal = {Journal of Machine Learning Research},
	volume = {12},
	year = {2011},
	pages = {2121--2159},
}

@misc{tieleman2012,
	title={{Lecture 6.5---RMSProp: Divide the gradient by a running average of its recent magnitude}},
	author={Tieleman, Tijmen and Hinton, Geoffrey},
	howpublished={Coursera: Neural Networks for Machine Learning},
	year={2012}
}

@inproceedings{you2020large,
	title={Large Batch Optimization for Deep Learning: Training {BERT} in 76 minutes},
	author={Yang You and Jing Li and Sashank Reddi and Jonathan Hseu and Sanjiv Kumar and Srinadh Bhojanapalli and Xiaodan Song and James Demmel and Kurt Keutzer and Cho-Jui Hsieh},
	booktitle={International Conference on Learning Representations (ICLR)},
	year={2020},
}

@inproceedings{hoffmann2022training,
	title={Training Compute-Optimal Large Language Models},
	author={Jordan Hoffmann and Sebastian Borgeaud and Arthur Mensch and Elena Buchatskaya and Trevor Cai and Eliza Rutherford and Diego de las Casas and Lisa Anne Hendricks and Johannes Welbl and Aidan Clark and Tom Hennigan and Eric Noland and Katherine Millican and George van den Driessche and Bogdan Damoc and Aurelia Guy and Simon Osindero and Karen Simonyan and Erich Elsen and Oriol Vinyals and Jack William Rae and Laurent Sifre},
	booktitle={Advances in Neural Information Processing Systems (NeurIPS)},
	year={2022}
}

@article{chowdhery2022palm,
	title={{PaLM}: Scaling language modeling with pathways},
	author	= {Aakanksha Chowdhery and Sharan Narang and Jacob Devlin and Maarten Bosma and Gaurav Mishra and Adam Roberts and Paul Barham and Hyung Won Chung and Charles Sutton and Sebastian Gehrmann and Parker Schuh and Kensen Shi and Sasha Tsvyashchenko and Joshua Maynez and Abhishek Rao and Parker Barnes and Yi Tay and Noam Shazeer and Vinodkumar Prabhakaran and Emily Reif and Nan Du and Ben Hutchinson and Reiner Pope and James Bradbury and Jacob Austin and Michael Isard and Guy Gur-Ari and Pengcheng Yin and Toju Duke and Anselm Levskaya and Sanjay Ghemawat and Sunipa Dev and Henryk Michalewski and Xavier Garcia and Vedant Misra and Kevin Robinson and Liam Fedus and Denny Zhou and Daphne Ippolito and David Luan and Hyeontaek Lim and Barret Zoph and Alexander Spiridonov and Ryan Sepassi and David Dohan and Shivani Agrawal and Mark Omernick and Andrew M. Dai and Thanumalayan Sankaranarayana Pillai and Marie Pellat and Aitor Lewkowycz and Erica Moreira and Rewon Child and Oleksandr Polozov and Katherine Lee and Zongwei Zhou and Xuezhi Wang and Brennan Saeta and Mark Diaz and Orhan Firat and Michele Catasta and Jason Wei and Kathy Meier-Hellstern and Douglas Eck and Jeff Dean and Slav Petrov and Noah Fiedel},
	journal = {Journal of Machine Learning Research},
	year    = {2023},
	volume  = {24},
	number  = {240},
	pages   = {1--113},
}

@inproceedings{brown2020language,
	title={Language models are few-shot learners},
	author = {Brown, Tom B. and Mann, Benjamin and Ryder, Nick and Subbiah, Melanie and Kaplan, Jared and Dhariwal, Prafulla and Neelakantan, Arvind and Shyam, Pranav and Sastry, Girish and Askell, Amanda and Agarwal, Sandhini and Herbert-Voss, Ariel and Krueger, Gretchen and Henighan, Tom and Child, Rewon and Ramesh, Aditya and Ziegler, Daniel M. and Wu, Jeffrey and Winter, Clemens and Hesse, Chris and Chen, Mark and Sigler, Eric and Litwin, Mateusz and Gray, Scott and Chess, Benjamin and Clark, Jack and Berner, Christopher and McCandlish, Sam and Radford, Alec and Sutskever, Ilya and Amodei, Dario},
	booktitle={Advances in Neural Information Processing Systems (NeurIPS)},
	year={2020}
}

@techreport{radford2019language,
  title = {Language models are unsupervised multitask learners},
  author = {Radford, Alec and Wu, Jeffrey and Child, Rewon and Luan, David and Amodei, Dario and Sutskever, Ilya},
  institution = {OpenAI},
  year = {2019}
}

@techreport{radford2018improving,
  title = {Improving language understanding by generative pre-training},
  author = {Radford, Alec and Narasimhan, Karthik and Salimans, Tim and Sutskever, Ilya},
  institution = {OpenAI},
  year = {2018}
}

@inproceedings{vaswani2017attention,
	title={Attention is all you need},
	author={Vaswani, Ashish and Shazeer, Noam and Parmar, Niki and Uszkoreit, Jakob and Jones, Llion and Gomez, Aidan N. and Kaiser, {\L}ukasz and Polosukhin, Illia},
	booktitle={Advances in Neural Information Processing Systems (NeurIPS)},
	year={2017}
}

@inproceedings{dosovitskiy2021an,
	title={An Image is Worth 16x16 Words: Transformers for Image Recognition at Scale},
	author={Alexey Dosovitskiy and Lucas Beyer and Alexander Kolesnikov and Dirk Weissenborn and Xiaohua Zhai and Thomas Unterthiner and Mostafa Dehghani and Matthias Minderer and Georg Heigold and Sylvain Gelly and Jakob Uszkoreit and Neil Houlsby},
	booktitle={International Conference on Learning Representations (ICLR)},
	year={2021},
}

@article{kaplan2020scaling,
	title={Scaling laws for neural language models},
	author={Kaplan, Jared and McCandlish, Sam and Henighan, Tom and Brown, Tom B. and Chess, Benjamin and Child, Rewon and Gray, Scott and Radford, Alec and Wu, Jeffrey and Amodei, Dario},
	journal={arXiv preprint arXiv:2001.08361},
	year={2020}
}

@article{dahl2023benchmarking,
	title={Benchmarking Neural Network Training Algorithms},
	author={George E. Dahl and Frank Schneider and Zachary Nado and Naman Agarwal and Chandramouli Shama Sastry and Philipp Hennig and Sourabh Medapati and Runa Eschenhagen and Priya Kasimbeg and Daniel Suo and Juhan Bae and Justin Gilmer and Abel L. Peirson and Bilal Khan and Rohan Anil and Mike Rabbat and Shankar Krishnan and Daniel Snider and Ehsan Amid and Kongtao Chen and Chris J. Maddison and Rakshith Vasudev and Michal Badura and Ankush Garg and Peter Mattson},
	journal={arXiv preprint arXiv:2306.07179},
	year={2023}
}

@inproceedings{loshchilov2019decoupled,
	title={Decoupled Weight Decay Regularization},
	author={Loshchilov, Ilya and Hutter, Frank},
	booktitle={International Conference on Learning Representations (ICLR)},
	year={2019}
}

@inproceedings{he2016deep,
	title={Deep residual learning for image recognition},
	author={He, Kaiming and Zhang, Xiangyu and Ren, Shaoqing and Sun, Jian},
	booktitle={Proceedings of the IEEE/CVF Conference on Computer Vision and Pattern Recognition (CVPR)},
	year={2016}
}

@inproceedings{wortsman2024small,
  title={Small-scale proxies for large-scale Transformer training instabilities},
  author={Mitchell Wortsman and Peter J. Liu and Lechao Xiao and Katie Everett and Alex Alemi and Ben Adlam and John D. Co-Reyes and Izzeddin Gur and Abhishek Kumar and Roman Novak and Jeffrey Pennington and Jascha Sohl-dickstein and Kelvin Xu and Jaehoon Lee and Justin Gilmer and Simon Kornblith},
  booktitle={International Conference on Learning Representations (ICLR)},
  year={2024}
}

@article{shi2023distributed,
    title={A Distributed Data-Parallel {PyTorch} Implementation of the Distributed {Shampoo} Optimizer for Training Neural Networks At-Scale},
    author={Shi, Hao-Jun Michael and Lee, Tsung-Hsien and Iwasaki, Shintaro and Gallego-Posada, Jose and Li, Zhijing and Rangadurai, Kaushik and Mudigere, Dheevatsa and Rabbat, Michael},
    journal={arXiv preprint arXiv:2309.06497},
    year={2023}
}

@inproceedings{mcmahan2010adaptive,
    title={Adaptive bound optimization for online convex optimization},
    author={McMahan, H. Brendan and Streeter, Matthew},
    booktitle = {Proceedings of the Conference on Learning Theory (COLT)},
    year={2010}
}

@inproceedings{yang2021tuning,
  title={Tuning Large Neural Networks via Zero-Shot Hyperparameter Transfer},
  author={Greg Yang and Edward J. Hu and Igor Babuschkin and Szymon Sidor and Xiaodong Liu and David Farhi and Nick Ryder and Jakub Pachocki and Weizhu Chen and Jianfeng Gao},
  booktitle={Advances in Neural Information Processing Systems (NeurIPS)},
  year={2021},
}

@inproceedings{lepikhin2021gshard,
  title={{GShard}: Scaling Giant Models with Conditional Computation and Automatic Sharding},
  author={Dmitry Lepikhin and HyoukJoong Lee and Yuanzhong Xu and Dehao Chen and Orhan Firat and Yanping Huang and Maxim Krikun and Noam Shazeer and Zhifeng Chen},
  booktitle={International Conference on Learning Representations (ICLR)},
  year={2021},
}

@article{team2024gemma2,
  title={Gemma 2: Improving open language models at a practical size},
  author={{Gemma Team} and Morgane Riviere and Shreya Pathak and Pier Giuseppe Sessa and Cassidy Hardin and Surya Bhupatiraju and Léonard Hussenot and Thomas Mesnard and Bobak Shahriari and Alexandre Ramé and Johan Ferret and Peter Liu and Pouya Tafti and Abe Friesen and Michelle Casbon and Sabela Ramos and Ravin Kumar and Charline Le Lan and Sammy Jerome and Anton Tsitsulin and Nino Vieillard and Piotr Stanczyk and Sertan Girgin and Nikola Momchev and Matt Hoffman and Shantanu Thakoor and Jean-Bastien Grill and Behnam Neyshabur and Olivier Bachem and Alanna Walton and Aliaksei Severyn and Alicia Parrish and Aliya Ahmad and Allen Hutchison and Alvin Abdagic and Amanda Carl and Amy Shen and Andy Brock and Andy Coenen and Anthony Laforge and Antonia Paterson and Ben Bastian and Bilal Piot and Bo Wu and Brandon Royal and Charlie Chen and Chintu Kumar and Chris Perry and Chris Welty and Christopher A. Choquette-Choo and Danila Sinopalnikov and David Weinberger and Dimple Vijaykumar and Dominika Rogozińska and Dustin Herbison and Elisa Bandy and Emma Wang and Eric Noland and Erica Moreira and Evan Senter and Evgenii Eltyshev and Francesco Visin and Gabriel Rasskin and Gary Wei and Glenn Cameron and Gus Martins and Hadi Hashemi and Hanna Klimczak-Plucińska and Harleen Batra and Harsh Dhand and Ivan Nardini and Jacinda Mein and Jack Zhou and James Svensson and Jeff Stanway and Jetha Chan and Jin Peng Zhou and Joana Carrasqueira and Joana Iljazi and Jocelyn Becker and Joe Fernandez and Joost van Amersfoort and Josh Gordon and Josh Lipschultz and Josh Newlan and Ju-yeong Ji and Kareem Mohamed and Kartikeya Badola and Kat Black and Katie Millican and Keelin McDonell and Kelvin Nguyen and Kiranbir Sodhia and Kish Greene and Lars Lowe Sjoesund and Lauren Usui and Laurent Sifre and Lena Heuermann and Leticia Lago and Lilly McNealus and Livio Baldini Soares and Logan Kilpatrick and Lucas Dixon and Luciano Martins and Machel Reid and Manvinder Singh and Mark Iverson and Martin Görner and Mat Velloso and Mateo Wirth and Matt Davidow and Matt Miller and Matthew Rahtz and Matthew Watson and Meg Risdal and Mehran Kazemi and Michael Moynihan and Ming Zhang and Minsuk Kahng and Minwoo Park and Mofi Rahman and Mohit Khatwani and Natalie Dao and Nenshad Bardoliwalla and Nesh Devanathan and Neta Dumai and Nilay Chauhan and Oscar Wahltinez and Pankil Botarda and Parker Barnes and Paul Barham and Paul Michel and Pengchong Jin and Petko Georgiev and Phil Culliton and Pradeep Kuppala and Ramona Comanescu and Ramona Merhej and Reena Jana and Reza Ardeshir Rokni and Rishabh Agarwal and Ryan Mullins and Samaneh Saadat and Sara Mc Carthy and Sarah Cogan and Sarah Perrin and Sébastien M. R. Arnold and Sebastian Krause and Shengyang Dai and Shruti Garg and Shruti Sheth and Sue Ronstrom and Susan Chan and Timothy Jordan and Ting Yu and Tom Eccles and Tom Hennigan and Tomas Kocisky and Tulsee Doshi and Vihan Jain and Vikas Yadav and Vilobh Meshram and Vishal Dharmadhikari and Warren Barkley and Wei Wei and Wenming Ye and Woohyun Han and Woosuk Kwon and Xiang Xu and Zhe Shen and Zhitao Gong and Zichuan Wei and Victor Cotruta and Phoebe Kirk and Anand Rao and Minh Giang and Ludovic Peran and Tris Warkentin and Eli Collins and Joelle Barral and Zoubin Ghahramani and Raia Hadsell and D. Sculley and Jeanine Banks and Anca Dragan and Slav Petrov and Oriol Vinyals and Jeff Dean and Demis Hassabis and Koray Kavukcuoglu and Clement Farabet and Elena Buchatskaya and Sebastian Borgeaud and Noah Fiedel and Armand Joulin and Kathleen Kenealy and Robert Dadashi and Alek Andreev},
  journal={arXiv preprint arXiv:2408.00118},
  year={2024}
}

@inproceedings{bernstein2024modular,
  title={Modular Duality in Deep Learning},
  author={Bernstein, Jeremy and Newhouse, Laker},
  booktitle={Proceedings of the International Conference on Machine Learning (ICML)},
  year={2025}
}

@inproceedings{bernstein2024old,
  title = {Old Optimizer, New Norm: An Anthology},
  author = {Bernstein, Jeremy and Newhouse, Laker},
  booktitle = {Annual Workshop on Optimization for Machine Learning (OPT)},
  year = {2024}
}

@inproceedings{large2024scalable,
  title={Scalable Optimization in the Modular Norm},
  author={Large, Tim and Liu, Yang and Huh, Minyoung and Bahng, Hyojin and Isola, Phillip and Bernstein, Jeremy},
  booktitle={Advances in Neural Information Processing Systems (NeurIPS)},
  year={2024}
}

@article{lewis1995convex,
  title={The convex analysis of unitarily invariant matrix functions},
  author={Lewis, Adrian S.},
  journal={Journal of Convex Analysis},
  volume={2},
  number={1},
  pages={173--183},
  year={1995}
}

@article{lewis1996eigenvalue,
  title={Eigenvalue Optimization},
  author={Lewis, Adrian S. and Overton, Michael L.},
  journal={Acta Numerica},
  volume={5},
  pages={149--190},
  year={1996},
  publisher={Cambridge University Press}
}

@article{lewis2003mathematics,
  title={The mathematics of eigenvalue optimization},
  author={Lewis, Adrian S.},
  journal={Mathematical Programming},
  volume={97},
  pages={155--176},
  year={2003},
  publisher={Springer}
}

@misc{jordan2024muon,
  author       = {Keller Jordan and Yuchen Jin and Vlado Boza and Jiacheng You and
                  Franz Cesista and Laker Newhouse and Jeremy Bernstein},
  title        = {Muon: An optimizer for hidden layers in neural networks},
  year         = {2024},
  url          = {https://kellerjordan.github.io/posts/muon/}
}

@inproceedings{gupta2018shampoo,
  title={Shampoo: Preconditioned stochastic tensor optimization},
  author={Gupta, Vineet and Koren, Tomer and Singer, Yoram},
  booktitle={Proceedings of the International Conference on Machine Learning (ICML)},
  year={2018},
}

@article{anil2020scalable,
  title={Scalable second order optimization for deep learning},
  author={Anil, Rohan and Gupta, Vineet and Koren, Tomer and Regan, Kevin and Singer, Yoram},
  journal={arXiv preprint arXiv:2002.09018},
  year={2020}
}

@inproceedings{vyas2024soap,
  title={{SOAP}: Improving and stabilizing {Shampoo} using {Adam}},
  author={Vyas, Nikhil and Morwani, Depen and Zhao, Rosie and Shapira, Itai and Brandfonbrener, David and Janson, Lucas and Kakade, Sham},
  booktitle={International Conference on Learning Representations (ICLR)},
  year={2025}
}

@misc{modded_nanogpt_2024,
  author       = {Keller Jordan and Jeremy Bernstein and Brendan Rappazzo and
                  @fernbear.bsky.social and Boza Vlado and You Jiacheng and
                  Franz Cesista and Braden Koszarsky and @Grad62304977},
  title        = {\texttt{modded-nanogpt}: Speedrunning the {NanoGPT} baseline},
  year         = {2024},
  url          = {https://github.com/KellerJordan/modded-nanogpt},
}

@article{li2017preconditioned,
  title={Preconditioned stochastic gradient descent},
  author={Li, Xi-Lin},
  journal={IEEE Transactions on Neural Networks and Learning Systems},
  volume={29},
  number={5},
  pages={1454--1466},
  year={2017},
  publisher={IEEE}
}

@inproceedings{bernstein2018signsgd,
  title={sign{SGD}: Compressed optimisation for non-convex problems},
  author={Bernstein, Jeremy and Wang, Yu-Xiang and Azizzadenesheli, Kamyar and Anandkumar, Animashree},
  booktitle={Proceedings of the International Conference on Machine Learning (ICML)},
  year={2018},
}

@inproceedings{carlson2015stochasticRBM,
  title={Stochastic spectral descent for restricted {B}oltzmann machines},
  author={Carlson, David and Cevher, Volkan and Carin, Lawrence},
  booktitle={Proceedings of the International Conference on Artificial Intelligence and Statistics (AISTATS)},
  year={2015},
}

@article{carlson2016stochastic,
  title={Stochastic spectral descent for discrete graphical models},
  author={Carlson, David and Hsieh, Ya-Ping and Collins, Edo and Carin, Lawrence and Cevher, Volkan},
  journal={IEEE Journal of Selected Topics in Signal Processing},
  volume={10},
  number={2},
  pages={296--311},
  year={2016},
  publisher={IEEE}
}

@inproceedings{carlson2015preconditioned,
  title={Preconditioned spectral descent for deep learning},
  author={Carlson, David and Collins, Edo and Hsieh, Ya-Ping and Carin, Lawrence and Cevher, Volkan},
  booktitle={Advances in Neural Information Processing Systems (NeurIPS)},
  year={2015}
}

@inproceedings{shazeer2018adafactor,
  title={Adafactor: Adaptive learning rates with sublinear memory cost},
  author={Shazeer, Noam and Stern, Mitchell},
  booktitle={Proceedings of the International Conference on Machine Learning (ICML)},
  year={2018},
}

@inproceedings{yuan2024mars,
  title={{MARS}: Unleashing the Power of Variance Reduction for Training Large Models},
  author={Yuan, Huizhuo and Liu, Yifeng and Wu, Shuang and Zhou, Xun and Gu, Quanquan},
  booktitle={Proceedings of the International Conference on Machine Learning (ICML)},
  year={2025}
}

@article{tuddenham2022orthogonalising,
  title={Orthogonalising gradients to speed up neural network optimisation},
  author={Tuddenham, Mark and Pr{\"u}gel-Bennett, Adam and Hare, Jonathan},
  journal={arXiv preprint arXiv:2202.07052},
  year={2022}
}

@inproceedings{dehghani2023scaling,
  title={Scaling vision transformers to 22 billion parameters},
  author = {{Dehghani}, Mostafa and {Djolonga}, Josip and {Mustafa}, Basil and {Padlewski}, Piotr and {Heek}, Jonathan and {Gilmer}, Justin and {Steiner}, Andreas and {Caron}, Mathilde and {Geirhos}, Robert and {Alabdulmohsin}, Ibrahim and {Jenatton}, Rodolphe and {Beyer}, Lucas and {Tschannen}, Michael and {Arnab}, Anurag and {Wang}, Xiao and {Riquelme}, Carlos and {Minderer}, Matthias and {Puigcerver}, Joan and {Evci}, Utku and {Kumar}, Manoj and {van Steenkiste}, Sjoerd and {Elsayed}, Gamaleldin F. and {Mahendran}, Aravindh and {Yu}, Fisher and {Oliver}, Avital and {Huot}, Fantine and {Bastings}, Jasmijn and {Collier}, Mark Patrick and {Gritsenko}, Alexey and {Birodkar}, Vighnesh and {Vasconcelos}, Cristina and {Tay}, Yi and {Mensink}, Thomas and {Kolesnikov}, Alexander and {Paveti{\'c}}, Filip and {Tran}, Dustin and {Kipf}, Thomas and {Lu{\v{c}}i{\'c}}, Mario and {Zhai}, Xiaohua and {Keysers}, Daniel and {Harmsen}, Jeremiah and {Houlsby}, Neil},
  booktitle={Proceedings of the International Conference on Machine Learning (ICML)},
  year={2023},
}

@article{chen2021spectral,
  title={Spectral methods for data science: A statistical perspective},
  author={Chen, Yuxin and Chi, Yuejie and Fan, Jianqing and Ma, Cong},
  journal={Foundations and Trends{\textregistered} in Machine Learning},
  volume={14},
  number={5},
  pages={566--806},
  year={2021},
  publisher={Now Publishers, Inc.}
}

@article{pooladzandi2024curvature,
  title={Curvature-Informed {SGD} via General Purpose {L}ie-Group Preconditioners},
  author={Pooladzandi, Omead and Li, Xi-Lin},
  journal={arXiv preprint arXiv:2402.04553},
  year={2024}
}

@article{nakatsukasa2016computing,
  title={Computing fundamental matrix decompositions accurately via the matrix sign function in two iterations: The power of {Z}olotarev's functions},
  author={Nakatsukasa, Yuji and Freund, Roland W.},
  journal={SIAM Review},
  volume={58},
  number={3},
  pages={461--493},
  year={2016},
}

@inproceedings{zhao2024galore,
title={{GaLore}: Memory-Efficient {LLM} Training by Gradient Low-Rank Projection},
author={Jiawei Zhao and Zhenyu Zhang and Beidi Chen and Zhangyang Wang and Anima Anandkumar and Yuandong Tian},
booktitle={Proceedings of the International Conference on Machine Learning (ICML)},
year={2024},
}

@inproceedings{wortsman2023stable,
title={Stable and low-precision training for large-scale vision-language models},
author={Mitchell Wortsman and Tim Dettmers and Luke Zettlemoyer and Ari S. Morcos and Ali Farhadi and Ludwig Schmidt},
booktitle={Advances in Neural Information Processing Systems (NeurIPS)},
year={2023},
}

@article{nakatsukasa2010optimizing,
  title={Optimizing {H}alley's iteration for computing the matrix polar decomposition},
  author={Nakatsukasa, Yuji and Bai, Zhaojun and Gygi, Fran{\c{c}}ois},
  journal={SIAM Journal on Matrix Analysis and Applications},
  volume={31},
  number={5},
  pages={2700--2720},
  year={2010},
  publisher={SIAM}
}

@inproceedings{kasimbeg2025accelerating,
title={Accelerating neural network training: An analysis of the {AlgoPerf} competition},
author={Priya Kasimbeg and Frank Schneider and Runa Eschenhagen and Juhan Bae and Chandramouli Shama Sastry and Mark Saroufim and Boyuan Feng and Less Wright and Edward Z. Yang and Zachary Nado and Sourabh Medapati and Philipp Hennig and Michael Rabbat and George E. Dahl},
booktitle={International Conference on Learning Representations (ICLR)},
year={2025},
}

@book{higham2008functions,
  title={Functions of Matrices: Theory and Computation},
  author={Higham, Nicholas J.},
  year={2008},
  publisher = {Society for Industrial and Applied Mathematics},
}

@inproceedings{martens2015optimizing,
  title={Optimizing neural networks with {K}ronecker-factored approximate curvature},
  author={Martens, James and Grosse, Roger},
  booktitle={Proceedings of the International Conference on Machine Learning (ICML)},
  year={2015},
}

@inproceedings{hu2025yulan,
  title = {{YuLan-Mini}: Pushing the Limits of Open Data-efficient Language Model},
  author = {Yiwen Hu and Huatong Song and Jia Deng and Jiapeng Wang and Jie Chen and Kun Zhou and Yutao Zhu and Jinhao Jiang and Zican Dong and Wayne Xin Zhao and Ji-Rong Wen},
  booktitle = {Proceedings of the Annual Meeting of the Association for Computational Linguistics (ACL), Long Papers},
  year = {2025}
}

@inproceedings{muennighoff2024olmoe,
  title={{OLMoE}: Open Mixture-of-Experts Language Models}, 
  author={Niklas Muennighoff and Luca Soldaini and Dirk Groeneveld and Kyle Lo and Jacob Morrison and Sewon Min and Weijia Shi and Pete Walsh and Oyvind Tafjord and Nathan Lambert and Yuling Gu and Shane Arora and Akshita Bhagia and Dustin Schwenk and David Wadden and Alexander Wettig and Binyuan Hui and Tim Dettmers and Douwe Kiela and Ali Farhadi and Noah A. Smith and Pang Wei Koh and Amanpreet Singh and Hannaneh Hajishirzi},
  booktitle={International Conference on Learning Representations (ICLR)},
  year={2025}
}

@article{higham1986computing,
  title={Computing the polar decomposition---with applications},
  author={Higham, Nicholas J.},
  journal={SIAM Journal on Scientific and Statistical Computing},
  volume={7},
  number={4},
  pages={1160--1174},
  year={1986},
  publisher={SIAM}
}

@article{higham1997stable,
  title={Stable iterations for the matrix square root},
  author={Higham, Nicholas J.},
  journal={Numerical Algorithms},
  volume={15},
  number={2},
  pages={227--242},
  year={1997},
  publisher={Springer}
}

@article{lewis1996group,
  title={Group invariance and convex matrix analysis},
  author={Lewis, Adrian S.},
  journal={SIAM Journal on Matrix Analysis and Applications},
  volume={17},
  number={4},
  pages={927--949},
  year={1996},
  publisher={SIAM}
}

@book{horn2012matrix,
  title={Matrix Analysis},
  author={Horn, Roger A. and Johnson, Charles R.},
  year={2012},
  edition={2nd},
  publisher={Cambridge University Press}
}

@book{horn1994topics,
  title={Topics in Matrix Analysis},
  author={Horn, Roger A. and Johnson, Charles R.},
  year={1994},
  publisher={Cambridge University Press}
}

@article{li2025muon,
      title={A Note on the Convergence of {Muon}}, 
      author={Jiaxiang Li and Mingyi Hong},
      year={2025},
      journal={arXiv preprint arXiv:2502.02900},
}

@inproceedings{pethick2025training,
      title={Training Deep Learning Models with Norm-Constrained {LMOs}}, 
      author={Thomas Pethick and Wanyun Xie and Kimon Antonakopoulos and Zhenyu Zhu and Antonio Silveti-Falls and Volkan Cevher},
      year={2025},
      booktitle={Proceedings of the International Conference on Machine Learning (ICML)},
}

@article{liu2025muon,
  author = {Jingyuan Liu and Jianlin Su and Xingcheng Yao and Zhejun Jiang and Guokun Lai and Yulun Du and Yidao Qin and Weixin Xu and Enzhe Lu and Junjie Yan and Yanru Chen and Huabin Zheng and Yibo Liu and Shaowei Liu and Bohong Yin and Weiran He and Han Zhu and Yuzhi Wang and Jianzhou Wang and Mengnan Dong and Zheng Zhang and Yongsheng Kang and Hao Zhang and Xinran Xu and Yutao Zhang and Yuxin Wu and Xinyu Zhou and Zhilin Yang},
  title = {Muon is Scalable For {LLM} Training},
  journal={arXiv preprint arXiv:2502.16982},
  year = {2025},
}

@misc{bernstein2025deriving,
  author = {Jeremy Bernstein},
  title = {Deriving {Muon}},
  howpublished = {\url{https://jeremybernste.in/writing/deriving-muon}},
  year = {2025}
}

@article{kovalev2025understanding,
      title={Understanding Gradient Orthogonalization for Deep Learning via Non-{E}uclidean Trust-Region Optimization}, 
      author={Dmitry Kovalev},
      year={2025},
      journal={arXiv preprint arXiv:2503.12645},
}

@inproceedings{xie2025structured,
  title={Structured Preconditioners in Adaptive Optimization: A Unified Analysis},
  author={Xie, Shuo and Wang, Tianhao and Reddi, Sashank and Kumar, Sanjiv and Li, Zhiyuan},
  booktitle={Proceedings of the International Conference on Machine Learning (ICML)},
  year={2025}
}

@inproceedings{an2025asgo,
      title={{ASGO}: Adaptive Structured Gradient Optimization}, 
      author={Kang An and Yuxing Liu and Rui Pan and Shiqian Ma and Donald Goldfarb and Tong Zhang},
      year={2025},
      booktitle={Advances in Neural Information Processing Systems (NeurIPS)},
}

@article{essentialai2025practical_full,
      title={Practical Efficiency of {Muon} for Pretraining}, 
      author={{Essential AI} and Ishaan Shah and Anthony M. Polloreno and Karl Stratos and Philip Monk and Adarsh Chaluvaraju and Andrew Hojel and Andrew Ma and Anil Thomas and Ashish Tanwer and Darsh J. Shah and Khoi Nguyen and Kurt Smith and Michael Callahan and Michael Pust and Mohit Parmar and Peter Rushton and Platon Mazarakis and Ritvik Kapila and Saurabh Srivastava and Somanshu Singla and Tim Romanski and Yash Vanjani and Ashish Vaswani},
      year={2025},
      journal={arXiv preprint arXiv:2505.02222},
}

@article{autonne1902sur,
     author = {Autonne, Léon},
     title = {Sur les groupes lin\'eaires, r\'eels et orthogonaux},
     journal = {Bulletin de la Soci\'et\'e Math\'ematique de France},
     pages = {121--134},
     publisher = {Soci\'et\'e math\'ematique de France},
     volume = {30},
     year = {1902},
}

@inproceedings{amsel2025polar,
title={The {Polar Express}: Optimal Matrix Sign Methods and Their Application to the {Muon} Algorithm},
author={Amsel, Noah and Persson, David and Musco, Christopher and Gower, Robert M.},
booktitle={International Conference on Learning Representations (ICLR)},
year={2026},
}

@article{shen2025convergence,
  title = {On the Convergence Analysis of {M}uon},
  author = {Shen, Wei and Huang, Ruichuan and Huang, Minhui and Shen, Cong and Zhang, Jiawei},
  journal = {Transactions on Machine Learning Research},
  year = {2026},
  issn = {2835-8856}
}

@article{chen2025muon,
  title = {Muon Optimizes Under Spectral Norm Constraints},
  author = {Chen, Lizhang and Li, Jonathan and Liu, Qiang},
  journal = {Transactions on Machine Learning Research},
  year = {2026},
  issn = {2835-8856}
}

@inproceedings{goldfarb2020practical,
  title={Practical quasi-{N}ewton methods for training deep neural networks},
  author={Goldfarb, Donald and Ren, Yi and Bahamou, Achraf},
  booktitle={Advances in Neural Information Processing Systems (NeurIPS)},
  year={2020}
}

@article{kimi2025kimik2,
    author = {{Kimi Team}},
    title = {Kimi {K}2: Open Agentic Intelligence},
    journal={arXiv preprint arXiv:2507.20534},
    year = {2025}
}

@inproceedings{riabinin2025gluon,
  title={From {Muon} to {Gluon}: Bridging Theory and Practice of {LMO}-based Optimizers for {LLMs}},
  author={Riabinin, Artem and Shulgin, Egor and Gruntkowska, Kaja and Richt{\'a}rik, Peter},
  booktitle={Proceedings of the International Conference on Machine Learning (ICML)},
  year={2026}
}

@article{you2017large,
  title={Large batch training of convolutional networks},
  author={You, Yang and Gitman, Igor and Ginsburg, Boris},
  journal={arXiv preprint arXiv:1708.03888},
  year={2017}
}

@article{ahn2025dion,
  title={Dion: Distributed Orthonormalized Updates},
  author={Kwangjun Ahn and Byron Xu and Natalie Abreu and Ying Fan and Gagik Magakyan and Pratyusha Sharma and Zheng Zhan and John Langford},
  journal={arXiv preprint arXiv:2504.05295},
  year={2025}
}

@misc{bernstein2025manifolds,
  title = {Modular Manifolds},
  author = {Jeremy Bernstein},
  howpublished = {Thinking Machines Lab: Connectionism},
  year = {2025},
  url = {https://thinkingmachines.ai/blog/modular-manifolds/}
}

@article{grishina2025accelerating,
  title={Accelerating {N}ewton-{S}chulz Iteration for Orthogonalization via {C}hebyshev-type Polynomials},
  author={Grishina, Ekaterina and Smirnov, Matvey and Rakhuba, Maxim},
  journal={arXiv preprint arXiv:2506.10935},
  year={2025}
}

@article{su2025isotropic,
  title={Isotropic Curvature Model for Understanding Deep Learning Optimization: Is Gradient Orthogonalization Optimal?},
  author={Su, Weijie},
  journal={arXiv preprint arXiv:2511.00674},
  year={2025}
}

@inproceedings{crawshaw2025exploration,
  title={An Exploration of Non-{E}uclidean Gradient Descent: {M}uon and its Many Variants},
  author={Crawshaw, Michael and Modi, Chirag and Liu, Mingrui and Gower, Robert M.},
  booktitle={Proceedings of the International Conference on Machine Learning (ICML)},
  year={2026}
}

@article{veprikov2025preconditioned,
  title={Preconditioned Norms: A Unified Framework for Steepest Descent, Quasi-{N}ewton and Adaptive Methods},
  author={Veprikov, Andrey and Bolatov, Arman and Horv{\'a}th, Samuel and Beznosikov, Aleksandr and Tak{\'a}{\v{c}}, Martin and Hanzely, Slavomir},
  journal={arXiv preprint arXiv:2510.10777},
  year={2025}
}

@article{filatov2025optimal,
	title={Optimal scaling needs optimal norm},
	author={Filatov, Oleg and Wang, Jiangtao and Ebert, Jan and Kesselheim, Stefan},
	journal={arXiv preprint arXiv:2510.03871},
	year={2025}
}

@article{anthony2025training,
	title={Training Foundation Models on a Full-Stack {AMD} Platform: Compute, Networking, and System Design},
	author={Quentin Anthony and Yury Tokpanov and Skyler Szot and Srivatsan Rajagopal and Praneeth Medepalli and Rishi Iyer and Vasu Shyam and Anna Golubeva and Ansh Chaurasia and Xiao Yang and Tomas Figliolia and Robert Washbourne and Drew Thorstensen and Amartey Pearson and Zack Grossbart and Jason van Patten and Emad Barsoum and Zhenyu Gu and Yao Fu and Beren Millidge},
	journal={arXiv preprint arXiv:2511.17127},
	year={2025}
}

@book{bhatia2013matrix,
	title={Matrix Analysis},
	author={Bhatia, Rajendra},
	volume={169},
	year={2013},
	publisher={Springer Science \& Business Media}
}

@inproceedings{wen2025fantastic,
	title={Fantastic pretraining optimizers and where to find them},
	author={Wen, Kaiyue and Hall, David and Ma, Tengyu and Liang, Percy},
	booktitle={International Conference on Learning Representations (ICLR)},
	year={2026}
}

@article{semenov2025benchmarking,
	title={Benchmarking optimizers for large language model pretraining},
	author={Semenov, Andrei and Pagliardini, Matteo and Jaggi, Martin},
	journal={arXiv preprint arXiv:2509.01440},
	year={2025}
}

@article{kravatskiy2025ky,
	title={The {Ky Fan} Norms and Beyond: Dual Norms and Combinations for Matrix Optimization},
	author={Kravatskiy, Alexey and Kozyrev, Ivan and Kozlov, Nikolai and Vinogradov, Alexander and Merkulov, Daniil and Oseledets, Ivan},
	journal={arXiv preprint arXiv:2512.09678},
	year={2025}
}

@article{pethick2025training_review,
	title={Training Neural Networks at Any Scale: An exposition},
	author={Pethick, Thomas and Antonakopoulos, Kimon and Silveti-Falls, Antonio and Vankadara, Leena Chennuru and Cevher, Volkan},
	journal={IEEE Signal Processing Magazine},
    volume={43},
    number={3},
    pages={21--36},
    year={2026},
    publisher={IEEE}
}

@article{boissin2025turbo,
	title={{Turbo-Muon}: Accelerating Orthogonality-Based Optimization with Pre-Conditioning},
	author={Boissin, Thibaut and Massena, Thomas and Mamalet, Franck and Serrurier, Mathieu},
	journal={arXiv preprint arXiv:2512.04632},
	year={2025}
}

@article{ma2026preconditioning,
  title={Preconditioning Benefits of Spectral Orthogonalization in {M}uon},
  author={Ma, Jianhao and Huang, Yu and Chi, Yuejie and Chen, Yuxin},
  journal={arXiv preprint arXiv:2601.13474},
  year={2026}
}

@article{davis2025spectral,
  title={When do spectral gradient updates help in deep learning?},
  author={Davis, Damek and Drusvyatskiy, Dmitriy},
  journal={arXiv preprint arXiv:2512.04299},
  year={2025}
}

@inproceedings{xie2026controlled,
  title={Controlled {LLM} Training on Spectral Sphere},
  author={Tian Xie and Haoming Luo and Haoyu Tang and Yiwen Hu and Jason Klein Liu and Qingnan Ren and Yang Wang and Wayne Xin Zhao and Rui Yan and Bing Su and Chong Luo and Baining Guo},
  booktitle={Proceedings of the International Conference on Machine Learning (ICML)},
  year={2026}
}

@article{yang2026manifold,
  title={Manifold constrained steepest descent},
  author={Yang, Kaiwei and Lai, Lexiao},
  journal={arXiv preprint arXiv:2601.21487},
  year={2026}
}

@article{gu2026mano,
  title={Mano: Restriking Manifold Optimization for {LLM} Training},
  author={Gu, Yufei and Xie, Zeke},
  journal={arXiv preprint arXiv:2601.23000},
  year={2026}
}

@inproceedings{qi2026delving,
  title={Delving into {M}uon and Beyond: Deep Analysis and Extensions},
  author={Qi, Xianbiao and Chen, Marco and Ye, Jiaquan and He, Yelin and Xiao, Rong},
  booktitle={Proceedings of the International Conference on Machine Learning (ICML)},
  year={2026}
}

@inproceedings{yang2026prism,
  title={{PRISM}: Distribution-free Adaptive Computation of Matrix Functions for Accelerating Neural Network Training},
  author={Yang, Shenghao and Wang, Zhichao and Balabanov, Oleg and Erichson, N. Benjamin and Mahoney, Michael W.},
  booktitle={Proceedings of the International Conference on Machine Learning (ICML)},
  year={2026}
}

@article{gong2026aro,
  title={{ARO}: A New Lens On Matrix Optimization For Large Models},
  author={Gong, Wenbo and Zazo, Javier and Luo, Qijun and Wang, Puqian and Hensman, James and Ma, Chao},
  journal={arXiv preprint arXiv:2602.09006},
  year={2026}
}

@inproceedings{zhao2022symmetry,
  title={Symmetry teleportation for accelerated optimization},
  author={Zhao, Bo and Dehmamy, Nima and Walters, Robin and Yu, Rose},
  booktitle={Advances in Neural Information Processing Systems (NeurIPS)},
  year={2022}
}

@inproceedings{frans2025stable,
title={A Stable Whitening Optimizer for Efficient Neural Network Training},
author={Kevin Frans and Sergey Levine and Pieter Abbeel},
booktitle={Advances in Neural Information Processing Systems (NeurIPS)},
year={2025},
}

@article{ahn2025dion2,
  title = {Dion2: A Simple Method to Shrink Matrix in {Muon}},
  author = {Ahn, Kwangjun and Amsel, Noah and Langford, John},
  journal = {arXiv preprint arXiv:2512.16928},
  year = {2025}
}

@article{ding2018spectral,
	title={Spectral operators of matrices},
	author={Ding, Chao and Sun, Defeng and Sun, Jie and Toh, Kim-Chuan},
	journal={Mathematical Programming},
	volume={168},
	number={1},
	pages={509--531},
	year={2018},
	publisher={Springer}
}

@article{ding2020spectral,
	title={Spectral operators of matrices: Semismoothness and characterizations of the generalized {J}acobian},
	author={Ding, Chao and Sun, Defeng and Sun, Jie and Toh, Kim-Chuan},
	journal={SIAM Journal on Optimization},
	volume={30},
	number={1},
	pages={630--659},
	year={2020},
	publisher={SIAM}
}

@inproceedings{kim2026convergence,
title={Convergence of {M}uon with {N}ewton-{S}chulz},
author={Kim, Gyu Yeol and Oh, Min-hwan},
booktitle={International Conference on Learning Representations (ICLR)},
year={2026},
}

@article{kovalev2025non,
  title={Non-{E}uclidean {SGD} for Structured Optimization: Unified Analysis and Improved Rates},
  author={Kovalev, Dmitry and Borodich, Ekaterina},
  journal={arXiv preprint arXiv:2511.11466},
  year={2025}
}

@article{chang2025convergence,
  title={On the Convergence of {M}uon and Beyond},
  author={Chang, Da and Liu, Yongxiang and Yuan, Ganzhao},
  journal={arXiv preprint arXiv:2509.15816},
  year={2025}
}

@article{qwen3technicalreport,
      title={Qwen3 Technical Report}, 
      author={{Qwen Team}},
      year={2025},
      journal={arXiv preprint arXiv:2505.09388},
}

@inproceedings{li2025normuon,
  title={{NorMuon}: Making {M}uon more efficient and scalable},
  author={Li, Zichong and Liu, Liming and Liang, Chen and Chen, Weizhu and Zhao, Tuo},
  booktitle={Proceedings of the International Conference on Machine Learning (ICML)},
  year={2026}
}

@inproceedings{li2026convergence,
  title={Convergence Rate Analysis of the {AdamW}-Style {S}hampoo: Unifying One-sided and Two-Sided Preconditioning},
  author={Li, Huan and Dong, Yiming and Lin, Zhouchen},
  booktitle={Proceedings of the International Conference on Machine Learning (ICML)},
  year={2026}
}

@article{zhang2026mousse,
      title={Mousse: Rectifying the Geometry of {M}uon with Curvature-Aware Preconditioning}, 
      author={Yechen Zhang and Shuhao Xing and Junhao Huang and Kai Lv and Yunhua Zhou and Xipeng Qiu and Qipeng Guo and Kai Chen},
      year={2026},
      journal={arXiv preprint arXiv:2603.09697},
}

@article{gemmateam2025gemma3technicalreport,
      title={Gemma 3 Technical Report}, 
      author={{Gemma Team} and Aishwarya Kamath and Johan Ferret and Shreya Pathak and Nino Vieillard and Ramona Merhej and Sarah Perrin and Tatiana Matejovicova and Alexandre Ramé and Morgane Rivière and Louis Rouillard and Thomas Mesnard and Geoffrey Cideron and Jean-Bastien Grill and Sabela Ramos and Edouard Yvinec and Michelle Casbon and Etienne Pot and Ivo Penchev and Gaël Liu and Francesco Visin and Kathleen Kenealy and Lucas Beyer and Xiaohai Zhai and Anton Tsitsulin and Robert Busa-Fekete and Alex Feng and Noveen Sachdeva and Benjamin Coleman and Yi Gao and Basil Mustafa and Iain Barr and Emilio Parisotto and David Tian and Matan Eyal and Colin Cherry and Jan-Thorsten Peter and Danila Sinopalnikov and Surya Bhupatiraju and Rishabh Agarwal and Mehran Kazemi and Dan Malkin and Ravin Kumar and David Vilar and Idan Brusilovsky and Jiaming Luo and Andreas Steiner and Abe Friesen and Abhanshu Sharma and Abheesht Sharma and Adi Mayrav Gilady and Adrian Goedeckemeyer and Alaa Saade and Alex Feng and Alexander Kolesnikov and Alexei Bendebury and Alvin Abdagic and Amit Vadi and András György and André Susano Pinto and Anil Das and Ankur Bapna and Antoine Miech and Antoine Yang and Antonia Paterson and Ashish Shenoy and Ayan Chakrabarti and Bilal Piot and Bo Wu and Bobak Shahriari and Bryce Petrini and Charlie Chen and Charline Le Lan and Christopher A. Choquette-Choo and CJ Carey and Cormac Brick and Daniel Deutsch and Danielle Eisenbud and Dee Cattle and Derek Cheng and Dimitris Paparas and Divyashree Shivakumar Sreepathihalli and Doug Reid and Dustin Tran and Dustin Zelle and Eric Noland and Erwin Huizenga and Eugene Kharitonov and Frederick Liu and Gagik Amirkhanyan and Glenn Cameron and Hadi Hashemi and Hanna Klimczak-Plucińska and Harman Singh and Harsh Mehta and Harshal Tushar Lehri and Hussein Hazimeh and Ian Ballantyne and Idan Szpektor and Ivan Nardini and Jean Pouget-Abadie and Jetha Chan and Joe Stanton and John Wieting and Jonathan Lai and Jordi Orbay and Joseph Fernandez and Josh Newlan and Ju-yeong Ji and Jyotinder Singh and Kat Black and Kathy Yu and Kevin Hui and Kiran Vodrahalli and Klaus Greff and Linhai Qiu and Marcella Valentine and Marina Coelho and Marvin Ritter and Matt Hoffman and Matthew Watson and Mayank Chaturvedi and Michael Moynihan and Min Ma and Nabila Babar and Natasha Noy and Nathan Byrd and Nick Roy and Nikola Momchev and Nilay Chauhan and Noveen Sachdeva and Oskar Bunyan and Pankil Botarda and Paul Caron and Paul Kishan Rubenstein and Phil Culliton and Philipp Schmid and Pier Giuseppe Sessa and Pingmei Xu and Piotr Stanczyk and Pouya Tafti and Rakesh Shivanna and Renjie Wu and Renke Pan and Reza Rokni and Rob Willoughby and Rohith Vallu and Ryan Mullins and Sammy Jerome and Sara Smoot and Sertan Girgin and Shariq Iqbal and Shashir Reddy and Shruti Sheth and Siim Põder and Sijal Bhatnagar and Sindhu Raghuram Panyam and Sivan Eiger and Susan Zhang and Tianqi Liu and Trevor Yacovone and Tyler Liechty and Uday Kalra and Utku Evci and Vedant Misra and Vincent Roseberry and Vlad Feinberg and Vlad Kolesnikov and Woohyun Han and Woosuk Kwon and Xi Chen and Yinlam Chow and Yuvein Zhu and Zichuan Wei and Zoltan Egyed and Victor Cotruta and Minh Giang and Phoebe Kirk and Anand Rao and Kat Black and Nabila Babar and Jessica Lo and Erica Moreira and Luiz Gustavo Martins and Omar Sanseviero and Lucas Gonzalez and Zach Gleicher and Tris Warkentin and Vahab Mirrokni and Evan Senter and Eli Collins and Joelle Barral and Zoubin Ghahramani and Raia Hadsell and Yossi Matias and D. Sculley and Slav Petrov and Noah Fiedel and Noam Shazeer and Oriol Vinyals and Jeff Dean and Demis Hassabis and Koray Kavukcuoglu and Clement Farabet and Elena Buchatskaya and Jean-Baptiste Alayrac and Rohan Anil and Dmitry and Lepikhin and Sebastian Borgeaud and Olivier Bachem and Armand Joulin and Alek Andreev and Cassidy Hardin and Robert Dadashi and Léonard Hussenot},
  journal={arXiv preprint arXiv:2503.19786},
  year={2025},
}

@article{glm2025glm,
  title={{GLM}-4.5: Agentic, Reasoning, and Coding ({ARC}) Foundation Models}, 
        author={{{GLM}-4.5 Team} and Aohan Zeng and Xin Lv and Qinkai Zheng and Zhenyu Hou and Bin Chen and Chengxing Xie and Cunxiang Wang and Da Yin and Hao Zeng and Jiajie Zhang and Kedong Wang and Lucen Zhong and Mingdao Liu and Rui Lu and Shulin Cao and Xiaohan Zhang and Xuancheng Huang and Yao Wei and Yean Cheng and Yifan An and Yilin Niu and Yuanhao Wen and Yushi Bai and Zhengxiao Du and Zihan Wang and Zilin Zhu and Bohan Zhang and Bosi Wen and Bowen Wu and Bowen Xu and Can Huang and Casey Zhao and Changpeng Cai and Chao Yu and Chen Li and Chendi Ge and Chenghua Huang and Chenhui Zhang and Chenxi Xu and Chenzheng Zhu and Chuang Li and Congfeng Yin and Daoyan Lin and Dayong Yang and Dazhi Jiang and Ding Ai and Erle Zhu and Fei Wang and Gengzheng Pan and Guo Wang and Hailong Sun and Haitao Li and Haiyang Li and Haiyi Hu and Hanyu Zhang and Hao Peng and Hao Tai and Haoke Zhang and Haoran Wang and Haoyu Yang and He Liu and He Zhao and Hongwei Liu and Hongxi Yan and Huan Liu and Huilong Chen and Ji Li and Jiajing Zhao and Jiamin Ren and Jian Jiao and Jiani Zhao and Jianyang Yan and Jiaqi Wang and Jiayi Gui and Jiayue Zhao and Jie Liu and Jijie Li and Jing Li and Jing Lu and Jingsen Wang and Jingwei Yuan and Jingxuan Li and Jingzhao Du and Jinhua Du and Jinxin Liu and Junkai Zhi and Junli Gao and Ke Wang and Lekang Yang and Liang Xu and Lin Fan and Lindong Wu and Lintao Ding and Lu Wang and Man Zhang and Minghao Li and Minghuan Xu and Mingming Zhao and Mingshu Zhai and Pengfan Du and Qian Dong and Shangde Lei and Shangqing Tu and Shangtong Yang and Shaoyou Lu and Shijie Li and Shuang Li and Shuang-Li and Shuxun Yang and Sibo Yi and Tianshu Yu and Wei Tian and Weihan Wang and Wenbo Yu and Weng Lam Tam and Wenjie Liang and Wentao Liu and Xiao Wang and Xiaohan Jia and Xiaotao Gu and Xiaoying Ling and Xin Wang and Xing Fan and Xingru Pan and Xinyuan Zhang and Xinze Zhang and Xiuqing Fu and Xunkai Zhang and Yabo Xu and Yandong Wu and Yida Lu and Yidong Wang and Yilin Zhou and Yiming Pan and Ying Zhang and Yingli Wang and Yingru Li and Yinpei Su and Yipeng Geng and Yitong Zhu and Yongkun Yang and Yuhang Li and Yuhao Wu and Yujiang Li and Yunan Liu and Yunqing Wang and Yuntao Li and Yuxuan Zhang and Zezhen Liu and Zhen Yang and Zhengda Zhou and Zhongpei Qiao and Zhuoer Feng and Zhuorui Liu and Zichen Zhang and Zihan Wang and Zijun Yao and Zikang Wang and Ziqiang Liu and Ziwei Chai and Zixuan Li and Zuodong Zhao and Wenguang Chen and Jidong Zhai and Bin Xu and Minlie Huang and Hongning Wang and Juanzi Li and Yuxiao Dong and Jie Tang},
  journal={arXiv preprint arXiv:2508.06471},
  year={2025}
}

@article{xu2026width,
  title={On the Width Scaling of Neural Optimizers Under Matrix Operator Norms {I}: Row/Column Normalization and Hyperparameter Transfer},
  author={Xu, Ruihan and Li, Jiajin and Lu, Yiping},
  journal={arXiv preprint arXiv:2603.09952},
  year={2026}
}

@article{eschenhagen2026clarifying,
	title={Clarifying {Shampoo}: Adapting Spectral Descent to Stochasticity and the Parameter Trajectory},
	author={Eschenhagen, Runa and Cai, Anna and Lee, Tsung-Hsien and Shi, Hao-Jun Michael},
	journal={arXiv preprint arXiv:2602.09314},
	year={2026}
}

@misc{lozhkov2024fineweb-edu,
	author       = {Lozhkov, Anton and Ben Allal, Loubna and von Werra, Leandro and Wolf, Thomas},  
	title        = {{FineWeb-Edu}: the Finest Collection of Educational Content}, 
	year         = {2024},  
	url          = {https://huggingface.co/datasets/HuggingFaceFW/fineweb-edu},  
}

@inproceedings{deng2026rmnp,
      title={{RMNP}: Row-Momentum Normalized Preconditioning for Scalable Matrix-Based Optimization}, 
      author={Shenyang Deng and Zhuoli Ouyang and Tianyu Pang and Zihang Liu and Ruochen Jin and Shuhua Yu and Yaoqing Yang},
      year={2026},
      booktitle={Proceedings of the International Conference on Machine Learning (ICML)},
}

@inproceedings{glentis2025minimalist,
  title={Memory-Efficient {LLM} Pretraining via Minimalist Optimizer Design},
  author={Glentis, Athanasios and Li, Jiaxiang and Han, Andi and Hong, Mingyi},
  booktitle={Proceedings of the International Conference on Machine Learning (ICML)},
  year={2026}
}

@article{lojasiewicz1993geometrie,
  title={Sur la g{\'e}om{\'e}trie semi-et sous-analytique},
  author={{\L}ojasiewicz, Stanislas},
  journal={Annales de l'institut Fourier},
  volume={43},
  number={5},
  pages={1575--1595},
  year={1993}
}

@article{kurdyka1998gradients,
  title={On gradients of functions definable in o-minimal structures},
  author={Kurdyka, Krzysztof},
  journal={Annales de l'institut Fourier},
  volume={48},
  number={3},
  pages={769--783},
  year={1998}
}

@article{glm5team2026glm5,
	title={{GLM-5}: from Vibe Coding to Agentic Engineering},
	author={{GLM-5 Team} and Aohan Zeng and Xin Lv and Zhenyu Hou and Zhengxiao Du and Qinkai Zheng and Bin Chen and Da Yin and Chendi Ge and Chenghua Huang and Chengxing Xie and Chenzheng Zhu and Congfeng Yin and Cunxiang Wang and Gengzheng Pan and Hao Zeng and Haoke Zhang and Haoran Wang and Huilong Chen and Jiajie Zhang and Jian Jiao and Jiaqi Guo and Jingsen Wang and Jingzhao Du and Jinzhu Wu and Kedong Wang and Lei Li and Lin Fan and Lucen Zhong and Mingdao Liu and Mingming Zhao and Pengfan Du and Qian Dong and Rui Lu and Shuang-Li and Shulin Cao and Song Liu and Ting Jiang and Xiaodong Chen and Xiaohan Zhang and Xuancheng Huang and Xuezhen Dong and Yabo Xu and Yao Wei and Yifan An and Yilin Niu and Yitong Zhu and Yuanhao Wen and Yukuo Cen and Yushi Bai and Zhongpei Qiao and Zihan Wang and Zikang Wang and Zilin Zhu and Ziqiang Liu and Zixuan Li and Bojie Wang and Bosi Wen and Can Huang and Changpeng Cai and Chao Yu and Chen Li and Chengwei Hu and Chenhui Zhang and Dan Zhang and Daoyan Lin and Dayong Yang and Di Wang and Ding Ai and Erle Zhu and Fangzhou Yi and Feiyu Chen and Guohong Wen and Hailong Sun and Haisha Zhao and Haiyi Hu and Hanchen Zhang and Hanrui Liu and Hanyu Zhang and Hao Peng and Hao Tai and Haobo Zhang and He Liu and Hongwei Wang and Hongxi Yan and Hongyu Ge and Huan Liu and Huanpeng Chu and Jia'ni Zhao and Jiachen Wang and Jiajing Zhao and Jiamin Ren and Jiapeng Wang and Jiaxin Zhang and Jiayi Gui and Jiayue Zhao and Jijie Li and Jing An and Jing Li and Jingwei Yuan and Jinhua Du and Jinxin Liu and Junkai Zhi and Junwen Duan and Kaiyue Zhou and Kangjian Wei and Ke Wang and Keyun Luo and Laiqiang Zhang and Leigang Sha and Liang Xu and Lindong Wu and Lintao Ding and Lu Chen and Minghao Li and Nianyi Lin and Pan Ta and Qiang Zou and Rongjun Song and Ruiqi Yang and Shangqing Tu and Shangtong Yang and Shaoxiang Wu and Shengyan Zhang and Shijie Li and Shuang Li and Shuyi Fan and Wei Qin and Wei Tian and Weining Zhang and Wenbo Yu and Wenjie Liang and Xiang Kuang and Xiangmeng Cheng and Xiangyang Li and Xiaoquan Yan and Xiaowei Hu and Xiaoying Ling and Xing Fan and Xingye Xia and Xinyuan Zhang and Xinze Zhang and Xirui Pan and Xu Zou and Xunkai Zhang and Yadi Liu and Yandong Wu and Yanfu Li and Yidong Wang and Yifan Zhu and Yijun Tan and Yilin Zhou and Yiming Pan and Ying Zhang and Yinpei Su and Yipeng Geng and Yong Yan and Yonglin Tan and Yuean Bi and Yuhan Shen and Yuhao Yang and Yujiang Li and Yunan Liu and Yunqing Wang and Yuntao Li and Yurong Wu and Yutao Zhang and Yuxi Duan and Yuxuan Zhang and Zezhen Liu and Zhengtao Jiang and Zhenhe Yan and Zheyu Zhang and Zhixiang Wei and Zhuo Chen and Zhuoer Feng and Zijun Yao and Ziwei Chai and Ziyuan Wang and Zuzhou Zhang and Bin Xu and Minlie Huang and Hongning Wang and Juanzi Li and Yuxiao Dong and Jie Tang},
	journal={arXiv preprint arXiv:2602.15763},
	year={2026}
}

@inproceedings{radford2021learning,
	title={Learning transferable visual models from natural language supervision},
	author={Alec Radford and Jong Wook Kim and Chris Hallacy and Aditya Ramesh and Gabriel Goh and Sandhini Agarwal and Girish Sastry and Amanda Askell and Pamela Mishkin and Jack Clark and Gretchen Krueger and Ilya Sutskever},
	booktitle={Proceedings of the International Conference on Machine Learning (ICML)},
	year={2021},
}

@article{bordes2024introduction,
	title={An introduction to vision-language modeling},
	author={Florian Bordes and Richard Yuanzhe Pang and Anurag Ajay and Alexander C. Li and Adrien Bardes and Suzanne Petryk and Oscar Mañas and Zhiqiu Lin and Anas Mahmoud and Bargav Jayaraman and Mark Ibrahim and Melissa Hall and Yunyang Xiong and Jonathan Lebensold and Candace Ross and Srihari Jayakumar and Chuan Guo and Diane Bouchacourt and Haider Al-Tahan and Karthik Padthe and Vasu Sharma and Hu Xu and Xiaoqing Ellen Tan and Megan Richards and Samuel Lavoie and Pietro Astolfi and Reyhane Askari Hemmat and Jun Chen and Kushal Tirumala and Rim Assouel and Mazda Moayeri and Arjang Talattof and Kamalika Chaudhuri and Zechun Liu and Xilun Chen and Quentin Garrido and Karen Ullrich and Aishwarya Agrawal and Kate Saenko and Asli Celikyilmaz and Vikas Chandra},
	journal={arXiv preprint arXiv:2405.17247},
	year={2024}
}

@inproceedings{lou2024discrete,
	title={Discrete Diffusion Modeling by Estimating the Ratios of the Data Distribution},
	author={Aaron Lou and Chenlin Meng and Stefano Ermon},
	booktitle={Proceedings of the International Conference on Machine Learning (ICML)},
	year={2024},
}

@inproceedings{gong2025scaling,
	title={Scaling Diffusion Language Models via Adaptation from Autoregressive Models},
	author={Shansan Gong and Shivam Agarwal and Yizhe Zhang and Jiacheng Ye and Lin Zheng and Mukai Li and Chenxin An and Peilin Zhao and Wei Bi and Jiawei Han and Hao Peng and Lingpeng Kong},
	booktitle={International Conference on Learning Representations (ICLR)},
	year={2025},
}

@inproceedings{nie2025large,
	title={Large Language Diffusion Models},
	author={Shen Nie and Fengqi Zhu and Zebin You and Xiaolu Zhang and Jingyang Ou and Jun Hu and Jun Zhou and Yankai Lin and Ji-Rong Wen and Chongxuan Li},
	booktitle={Advances in Neural Information Processing Systems (NeurIPS)},
	year={2025},
}

@inproceedings{shazeer2017outrageously,
	title={Outrageously Large Neural Networks: The Sparsely-Gated Mixture-of-Experts Layer},
	author={Noam Shazeer and Azalia Mirhoseini and Krzysztof Maziarz and Andy Davis and Quoc Le and Geoffrey Hinton and Jeff Dean},
	booktitle={International Conference on Learning Representations (ICLR)},
	year={2017},
}

@inproceedings{ainslie2023gqa,
	title={{GQA}: Training generalized multi-query transformer models from multi-head checkpoints},
	author={Ainslie, Joshua and Lee-Thorp, James and De Jong, Michiel and Zemlyanskiy, Yury and Lebr{\'o}n, Federico and Sanghai, Sumit},
	booktitle={Proceedings of the Conference on Empirical Methods in Natural Language Processing (EMNLP)},
	year={2023}
}

@article{primeintellectteam2025intellect3,
  title={{INTELLECT}-3: Technical Report},
  author={{Prime Intellect Team} and Mika Senghaas and Fares Obeid and Sami Jaghouar and William Brown and Jack Min Ong and Daniel Auras and Matej Sirovatka and Jannik Straube and Andrew Baker and Sebastian Müller and Justus Mattern and Manveer Basra and Aiman Ismail and Dominik Scherm and Cooper Miller and Ameen Patel and Simon Kirsten and Mario Sieg and Christian Reetz and Kemal Erdem and Vincent Weisser and Johannes Hagemann},
  journal={arXiv preprint arXiv:2512.16144},
  year={2025}
}

@article{singh2026arcee,
  title={{Arcee Trinity Large} Technical Report},
  author={Varun Singh and Lucas Krauss and Sami Jaghouar and Matej Sirovatka and Charles Goddard and Fares Obied and Jack Min Ong and Jannik Straube and Fern and Aria Harley and Conner Stewart and Colin Kealty and Maziyar Panahi and Simon Kirsten and Anushka Deshpande and Anneketh Vij and Arthur Bresnu and Pranav Veldurthi and Raghav Ravishankar and Hardik Bishnoi and {DatologyAI Team} and {Arcee AI Team} and {Prime Intellect Team} and Mark McQuade and Johannes Hagemann and Lucas Atkins},
  journal={arXiv preprint arXiv:2602.17004},
  year={2026}
}

@article{stepfun2026step,
  title={{Step 3.5 Flash}: Open Frontier-Level Intelligence with 11{B} Active Parameters},
  author={{StepFun Team} and Ailin Huang and Ang Li and Aobo Kong and Bin Wang and Binxing Jiao and Bo Dong and Bojun Wang and Boyu Chen and Brian Li and Buyun Ma and Chang Su and Changxin Miao and Changyi Wan and Chao Lou and Chen Hu and Chen Xu and Chenfeng Yu and Chengting Feng and Chengyuan Yao and Chunrui Han and Dan Ma and Dapeng Shi and Daxin Jiang and Dehua Ma and Deshan Sun and Di Qi and Enle Liu and Fajie Zhang and Fanqi Wan and Guanzhe Huang and Gulin Yan and Guoliang Cao and Guopeng Li and Han Cheng and Hangyu Guo and Hanshan Zhang and Hao Nie and Haonan Jia and Haoran Lv and Hebin Zhou and Hekun Lv and Heng Wang and Heung-Yeung Shum and Hongbo Huang and Hongbo Peng and Hongyu Zhou and Hongyuan Wang and Houyong Chen and Huangxi Zhu and Huimin Wu and Huiyong Guo and Jia Wang and Jian Zhou and Jianjian Sun and Jiaoren Wu and Jiaran Zhang and Jiashu Lv and Jiashuo Liu and Jiayi Fu and Jiayu Liu and Jie Cheng and Jie Luo and Jie Yang and Jie Zhou and Jieyi Hou and Jing Bai and Jingcheng Hu and Jingjing Xie and Jingwei Wu and Jingyang Zhang and Jishi Zhou and Junfeng Liu and Junzhe Lin and Ka Man Lo and Kai Liang and Kaibo Liu and Kaijun Tan and Kaiwen Yan and Kaixiang Li and Kang An and Kangheng Lin and Lei Yang and Liang Lv and Liang Zhao and Liangyu Chen and Lieyu Shi and Liguo Tan and Lin Lin and Lina Chen and Luck Ma and Mengqiang Ren and Michael Li and Ming Li and Mingliang Li and Mingming Zhang and Mingrui Chen and Mitt Huang and Na Wang and Peng Liu and Qi Han and Qian Zhao and Qinglin He and Qinxin Du and Qiuping Wu and Quan Sun and Rongqiu Yang and Ruihang Miao and Ruixin Han and Ruosi Wan and Ruyan Guo and Shan Wang and Shaoliang Pang and Shaowen Yang and Shengjie Fan and Shijie Shang and Shiliang Yang and Shiwei Li and Shuangshuang Tian and Siqi Liu and Siye Wu and Siyu Chen and Song Yuan and Tiancheng Cao and Tianchi Yue and Tianhao Cheng and Tianning Li and Tingdan Luo and Wang You and Wei Ji and Wei Yuan and Wei Zhang and Weibo Wu and Weihao Xie and Wen Sun and Wenjin Deng and Wenzhen Zheng and Wuxun Xie and Xiangfeng Wang and Xiangwen Kong and Xiangyu Liu and Xiangyu Zhang and Xiaobo Yang and Xiaojia Liu and Xiaolan Yuan and Xiaoran Jiao and Xiaoxiao Ren and Xiaoyun Zhang and Xin Li and Xin Liu and Xin Wu and Xing Chen and Xingping Yang and Xinran Wang and Xu Zhao and Xuan He and Xuanti Feng and Xuedan Cai and Xuqiang Zhou and Yanbo Yu and Yang Li and Yang Xu and Yanlin Lai and Yanming Xu and Yaoyu Wang and Yeqing Shen and Yibo Zhu and Yichen Lv and Yicheng Cao and Yifeng Gong and Yijing Yang and Yikun Yang and Yin Zhao and Yingxiu Zhao and Yinmin Zhang and Yitong Zhang and Yixuan Zhang and Yiyang Chen and Yongchi Zhao and Yongshen Long and Yongyao Wang and Yousong Guan and Yu Zhou and Yuang Peng and Yuanhao Ding and Yuantao Fan and Yuanwei Lu and Yuanzhen Yang and Yuchu Luo and Yudi Zhao and Yue Peng and Yueqiang Lin and Yufan Lu and Yuling Zhao and Yunzhou Ju and Yurong Zhang and Yusheng Li and Yuxiang Yang and Yuyang Chen and Yuzhu Cai and Zejia Weng and Zetao Hong and Zexi Li and Zhe Xie and Zheng Ge and Zheng Gong and Zheng Zeng and Zhenyi Lu and Zhewei Huang and Zhichao Chang and Zhiguo Huang and Zhiheng Hu and Zidong Yang and Zili Wang and Ziqi Ren and Zixin Zhang and Zixuan Wang},
  journal={arXiv preprint arXiv:2602.10604},
  year={2026}
}

@article{lim2025motif,
	title={Motif 2 12.7{B} technical report},
	author={Junghwan Lim and Sungmin Lee and Dongseok Kim and Taehyun Kim and Eunhwan Park and Jeesoo Lee and Jeongdoo Lee and Junhyeok Lee and Wai Ting Cheung and Dahye Choi and Jaeheui Her and Jaeyeon Huh and Hanbin Jung and Changjin Kang and Beomgyu Kim and Minjae Kim and Taewhan Kim and Youngrok Kim and Hyukjin Kweon and Haesol Lee and Kungyu Lee and Dongpin Oh and Yeongjae Park and Bokki Ryu and Dongjoo Weon},
	journal={arXiv preprint arXiv:2511.07464},
	year={2025}
}

@misc{lin2025flash,
	author       = {Tianyang Lin},
	title        = {{Flash-Muon}: An Efficient Implementation of {M}uon Optimizer},
	year         = {2025},
	url          = {https://github.com/nil0x9/flash-muon}
}

@inproceedings{jiang2026adaptive,
  title={Adaptive Matrix Online Learning through Smoothing with Guarantees for Nonsmooth Nonconvex Optimization},
  author={Jiang, Ruichen and Mhammedi, Zakaria and Mohri, Mehryar and Mokhtari, Aryan},
  booktitle={Proceedings of the Conference on Learning Theory (COLT)},
  year={2026}
}

@inproceedings{huang2025limuon,
  title={{LiMuon}: Light and fast {M}uon optimizer for large models},
  author={Huang, Feihu and Luo, Yuning and Chen, Songcan},
  booktitle={Proceedings of the International Conference on Machine Learning (ICML)},
  year={2026}
}

@article{zhang2025adagrad,
  title={{AdaGrad} meets {M}uon: Adaptive stepsizes for orthogonal updates},
  author={Zhang, Minxin and Liu, Yuxuan and Schaeffer, Hayden},
  journal={arXiv preprint arXiv:2509.02981},
  year={2025}
}

@article{si2025adamuon,
  title={Ada{M}uon: Adaptive {M}uon optimizer},
  author={Si, Chongjie and Zhang, Debing and Shen, Wei},
  journal={arXiv preprint arXiv:2507.11005},
  year={2025}
}

@article{gonon2026insights,
  title={Insights on {M}uon from Simple Quadratics},
  author={Gonon, Antoine and Mu{\c{s}}at, Andreea-Alexandra and Boumal, Nicolas},
  journal={arXiv preprint arXiv:2602.11948},
  year={2026}
}

@inproceedings{eschenhagen2023kronecker,
  title={Kronecker-factored approximate curvature for modern neural network architectures},
  author={Eschenhagen, Runa and Immer, Alexander and Turner, Richard and Schneider, Frank and Hennig, Philipp},
  booktitle={Advances in Neural Information Processing Systems (NeurIPS)},
  year={2023}
}

@misc{buchanan2025mmuonadmm,
  author = {Buchanan, Sam},
  title = {A Faster Manifold {M}uon with {ADMM}},
  year = {2025},
  howpublished = {\url{https://sdbuchanan.com/blog/manifold-muon/}}
}

@misc{zhang2026gram,
	title   = {{Gram Newton-Schulz}},
	author  = {Jack Zhang and Noah Amsel and Berlin Chen and Tri Dao},
	year    = {2026},
	url     = {https://dao-ailab.github.io/blog/2026/gram-newton-schulz/}
}

@misc{essentialai2025layer,
	author = {{Essential AI}},
	title = {Layer Sharding for Large‑Scale Training with {M}uon},
	year = {2025},
	howpublished = {\url{https://www.essential.ai/research/infra}}
}

@article{zhao2026symmetry,
  title = {Symmetry in Neural Network Parameter Spaces},
  author = {Bo Zhao and Robin Walters and Rose Yu},
  journal = {Transactions on Machine Learning Research},
  year = {2026},
  issn = {2835-8856}
}

@inproceedings{zhao2024improving,
title={Improving Convergence and Generalization Using Parameter Symmetries},
author={Bo Zhao and Robert M. Gower and Robin Walters and Rose Yu},
booktitle={International Conference on Learning Representations (ICLR)},
year={2024},
}

@inproceedings{putterman2025gl,
  title={{GL} Equivariant Metanetworks for Learning on Low Rank Weight Spaces},
  author={Putterman, Theo and Lim, Derek and Gelberg, Yoav and Bronstein, Michael M and Jegelka, Stefanie and Maron, Haggai},
  booktitle={Learning on Graphs Conference (LoG)},
  year={2025}
}

@inproceedings{abbe2022non,
  title={On the non-universality of deep learning: quantifying the cost of symmetry},
  author={Abbe, Emmanuel and Boix-Adsera, Enric},
  booktitle={Advances in Neural Information Processing Systems (NeurIPS)},
  year={2022}
}

@inproceedings{ng2004feature,
  title={Feature selection, $L_1$ vs.~$L_2$ regularization, and rotational invariance},
  author={Ng, Andrew Y.},
  booktitle={Proceedings of the International Conference on Machine Learning (ICML)},
  year={2004}
}

@inproceedings{schubert2019circular,
  title={Circular convolutional neural networks for panoramic images and laser data},
  author={Schubert, Stefan and Neubert, Peer and P{\"o}schmann, Johannes and Protzel, Peter},
  booktitle={IEEE Intelligent Vehicles Symposium (IV)},
  year={2019},
}

@inproceedings{li2021why,
title={Why Are Convolutional Nets More Sample-Efficient than Fully-Connected Nets?},
author={Zhiyuan Li and Yi Zhang and Sanjeev Arora},
booktitle={International Conference on Learning Representations (ICLR)},
year={2021},
}

@article{zhang2026muon+,
  title={Muon+: Towards Better {M}uon via One Additional Normalization Step},
  author={Zhang, Ruijie and Zhao, Yequan and Liu, Ziyue and Wang, Zhengyang and Zhang, Zheng},
  journal={arXiv preprint arXiv:2602.21545},
  year={2026}
}

@article{xu2026fismo,
  title={{FISMO}: Fisher-Structured Momentum-Orthogonalized Optimizer},
  author={Xu, Chenrui and Yan, Wenjing and Zhang, Ying-Jun Angela},
  journal={arXiv preprint arXiv:2601.21750},
  year={2026}
}

@article{su2025galore,
  title={Ga{L}ore 2: Large-scale {LLM} pre-training by gradient low-rank projection},
  author={Su, DiJia and Gu, Andrew and Xu, Jane and Tian, Yuandong and Zhao, Jiawei},
  journal={arXiv preprint arXiv:2504.20437},
  year={2025}
}

@article{du2026newton,
  title={The {N}ewton--{M}uon Optimizer},
  author={Du, Zhehang and Su, Weijie},
  journal={arXiv preprint arXiv:2604.01472},
  year={2026}
}

@article{jiang2024mixtral,
  title={Mixtral of experts},
  author={Albert Q. Jiang and Alexandre Sablayrolles and Antoine Roux and Arthur Mensch and Blanche Savary and Chris Bamford and Devendra Singh Chaplot and Diego de las Casas and Emma Bou Hanna and Florian Bressand and Gianna Lengyel and Guillaume Bour and Guillaume Lample and Lélio Renard Lavaud and Lucile Saulnier and Marie-Anne Lachaux and Pierre Stock and Sandeep Subramanian and Sophia Yang and Szymon Antoniak and Teven Le Scao and Théophile Gervet and Thibaut Lavril and Thomas Wang and Timothée Lacroix and William El Sayed},
  journal={arXiv preprint arXiv:2401.04088},
  year={2024}
}

@article{liu2024deepseek,
  title={{DeepSeek-V3} Technical Report}, 
        author={DeepSeek-AI and Aixin Liu and Bei Feng and Bing Xue and Bingxuan Wang and Bochao Wu and Chengda Lu and Chenggang Zhao and Chengqi Deng and Chenyu Zhang and Chong Ruan and Damai Dai and Daya Guo and Dejian Yang and Deli Chen and Dongjie Ji and Erhang Li and Fangyun Lin and Fucong Dai and Fuli Luo and Guangbo Hao and Guanting Chen and Guowei Li and H. Zhang and Han Bao and Hanwei Xu and Haocheng Wang and Haowei Zhang and Honghui Ding and Huajian Xin and Huazuo Gao and Hui Li and Hui Qu and J. L. Cai and Jian Liang and Jianzhong Guo and Jiaqi Ni and Jiashi Li and Jiawei Wang and Jin Chen and Jingchang Chen and Jingyang Yuan and Junjie Qiu and Junlong Li and Junxiao Song and Kai Dong and Kai Hu and Kaige Gao and Kang Guan and Kexin Huang and Kuai Yu and Lean Wang and Lecong Zhang and Lei Xu and Leyi Xia and Liang Zhao and Litong Wang and Liyue Zhang and Meng Li and Miaojun Wang and Mingchuan Zhang and Minghua Zhang and Minghui Tang and Mingming Li and Ning Tian and Panpan Huang and Peiyi Wang and Peng Zhang and Qiancheng Wang and Qihao Zhu and Qinyu Chen and Qiushi Du and R. J. Chen and R. L. Jin and Ruiqi Ge and Ruisong Zhang and Ruizhe Pan and Runji Wang and Runxin Xu and Ruoyu Zhang and Ruyi Chen and S. S. Li and Shanghao Lu and Shangyan Zhou and Shanhuang Chen and Shaoqing Wu and Shengfeng Ye and Shengfeng Ye and Shirong Ma and Shiyu Wang and Shuang Zhou and Shuiping Yu and Shunfeng Zhou and Shuting Pan and T. Wang and Tao Yun and Tian Pei and Tianyu Sun and W. L. Xiao and Wangding Zeng and Wanjia Zhao and Wei An and Wen Liu and Wenfeng Liang and Wenjun Gao and Wenqin Yu and Wentao Zhang and X. Q. Li and Xiangyue Jin and Xianzu Wang and Xiao Bi and Xiaodong Liu and Xiaohan Wang and Xiaojin Shen and Xiaokang Chen and Xiaokang Zhang and Xiaosha Chen and Xiaotao Nie and Xiaowen Sun and Xiaoxiang Wang and Xin Cheng and Xin Liu and Xin Xie and Xingchao Liu and Xingkai Yu and Xinnan Song and Xinxia Shan and Xinyi Zhou and Xinyu Yang and Xinyuan Li and Xuecheng Su and Xuheng Lin and Y. K. Li and Y. Q. Wang and Y. X. Wei and Y. X. Zhu and Yang Zhang and Yanhong Xu and Yanhong Xu and Yanping Huang and Yao Li and Yao Zhao and Yaofeng Sun and Yaohui Li and Yaohui Wang and Yi Yu and Yi Zheng and Yichao Zhang and Yifan Shi and Yiliang Xiong and Ying He and Ying Tang and Yishi Piao and Yisong Wang and Yixuan Tan and Yiyang Ma and Yiyuan Liu and Yongqiang Guo and Yu Wu and Yuan Ou and Yuchen Zhu and Yuduan Wang and Yue Gong and Yuheng Zou and Yujia He and Yukun Zha and Yunfan Xiong and Yunxian Ma and Yuting Yan and Yuxiang Luo and Yuxiang You and Yuxuan Liu and Yuyang Zhou and Z. F. Wu and Z. Z. Ren and Zehui Ren and Zhangli Sha and Zhe Fu and Zhean Xu and Zhen Huang and Zhen Zhang and Zhenda Xie and Zhengyan Zhang and Zhewen Hao and Zhibin Gou and Zhicheng Ma and Zhigang Yan and Zhihong Shao and Zhipeng Xu and Zhiyu Wu and Zhongyu Zhang and Zhuoshu Li and Zihui Gu and Zijia Zhu and Zijun Liu and Zilin Li and Ziwei Xie and Ziyang Song and Ziyi Gao and Zizheng Pan},
  journal={arXiv preprint arXiv:2412.19437},
  year={2024}
}

@article{agarwal2025gpt,
  title={gpt-oss-120b \& gpt-oss-20b model card},
  author={{OpenAI} and Sandhini Agarwal and Lama Ahmad and Jason Ai and Sam Altman and Andy Applebaum and Edwin Arbus and Rahul K. Arora and Yu Bai and Bowen Baker and Haiming Bao and Boaz Barak and Ally Bennett and Tyler Bertao and Nivedita Brett and Eugene Brevdo and Greg Brockman and Sebastien Bubeck and Che Chang and Kai Chen and Mark Chen and Enoch Cheung and Aidan Clark and Dan Cook and Marat Dukhan and Casey Dvorak and Kevin Fives and Vlad Fomenko and Timur Garipov and Kristian Georgiev and Mia Glaese and Tarun Gogineni and Adam Goucher and Lukas Gross and Katia Gil Guzman and John Hallman and Jackie Hehir and Johannes Heidecke and Alec Helyar and Haitang Hu and Romain Huet and Jacob Huh and Saachi Jain and Zach Johnson and Chris Koch and Irina Kofman and Dominik Kundel and Jason Kwon and Volodymyr Kyrylov and Elaine Ya Le and Guillaume Leclerc and James Park Lennon and Scott Lessans and Mario Lezcano-Casado and Yuanzhi Li and Zhuohan Li and Ji Lin and Jordan Liss and Lily and Liu and Jiancheng Liu and Kevin Lu and Chris Lu and Zoran Martinovic and Lindsay McCallum and Josh McGrath and Scott McKinney and Aidan McLaughlin and Song Mei and Steve Mostovoy and Tong Mu and Gideon Myles and Alexander Neitz and Alex Nichol and Jakub Pachocki and Alex Paino and Dana Palmie and Ashley Pantuliano and Giambattista Parascandolo and Jongsoo Park and Leher Pathak and Carolina Paz and Ludovic Peran and Dmitry Pimenov and Michelle Pokrass and Elizabeth Proehl and Huida Qiu and Gaby Raila and Filippo Raso and Hongyu Ren and Kimmy Richardson and David Robinson and Bob Rotsted and Hadi Salman and Suvansh Sanjeev and Max Schwarzer and D. Sculley and Harshit Sikchi and Kendal Simon and Karan Singhal and Yang Song and Dane Stuckey and Zhiqing Sun and Philippe Tillet and Sam Toizer and Foivos Tsimpourlas and Nikhil Vyas and Eric Wallace and Xin Wang and Miles Wang and Olivia Watkins and Kevin Weil and Amy Wendling and Kevin Whinnery and Cedric Whitney and Hannah Wong and Lin Yang and Yu Yang and Michihiro Yasunaga and Kristen Ying and Wojciech Zaremba and Wenting Zhan and Cyril Zhang and Brian Zhang and Eddie Zhang and Shengjia Zhao},
  journal={arXiv preprint arXiv:2508.10925},
  year={2025}
}

@misc{qwen3.5,
    title  = {{Qwen3.5}: Towards Native Multimodal Agents},
    author = {{Qwen Team}},
    month  = {February},
    year   = {2026},
    url    = {https://qwen.ai/blog?id=qwen3.5}
}

@misc{gemma4,
    title  = {Gemma 4 model card},
    author = {{Google DeepMind}},
    month  = {April},
    year   = {2026},
    url    = {https://ai.google.dev/gemma/docs/core/model_card_4}
}

@article{zoph2022st,
  title={{ST-MoE}: Designing stable and transferable sparse expert models},
  author={Zoph, Barret and Bello, Irwan and Kumar, Sameer and Du, Nan and Huang, Yanping and Dean, Jeff and Shazeer, Noam and Fedus, William},
  journal={arXiv preprint arXiv:2202.08906},
  year={2022}
}

@article{liu2025reg,
	title={{REG}: A Regularization Optimizer for Robust Training Dynamics},
	author={Liu, Zehua and Wu, Han and Fu, Xiaojin and Liu, Shuqi and Han, Xiongwei and Zhong, Tao and Yuan, Mingxuan},
	journal={arXiv preprint arXiv:2510.03691},
	year={2025}
}

@misc{deepseekai2026deepseekv4,
      title={{DeepSeek-V4}: Towards Highly Efficient Million-Token Context Intelligence},
      author={DeepSeek-AI},
      year={2026},
}

@inproceedings{dauphin2017language,
  title={Language modeling with gated convolutional networks},
  author={Dauphin, Yann N. and Fan, Angela and Auli, Michael and Grangier, David},
  booktitle={Proceedings of the International Conference on Machine Learning (ICML)},
  year={2017},
}

@article{shazeer2020glu,
  title={{GLU} variants improve transformer},
  author={Shazeer, Noam},
  journal={arXiv preprint arXiv:2002.05202},
  year={2020}
}

@article{su2024roformer,
  title = {{RoFormer}: Enhanced transformer with Rotary Position Embedding},
  author = {Jianlin Su and Murtadha Ahmed and Yu Lu and Shengfeng Pan and Wen Bo and Yunfeng Liu},
  journal={Neurocomputing},
  volume={568},
  pages={127063},
  year={2024},
  publisher={Elsevier},
  doi = {10.1016/j.neucom.2023.127063},
}

@inproceedings{jiang2023pre,
  title={{Pre-RMSNorm} and {Pre-CRMSNorm} transformers: equivalent and efficient {Pre-LN} transformers},
  author={Jiang, Zixuan and Gu, Jiaqi and Zhu, Hanqing and Pan, David},
  booktitle={Advances in Neural Information Processing Systems (NeurIPS)},
  year={2023}
}

@inproceedings{xiong2020layer,
  title={On layer normalization in the transformer architecture},
  author={Xiong, Ruibin and Yang, Yunchang and He, Di and Zheng, Kai and Zheng, Shuxin and Xing, Chen and Zhang, Huishuai and Lan, Yanyan and Wang, Liwei and Liu, Tieyan},
  booktitle={Proceedings of the International Conference on Machine Learning (ICML)},
  year={2020},
}

@inproceedings{zhang2019root,
	title={Root mean square layer normalization},
	author={Zhang, Biao and Sennrich, Rico},
	booktitle={Advances in Neural Information Processing Systems (NeurIPS)},
	year={2019}
}

@inproceedings{lin2025understanding,
  title={Understanding and Improving {Shampoo} and {SOAP} via {K}ullback--{L}eibler Minimization},
  author={Lin, Wu and Lowe, Scott C. and Dangel, Felix and Eschenhagen, Runa and Xu, Zikun and Grosse, Roger B.},
  booktitle={International Conference on Learning Representations (ICLR)},
  year={2026}
}

@article{sun2008lowner,
  title={L{\"o}wner's operator and spectral functions in {E}uclidean {J}ordan algebras},
  author={Sun, Defeng and Sun, Jie},
  journal={Mathematics of Operations Research},
  volume={33},
  number={2},
  pages={421--445},
  year={2008},
  publisher={INFORMS}
}

@article{chang2026muoneq,
	title={{MuonEq}: Balancing Before Orthogonalization with Lightweight Equilibration},
	author={Chang, Da and Shi, Qiankun and Zhang, Lvgang and Li, Yu and Zhang, Ruijie and Lu, Yao and Liu, Yongxiang and Yuan, Ganzhao},
	journal={arXiv preprint arXiv:2603.28254},
	year={2026}
}

@article{liu2025mars,
  title={{MARS-M}: When variance reduction meets matrices},
  author={Liu, Yifeng and Yuan, Angela and Gu, Quanquan},
  journal={arXiv preprint arXiv:2510.21800},
  year={2025}
}

@misc{aurora2026leverage,
  author = {Alec Dewulf and Dhruv Pai and Li Yang and Ashley Zhang and Ben Keigwin},
  title = {Aurora: A Leverage-Aware Optimizer for Rectangular Matrices},
  year = {2026},
  month = {May},
  day = {5},
  url = {https://blog.tilderesearch.com/blog/aurora},
}

@inproceedings{gu2024mamba,
title={Mamba: Linear-Time Sequence Modeling with Selective State Spaces},
author={Albert Gu and Tri Dao},
booktitle={Proceedings of the Conference on Language Modeling (COLM)},
year={2024},
}

@inproceedings{dao2024transformers,
  title={Transformers are {SSM}s: Generalized Models and Efficient Algorithms Through Structured State Space Duality},
  author={Dao, Tri and Gu, Albert},
  booktitle={Proceedings of the International Conference on Machine Learning (ICML)},
  year={2024},
}

@inproceedings{lahoti2026mamba,
title={Mamba-3: Improved Sequence Modeling using State Space Principles},
author={Aakash Lahoti and Kevin Li and Berlin Chen and Caitlin Wang and Aviv Bick and J Zico Kolter and Tri Dao and Albert Gu},
booktitle={International Conference on Learning Representations (ICLR)},
year={2026},
}

@article{washbourne2026zaya18b,
      title={{ZAYA1-8B} Technical Report}, 
      author={Robert Washbourne and Rishi Iyer and Tomas Figliolia and Henry Zheng and Ryan Lorig-Roach and Sungyeon Yang and Pritish Yuvraj and Quentin Anthony and Yury Tokpanov and Xiao Yang and Ganesh Nanduru and Stephen Ebert and Praneeth Medepalli and Skyler Szot and Srivatsan Rajagopal and Alex Ong and Bhavana Mehta and Beren Millidge},
      year={2026},
      journal={arXiv preprint arXiv:2605.05365},
}

@article{lim2023equivariant,
  author  = {Lim, Lek-Heng and Nelson, Bradley J.},
  title   = {What Is an Equivariant Neural Network?},
  journal = {Notices of the American Mathematical Society},
  volume  = {70},
  number  = {4},
  pages   = {619--625},
  year    = {2023},
}

@article{kondor2025principles,
  author  = {Kondor, Risi},
  title   = {The Principles Behind Equivariant Neural Networks for Physics and Chemistry},
  journal = {Proceedings of the National Academy of Sciences},
  volume  = {122},
  number  = {41},
  pages   = {e2415656122},
  year    = {2025},
}

@article{bao2019equivariant,
  title={Equivariant neural networks and equivarification},
  author={Bao, Erkao and Lu, Jingcheng and Song, Linqi and Hart-Hodgson, Nathan and Parson, William and Zhou, Yanheng},
  journal={arXiv preprint arXiv:1906.07172},
  year={2019}
}

@article{du2026uncovering,
  title={Uncovering Symmetry Transfer in Large Language Models via Layer-Peeled Optimization},
  author={Du, Zhehang and He, Hangfeng and Su, Weijie},
  journal={arXiv preprint arXiv:2605.12756},
  year={2026}
}

@article{yuan2026nora,
  title={Nora: Normalized Orthogonal Row Alignment for Scalable Matrix Optimizer},
  author={Yuan, Jinghui and Zou, Jiaxuan and Wang, Shuo and Liu, Yong and Nie, Feiping},
  journal={arXiv preprint arXiv:2605.03769},
  year={2026}
}

@inproceedings{xue2024openmoe,
  title = {{OpenMoE}: An Early Effort on Open {Mixture-of-Experts} Language Models},
  author = {Xue, Fuzhao and Zheng, Zian and Fu, Yao and Ni, Jinjie and Zheng, Zangwei and Zhou, Wangchunshu and You, Yang},
  booktitle = {Proceedings of the International Conference on Machine Learning (ICML)},
  year = {2024},
}

@article{shen2024jetmoe,
  title={{JetMoE}: Reaching {Llama2} Performance with 0.1{M} Dollars}, 
  author={Shen, Yikang and Guo, Zhen and Cai, Tianle and Qin, Zengyi},
  journal={arXiv preprint arXiv:2404.07413},
  year={2024}
}

@article{lau2025polargrad,
	title={\textsc{PolarGrad}: A Class of Matrix-Gradient Optimizers from a Unifying Preconditioning Perspective},
	author={Lau, Tim Tsz-Kit and Qi Long and Weijie Su},
	year={2025},
	journal={arXiv preprint arXiv:2505.21799}
}

@InProceedings{zeng2019global,
  title = "Global convergence of block coordinate descent in deep learning",
  author = 	 {Jinshan Zeng and Tim Tsz-Kit Lau and Shaobo Lin and Yuan Yao},
  booktitle = {Proceedings of the International Conference on Machine Learning (ICML)},
  year = 	 {2019},
}

@InProceedings{lau2018proximal,
	title = "A Proximal Block Coordinate Descent Algorithm for Deep Neural Network Training",
	author = 	 {Tim Tsz-Kit Lau and Jinshan Zeng and Baoyuan Wu and Yuan Yao},
	booktitle = {International Conference on Learning Representations (ICLR), Workshop Track},
	year = 	 {2018},
}

\newpage
\appendix
\addcontentsline{toc}{section}{\protect\textbf{Appendix}}
\numberwithin{equation}{section}
\numberwithin{theorem}{section}
\numberwithin{algorithm}{section}
\numberwithin{figure}{section}
\numberwithin{table}{section}

\begin{center}
    {\LARGE\bfseries\sffamily Appendix}
\end{center}

\tableofcontents

\newpage
\paragraph{Organization.}
In the appendix, we provide discussion on further related work (\Cref{sec:further_related}), and supplementary technical background materials on amtrix analysis (\Cref{sec:supp}). Furthermore, we illustrate the geometric misalignment of coordinate-wise adaptive gradient methods in \Cref{sec:misalignment}. We then provide omitted proofs from the main text in \Cref{sec:proofs}, as well as details of the implementation of practical optimizers \Cref{sec:practical}. We then give the convergence analysis in \Cref{sec:conv}. We also provide details of numerical experiments in \Cref{sec:details_expt}.

\section{Further Related Work}
\label{sec:further_related}
We further discuss connections between our framework and other existing paradigms for designing matrix-gradient optimizers in modern deep learning.

\subsection{Non-Euclidean Norm-Based Steepest Descent and LMO-Based Frameworks}
A prominent line of recent work interprets \Muon and related methods through the lens of non-Euclidean steepest descent, trust-region methods, and linear minimization oracle (LMO) frameworks \citep{jordan2024muon,bernstein2025deriving,bernstein2024old,bernstein2024modular,pethick2025training,kovalev2025understanding,riabinin2025gluon}. In these views, \Muon can be understood as a normalized steepest descent method with respect to a non-Euclidean matrix norm---most notably the spectral norm---and this perspective has also been used to connect \Muon to optimizers such as \signSGD\ \citep{bernstein2018signsgd} and \Shampoo \citep{gupta2018shampoo}. 
Related work has further explored alternative norm choices and generalized preconditioned steepest descent constructions \citep{crawshaw2025exploration,veprikov2025preconditioned,kravatskiy2025ky,xu2026width}.

While these frameworks provide useful algorithmic interpretations of \Muon and its variants, they leave open a fundamental question: what is the \emph{principled} criterion for choosing the ``correct'' norm for a given layer? In practice, existing prescriptions are not always fully aligned with actual usage. For example, \citet{large2024scalable} suggest using the maximum column $\ell_2$-norm for embedding layers, whereas in practice \AdamW is often used for embeddings while \Muon is applied to other matrix parameters in \texttt{modded-nanogpt} speedrunning \citep{modded_nanogpt_2024}. More broadly, several works suggest that one may use various $\ell_p\to\ell_q$ operator norms for steepest descent or LMO for embedding and LM head layers (see e.g., \citep{pethick2025training}). Our theoretical development suggests that such norm choices might generally not be geometrically principled for matrix-valued parameters.

In contrast, our symmetry-based analysis provides a concrete design principle. For matrix parameters whose geometry should respect orthogonal changes of basis, only unitarily invariant matrix norms are appropriate, since they are precisely the norms compatible with orthogonal invariance. By comparison, the $\ell_p\to\ell_q$ operator norm is generally \emph{not} unitarily invariant, except in the special case $p=q=2$, which reduces to the spectral norm. Accordingly, non-unitarily invariant norms typically induce an incorrect matrix geometry. This offers a possible explanation for the gap between certain norm-based theoretical prescriptions and empirical practice.

\subsection{Modular Norm Theory}
Our right-spectral viewpoint is closely related in spirit to modular norm theory \citep{bernstein2024modular,large2024scalable}: both seek architecture-aware optimizer geometries derived from structural properties of the module rather than from an ambient coordinate system. The two approaches are nevertheless conceptually distinct.

Modular norm theory derives updates from steepest descent under module-adapted operator norms, and is primarily motivated by scale transfer across width and depth through recursively constructed global norms. By contrast, right-spectral optimizers are derived directly from left-permutation and right-orthogonal equivariance, leading to update rules of the form \eqref{eqn:right_spectral}, namely spectral transformations determined by the right Gram matrix $D^\top D$. In this sense, right-spectral optimizers form a symmetry-derived subclass of LPRO-equivariant updates, whereas modular-norm-based methods may also include coordinate-dependent row-wise or column-wise transformations that are compatible with only part of the underlying symmetry.

From this perspective, the two frameworks may be viewed as complementary. Modular norm theory aims to provide the correct \emph{scale invariance}, namely how learning rates and normalized updates should behave as architecture width and depth vary. Our spectral framework instead aims to provide the correct \emph{directional geometry}, namely how update directions should respect the invariance structure of different layer types. Their combination suggests the possibility of a fully geometry-aware optimizer that uses modular-norm-based global scaling together with layerwise spectral, right-spectral, or left-spectral update directions dictated by symmetry.

\subsection{Rotation-Based Optimizers}
The recent work \citep{gong2026aro} is perhaps the closest in spirit to our approach, but there are several important differences. First, we do not assume that the layerwise loss function itself is rotationally, or more generally bi-orthogonally, invariant. Rather, we derive optimizer classes from the symmetry structure of the parameter geometry and the corresponding equivariance properties of update rules. This distinction matters, since exact rotational invariance of the layerwise loss may be difficult to verify and can fail in practice.

Second, we do not advocate a single update rule for all matrix-valued parameters. Instead, our framework emphasizes that different layer types admit different symmetry groups and therefore should, in principle, use different optimizer geometries. In particular, embedding and LM head matrices possess different equivariance properties from weight matrices in linear and attention layers, which leads naturally to an architecture--optimizer co-design perspective.

Third, the update rules in ARO for matrix parameters are based only on left multiplication by a rotation matrix. By contrast, our framework includes full spectral, right-spectral, left-spectral, row-norm-based, and hybrid classes, depending on the relevant symmetry structure of the layer. In this sense, our theory provides a broader symmetry-based taxonomy of matrix-gradient optimizers.

\section{Supplemental Technical Background on Matrix Analysis}
\label{sec:supp}
We include several technical definitions arising in matrix analysis, which is closely related to the derivation of spectral optimizers via steepest descent. Details of these materials can be found in \citep{lewis1996eigenvalue,horn2012matrix,horn1994topics,bhatia2013matrix}. 

Let us recall the notion of unitarily invariant norms, which provide the norm-level analogue of left-right orthogonal symmetry. Since we work over the field of real numbers $\RR$, ``unitarily invariant'' here is equivalent to invariance under left and right orthogonal transformations. 

\begin{definition}[Unitarily invariant norms]
    \label{def:unitarily_invariant}
	A norm $\norm{\cdot}\colon\RR^{m\times n}\to\Rp$ is said to be \emph{unitarily invariant} if $\norm{A} = \norm{UAV}$ for all $A\in\RR^{m\times n}$ and all orthogonal matrices $U\in\OO^m$ and $V\in\OO^n$. A \emph{unitarily invariant matrix norm} is a unitarily invariant norm on $\mathbb{R}^{m\times n}$ that is also submultiplicative.
\end{definition}

Unitarily invariant norms can be completely characterized through symmetric gauge functions acting on singular values.

\begin{definition}[Symmetric gauge functions]
	A function $\psi\colon\RR^r\to\Rp$ is a \emph{symmetric gauge function} if:
	(i) $\psi$ is a norm;
	(ii) $\psi(|x|)=\psi(x)$ for all $x\in\RR^r$, where $|\cdot|$ is understood coordinatewise; and
	(iii) $\psi(Px)=\psi(x)$ for all permutation matrices $P\in\Prob^r$ and all $x\in\RR^r$.
\end{definition}

\begin{proposition}
	A norm $\norm{\cdot}\colon\RR^{m\times n}\to\Rp$ is unitarily invariant if and only if there exists a symmetric gauge function $\psi$ such that $\norm{W} = \psi(\sigma(W)) = (\psi\circ\sigma)(W)$, where $\sigma(W)$ is the vector of singular values of $W$ arranged in descending order.
\end{proposition}

\begin{example}
	Important examples of unitarily invariant matrix norms include the Schatten $p$-norms $\matsnorm{\cdot}{p}$ for $p\in\left[1, \infty\right]$, including the nuclear norm $\nucnorm{\cdot}$, Frobenius norm $\fronorm{\cdot}$, and the spectral norm $\specnorm{\cdot}$. Another notable example is the Ky Fan $k$-norm.
\end{example}

\section{Geometric Misalignment of Coordinate-wise Adaptive Gradient Methods}
\label{sec:misalignment}
If an optimizer for matrix parameters in linear layers does not respect the intrinsic geometry of the parameter space through bi-orthogonal equivariance, several issues arise. In particular, for coordinate-wise adaptive gradient methods, the optimizer iterates depend on arbitrary coordinate choices: rotating the input space can change the optimizer itself, leading to different training dynamics under equivalent reparameterizations.

For a vector parameter $w\in\mathbb{R}^n$, the general form of a coordinate-wise adaptive gradient method is
\[
(\forall k\in\NN)\qquad w_{k+1} = w_k - \gamma_k\scrU(d_k), 
\]
where $d_k$ is an update direction and $\scrU\colon\RR^n\to\RR^n$ is a vector-valued map. A simple but important observation is that optimizers built only from linear operations, such as additions, subtractions, scalar multiplications, and their linear combinations, behave analogously for vector and matrix parameters. By contrast, coordinate-wise adaptive methods such as \Adam rely on elementwise nonlinear operations, including division, squaring, and square roots. These operations distort the geometry of the original update direction, whether it is formed from the gradient itself or from a momentum term.

While spectral methods (e.g., spectral descent and \Muon; formally introduced in \Cref{subsec:spectral_optim}) respect bi-orthogonal equivariance, coordinate-wise methods are typically equivariant only under the much smaller signed permutation group. As a result, they fail to respect the intrinsic matrix geometry of the optimization problem. This mismatch is especially problematic in deep learning, where the optimization landscape often exhibits an intrinsically low-dimensional structure. Spectral optimizers are designed to adapt to this geometry, whereas coordinate-wise methods largely ignore it. Moreover, as model size increases, the advantage of such geometry-aware methods typically becomes more pronounced.

\paragraph{Sign descent. }
For a matrix gradient $G$, sign descent or \signSGD \citep{bernstein2018signsgd} uses the coordinate-wise signs $\sgn(G)$ of the gradient as the update direction, which satisfies $\dotpF{G}{\sgn(G)} = \onenorm{G}$. 
Thus, sign descent is governed by the duality associated with the coordinate-wise $\ell_\infty$-norm, rather than by any unitarily invariant matrix geometry (see \Cref{def:unitarily_invariant}).

\paragraph{\Adam.}
More generally, \Adam \citep{kingma2015} can be viewed as a smoothed version of sign descent, since its update is dominated by coordinate-wise normalization. In particular, \Adam applies the coordinate-wise scaling (omitting bias corrections for simplicity)
\[
W_{ij}\leftarrow W_{ij}-\gamma\cdot \frac{M_{ij}}{\sqrt{V_{ij}}+\varepsilon},
\]
where $M_{ij}$ and $V_{ij}$ denote the first- and second-moment statistics (with bias corrections) at coordinate $(i,j)$, $\gamma>0$ is the learning rate , and $\varepsilon>0$ is a small constant. This update does not transform equivariantly under bi-orthogonal reparameterizations $W\mapsto \tW\coloneqq PWQ^\top$ for $P\in\OO^m$ and $Q\in\OO^n$. Consequently, the update direction of \Adam depends on how the entries of $W$ are indexed, rather than only on the intrinsic geometry of $W$ as a matrix.

Furthermore, coordinate-wise methods tend to inject high-rank coordinate noise even when the underlying gradient $G$ is low rank. In contrast, spectral methods act directly on the matrix structure of $G$, and therefore remain aligned with the low-dimensional geometry that frequently governs deep network optimization.

\paragraph{Sign descent as a special case of spectral descent.}
Following similar arguments in \citep{lau2025polargrad}, we further interpret sign descent as spectral descent applied to the diagonal matrization of the vectorized matrix parameter. Define the \emph{diagonal matrization of the vectorization} of $G$ by $\tG \coloneqq \Diag(\vec(G))\in\RR^{mn\times mn}$. The following lemma makes this connection precise.

\begin{lemma}
   	\label{lem:sign_as_polar_diagonal}
   	Let $G\coloneqq \nabla_W f(W)\in\RR^{m\times n}$ contain no exact zero entries, and define $\tG \coloneqq \Diag(\vec(G))\in\RR^{mn\times mn}$. Then the orthogonal polar factor of $\tG$ is $\polar(\tG) = \Diag(\sgn(\vec(G)))$. 
   	Consequently, $\sgn(G)=\reshape(\diag(\polar(\tG)), m,n)$.
\end{lemma}

\begin{proof}
	Since $\tG=\Diag(\vec(G))$ is diagonal, its singular values are given by the absolute values of its diagonal entries. Hence we may write its singular value decomposition as
	\[ \tG = I_{mn}\,\Diag(|\vec(G)|)\,\Diag(\sgn(\vec(G))), \]
	or equivalently,
	\[ \tG = \Diag(\sgn(\vec(G)))\,\Diag(|\vec(G)|). \]
	Therefore, by the definition of the orthogonal polar factor, we have $\polar(\tG)=\Diag(\sgn(\vec(G)))$. 
	Taking the diagonal of $\polar(\tG)$ and reshaping it back to an $m\times n$ matrix yields
	\[ \reshape(\diag(\polar(\tG)), m, n) = \reshape(\sgn(\vec(G)), m, n) = \sgn(G). \]
\end{proof}
In other words, coordinate-wise sign descent can be viewed as spectral gradient descent applied to the highly degenerate diagonal lifting $\tG=\Diag(\vec(G))\in\RR^{mn\times mn}$. Thus, sign-based methods inherit only the trivial geometry of this pathological diagonal representation, rather than the intrinsic matrix geometry of $G$ itself. This highlights a \emph{fundamental geometric mismatch} in coordinate-wise updates, which is further worsened when the parameter dimensions $m$ and $n$ and hence model sizes grow.

\section{Proofs of Main Text}
\label{sec:proofs}

\begin{proof}[Proof of \Cref{prop:momentum}]
We also state the desired results for Polyak and Nesterov momentum. 

For EMA momentum, we define $M_k = \beta M_{k-1} + (1-\beta)G_k$ with $M_{-1}=0$, and, for the transformed sequence, $\tM_k = \beta \tM_{k-1} + (1-\beta)\tG_k$ with $\tM_{-1}=0$. Then we have $\tM_k = PM_kQ^\top$ and $\polar(\tM_k)=P\,\polar(M_k)Q^\top$. 

For Polyak heavy-ball momentum, we define $M_k = \beta M_{k-1} + G_k$ with $M_{-1}=0$, and $\tM_k = \beta \tM_{k-1} + \tG_k$ with $\tM_{-1}=0$. Then we have $\tM_k = PM_kQ^\top$ and $\polar(\tM_k)=P\,\polar(M_k)Q^\top$.

We first prove the stated transformation for the momentum sequences. For the EMA recursion, assume inductively that $\tM_{k-1}=PM_{k-1}Q^\top$. 
Using the bi-orthogonal equivariance of the gradient, we have 
\[\tG_k = \nabla_W f(\tW_k) = \nabla_W f(PW_kQ^\top) = P\nabla_W f(W_k)Q^\top = PG_kQ^\top. \]
Hence
\begin{multline*}
\tM_k
= \beta \tM_{k-1} + (1-\beta)\tG_k
= \beta PM_{k-1}Q^\top + (1-\beta)PG_kQ^\top \\
= P(\beta M_{k-1} + (1-\beta)G_k)Q^\top 
= PM_kQ^\top.
\end{multline*}
Thus the EMA momentum is bi-orthogonally equivariant. The Polyak case is identical, replacing $(1-\beta)$ by $1$, which gives $\tM_k = PM_kQ^\top$.

For Nesterov momentum, the momentum recursion is the same as in the Polyak case, so $\tM_k = PM_kQ^\top$.
Therefore, 
\begin{equation*}
\tN_k 
= \tG_k + \beta \tM_k 
= PG_kQ^\top + \beta PM_kQ^\top 
= P(G_k + \beta M_k)Q^\top 
= PN_kQ^\top.
\end{equation*}
All claims about orthogonal polar factors follow from \eqref{eqn:polar}. For example,
\[
\polar(\tN_k) = \polar(PN_kQ^\top) = P\,\polar(N_k)Q^\top.
\]
The proofs for the momentum polar factors are identical.
\end{proof}

\begin{proof}[Proof of \Cref{thm:right_spectral}]
   	Let $P\in\Prob^v$, $R\in\OO^d$, and $D\in\RR^{v\times d}$. Since $P$ is a permutation matrix, we have $P^\top P=I_v$. Therefore,
   	\[ (PDR^\top)^\top(PDR^\top) = R D^\top P^\top P D R^\top = R D^\top D R^\top. \]
   	Using the orthogonal equivariance of $\Phi$, it follows that 
    \[\Phi((PDR^\top)^\top(PDR^\top)) = \Phi(RD^\top D R^\top) = R\Phi(D^\top D)R^\top. \]
   	Hence we obtain 
   	\begin{multline*}
   		\scrU_{\sfR}(PDR^\top) = PDR^\top\Phi((PDR^\top)^\top(PDR^\top))
   		= PDR^\top (R\,\Phi(D^\top D)R^\top) \\
   		= PD\Phi(D^\top D)R^\top 
   		= P\,\scrU_{\sfR}(D)R^\top.
   	\end{multline*}
\end{proof}

\begin{proof}[Proof of \Cref{prop:closure_composition}]
For any $D\in\RR^{v\times d}$, permutation matrix $P\in\Prob^v$, and orthogonal matrix $R\in\OO^d$,
\[
(\scrU_2\circ \scrU_1)(PDR^\top) = \scrU_2(\scrU_1(PDR^\top)) = \scrU_2(P\,\scrU_1(D)R^\top) = P\,\scrU_2(\scrU_1(D))R^\top.
\]
Hence $\scrU_2\circ\scrU_1$ is left-permutation and right-orthogonal equivariant.
\end{proof}

\begin{proof}[Proof of \Cref{prop:swiglu_permutation_symmetry}]
	Since $\sigma$ is applied coordinatewise and $P$ is a permutation matrix, $\sigma(PW_{\mathrm{gate}}x)=P\sigma(W_{\mathrm{gate}}x)$. Moreover, $(Pu)\odot(Pv)=P(u\odot v)$ for all $u,v\in\RR^{d_{\mathrm{ff}}}$. Hence
	\begin{align*}
		\mathrm{SwiGLU}(x; \tW_{\mathrm{gate}}, \tW_{\mathrm{up}}, \tW_{\mathrm{down}})
		&= W_{\mathrm{down}}P^\top \left(\sigma(PW_{\mathrm{gate}}x)\odot (PW_{\mathrm{up}}x)\right) \\
		&= W_{\mathrm{down}}P^\top P\left(\sigma(W_{\mathrm{gate}}x)\odot(W_{\mathrm{up}}x)\right) \\
		&= W_{\mathrm{down}}\left(\sigma(W_{\mathrm{gate}}x)\odot(W_{\mathrm{up}}x)\right).
	\end{align*}
\end{proof}

\begin{proof}[Proof of \Cref{prop:router_compatible_updates}]
We first record two elementary identities. Since $P\in\Prob^e$ satisfies $P\One_e=\One_e$ and $P^\top\One_e=\One_e$, the centering projector $    \Pi_\perp=I_e-\frac1e\One_e\One_e^\top$ commutes with $P$, i.e., $\Pi_\perp P=P\Pi_\perp$. Moreover, for any $a\in\RR^d$, $\Pi_\perp \One_e a^\top=0$. Therefore, if $\widetilde D = PD+\One_e a^\top$, then its centered direction is
\[
    \widetilde D_c = \Pi_\perp \widetilde D = \Pi_\perp(PD+\One_e a^\top) = P\Pi_\perp D = P D_c.
\]
We first consider the centered left-spectral update. Since $\widetilde D_c=P D_c$, we have $\widetilde D_c\widetilde D_c^\top = P D_cD_c^\top P^\top$. 
Using the permutation equivariance of $\Psi$,
\[
    \Psi(\widetilde D_c\widetilde D_c^\top)
    =
    \Psi(PD_cD_c^\top P^\top)
    =
    P\Psi(D_cD_c^\top)P^\top .
\]
Hence
\[
    \scrU_{\sfL}^{\router}(\widetilde D)
    =
    \Psi(\widetilde D_c\widetilde D_c^\top)\widetilde D_c 
    =
    P\Psi(D_cD_c^\top)P^\top P D_c 
    =
    P\Psi(D_cD_c^\top)D_c 
    =
    P\,\scrU_{\sfL}^{\router}(D).
\]
This proves expert-permutation equivariance and shared-row-shift invariance for $\scrU_{\sfL}^{\router}$.

It remains to verify horizontality. Since $D_c=\Pi_\perp D$, every column of $D_c$ lies in $\One_e^\perp$. By assumption, $\Psi(D_cD_c^\top)$ maps the centered expert subspace to itself. Therefore every column of $\Psi(D_cD_c^\top)D_c$ also lies in $\One_e^\perp$, and hence
\[
    \One_e^\top \scrU_{\sfL}^{\router}(D) = 0 .
\]

We next consider the projected centered row-norm update. Define
\[
    \sfA(D_c) \coloneqq \Diag(\eta(\|D_{c,1:}\|_2),\dots,\eta(\|D_{c,e:}\|_2)).
\]
Since $\widetilde D_c=P D_c$, the row norms of $\widetilde D_c$ are the corresponding permutation of the row norms of $D_c$. Hence
\[
    \sfA(\widetilde D_c) = P \sfA(D_c)P^\top .
\]
Using again that $\Pi_\perp P=P\Pi_\perp$, we obtain
\[
    \scrU_{\row}^{\router}(\widetilde D)
    =
    \Pi_\perp \sfA(\widetilde D_c)\widetilde D_c 
    =
    \Pi_\perp P \sfA(D_c)P^\top P D_c 
    =
    P\Pi_\perp \sfA(D_c)D_c 
    =
    P\,\scrU_{\row}^{\router}(D).
\]
Thus the projected centered row-norm update is also expert-permutation equivariant and shared-row-shift invariant.

Finally, horizontality of the row-norm update follows directly from the final projection:
\[
    \One_e^\top \scrU_{\row}^{\router}(D) = \One_e^\top \Pi_\perp \sfA(D_c)D_c = 0.
\]
Therefore both update families are router-compatible.
\end{proof}

\begin{proof}[Proof of \Cref{thm:orth_invar}]
	Suppose first that $\scrU$ is of the stated spectral form. Let $P\in\OO^m$ and $Q\in\OO^n$ be orthogonal. If $D=U\Diag(\sigma(D))V^\top$, then $PDQ^\top=(PU)\Diag(\sigma(D))(QV)^\top$ is a singular value decomposition of $PDQ^\top$. Therefore, we have
	\[ \scrU(PDQ^\top) = (PU)\Diag(\psi(\sigma(D)))(QV)^\top = P\,\scrU(D)Q^\top, \]
	so $\scrU$ is bi-orthogonally equivariant.
	
	Conversely, suppose that $\scrU$ is bi-orthogonally equivariant. Let $D=U\Diag(\sigma(D))V^\top$ be a singular value decomposition of $D$. Then we have
	\[ \scrU(D) = \scrU(U\Diag(\sigma(D))V^\top) = U\scrU(\Diag(\sigma(D)))V^\top. \]
	Thus the action of $\scrU$ is completely determined by its action on diagonal matrices of singular values. 
    Since bi-orthogonal equivariance ensures $\scrU$ commutes with arbitrary diagonal sign matrices, forcing all off-diagonal elements to be zero, let us define $\Diag(\psi(s))\coloneqq \scrU(\Diag(s))$, where $s\in\Rp^r$. 
	The bi-orthogonal equivariance of $\scrU$ further implies that this definition is independent of the particular singular value decomposition and that $\psi$ is absolutely symmetric with respect to permutations and sign changes compatible with singular-value representations. Hence we have $\scrU(D)=U\Diag(\psi(\sigma(D)))V^\top$, which is exactly the claimed spectral form.
\end{proof}

\section{Implementation Details of Practical Optimizers}
\label{sec:practical}
We provide the implementation details of our proposed practical optimizers in this section. 

\subsection{Numerical Algorithm for Matrix Inverse Square Root via \textsc{Polar Express}}
We illustrate below the algorithm for computing matrix inverse square root via \textsc{Polar Express}, which is based on the intrinsic connection between polynomial iterations for computing the orthogonal polar factor and those for computing the inverse square root of a square matrix, stated in the following theorem. 

\begin{theorem}[\citet{higham1997stable}]
Let $A\in\RR^{n\times n}$ be a square matrix with nonnegative eigenvalues. Consider any iteration of the form $X_{k+1} = X_k h(X_k^2)$ that converges to $\msgn(X_0)$ for $X_0 = \left(\begin{smallmatrix}
0 & A\\ I_n & 0
\end{smallmatrix}\right)$ with order of convergence $q$. Then in the coupled iteration $X_{k+1} = X_k h(Y_k X_k)$, $Y_{k+1} = h(Y_k X_k)Y_k$, with $X_0=A$ and $Y_0=I_n$, we have $X_k \to A^{\half}$ and $Y_k \to A^{-\half}$, both with order of convergence $q$. Here $\msgn$ denotes the matrix sign function. 
\end{theorem}

\begin{algorithm}[h!]
	\caption{Matrix Inverse Square Root via \textsc{Polar Express} \citep{amsel2025polar}}
	\label{alg:inv_sqrt}
    	\begin{algorithmic}[1]
    		\REQUIRE $A\in\RR^{n\times n}$, $\varepsilon>0$, $K\in\NN^*$, sequence of \textsc{Polar Express} coefficients $\{(a_k, b_k, c_k)\}_{k=1}^K$
            \STATE $A = (A + A^\top)/2 + \varepsilon I_n$
            \STATE $\alpha = 1.02\fronorm{A} + \varepsilon$
            \STATE $Y_1 = A/\alpha$
            \STATE $Z_1 = I_n$
    		\FOR{$k=1, \ldots, K$}
                \STATE $(\bar{a}_k, \bar{b}_k, \bar{c}_k) = (a_k+b_k+c_k, -b_k-2c_k, c_k)$
    			\STATE $R_k = I_n - Z_kY_k$
    			\STATE $R_k = (R_k + R_k^\top)/2$ (optional matrix symmetrization)
   				\STATE $M_k = \bar{a}_kI_n + \bar{b}_kR_k + \bar{c}_kR_k^2$
   				\STATE $Y_{k+1} = Y_k M_k$
   				\STATE $Z_{k+1} = M_k Z_k$
    		\ENDFOR
    		\ENSURE $Z_{K+1}/\sqrt{\alpha}$
	\end{algorithmic}
\end{algorithm}

\newpage
\section{Convergence Analysis of Symmetry-Compatible Optimizers}
\label{sec:conv}
In this section, we study the convergence of several symmetry-compatible optimizer classes introduced above, including full spectral, one-sided spectral, row-norm-based, and hybrid optimizers. To present the analysis in a unified way, we begin from a general first-order iteration and state the basic assumptions once. The subsequent subsections then specialize these assumptions to each optimizer class.

\subsection{General Update Scheme and Standing Assumptions}
We consider the generic iteration
\begin{equation}\label{eqn:general_update_scheme}
(\forall k\in\NN)\qquad W_{k+1}=W_k-\gamma_k\calT(G_k), \qquad G_k\coloneqq \nabla f(W_k),
\end{equation}
where $W_k\in\RR^{m\times n}$, $\gamma_k>0$ is the learning rate, and $\calT\colon\RR^{m\times n}\to\RR^{m\times n}$ is an update map whose precise form depends on the optimizer class under consideration. We do not consider momentum in our analysis for simplicity and leave it to future work. The layerwise loss function $f$ is assumed to satisfy the following standard regularity conditions.

\begin{assumption}[$L$-smoothness]
\label{ass:Lipschitz_smooth}
Let $f\colon\RR^{m\times n}\to\oRR$ be differentiable and $L$-smooth with respect to the Frobenius norm, i.e., there exists $L\in(0,\infty)$ such that 
\[
(\forall X,Y\in\RR^{m\times n})\qquad \fronorm{\nabla f(X)-\nabla f(Y)} \le L\fronorm{X-Y}.
\]
\end{assumption}

\begin{assumption}[$\mu$-Polyak--\L{}ojasiewicz]
\label{ass:PL}
Let $f\colon\RR^{m\times n}\to\oRR$ satisfy the $\mu$-Polyak--\L{}ojasiewicz (\PL) inequality: 
\[
(\forall X\in\RR^{m\times n})\qquad \fronorm{\nabla f(X)}^2 \ge 2\mu\bigl(f(X)-f^\star\bigr),
\]
where $f^\star\coloneqq \inf f$. 
\end{assumption}
The \PL condition will only be needed to obtain linear convergence. Under smoothness alone, we will still obtain monotonic descent and sublinear convergence to stationarity.

\begin{remark}[On the stochastic setting]
In large-scale deep learning, one typically replaces the full gradient $G$ by a stochastic gradient estimator $\widehat G$. However, for the symmetry-compatible update maps considered here, including spectral, one-sided spectral, row-norm-based, and hybrid updates, the map $\calT$ is generally nonlinear. As a result, unbiasedness of the stochastic gradient does not imply unbiasedness of the resulting update direction. A rigorous stochastic analysis therefore requires additional assumptions directly on the expected alignment and norm of $\calT(\widehat G)$, and is left for future work.
\end{remark}

Throughout the section, the basic descent estimate follows from the standard smoothness inequality:
\begin{equation*}
f(W_{k+1}) \le f(W_k) -\gamma_k\dotpF{G_k}{\calT(G_k)} +\frac{L\gamma_k^2}{2}\,\fronorm{\calT(G_k)}^2.
\end{equation*}
Thus, the convergence analysis for each optimizer class reduces to controlling two quantities $\dotpF{G}{\calT(G)}$ and $\fronorm{\calT(G)}^2$. The relevant bounds depend on the geometry of the update map $\calT$, and are stated separately in the subsections below.

We now specialize \eqref{eqn:general_update_scheme} to each symmetry-compatible optimizer class introduced earlier. In each case, the key step is to identify the geometry-dependent alignment term $\dotpF{G}{\calT(G)}$ and the corresponding update norm $\fronorm{\calT(G)}^2$, which together determine the admissible learning rates and convergence rates. 

The optimizer classes considered in this section fall into two broad regimes. The first consists of scale-compatible updates, for which $\fronorm{\calT(G)}$ scales proportionally to $\fronorm{G}$. This includes standard spectral, one-sided spectral, and bounded row-norm-based updates. The second consists of fully normalized updates, such as polar or row-normalized directions, whose magnitude is controlled by rank or support rather than gradient norm. The former admit convergence bounds through uniform alignment and norm-control constants, whereas the latter are more naturally analyzed through geometry-dependent ratio quantities.

\subsection{Full Spectral Optimizers}
We begin with full spectral optimizers. The following assumption captures the two structural properties needed for convergence: positive alignment with the gradient and control of the update norm. We specialize \eqref{eqn:general_update_scheme} to the full spectral case, where $\calT$ is a spectral operator.

\begin{assumption}[Singular-value alignment and boundedness]
\label{ass:singular_alignment}
Let $\calT\colon\RR^{m\times n}\to\RR^{m\times n}$ be the spectral update map defined by $\calT(G)=U\Diag(\psi(\sigma(G)))V^\top$ whenever $G=U\Diag(\sigma(G))V^\top$ is a singular value decomposition of $G$, for some absolutely symmetric map $\psi\colon\Rp^r\to\RR^r$. Assume there exist constants $0<c_1\le c_2<\infty$ such that for all $s\in\Rp^r$, 
\[\sum_{i=1}^r s_i\psi_i(s)\ge c_1\sum_{i=1}^r s_i^2 \quad \text{and} \quad \sum_{i=1}^r \psi_i(s)^2\le c_2\sum_{i=1}^r s_i^2, \]
where $r=\min\{m,n\}$.
\end{assumption}

\begin{lemma}[Alignment and norm bounds]
\label{lem:T_alignment}
Under \Cref{ass:singular_alignment}, for all $G\in\RR^{m\times n}$, we have $\dotpF{G}{\calT(G)}\ge c_1\fronorm{G}^2$ and $\fronorm{\calT(G)}^2\le c_2\fronorm{G}^2$. 
\end{lemma}

\begin{proof}
If we write $s=\sigma(G)$, then we have $\calT(G)=U\Diag(\psi(s))V^\top$. Using orthogonal invariance of the Frobenius inner product, we have $\dotpF{G}{\calT(G)} = \dotpF{\Diag(s)}{\Diag(\psi(s))} = \sum_{i=1}^r s_i\psi_i(s)$. By \Cref{ass:singular_alignment}, we have $\dotpF{G}{\calT(G)} \ge c_1\sum_{i=1}^r s_i^2 = c_1\fronorm{G}^2$. Similarly, we also have $\fronorm{\calT(G)}^2 = \sum_{i=1}^r \psi_i(s)^2 \le c_2\sum_{i=1}^r s_i^2 = c_2\fronorm{G}^2$. 
\end{proof}

\begin{theorem}[Descent lemma for spectral optimizers]
\label{thm:descent_spectral}
Suppose \Cref{ass:Lipschitz_smooth,ass:singular_alignment} hold. Then the iteration \eqref{eqn:general_update_scheme} satisfies 
\begin{equation}\label{eqn:descent_spectral}
    f(W_{k+1}) \le f(W_k) -\gamma_k\dotpF{G_k}{\calT(G_k)} +\frac{L\gamma_k^2}{2}\fronorm{\calT(G_k)}^2.
\end{equation}
Consequently, we have $f(W_{k+1}) \le f(W_k) -\left(c_1\gamma_k-Lc_2\gamma_k^2/2\right)\fronorm{G_k}^2$. In particular, if $\gamma_k\in\left(0, 2c_1/(Lc_2)\right)$, then $f(W_{k+1})\le f(W_k)$.
\end{theorem}

\begin{proof}
By $L$-smoothness of $f$, we have 
\[
f(W_{k+1}) \le f(W_k) + \dotpF{\nabla f(W_k)}{W_{k+1}-W_k} + \frac{L}{2}\fronorm{W_{k+1}-W_k}^2.
\]
Using $W_{k+1}-W_k=-\gamma_k\calT(G_k)$ and $G_k=\nabla f(W_k)$, we obtain \eqref{eqn:descent_spectral}. 
Now apply \Cref{lem:T_alignment} to obtain $\dotpF{G_k}{\calT(G_k)}\ge c_1\fronorm{G_k}^2$ and $\fronorm{\calT(G_k)}^2\le c_2\fronorm{G_k}^2$, which yields $f(W_{k+1}) \le f(W_k) -\left(c_1\gamma_k-Lc_2\gamma_k^2/2\right)\fronorm{G_k}^2$. 
\end{proof}

\begin{theorem}[Sublinear convergence to stationarity]
\label{thm:spectral_stationarity}
Suppose \Cref{ass:Lipschitz_smooth,ass:singular_alignment} hold. If the learning rate is constant and satisfies $\gamma\in(0, 2c_1/(Lc_2))$, then
\[
\sum_{k=0}^{T-1}\fronorm{G_k}^2 \le \frac{f(W_0)-f^\star} {\gamma\left(c_1-Lc_2\gamma/2\right)}, 
\quad\text{ and therefore }\quad 
\min_{0\le k<T}\,\fronorm{G_k}^2 \le \frac{f(W_0)-f^\star} {T\gamma\left(c_1-Lc_2\gamma/2\right)}.
\]
\end{theorem}

\begin{proof}
By \Cref{thm:descent_spectral}, $f(W_{k+1}) \le f(W_k) -\gamma\left(c_1-Lc_2\gamma/2\right)\fronorm{G_k}^2$. 
Summing from $k=0$ to $T-1$ yields
\[
f(W_T) \le f(W_0) -\gamma\left(c_1-\frac{Lc_2}{2}\gamma\right) \sum_{k=0}^{T-1}\fronorm{G_k}^2.
\]
Since $f(W_T)\ge f^\star$, rearranging gives $\sum_{k=0}^{T-1}\fronorm{G_k}^2 \le (f(W_0)-f^\star)/(\gamma(c_1-Lc_2\gamma/2))$. The second inequality follows from $\min_{0\le k<T}\,\fronorm{G_k}^2 \le \frac1T\sum_{k=0}^{T-1}\fronorm{G_k}^2$. 
\end{proof}

\begin{theorem}[Linear convergence under the \PL condition]
\label{thm:PL_spectral}
Suppose \Cref{ass:Lipschitz_smooth,ass:PL,ass:singular_alignment} hold. If the learning rate is constant and satisfies $\gamma\in(0, 2c_1/(Lc_2))$, then the spectral iteration \eqref{eqn:general_update_scheme} obeys
\[
f(W_{k+1})-f^\star \le \left(1-2\mu\gamma\left(c_1-\frac{Lc_2}{2}\gamma\right)\right) \left(f(W_k)-f^\star\right).
\]
Hence $f(W_k)-f^\star \le \rho^k\bigl(f(W_0)-f^\star\bigr)$, where $\rho \coloneqq 1-2\mu\gamma\left(c_1-Lc_2\gamma/2\right)\in(0,1)$. 
\end{theorem}

\begin{proof}
By \Cref{thm:descent_spectral}, subtracting $f^\star$ from both sides gives
\[
f(W_{k+1})-f^\star \le f(W_k)-f^\star -\left(c_1\gamma-\frac{Lc_2}{2}\gamma^2\right)\fronorm{G_k}^2.
\]
Using the \PL inequality, we obtain the desired inequality. Since $0<\gamma<2c_1/(Lc_2)$, the factor $c_1-Lc_2\gamma/2$ is positive, so $\rho\in(0,1)$. Iterating the recursion yields $f(W_k)-f^\star\le \rho^k(f(W_0)-f^\star)$. 
\end{proof}

\begin{corollary}[Optimal constant learning rate within this bound]
\label{cor:best_stepsize}
Under the assumptions of \Cref{thm:PL_spectral}, the contraction factor in
\Cref{thm:PL_spectral} is minimized over constant learning rates $\gamma>0$ by $\gamma^\star=c_1/(Lc_2)$, for which $\rho^\star = 1-\mu c_1^2/(Lc_2)$. 
\end{corollary}

\begin{proof}
The contraction factor $\rho(\gamma)=1-2\mu\gamma\left(c_1-Lc_2\gamma/2\right)$ is minimized when the quadratic $-2\mu c_1\gamma+\mu Lc_2\gamma^2$ is minimized, equivalently when $c_1\gamma-Lc_2\gamma^2/2$ is maximized. This occurs at $\gamma^\star=c_1/(Lc_2)$. Substituting into $\rho(\gamma)$ gives $\rho^\star = 1-2\mu\frac{c_1}{Lc_2}\left(c_1-\frac{Lc_2}{2}\frac{c_1}{Lc_2}\right) = 1-\mu c_1^2/(Lc_2)$. 
\end{proof}

\subsubsection{Specialization to Normalized Polar-Type Spectral Methods}
The abstract convergence result above assumes that the spectral update $\calT(G)$ grows proportionally to the gradient norm. This covers many spectral maps, but excludes fully normalized polar updates such as $\calT(G)=\polar(G)=UV^\top$, for which $\fronorm{\calT(G)}^2=\rank(G)$ does not vanish as $G\to 0$ (cf.~\emph{null-gradient consistency} defined in \citep{lau2025polargrad}). To analyze such methods, it is more natural to work with a geometry-dependent ratio between the update norm and its alignment with the gradient, defined below.

\begin{assumption}[Positive alignment]
\label{ass:positive_alignment}
For every $G\in\RR^{m\times n}\setminus\{0\}$, we assume that $\dotpF{G}{\calT(G)} > 0$. 
\end{assumption}

\begin{definition}[Spectral advantage ratio]
\label{def:spectral_adv_ratio}
For a spectral update map $\calT$, define its \emph{spectral advantage ratio}
at $G\neq 0$ by
\[
\mathfrak R_{\calT}(G)
\coloneqq
\frac{\dotpF{G}{\calT(G)}}{\fronorm{\calT(G)}^2}.
\]
Equivalently, its reciprocal is $\mathfrak R_{\calT}^{-1}(G) = \fronorm{\calT(G)}^2/\dotpF{G}{\calT(G)}$. 
\end{definition}

The quantity $\mathfrak R_{\calT}(G)$ measures how much descent is obtained per unit squared update norm. Larger values correspond to more favorable geometry.

\begin{lemma}[Descent lemma in ratio form]
\label{lem:ratio_descent}
Suppose \Cref{ass:Lipschitz_smooth,ass:positive_alignment} hold. Then the iteration \eqref{eqn:general_update_scheme} satisfies
\[
f(W_{k+1}) \le f(W_k) - \left(\gamma_k-\frac{L\gamma_k^2}{2\,\mathfrak R_{\calT}(G_k)}\right)\dotpF{G_k}{\calT(G_k)}.
\]
In particular, if $0<\gamma_k<2\mathfrak R_{\calT}(G_k)/L$, then $f(W_{k+1})\le f(W_k)$.
\end{lemma}

\begin{proof}
The inequality follows from $L$-smoothness of $f$ and the definition of $\mathfrak R_{\calT}(G_k)$. The final claim follows immediately.
\end{proof}

The preceding lemma shows that normalized spectral methods are governed by two distinct quantities: the ratio term $\mathfrak R_{\calT}(G)$, which controls the admissible learning rate through the smoothness bound, and the alignment term $\dotpF{G}{\calT(G)}$, which must still be related to $\fronorm{G}^2$ in order to combine the descent estimate with the \PL inequality.

\begin{theorem}[Linear convergence under \PL and ratio/alignment bounds]
\label{thm:PL_ratio}
Suppose \Cref{ass:Lipschitz_smooth,ass:PL,ass:positive_alignment} hold. Assume there exists a constant $\underline{\mathfrak R}>0$ such that $\mathfrak R_{\calT}(G)\ge \underline{\mathfrak R}$ for all $G\neq 0$, and that there exists $\alpha>0$ such that $\dotpF{G}{\calT(G)}\ge \alpha \fronorm{G}^2$ for all $G\in\RR^{m\times n}$. 
If the learning rate is constant and satisfies $0<\gamma<2\underline{\mathfrak R}/L$, then
\[
f(W_{k+1})-f^\star \le \left(1-2\mu\alpha\gamma \left(1-\frac{L\gamma}{2\underline{\mathfrak R}} \right) \right)\bigl(f(W_k)-f^\star\bigr).
\]
Hence $f(W_k)-f^\star \le \rho^k\bigl(f(W_0)-f^\star\bigr)$, where $\rho \coloneqq 1-2\mu\alpha\gamma\left( 1-L\gamma/(2\underline{\mathfrak R})\right)\in(0,1)$. 
\end{theorem}

\begin{proof}
By \Cref{lem:ratio_descent}, $f(W_{k+1}) \le f(W_k) - \left(\gamma-\frac{L\gamma^2}{2\mathfrak R_{\calT}(G_k)}\right) \dotpF{G_k}{\calT(G_k)}$. 
Using $\mathfrak R_{\calT}(G)\ge \underline{\mathfrak R}$ and $\dotpF{G}{\calT(G)}\ge \alpha \fronorm{G}^2$, we obtain
\[
f(W_{k+1})
\le
f(W_k)
-
\alpha\gamma
\left(
1-\frac{L\gamma}{2\underline{\mathfrak R}}
\right)
\fronorm{G_k}^2.
\]
Applying the \PL inequality yields the desired inequality. Since $0<\gamma<2\underline{\mathfrak R}/L$, the factor $1-L\gamma/(2\underline{\mathfrak R})$ is positive, so $\rho\in(0,1)$. Iterating the recursion proves the claim.
\end{proof}

\paragraph{Specialization to the polar update.}
For the normalized polar update $\calT(G)=\polar(G)=UV^\top$ when $G=U\Sigma V^\top$, one has $\dotpF{G}{\polar(G)}=\nucnorm{G}$ and $\fronorm{\polar(G)}^2=\rank(G)$, and therefore $\mathfrak R_{\polar}(G) = \nucnorm{G}/\rank(G)$. This is the average nonzero singular value of $G$. In particular, the
descent condition becomes
\[
0 < \gamma_k < \frac{2\nucnorm{G_k}}{L\,\rank(G_k)}.
\]
Moreover, $\dotpF{G}{\polar(G)}=\nucnorm{G}\ge \fronorm{G}$, but in general one does not have a uniform lower bound of the form
\[
\dotpF{G}{\polar(G)}\ge \alpha \fronorm{G}^2
\]
without additional scale control. Thus, for normalized polar updates, the ratio $\mathfrak R_{\polar}(G)$ naturally governs the admissible learning rate, whereas stronger convergence rates require an additional lower bound relating $\nucnorm{G}$ to $\fronorm{G}^2$. Using the stable rank defined by $\srank(G)\coloneqq \fronorm{G}^2/\specnorm{G}^2$, 
one has
\[
\nucnorm{G} \ge \frac{\fronorm{G}^2}{\specnorm{G}} = \srank(G)\specnorm{G},
\]
and hence
\[
\mathfrak R_{\polar}(G) = \frac{\nucnorm{G}}{\rank(G)} \ge \frac{\srank(G)}{\rank(G)}\specnorm{G}.
\]
This lower bound shows explicitly how the descent margin depends on the
spectral spread of the gradient: flatter spectra, corresponding to larger
stable rank, lead to more favorable ratio bounds for the polar update.

\subsubsection{Specialization to \PolarGrad with Nuclear-Norm Scaling}
For \PolarGrad with nuclear-norm scaling \citep{lau2025polargrad}, the normalized polar direction is rescaled in such a way that both its alignment with the gradient and its squared Frobenius norm admit exact closed-form expressions. This removes the scale ambiguity present in fully normalized polar updates and leads to a particularly transparent descent analysis in terms of gradient rank and stable rank.

Given a gradient matrix $G=U\Diag(\sigma(G))V^\top$, define the update map 
\[
\calT_{\mathrm{PG}}(G) \coloneqq \nucnorm{G}\,\polar(G) = \nucnorm{G}\,UV^\top.
\]
The corresponding iteration is
\begin{equation}\label{eq:polargrad_nuc_iteration}
(\forall k\in\NN)\qquad
W_{k+1} = W_k-\gamma_k\nucnorm{G_k}\,\polar(G_k), \qquad G_k=\nabla f(W_k).
\end{equation}
In what follows, we write $r_k\coloneqq\rank(G_k)$. 

\begin{lemma}[Alignment and norm identities for \PolarGrad]
\label{lem:polargrad_nuc_identities}
For every $G\in\RR^{m\times n}$, we have $\dotpF{G}{\calT_{\mathrm{PG}}(G)} = \nucnorm{G}^2$ and $\fronorm{\calT_{\mathrm{PG}}(G)}^2 = \rank(G)\,\nucnorm{G}^2$. 
Consequently,
\[
\mathfrak R_{\calT_{\mathrm{PG}}}(G)
=
\frac{\dotpF{G}{\calT_{\mathrm{PG}}(G)}}{\fronorm{\calT_{\mathrm{PG}}(G)}^2}
=
\frac{1}{\rank(G)}.
\]
\end{lemma}

\begin{proof}
Using orthogonal invariance of the Frobenius inner product,
\[
\dotpF{G}{\calT_{\mathrm{PG}}(G)}
=
\nucnorm{G}\,\dotpF{U\Diag(\sigma(G))V^\top}{UV^\top}
=
\nucnorm{G}\sum_{i=1}^r \sigma_i(G)
=
\nucnorm{G}^2.
\]
Also, $\fronorm{\calT_{\mathrm{PG}}(G)}^2 = \nucnorm{G}^2\,\fronorm{UV^\top}^2 = \nucnorm{G}^2\,\rank(G)$. 
The ratio identity follows immediately.
\end{proof}

\begin{theorem}[Descent lemma for nuclear-norm-scaled \PolarGrad]
\label{thm:polargrad_nuc_descent}
Suppose $f$ satisfies \Cref{ass:Lipschitz_smooth}. Then the iteration
\eqref{eq:polargrad_nuc_iteration} satisfies
\[
f(W_{k+1}) \le f(W_k) -\gamma_k \nucnorm{G_k}^2 +\frac{L\gamma_k^2}{2}r_k\nucnorm{G_k}^2.
\]
Equivalently,
\[
f(W_{k+1}) \le f(W_k) - \gamma_k\left(1-\frac{L\gamma_k}{2}r_k\right)\nucnorm{G_k}^2.
\]
In particular, if $\gamma_k<2/(Lr_k)$, then $f(W_{k+1})\le f(W_k)$.
\end{theorem}

\begin{proof}
By $L$-smoothness of $f$ and $W_{k+1}-W_k=-\gamma_k\calT_{\mathrm{PG}}(G_k)$, we obtain
\[
f(W_{k+1}) \le f(W_k) -\gamma_k\dotpF{G_k}{\calT_{\mathrm{PG}}(G_k)} +\frac{L\gamma_k^2}{2}\fronorm{\calT_{\mathrm{PG}}(G_k)}^2.
\]
Now apply \Cref{lem:polargrad_nuc_identities}.
\end{proof}

\begin{corollary}[Stable-rank improvement]
\label{cor:polargrad_nuc_srank}
Under the assumptions of \Cref{thm:polargrad_nuc_descent},
\[
f(W_{k+1})
\le
f(W_k)
-
\gamma_k\left(1-\frac{L\gamma_k}{2}r_k\right)\nucnorm{G_k}^2.
\]
Since $\nucnorm{G_k}^2 \ge \srank(G_k)\,\fronorm{G_k}^2$ with $\srank(G_k)\coloneqq \fronorm{G_k}^2/\specnorm{G_k}^2$, it follows that
\[
f(W_{k+1}) \le f(W_k) - \gamma_k\left(1-\frac{L\gamma_k}{2}r_k\right) \srank(G_k)\fronorm{G_k}^2.
\]
\end{corollary}
Thus, compared with vanilla gradient descent, nuclear-norm-scaled \PolarGrad admits a potentially stronger one-step decrease when the gradient has nontrivial spectral spread, measured by its stable rank. 

\begin{theorem}[Sublinear convergence to stationarity for nuclear-norm-scaled \PolarGrad]
\label{thm:polargrad_nuc_stationarity}
Suppose $f$ satisfies \Cref{ass:Lipschitz_smooth}. Consider the iteration \eqref{eq:polargrad_nuc_iteration} with constant learning rate $\gamma>0$. Assume there exists
$\bar r\in\NN^*$ such that $r_k\le \bar r$ for all $k\in\NN$, and that $\gamma\in(0, 2/(L\bar r))$. Then
\[
f(W_{k+1}) \le f(W_k) - \gamma\left(1-\frac{L\gamma}{2}\bar r\right)\nucnorm{G_k}^2. 
\]
Consequently, we have 
\[
\sum_{k=0}^{T-1}\fronorm{G_k}^2 \le \sum_{k=0}^{T-1}\nucnorm{G_k}^2 \le \frac{f(W_0)-f^\star}{\gamma\left(1-L\gamma\bar r/2\right)}
\quad\text{ and }\quad
\min_{0\le k<T}\,\fronorm{G_k}^2 \le \frac{f(W_0)-f^\star} {T\gamma\left(1-L\gamma\bar r/2\right)}.
\]
Hence the method converges to stationarity at the standard $\scrO(1/T)$ rate in
the minimum gradient norm.
\end{theorem}

\begin{proof}
By \Cref{thm:polargrad_nuc_descent}, $f(W_{k+1}) \le f(W_k) - \gamma\left(1-\frac{L\gamma}{2}r_k\right)\nucnorm{G_k}^2$. Using $r_k\le \bar r$, we obtain $f(W_{k+1}) \le f(W_k) - \gamma\left(1-\frac{L\gamma}{2}\bar r\right)\nucnorm{G_k}^2$. Summing from $k=0$ to $T-1$ yields
\[
f(W_T) \le f(W_0) - \gamma\left(1-\frac{L\gamma}{2}\bar r\right) \sum_{k=0}^{T-1}\nucnorm{G_k}^2.
\]
Since $f(W_T)\ge f^\star$ and $\nucnorm{G_k}\ge \fronorm{G_k}$,
\[
\sum_{k=0}^{T-1}\fronorm{G_k}^2 \le \sum_{k=0}^{T-1}\nucnorm{G_k}^2 \le \frac{f(W_0)-f^\star}{\gamma\left(1-L\gamma\bar r/2\right)}.
\]
Finally, $\min_{0\le k<T}\,\fronorm{G_k}^2 \le \frac1T\sum_{k=0}^{T-1}\fronorm{G_k}^2$ proves the stated $\scrO(1/T)$ bound.
\end{proof}

\begin{theorem}[Linear convergence under \PL for nuclear-norm-scaled \PolarGrad]
\label{thm:polargrad_nuc_PL}
Suppose $f$ satisfies \Cref{ass:Lipschitz_smooth,ass:PL}. Assume there exists $\bar r\in\NN^*$ such that $r_k\le \bar r$ for all $k\in\NN$. Then any constant learning rate satisfying $\gamma\in(0, 2/(L\bar r))$ yields
\[
f(W_{k+1})-f^\star \le \left(1-2\mu\gamma\left(1-\frac{L\gamma}{2}\bar r\right)\right) \bigl(f(W_k)-f^\star\bigr).
\]
Hence $f(W_k)-f^\star \le \rho^k\bigl(f(W_0)-f^\star\bigr)$, where $\rho = 1-2\mu\gamma\left(1-L\gamma\bar r/2\right)\in(0,1)$.
\end{theorem}

\begin{proof}
By \Cref{thm:polargrad_nuc_descent}, and using $r_k\le \bar r$ and $\nucnorm{G_k}^2\ge \fronorm{G_k}^2$, we obtain
\[
f(W_{k+1}) \le f(W_k) -\gamma\left(1-\frac{L\gamma}{2}\bar r\right)\fronorm{G_k}^2.
\]
Applying the \PL inequality gives
\[
f(W_{k+1})-f^\star \le \left(1-2\mu\gamma\left(1-\frac{L\gamma}{2}\bar r\right)\right) \bigl(f(W_k)-f^\star\bigr).
\]
Since $0<\gamma<2/(L\bar r)$, the factor $1-L\gamma\bar r/2$ is positive, hence $\rho\in(0,1)$.
\end{proof}

\begin{remark}[Relaxing the global \PL condition]
The global \PL condition is used here only to obtain global linear convergence. Under smoothness alone, nuclear-norm-scaled \PolarGrad still enjoys monotonic descent and an $\scrO(1/T)$ convergence rate to stationarity in the minimum gradient norm. If the \PL inequality holds only locally in a neighborhood of a limit point, then the same descent argument yields eventual linear convergence once the iterates enter that neighborhood. More generally, one may replace the \PL condition by a Kurdyka--\L{}ojasiewicz (\KLo) inequality \citep{lojasiewicz1993geometrie,kurdyka1998gradients}, which provides a broader convergence framework and recovers linear or sublinear rates depending on the local \KLo exponent.
\end{remark}

\subsection{One-Sided Spectral Optimizers}
We now specialize the preceding convergence analysis to one-sided spectral optimizers. For one-sided spectral optimizers, the convergence analysis follows the same smoothness-based template as in the full spectral case. The only additional ingredients are alignment and norm-control conditions adapted to the relevant one-sided Gram operator. These yield standard descent, sublinear stationarity, and linear convergence under the \PL condition.

Recall that right-spectral updates take the form $\calT_{\sfR}(G)=G\,\Psi(G^\top G)$, while left-spectral updates take the form $\calT_{\sfL}(G)=\Phi(GG^\top)G$, where $\Psi$ and $\Phi$ are orthogonally equivariant spectral operators on the corresponding Gram matrices.

\subsubsection{Right-Spectral Optimizers}
Consider the iteration
\begin{equation}\label{eq:right_spectral_iteration}
(\forall k\in\NN)\qquad W_{k+1}=W_k-\gamma_k\calT_{\sfR}(G_k), \qquad G_k\coloneqq \nabla f(W_k), 
\end{equation}
with $\calT_{\sfR}(G)=G\,\Psi(G^\top G)$. 

\begin{assumption}[Right-spectral eigenvalue alignment and boundedness]
\label{ass:right_spectral_eig}
For every $G\in\RR^{m\times n}$, let $G^\top G = V \Diag(\lambda(G^\top G))V^\top$ and $\Psi(G^\top G) = V\Diag(\psi(\lambda(G^\top G)))V^\top$, 
where $\lambda(G^\top G)=(\lambda_1,\dots,\lambda_n)\in\Rp^n$. Assume there exist constants $0<c_{\sfR,1}\le c_{\sfR,2}<\infty$ such that for all $\lambda\in\Rp^n$, 
\[
\sumn \lambda_i\psi_i(\lambda)\ge c_{\sfR,1}\sumn \lambda_i \quad \text{and} \quad \sumn \lambda_i\psi_i(\lambda)^2\le c_{\sfR,2}\sumn \lambda_i.
\]
\end{assumption}

\begin{lemma}[Right-spectral alignment identities]
\label{lem:right_spectral_alignment}
Under \Cref{ass:right_spectral_eig}, for all $G\in\RR^{m\times n}$, $\dotpF{G}{\calT_{\sfR}(G)} = \sumn \lambda_i\psi_i(\lambda)$ and $\fronorm{\calT_{\sfR}(G)}^2 = \sumn \lambda_i\psi_i(\lambda)^2$, where $\lambda=\lambda(G^\top G)$. Consequently, $\dotpF{G}{\calT_{\sfR}(G)}\ge c_{\sfR,1}\fronorm{G}^2$ and $\fronorm{\calT_{\sfR}(G)}^2\le c_{\sfR,2}\fronorm{G}^2$. 
\end{lemma}

\begin{proof}
Let $G=U\Diag(\sigma(G))V^\top$ be a singular value decomposition of $G$. Since $\lambda_i(G^\top G)=\sigma_i(G)^2$, we may write
$G^\top G = V\Diag(\lambda)V^\top$ and $\Psi(G^\top G)=V\Diag(\psi(\lambda))V^\top$.
Therefore,
\[
\calT_{\sfR}(G)
=
G\Psi(G^\top G)
=
U\Diag(\sigma(G))V^\top V\Diag(\psi(\lambda))V^\top
=
U\Diag(\sigma_i(G)\psi_i(\lambda))V^\top.
\]
Hence $\dotpF{G}{\calT_{\sfR}(G)} = \sumn \sigma_i(G)^2\psi_i(\lambda) = \sumn \lambda_i\psi_i(\lambda)$, and $\fronorm{\calT_{\sfR}(G)}^2 = \sumn \sigma_i(G)^2\psi_i(\lambda)^2 = \sumn \lambda_i\psi_i(\lambda)^2$. 
The final inequalities follow from \Cref{ass:right_spectral_eig} and $\sumn \lambda_i=\fronorm{G}^2$. 
\end{proof}

\begin{theorem}[Right-spectral descent and convergence]
\label{thm:right_spectral_conv}
Suppose \Cref{ass:Lipschitz_smooth,ass:right_spectral_eig} hold. Then the
iteration \eqref{eq:right_spectral_iteration} satisfies
\[
f(W_{k+1})
\le
f(W_k)
-
\left(
c_{\sfR,1}\gamma_k-\frac{Lc_{\sfR,2}}{2}\gamma_k^2
\right)\fronorm{G_k}^2.
\]
In particular, if $\gamma_k\in(0, 2c_{\sfR,1}/(Lc_{\sfR,2}))$, then $f(W_{k+1})\le f(W_k)$. 
If, in addition, $f$ satisfies \Cref{ass:PL} and $\gamma_k\equiv \gamma$ is constant with $\gamma\in(0, 2c_{\sfR,1}/(Lc_{\sfR,2}))$, then
\[
f(W_{k+1})-f^\star
\le
\left(
1-2\mu\gamma\left(c_{\sfR,1}-\frac{Lc_{\sfR,2}}{2}\gamma\right)
\right)
\bigl(f(W_k)-f^\star\bigr),
\]
and therefore $f(W_k)-f^\star \le \rho_{\sfR}^k\bigl(f(W_0)-f^\star\bigr)$, where $\rho_{\sfR} = 1-2\mu\gamma\left(c_{\sfR,1}-Lc_{\sfR,2}\gamma/2\right)\in(0,1)$. 
Moreover, without the \PL condition, if $\gamma_k\equiv\gamma$ is constant,
then
\[
\min_{0\le k<T}\,\fronorm{G_k}^2 \le \frac{f(W_0)-f^\star} {T\gamma\left(c_{\sfR,1}-Lc_{\sfR,2}\gamma/2\right)}.
\]
\end{theorem}

\begin{proof}
Combine the smoothness inequality with \Cref{lem:right_spectral_alignment}, as
in the proof of \Cref{thm:descent_spectral}, and then use either the \PL inequality or summation of the descent bound.
\end{proof}

\subsubsection{Left-Spectral Optimizers}
Consider the iteration
\begin{equation}\label{eq:left_spectral_iteration}
(\forall k\in\NN)\qquad W_{k+1}=W_k-\gamma_k\calT_{\sfL}(G_k), \qquad G_k\coloneqq \nabla f(W_k),
\end{equation}
with $\calT_{\sfL}(G)=\Phi(GG^\top)\,G$. 

\begin{assumption}[Left-spectral eigenvalue alignment and boundedness]
\label{ass:left_spectral_eig}
For every $G\in\RR^{m\times n}$, let $GG^\top = U \Diag(\lambda(GG^\top))U^\top$ and $\Phi(GG^\top)=U\Diag(\phi(\lambda(GG^\top)))U^\top$, where $\lambda(GG^\top)=(\lambda_1,\dots,\lambda_m)\in\Rp^m$. Assume there exist constants $0<c_{\sfL,1}\le c_{\sfL,2}<\infty$ such that for all
$\lambda\in\Rp^m$, 
\[\summ \lambda_i\phi_i(\lambda)\ge c_{\sfL,1}\summ \lambda_i \quad \text{and}\quad \summ \lambda_i\phi_i(\lambda)^2\le c_{\sfL,2}\summ \lambda_i. \]
\end{assumption}

\begin{lemma}[Left-spectral alignment identities]
\label{lem:left_spectral_alignment}
Under \Cref{ass:left_spectral_eig}, for all $G\in\RR^{m\times n}$, $\dotpF{G}{\calT_{\sfL}(G)} = \summ \lambda_i\,\phi_i(\lambda)$ and $\fronorm{\calT_{\sfL}(G)}^2 = \sum_{i=1}^m \lambda_i\,\phi_i(\lambda)^2$, where $\lambda=\lambda(GG^\top)$. Consequently, $\dotpF{G}{\calT_{\sfL}(G)}\ge c_{\sfL,1}\fronorm{G}^2$ and $\fronorm{\calT_{\sfL}(G)}^2\le c_{\sfL,2}\fronorm{G}^2$. 
\end{lemma}

\begin{proof}
Let $G=U\Diag(\sigma(G))V^\top$ be a singular value decomposition of $G$. Since $\lambda_i(GG^\top)=\sigma_i(G)^2$, we may write $GG^\top = U\Diag(\lambda)U^\top$ and $\Phi(GG^\top)=U\Diag(\phi(\lambda))U^\top$. 
Therefore,
\[
\calT_{\sfL}(G)
=
\Phi(GG^\top)G
=
U\Diag(\phi(\lambda))U^\top U\Diag(\sigma(G))V^\top
=
U\Diag(\phi_i(\lambda)\sigma_i(G))V^\top.
\]
Hence $\dotpF{G}{\calT_{\sfL}(G)} = \summ \sigma_i(G)^2\phi_i(\lambda) = \summ \lambda_i\phi_i(\lambda)$, and $\fronorm{\calT_{\sfL}(G)}^2 = \summ \sigma_i(G)^2\phi_i(\lambda)^2 = \summ \lambda_i\phi_i(\lambda)^2$. The final inequalities follow from \Cref{ass:left_spectral_eig} and $\summ \lambda_i=\fronorm{G}^2$. 
\end{proof}

\begin{theorem}[Left-spectral descent and convergence]
\label{thm:left_spectral_conv}
Suppose \Cref{ass:Lipschitz_smooth,ass:left_spectral_eig} hold. Then the
iteration \eqref{eq:left_spectral_iteration} satisfies
\[
f(W_{k+1}) \le f(W_k) - \left( c_{\sfL,1}\gamma_k-\frac{Lc_{\sfL,2}}{2}\gamma_k^2 \right)\fronorm{G_k}^2.
\]
In particular, if $\gamma_k\in(0, 2c_{\sfL,1}/(Lc_{\sfL,2}))$, then $f(W_{k+1})\le f(W_k)$. 
If, in addition, $f$ satisfies \Cref{ass:PL} and $\gamma_k\equiv \gamma$ is constant with $\gamma\in(0, 2c_{\sfL,1}/(Lc_{\sfL,2}))$, then $f(W_{k+1})\le f(W_k)$ then \[
f(W_{k+1})-f^\star \le \left(1-2\mu\gamma\left(c_{\sfL,1}-\frac{Lc_{\sfL,2}}{2}\gamma\right)\right)\bigl(f(W_k)-f^\star\bigr), 
\]
and therefore $f(W_k)-f^\star \le \rho_{\sfL}^k\bigl(f(W_0)-f^\star\bigr)$, where $\rho_{\sfL} = 1-2\mu\gamma\left(c_{\sfL,1}-Lc_{\sfL,2}\gamma/2\right)\in(0,1)$. 
Moreover, without the \PL condition, if $\gamma_k\equiv\gamma$ is constant,
then
\[
\min_{0\le k<T}\,\fronorm{G_k}^2 \le \frac{f(W_0)-f^\star}{T\gamma\left(c_{\sfL,1}-Lc_{\sfL,2}\gamma/2\right)}.
\]
\end{theorem}

\begin{proof}
Identical to the proof of \Cref{thm:right_spectral_conv}, replacing $\Psi(G^\top G)$ by $\Phi(GG^\top)$.
\end{proof}

For the canonical one-sided polar updates, the abstract $(c_1,c_2)$-based analysis is less natural, and is better replaced by the same ratio-style viewpoint used for normalized \PolarGrad. After nuclear-norm scaling, however, both one-sided variants recover the same closed-form alignment and norm identities as full nuclear-norm-scaled \PolarGrad.

\begin{theorem}[Convergence of nuclear-norm-scaled one-sided \PolarGrad]
\label{thm:one_sided_polargrad_nuc}
Consider either the right-sided update 
\[\calT_{\mathsf{PG},\sfR}(G) = \nucnorm{G}\,G(G^\top G)^{\sfrac{\dagger}{2}}\] 
or the left-sided update 
\[\calT_{\mathsf{PG},\sfL}(G) = \nucnorm{G}\,(GG^\top)^{\sfrac{\dagger}{2}}G. \]
Then for both updates, $\dotpF{G}{\calT(G)}=\nucnorm{G}^2$ and $\fronorm{\calT(G)}^2=\rank(G)\nucnorm{G}^2$. Hence the identities in \Cref{lem:polargrad_nuc_identities} remain valid for both one-sided nuclear-norm-scaled variants. Therefore, the descent, stationarity, and \PL linear convergence results of \Cref{thm:polargrad_nuc_descent,thm:polargrad_nuc_stationarity,thm:polargrad_nuc_PL} apply verbatim.
\end{theorem}

\begin{proof}
We first consider the right-sided update $\calT_{\mathsf{PG},\sfR}(G)$. Let $G=U\Diag(\sigma(G))V^\top$ be a singular value decomposition of $G$. Then
\[
G(G^\top G)^{\sfrac{\dagger}{2}} = U\Diag(\sigma(G))V^\top \left(V\Diag(\sigma(G)^2)V^\top \right)^{\negthickspace\sfrac{\dagger}{2}} = UV^\top = \polar(G),
\]
so $\calT_{\mathsf{PG},\sfR}(G) = \nucnorm{G}\,\polar(G)$. 
Therefore, $\dotpF{G}{\calT_{\mathsf{PG},\sfR}(G)} = \nucnorm{G}\,\dotpF{G}{\polar(G)} = \nucnorm{G}^2$, and $\fronorm{\calT_{\mathsf{PG},\sfR}(G)}^2 = \nucnorm{G}^2\,\fronorm{\polar(G)}^2 = \rank(G)\nucnorm{G}^2$. The left-sided update follows similarly. The final claim follows immediately by invoking
\Cref{thm:polargrad_nuc_descent,thm:polargrad_nuc_stationarity,thm:polargrad_nuc_PL}.
\end{proof}

\subsection{Row-Norm-Based Optimizers}
We next study row-norm-based optimizers, whose update maps act locally on each row of the gradient matrix. Such methods are especially natural for parameter matrices whose row axis carries the relevant structural symmetry, as in embeddings, LM heads, and \MoE routers.

Consider the iteration
\begin{equation}\label{eq:row_iteration}
(\forall k\in\NN)\qquad W_{k+1}=W_k-\gamma_k\calT_{\row}(G_k), \qquad G_k\coloneqq \nabla f(W_k),
\end{equation}
where $\calT_{\row}(G) = D_\eta(G)\,G$ and $D_\eta(G) \coloneqq \Diag(\eta(\|G_{1:}\|_2),\dots,\eta(\|G_{v:}\|_2))$, for some scalar function $\eta\colon\Rp\to\RR$.

\begin{assumption}[Uniform row-scaling bounds]
\label{ass:row_eta_bounds}
There exist constants $0<\underline\eta\le \overline\eta<\infty$ such that $\underline\eta\le \eta(t)\le \overline\eta$ for all $ t\ge 0$. 
\end{assumption}

\begin{lemma}[Alignment and norm bounds for row-norm updates]
\label{lem:row_alignment}
Under \Cref{ass:row_eta_bounds}, for all $G\in\RR^{v\times d}$,
\[
\dotpF{G}{\calT_{\row}(G)} = \sum_{i=1}^v \eta(\|G_{i:}\|_2)\,\|G_{i:}\|_2^2 \ge \underline\eta\,\fronorm{G}^2, \]
and
\[
\fronorm{\calT_{\row}(G)}^2 = \sum_{i=1}^v \eta(\|G_{i:}\|_2)^2\,\|G_{i:}\|_2^2 \le \overline\eta^2\,\fronorm{G}^2.
\]
\end{lemma}

\begin{proof}
By definition, we have 
\[\dotpF{G}{\calT_{\row}(G)} = \sum_{i=1}^v \dotpF{G_{i:}}{\eta(\|G_{i:}\|_2)G_{i:}} = \sum_{i=1}^v \eta(\|G_{i:}\|_2)\,\|G_{i:}\|_2^2. \]
Using $\eta(\|G_{i:}\|_2)\ge \underline\eta$, we obtain $\dotpF{G}{\calT_{\row}(G)} \ge \underline\eta\sum_{i=1}^v \|G_{i:}\|_2^2 = \underline\eta\,\fronorm{G}^2$. 
Similarly, we also have 
\[\fronorm{\calT_{\row}(G)}^2 = \sum_{i=1}^v \eta(\|G_{i:}\|_2)^2\,\|G_{i:}\|_2^2 \le \overline\eta^2\sum_{i=1}^v \|G_{i:}\|_2^2 = \overline\eta^2\,\fronorm{G}^2. \]
\end{proof}

\begin{theorem}[One-step descent for row-norm-based optimizers]
\label{thm:row_descent}
Suppose \Cref{ass:Lipschitz_smooth,ass:row_eta_bounds} hold. Then the
iteration \eqref{eq:row_iteration} satisfies
\[
f(W_{k+1})
\le
f(W_k)
-
\left(
\underline\eta\,\gamma_k-\frac{L\overline\eta^2}{2}\gamma_k^2
\right)\fronorm{G_k}^2.
\]
In particular, if $\gamma_k \in (0,2\underline\eta/(L\overline\eta^2))$, then $f(W_{k+1})\le f(W_k)$.
\end{theorem}

\begin{proof}
By $L$-smoothness of $f$ and $W_{k+1}-W_k=-\gamma_k\calT_{\row}(G_k)$, we obtain
\[
f(W_{k+1})
\le
f(W_k)
-\gamma_k\dotpF{G_k}{\calT_{\row}(G_k)}
+\frac{L\gamma_k^2}{2}\fronorm{\calT_{\row}(G_k)}^2.
\]
Apply \Cref{lem:row_alignment}.
\end{proof}

\begin{theorem}[Sublinear convergence to stationarity]
\label{thm:row_stationarity}
Suppose \Cref{ass:Lipschitz_smooth,ass:row_eta_bounds} hold, and let $\gamma>0$ be constant with $\gamma \in (0,2\underline\eta/(L\overline\eta^2))$. Then
\[
\sum_{k=0}^{T-1}\fronorm{G_k}^2 \le \frac{f(W_0)-f^\star}{\gamma\left(\underline\eta-L\overline\eta^2\gamma/2\right)},
\quad\text{ and therefore }\quad 
\min_{0\le k<T}\,\fronorm{G_k}^2 \le \frac{f(W_0)-f^\star} {T\gamma\left(\underline\eta-L\overline\eta^2\gamma/2\right)}.
\]
\end{theorem}

\begin{proof}
Sum the descent inequality in \Cref{thm:row_descent}.
\end{proof}

\begin{theorem}[Linear convergence under the \PL condition]
\label{thm:row_PL}
Suppose \Cref{ass:Lipschitz_smooth,ass:PL,ass:row_eta_bounds} hold, and let $\gamma>0$ be constant with $\gamma \in (0,2\underline\eta/(L\overline\eta^2))$.
Then
\[
f(W_{k+1})-f^\star
\le
\left(
1-2\mu\gamma\left(\underline\eta-\frac{L\overline\eta^2}{2}\gamma\right)
\right)
\bigl(f(W_k)-f^\star\bigr).
\]
Hence $f(W_k)-f^\star \le \rho_{\row}^k\bigl(f(W_0)-f^\star\bigr)$, where $\rho_{\row} = 1-2\mu\gamma\left(\underline\eta-L\overline\eta^2\gamma/2\right)\in(0,1)$. 
\end{theorem}

\begin{proof}
Combine \Cref{thm:row_descent} with the \PL inequality.
\end{proof}

\subsubsection{Specialization to Smoothed Row Normalization}
A useful smoothed variant of row normalization is obtained by taking $\eta(t)=1/(t+\varepsilon)$ for some $\varepsilon>0$. 
The corresponding update map is
\[
\calT_{\row,\varepsilon}(G) = \Diag\left( \frac{1}{\|G_{1:}\|_2+\varepsilon}, \dots, \frac{1}{\|G_{v:}\|_2+\varepsilon} \right)G.
\]
Unlike the fully normalized choice $\eta(t)=1/t$, this smoothed row-norm-based map remains bounded at zero and therefore fits naturally into the preceding  bounded-scaling framework.

\begin{assumption}[Uniform row-norm upper bound]
\label{ass:row_norm_upper_bound}
There exists $M>0$ such that $\|G_{k,i:}\|_2\le M$ for all $k\in\NN$, $i\in\set{v}\coloneqq\{1,\ldots, v\}$.
\end{assumption}

\begin{corollary}[Descent and convergence for smoothed row normalization]
\label{cor:row_eps}
Suppose \Cref{ass:Lipschitz_smooth,ass:row_norm_upper_bound} hold, and let $\eta(t)=1/(t+\varepsilon)$ for some $\varepsilon>0$. 
Then $1/(M+\varepsilon)\le \eta(\|G_{k,i:}\|_2)\le 1/\varepsilon$ for all $k\in\NN$ and $i\in\set{v}$. Hence \Cref{ass:row_eta_bounds} holds with $\underline\eta=1/(M+\varepsilon)$ and $\overline\eta=1/\varepsilon$. Consequently, the descent, stationarity, and \PL linear-convergence results of \Cref{thm:row_descent,thm:row_stationarity,thm:row_PL} apply directly. In particular, any constant learning rate satisfying $\gamma\in(0, 2\varepsilon^2/(L(M+\varepsilon)))$ guarantees monotonic descent, and under \Cref{ass:PL} one obtains linear convergence with contraction factor
\[
\rho_{\row,\varepsilon} = 1-2\mu\gamma\left( \frac{1}{M+\varepsilon}-\frac{L\gamma}{2\varepsilon^2} \right).
\]
\end{corollary}

\begin{proof}
For $t\ge 0$, the map $t\mapsto 1/(t+\varepsilon)$ is decreasing, so $1/(M+\varepsilon) \le 1/(\|G_{k,i:}\|_2+\varepsilon) \le 1/\varepsilon$. 
Thus \Cref{ass:row_eta_bounds} holds with the stated constants. The conclusions then follow immediately from \Cref{thm:row_descent,thm:row_stationarity,thm:row_PL}.
\end{proof}

Thus, the choice $\eta(t)=1/(t+\varepsilon)$ interpolates between the bounded row-norm regime and the fully normalized regime: it preserves the local row-adaptive flavor of normalization while avoiding the singular behavior of $\eta(t)=1/t$ at small row norms.

Having established convergence guarantees for right-spectral and row-norm-based optimizers separately, we now turn to their finite compositions. The resulting hybrid methods inherit the geometric structure of the right polar factor together with a local row-wise normalization, and their analyses are naturally expressed through the preserved alignment ratio and the active row support.

\subsection{Nuclear-Norm-Scaled Right-Spectral/Row-Norm Hybrid Optimizers}
\label{subsec:hyb_nuc}
We now study the hybrid optimizer obtained by composing the right polar factor with row-wise normalization. In this construction, the update is first normalized with respect to the feature geometry through the right polar factor, and then normalized locally across rows. After an additional nuclear-norm scaling, both the alignment and the squared Frobenius norm of the resulting update admit explicit expressions, leading to a clean descent analysis in terms of two interpretable quantities: the preserved alignment ratio $\mathfrak A_{\hyb}(G)$ and the active row support $s_{\row}(G)$.

Define the right polar factor
\[
Z(G)\coloneqq G(G^\top G)^{\sfrac{\dagger}{2}},
\]
and let the row-normalized hybrid map
$\calT_{\hyb}\colon\RR^{v\times d}\to\RR^{v\times d}$ be given row-wise by
\[
\calT_{\hyb}(G)_{i:} 
= \begin{cases}
\dfrac{Z(G)_{i:}}{\|Z(G)_{i:}\|_2}, & Z(G)_{i:}\neq 0,\\[6pt]
0, & Z(G)_{i:}=0.
\end{cases}
\]
Its nuclear-norm-scaled version is defined by $\calT_{\hyb,\nuc}(G)\coloneqq \nucnorm{G}\,\calT_{\hyb}(G)$. 
The corresponding iteration is
\begin{equation}\label{eq:hyb_nuc_iteration}
(\forall k\in\NN)\qquad
W_{k+1}=W_k-\gamma_k\calT_{\hyb,\nuc}(G_k),
\qquad
G_k\coloneqq \nabla f(W_k).
\end{equation}

\begin{definition}[Active row support of the right polar factor]
For $G\in\RR^{v\times d}$, define
\[
s_{\row}(G)\coloneqq \sharp\bigl\{i\in\set{v} : Z(G)_{i:}\neq 0\bigr\},
\qquad
Z(G)=G(G^\top G)^{\sfrac{\dagger}{2}}.
\]
\end{definition}

\begin{definition}[Hybrid row-polar alignment ratio]
For $G\neq 0$, define
\[
\mathfrak A_{\hyb}(G)
\coloneqq
\frac{\dotpF{G}{\calT_{\hyb}(G)}}{\nucnorm{G}}.
\]
\end{definition}

The quantity $\mathfrak A_{\hyb}(G)$ measures how much of the nuclear-norm alignment of the right polar factor is preserved after row normalization.

\begin{lemma}[Norm and alignment identities]
\label{lem:hyb_nuc_identities}
For every $G\in\RR^{v\times d}$, we have $\fronorm{\calT_{\hyb}(G)}^2=s_{\row}(G)$, $\fronorm{\calT_{\hyb,\nuc}(G)}^2=\nucnorm{G}^2\,s_{\row}(G)$, and $\dotpF{G}{\calT_{\hyb,\nuc}(G)} = \nucnorm{G}\,\dotpF{G}{\calT_{\hyb}(G)} = \mathfrak A_{\hyb}(G)\,\nucnorm{G}^2$. 
\end{lemma}

\begin{proof}
By construction, every nonzero row of $\calT_{\hyb}(G)$ has Euclidean norm $1$, while every zero row remains zero. Hence $\fronorm{\calT_{\hyb}(G)}^2=s_{\row}(G)$. 
Multiplying by $\nucnorm{G}$ yields $\fronorm{\calT_{\hyb,\nuc}(G)}^2 = \nucnorm{G}^2\,\fronorm{\calT_{\hyb}(G)}^2 = \nucnorm{G}^2\,s_{\row}(G)$. 
Similarly, $\dotpF{G}{\calT_{\hyb,\nuc}(G)} = \nucnorm{G}\,\dotpF{G}{\calT_{\hyb}(G)} = \mathfrak A_{\hyb}(G)\,\nucnorm{G}^2$. 
\end{proof}

\begin{theorem}[Descent lemma for the nuclear-norm-scaled hybrid update]
\label{thm:hyb_nuc_descent}
Suppose $f$ satisfies \Cref{ass:Lipschitz_smooth}. Then the iteration
\eqref{eq:hyb_nuc_iteration} satisfies
\[
f(W_{k+1}) \le f(W_k) -\gamma_k\mathfrak A_{\hyb}(G_k)\,\nucnorm{G_k}^2 +\frac{L\gamma_k^2}{2}\,s_{\row}(G_k)\,\nucnorm{G_k}^2.
\]
Equivalently,
\[
f(W_{k+1})
\le
f(W_k)
-\gamma_k\left(
\mathfrak A_{\hyb}(G_k)-\frac{L\gamma_k}{2}s_{\row}(G_k)
\right)\nucnorm{G_k}^2.
\]
In particular, if $\gamma_k\in(0, 2\,\mathfrak A_{\hyb}(G_k)/(Ls_{\row}(G_k)))$, then $f(W_{k+1})\le f(W_k)$.
\end{theorem}

\begin{proof}
By $L$-smoothness of $f$ and \eqref{eq:hyb_nuc_iteration}, we obtain
\[
f(W_{k+1}) \le f(W_k) -\gamma_k\dotpF{G_k}{\calT_{\hyb,\nuc}(G_k)} + \frac{L\gamma_k^2}{2}\fronorm{\calT_{\hyb,\nuc}(G_k)}^2.
\]
Now apply \Cref{lem:hyb_nuc_identities}.
\end{proof}

\begin{assumption}[Uniform hybrid alignment and row-support bounds]
\label{ass:hyb_nuc_uniform}
There exist constants $\underline a>0$ and $\overline s\in\NN^*$ such that for all $k\in\NN$, $\mathfrak A_{\hyb}(G_k)\ge \underline a$ and $s_{\row}(G_k)\le \overline s$. 
\end{assumption}

\begin{theorem}[Sublinear convergence to stationarity]
\label{thm:hyb_nuc_stationarity}
Suppose \Cref{ass:Lipschitz_smooth,ass:hyb_nuc_uniform} hold. If the learning rate $\gamma$ is constant and satisfies $\gamma\in(0, 2\underline a/(L\overline s))$, then
\[
f(W_{k+1}) \le f(W_k) -\gamma\left(\underline a-\frac{L\gamma}{2}\overline s\right)\nucnorm{G_k}^2.
\]
Consequently,
\[
\sum_{k=0}^{T-1}\fronorm{G_k}^2 \le \sum_{k=0}^{T-1}\nucnorm{G_k}^2 \le \frac{f(W_0)-f^\star}{\gamma\left(\underline a-L\gamma\overline s/2\right)}
\quad \text{ and } \quad
\min_{0\le k<T}\,\fronorm{G_k}^2 \le \frac{f(W_0)-f^\star}{T\gamma\left(\underline a-L\gamma\overline s/2\right)}.
\]
\end{theorem}

\begin{proof}
By \Cref{thm:hyb_nuc_descent} and \Cref{ass:hyb_nuc_uniform},
\[
f(W_{k+1})
\le
f(W_k)
-\gamma\left(\underline a-\frac{L\gamma}{2}\overline s\right)\nucnorm{G_k}^2.
\]
Summing from $k=0$ to $T-1$ gives
\[
f(W_T) \le f(W_0) -\gamma\left(\underline a-\frac{L\gamma}{2}\overline s\right) \sum_{k=0}^{T-1}\nucnorm{G_k}^2.
\]
Using $f(W_T)\ge f^\star$ yields
\[
\sum_{k=0}^{T-1}\nucnorm{G_k}^2 \le \frac{f(W_0)-f^\star}{\gamma\left(\underline a-L\gamma\overline s/2\right)}.
\]
The Frobenius-norm and minimum-gradient-norm bounds follow from
$\nucnorm{G_k}\ge \fronorm{G_k}$.
\end{proof}

\begin{theorem}[Linear convergence under the \PL condition]
\label{thm:hyb_nuc_PL}
Suppose \Cref{ass:Lipschitz_smooth,ass:PL,ass:hyb_nuc_uniform} hold. If the
learning rate is constant and satisfies $\gamma\in(0, 2\underline a/(L\overline s))$, then
\[
f(W_{k+1})-f^\star \le \left( 1-2\mu\gamma\left(\underline a-\frac{L\gamma}{2}\overline s\right) \right) \bigl(f(W_k)-f^\star\bigr).
\]
Hence $f(W_k)-f^\star \le \rho_{\hyb}^k\bigl(f(W_0)-f^\star\bigr)$, where $\rho_{\hyb} = 1-2\mu\gamma\left(\underline a-L\gamma \overline s/2\right)\in(0,1)$. 
\end{theorem}

\begin{proof}
By \Cref{thm:hyb_nuc_descent} and \Cref{ass:hyb_nuc_uniform},
\[
f(W_{k+1}) \le f(W_k) -\gamma\left(\underline a-\frac{L\gamma}{2}\overline s\right)\nucnorm{G_k}^2.
\]
Applying $\nucnorm{G_k}^2\ge \fronorm{G_k}^2$ and the \PL inequality gives the desired inequality. 
\end{proof}

Switching the order of the two normalizations leads to a genuinely different optimizer. If row normalization is applied first, then the subsequent right-spectral step is computed from the modified Gram matrix $\tG^\top \tG = G^\top D_\eta(G)^2 G$, rather than from the original gradient Gram matrix $G^\top G$. Thus, right-polar-first preserves the original feature geometry before applying local row-wise normalization, whereas row-normalize-first alters that geometry prior to the spectral step.

\subsection{Nuclear-Norm-Scaled Row-Norm/Right-Spectral Hybrid Optimizers}
\label{subsec:row_right_hybrid}
We now turn to the hybrid optimizer obtained by reversing the order of the two normalizations. Instead of first applying the right polar factor and then normalizing rows, we first apply row-wise normalization to the gradient and then compute a right-spectral update from the resulting row-normalized matrix. This construction is genuinely different from the right-spectral/row-norm hybrid in \Cref{subsec:hyb_nuc}. Indeed, the spectral step is now computed from the modified Gram matrix
\[\tG^\top \tG = G^\top D_\eta(G)^2G, \]
rather than from the original feature Gram matrix $G^\top G$.

Let $\eta\colon \RR_+\to\RR_+$ be a row-scaling function, and define
\[
    D_\eta(G) \coloneqq \Diag\bigl(\eta(\|G_{1:}\|_2),\ldots,\eta(\|G_{v:}\|_2)\bigr).
\]
The row-normalized gradient is defined by $\tG \coloneqq D_\eta(G)G$. For example, the normalized row-norm choice corresponds to
\[
    \eta(t)=
    \begin{cases}
        1/t, & t>0,\\
        0, & t=0,
    \end{cases}
\]
or, in practical implementations, the smoothed variant $\eta(t)=1/(t+\varepsilon)$. 

Define the row-norm/right-spectral hybrid map by $\calT_{\row\hyb}(G) \coloneqq Z(\tG) = \tG (\tG^\top\tG)^{\sfrac{\dagger}{2}}$. Its nuclear-norm-scaled version is
\[
    \calT_{\row\hyb,\nuc}(G) \coloneqq \nucnorm{\tG}\,\calT_{\row\hyb}(G).
\]
The corresponding iteration is
\begin{equation}
\label{eq:row_hyb_nuc_iteration}
    W_{k+1} = W_k-\gamma_k\calT_{\row\hyb,\nuc}(G_k), \qquad G_k=\nabla f(W_k).
\end{equation}

\begin{definition}[Row-normalized effective rank]
For $G\in\RR^{v\times d}$, define $r_{\row\hyb}(G) \coloneqq \rank(\tG)$, where $\tG = D_\eta(G)G$.
\end{definition}

\begin{definition}[Row-norm/right-spectral alignment ratio]
For $G\neq 0$, define
\[
    \mathfrak A_{\row\hyb}(G) \coloneqq \frac{\dotpF{G}{\calT_{\row\hyb}(G)}}{\nucnorm{\tG}} \quad \text{ with } \quad \tG=D_\eta(G)G.
\]
\end{definition}

The quantity $\mathfrak A_{\row\hyb}(G)$ measures the alignment between the
original gradient $G$ and the right polar factor of its row-normalized
version $\tG$. Unlike the polar-first hybrid, this alignment is not
automatically tied to the nuclear norm of $G$, because the spectral step is
computed after row-wise reweighting.

\begin{lemma}[Norm and alignment identities]
\label{lem:row_hyb_nuc_identities}
For every $G\in\RR^{v\times d}$, let
$\tG=D_\eta(G)G$. Then $\fronorm{\calT_{\row\hyb}(G)}^2 = r_{\row\hyb}(G)$, $\fronorm{\calT_{\row\hyb,\nuc}(G)}^2 = \nucnorm{\tG}^2\,r_{\row\hyb}(G)$, and $    \dotpF{G}{\calT_{\row\hyb,\nuc}(G)} = \mathfrak A_{\row\hyb}(G)\nucnorm{\tG}^2$. 
\end{lemma}

\begin{proof}
Since $\calT_{\row\hyb}(G)$ is the right polar factor of $\tG$, its Frobenius norm squared equals $\rank(\tG)=r_{\row\hyb}(G)$. Multiplying by $\nucnorm{\tG}$ gives
\[
    \fronorm{\calT_{\row\hyb,\nuc}(G)}^2 = \nucnorm{\tG}^2 \fronorm{\calT_{\row\hyb}(G)}^2 = \nucnorm{\tG}^2 r_{\row\hyb}(G).
\]
The alignment identity follows directly from the definition of $\mathfrak A_{\row\hyb}(G)$.
\end{proof}

\begin{theorem}[Descent lemma for the nuclear-norm-scaled row-norm/right-spectral update]
\label{thm:row_hyb_nuc_descent}
Suppose $f$ satisfies \Cref{ass:Lipschitz_smooth}. Then the iteration \eqref{eq:row_hyb_nuc_iteration} satisfies
\[
    f(W_{k+1}) \le f(W_k) - \gamma_k \mathfrak A_{\row\hyb}(G_k)\nucnorm{\tG_k}^2 + \frac{L\gamma_k^2}{2} r_{\row\hyb}(G_k) \nucnorm{\tG_k}^2,
\]
where $\tG_k=D_\eta(G_k)G_k$. Equivalently,
\[
    f(W_{k+1}) \le f(W_k) - \gamma_k \left(\mathfrak A_{\row\hyb}(G_k) - \frac{L\gamma_k}{2}r_{\row\hyb}(G_k)\right)\nucnorm{\tG_k}^2.
\]
In particular, if
\[
    \gamma_k
    \in
    \left(
        0,\,
        \frac{2\mathfrak A_{\row\hyb}(G_k)}
             {Lr_{\row\hyb}(G_k)}
    \right),
\]
then $f(W_{k+1})\le f(W_k)$.
\end{theorem}

\begin{proof}
By $L$-smoothness and \eqref{eq:row_hyb_nuc_iteration},
\[
    f(W_{k+1}) \le (W_k) - \gamma_k \dotpF{G_k}{\calT_{\row\hyb,\nuc}(G_k)} + \frac{L\gamma_k^2}{2} \fronorm{\calT_{\row\hyb,\nuc}(G_k)}^2.
\]
Applying \Cref{lem:row_hyb_nuc_identities} gives the claim.
\end{proof}

\begin{assumption}[Uniform row-norm/right-spectral alignment and rank bounds]
\label{ass:row_hyb_nuc_uniform}
There exist constants $\underline a_{\row}>0$ and $\overline r_{\row}\in\NN^*$ such that for all $k\in\NN$, $\mathfrak A_{\row\hyb}(G_k)\ge \underline a_{\row}$ and $r_{\row\hyb}(G_k)\le \overline r_{\row}$. 
\end{assumption}

We need to make the following extra comparability assumption, which is necessary if we want convergence to stationarity in terms of $\fronorm{G_k}^2$, instead of $\nucnorm{\tG_k}^2$. 
\begin{assumption}[Row-normalization comparability]
\label{ass:row_hyb_comparability}
There exists a constant $\kappa_{\row}>0$ such that for all $k\in\NN$, $\nucnorm{\tG_k} \ge \kappa_{\row}\fronorm{G_k}$, where $\tG_k=D_\eta(G_k)G_k$. 
\end{assumption}

\begin{theorem}[Sublinear convergence to stationarity]
\label{thm:row_hyb_nuc_stationarity}
Suppose
\Cref{ass:Lipschitz_smooth,ass:row_hyb_nuc_uniform,ass:row_hyb_comparability}
hold. If the learning rate $\gamma$ is constant and satisfies $\gamma\in\left(0, 2\underline a_{\row}/(L\overline r_{\row})\right)$, then
\[
    f(W_{k+1}) \le f(W_k) - \gamma\left(\underline a_{\row} - \frac{L\gamma}{2}\overline r_{\row}\right)\nucnorm{\tG_k}^2.
\]
Consequently,
\[
    \sum_{k=0}^{T-1}\fronorm{G_k}^2 \le \frac{f(W_0)-f^\star}{ \kappa_{\row}^2 \gamma\left(\underline a_{\row} - L\gamma\overline r_{\row}/2\right)},
\]
and hence
\[
    \min_{0\le k<T}\fronorm{G_k}^2 \le \frac{f(W_0)-f^\star}{T\kappa_{\row}^2\gamma\left(\underline a_{\row} - L\gamma\overline r_{\row}/2\right)}.
\]
\end{theorem}

\begin{proof}
By \Cref{thm:row_hyb_nuc_descent} and
\Cref{ass:row_hyb_nuc_uniform},
\[
    f(W_{k+1}) \le f(W_k) - \gamma \left(\underline a_{\row} - \frac{L\gamma}{2}\overline r_{\row}\right)\nucnorm{\tG_k}^2.
\]
Summing from $k=0$ to $T-1$ and using $f(W_T)\ge f^\star$ yields
\[
    \sum_{k=0}^{T-1}\nucnorm{\tG_k}^{\,2} \le \frac{f(W_0)-f^\star}{\gamma\left(\underline a_{\row} - L\gamma\overline r_{\row}/2\right)}.
\]
The comparability assumption gives $\nucnorm{\tG_k}^{\,2}\ge\kappa_{\row}^2\fronorm{G_k}^2$, which proves the claim. 
\end{proof}

\begin{theorem}[Linear convergence under the \PL condition]
\label{thm:row_hyb_nuc_PL}
Suppose
\Cref{ass:Lipschitz_smooth,ass:PL,ass:row_hyb_nuc_uniform,ass:row_hyb_comparability}
hold. If the learning rate $\gamma$ is constant and satisfies $\gamma\in\left(0, 2\underline a_{\row}/(L\overline r_{\row})\right)$, then
\[
    f(W_{k+1})-f^\star \le \left(1 - 2\mu\kappa_{\row}^2\gamma\left(\underline a_{\row} - \frac{L\gamma}{2}\overline r_{\row}\right)\right) \bigl(f(W_k)-f^\star\bigr).
\]
Hence
\[
    f(W_k)-f^\star \le \rho_{\row\hyb}^k\bigl(f(W_0)-f^\star\bigr),
\]
where
\[
    \rho_{\row\hyb} = 1 - 2\mu\kappa_{\row}^2\gamma\left(\underline a_{\row} - L\gamma\overline r_{\row}/2\right) \in(0,1).
\]
\end{theorem}

\begin{proof}
By \Cref{thm:row_hyb_nuc_descent} and
\Cref{ass:row_hyb_nuc_uniform},
\[
    f(W_{k+1}) \le f(W_k) - \gamma\left(\underline a_{\row} - \frac{L\gamma}{2}\overline r_{\row}\right)\nucnorm{\tG_k}^2.
\]
Using \Cref{ass:row_hyb_comparability},
\[
    \nucnorm{\tG_k}^2 \ge \kappa_{\row}^2\fronorm{G_k}^2.
\]
The \PL inequality then proves the claimed recursion. Iterating the recursion yields the linear
rate.
\end{proof}

\begin{remark}[Comparison with the right-spectral/row-norm hybrid]
The right-spectral/row-norm hybrid in \Cref{subsec:hyb_nuc} first computes the right polar factor from the original gradient Gram matrix $G^\top G$ and then applies row-wise normalization. In contrast, the row-norm/right-spectral hybrid first replaces $G$ by $\tG=D_\eta(G)G$ and then computes the right polar factor from $\tG^\top\tG$. Thus, the former preserves the original feature geometry before applying local row-wise normalization, while the latter modifies the feature geometry before the spectral step. This distinction is important in practice: row-normalize-first can better suppress row-scale imbalance before the spectral computation, whereas polar-first preserves more of the original right singular geometry.
\end{remark}

Now, we specialize the convergence results when $\eta(t) = 1/(t+\varepsilon)$ for $\varepsilon>0$, further assuming a uniform row-norm bound. 

\begin{assumption}[Uniform row-norm bound]
\label{ass:row_hyb_row_bound}
There exists $R>0$ such that, for all $k\in\NN$, $\max_{i\in\set{v}}\|G_{k,i:}\|_2 \le R$. 
\end{assumption}

\begin{lemma}[Verification for $\eta(t)=1/(t+\varepsilon)$]
\label{lem:row_hyb_eta_eps}
Let $\varepsilon>0$ and define
\[
    D_\varepsilon(G) = \Diag\left(\frac{1}{\|G_{1:}\|_2+\varepsilon}, \ldots, \frac{1}{\|G_{v:}\|_2+\varepsilon} \right),
\]
and $\tG = D_\varepsilon(G)G$. Then, for every $G\neq 0$,
\[
    \mathfrak A_{\row\hyb}(G) = \frac{\dotpF{G}{\polar(\tG)}}{\nucnorm{\tG}} \ge \varepsilon.
\]
Moreover, if $\max_{i\in\set{v}}\|G_{i:}\|_2\le R$, then
\[
    \nucnorm{\tG} \ge \fronorm{\tG} \ge \frac{1}{R+\varepsilon}\fronorm{G}.
\]
Consequently, along any sequence satisfying \Cref{ass:row_hyb_row_bound}, \Cref{ass:row_hyb_nuc_uniform} and \Cref{ass:row_hyb_comparability} hold with
$\underline a_{\row}=\varepsilon$, $\overline r_{\row}=d$, and $\kappa_{\row}=\frac{1}{R+\varepsilon}$. 
\end{lemma}

\begin{proof}
Let $\tG=D_\varepsilon(G)G$ and write $\tU_\sfp=\polar(\tG)$. Since $G=D_\varepsilon(G)^{-1}\tG$, we have $\dotpF{G}{\tU_\sfp} = \dotpF{D_\varepsilon(G)^{-1}\tG}{\tU_\sfp}$. 
Let $\tG=\tU\tSigma\tV^\top$ be a compact singular value decomposition. Then $\tU_\sfp=\tU\tV^\top$, and therefore
\[
    \dotpF{D_\varepsilon(G)^{-1}\tG}{\tU_\sfp} = \tr\left( \tSigma \tU^\top D_\varepsilon(G)^{-1}\tU \right).
\]
Because $D_\varepsilon(G)^{-1} = \Diag(\|G_{1:}\|_2+\varepsilon,\ldots,\|G_{v:}\|_2+\varepsilon) \succeq \varepsilon I_v$, we obtain
\[
    \tr\left( \tSigma \tU^\top D_\varepsilon(G)^{-1}\tU \right) \ge \varepsilon\,\tr(\tSigma) = \varepsilon\nucnorm{\tG}.
\]
This proves $\mathfrak A_{\row\hyb}(G)\ge \varepsilon$. 

For the comparability bound, if $\max_{i\in\set{v}}\|G_{i:}\|_2\le R$, then
\[
    \frac{1}{\|G_{i:}\|_2+\varepsilon} \ge \frac{1}{R+\varepsilon}.
\]
Hence
\[
    \fronorm{\tG}^2 = \sum_{i=1}^v \frac{\|G_{i:}\|_2^2}{(\|G_{i:}\|_2+\varepsilon)^2} \ge \frac{1}{(R+\varepsilon)^2} \fronorm{G}^2.
\]
Since $\nucnorm{\tG}\ge \fronorm{\tG}$, the claimed comparability bound follows. Finally, $r_{\row\hyb}(G)=\rank(\tG)\le d$, so we may take $\overline r_{\row}=d$.
\end{proof}

\begin{corollary}[Convergence for $\eta(t)=1/(t+\varepsilon)$]
\label{cor:row_hyb_eta_eps_convergence}
Suppose \Cref{ass:Lipschitz_smooth,ass:row_hyb_row_bound} hold and consider the row-norm/right-spectral hybrid optimizer with $\eta(t)=1/(t+\varepsilon)$. If the learning rate $\gamma$ is constant and satisfies $\gamma\in(0, 2\varepsilon/(Ld))$, then
\[
    f(W_{k+1}) \le f(W_k) - \gamma\left(\varepsilon-\frac{L\gamma}{2}d\right)\nucnorm{\tG_k}^2.
\]
Moreover,
\[
    \min_{0\le k<T}\fronorm{G_k}^2 \le \frac{(R+\varepsilon)^2(f(W_0)-f^\star)}{\gamma(\varepsilon-L\gamma d/2)}.
\]
If, in addition, $f$ satisfies the $\mu$-\PL condition, then
\[
    f(W_k)-f^\star \le \left(1 - \frac{ 2\mu\gamma (\varepsilon-L\gamma d/2)}{(R+\varepsilon)^2}\right)^{\negthickspace k}\bigl(f(W_0)-f^\star\bigr).
\]
\end{corollary}

\newpage
\section{Experimental Details}
\label{sec:details_expt}

Note that we reduce the number of hidden layers and the number of experts from the original architecture. We give the main modified model architecture specifications of the model in \Cref{table:architectures}. Their detailed designs can be found in their corresponding technical reports \citep{qwen3technicalreport,gemmateam2025gemma3technicalreport,muennighoff2024olmoe,agarwal2025gpt}. We initialize all 2D weights by Gaussian random numbers with zero mean and standard deviation 0.02. 

\begin{table}[H]    
    \centering
    \caption{Modified model architectures of language models.}
    \label{table:architectures}
    \small
    \begin{tabular}{rrrrr}
    	\toprule
    	Model & Qwen3-0.6B & Gemma 3 1B & \OLMoE-1B-7B & gpt-oss \\
    	\midrule
    	\# trainable parameters & 625,784,832 & 1,087,138,944  & 2,824,177,664 & 3,467,779,008 \\
    	$d_{\mathrm{model}}$ & 1024 & 1152 & 2048 & 2048 \\
        $d_{\mathrm{ff}}$ & 3072 & 6912 & 1024 & 2048 \\
    	$n_{\mathrm{layers}}$ & 20 & 18 & 12 & 12 \\
    	$n_{\mathrm{heads}}$ (Q / KV) & 16 / 8 & 4 / 1 & 16 / 16 & 64 / 8 \\
        $d_{\mathrm{heads}}$ & 128 & 256 & 128 & 64 \\
        $n_{\mathrm{experts}}$ & 1 & 1 & 32 & 16 \\
        $n_{\mathrm{experts}}$ activated & N/A & N/A & 8 & 4 \\
    	vocabulary size & 151,936 & 262,144 & 50,304 & 201,088\\
    	layer norm & RMSNorm & RMSNorm & RMSNorm & RMSNorm \\
    	activation function & SwiGLU & GeGLU with $\tanh$ & SwiGLU & SwiGLU\tablefootnote{The SwiGLU implementation in gpt-oss is unconventional as it includes clamping and residual connection. } \\
    	\bottomrule
    \end{tabular}
\end{table}

In the following experiments, we use \textsc{Polar Express} \citep{amsel2025polar} for computing the matrix inverse square root in \LeftPolarGradM and Gram Newton--Schulz \citep{zhang2026gram} for computing the orthogonal polar factor directly in \HybridPolarGradM, respectively. Unless otherwise specified, for \textsc{Polar Express} and Gram Newton--Schulz, we use 5 inner steps with $\varepsilon_{\mathrm{NS}} = 10^{-7}$. We use $\varepsilon=10^{-8}$ for all of \RowNormM, \LeftPolarGradM and \HybridPolarGradM. The row-scaling rule in \RowNormM and \HybridPolarGradM is chosen to be the \emph{smoothed row normalization}, i.e., $\eta(t) = 1/(t+\varepsilon)$ with $\varepsilon=10^{-8}$. 

For \AdamW on scalar and vector parameters, we use a linear warmup with cosine decay learning rate schedule with 100 warmup steps and a half-cosine decay to $0$. For all other parameters, regardless of the choice of the optimizers, we use a stable-decay schedule with an initial learning rate $\gamma_0$ for the first 60\% of training steps and linear decay to $0$ for the last 40\% training steps. We use the fused implementation of \AdamW in PyTorch, while the implementations of \RowNormM, \LeftPolarGradM, \RightPolarGradM and \HybridPolarGradM are not optimized with customized kernels, except for the usage of Gram Newton--Schulz \citep{zhang2026gram}. 

We give the training configurations and optimizer hyperparameters of all four model pre-training experiments in the following subsections.

\newpage
\subsection{Qwen3-0.6B-Style Pre-Training}

In terms of wall-clock training time, for (a), configurations (i)--(iii) take 2.273 hours, 2.406 hours and 2.248 hours respectively, while for (b), configurations (i)--(iii) take 2.685 hours, 2.790 hours and 2.678 hours respectively. The time difference between (a) and (b) is expected as \HybridPolarGradM for SwiGLU MLP projection matrices have additional computational overheads than \Muon. 
\begin{table}[H]
   	\centering
   	\caption{Training configurations for Qwen3-0.6B-style pre-training.}
   	\label{table:config_qwen3}
    \small
   	\begin{tabular}{lc}
        \toprule
        Model & Qwen3-0.6B on FineWeb-Edu-10B \\
        \midrule
        Context length & 1024 tokens \\
        Per-device batch size & 28 sequences \\
        Training steps & 30,000 \\
        Training tokens & 6,881,280,000 \\        
        Validation steps & 46 \\
        Validation tokens & 10,551,296 \\
        Precision & \texttt{bfloat16} \\         
        Data-parallel size & 8 (Nvidia H200) \\   
        \bottomrule
   	\end{tabular}
\end{table}

\begin{table}[H]
    \centering
    \caption{Optimizer hyperparameters for Qwen3-0.6B-style pre-training.}
    \label{table:optim_hyperparams_qwen3}
    \scriptsize
    \begin{tabular}{cccccc}
        \toprule
        Configuration & Parameter type & Optimizer & $\gamma_0$ & momentum $\beta$ & weight decay $\lambda$ \\
        \midrule
        \multirow{5}{*}{\shortstack{(a)(i) \AdamW + \Muon \\ + \Muon \\ + \RowNormM}}
            & scalar / vector    & \AdamW            & $0.05$     & $(0.9, 0.95)$ & $0.001$ \\
            & linear / attention & \Muon             & $0.02$     & $0.95$        & $0.001$ \\
            & SwiGLU MLP         & \Muon             & $0.02$     & $0.95$        & $0.001$ \\
            & embedding          & \RowNormM         & $0.50$     & $0.95$        & $0$ \\
            & head               & \RowNormM         & $0.005$    & $0.95$        & $0$ \\
        \midrule
        \multirow{5}{*}{\shortstack{(a)(ii) \AdamW + \Muon \\ + \Muon \\ + \HybridPolarGradM \\ (row-norm/right-spectral; $\alpha=1$)}}
            & scalar / vector    & \AdamW            & $0.05$     & $(0.9, 0.95)$ & $0.001$ \\
            & linear / attention & \Muon             & $0.02$     & $0.95$        & $0.001$ \\
            & SwiGLU MLP         & \Muon             & $0.02$     & $0.95$        & $0.001$ \\
            & embedding          & \HybridPolarGradM & $0.005$     & $0.95$        & $0$ \\
            & head               & \HybridPolarGradM & $0.0001$     & $0.95$        & $0$ \\
        \midrule
        \multirow{5}{*}{\shortstack{(a)(iii) \AdamW + \Muon \\ + \Muon + \AdamW}}
            & scalar / vector    & \AdamW            & $0.05$     & $(0.9, 0.95)$ & $0.001$ \\
            & linear / attention & \Muon             & $0.02$     & $0.95$        & $0.001$ \\
            & SwiGLU MLP         & \Muon             & $0.02$     & $0.95$        & $0.001$ \\
            & embedding          & \AdamW            & $0.10$     & $(0.9, 0.95)$ & $0$ \\
            & head               & \AdamW            & $0.001$    & $(0.9, 0.95)$ & $0$ \\
        \midrule
        \multirow{6}{*}{\shortstack{(b)(i) \AdamW + \Muon \\ + \HybridPolarGradM \\ (row-norm/right-spectral; $\alpha=0$) \\ + \RowNormM}}
            & scalar / vector    & \AdamW            & $0.05$     & $(0.9, 0.95)$ & $0.001$ \\
            & linear / attention & \Muon             & $0.02$     & $0.95$        & $0.001$ \\
            & SwiGLU gate/up     & \HybridPolarGradM & $0.02\sqrt{3}$ & $0.95$        & $0.001$ \\
            & SwiGLU down        & \HybridPolarGradM & $0.02$     & $0.95$        & $0.001$ \\
            & embedding          & \RowNormM         & $0.50$     & $0.95$        & $0$ \\
            & head               & \RowNormM         & $0.005$    & $0.95$        & $0$ \\
        \midrule
        \multirow{6}{*}{\shortstack{(b)(ii) \AdamW + \Muon \\ + \HybridPolarGradM \\ + \HybridPolarGradM \\ (row-norm/right-spectral) \\ (MLP: $\alpha=0$; \\ embedding/head: $\alpha=1$)}}
            & scalar / vector    & \AdamW            & $0.05$     & $(0.9, 0.95)$ & $0.001$ \\
            & linear / attention & \Muon             & $0.02$     & $0.95$        & $0.001$ \\
            & SwiGLU gate/up     & \HybridPolarGradM & $0.02\sqrt{3}$ & $0.95$        & $0.001$ \\
            & SwiGLU down        & \HybridPolarGradM & $0.02$     & $0.95$        & $0.001$ \\
            & embedding          & \HybridPolarGradM & $0.005$     & $0.95$        & $0$ \\
            & head               & \HybridPolarGradM & $0.0001$     & $0.95$        & $0$ \\
        \midrule
        \multirow{6}{*}{\shortstack{(b)(iii) \AdamW + \Muon \\ + \HybridPolarGradM \\ (row-norm/right-spectral; $\alpha=0$) \\ + \AdamW}}
            & scalar / vector    & \AdamW            & $0.05$     & $(0.9, 0.95)$ & $0.001$ \\
            & linear / attention & \Muon             & $0.02$     & $0.95$        & $0.001$ \\
            & SwiGLU gate/up     & \HybridPolarGradM & $0.02\sqrt{3}$ & $0.95$        & $0.001$ \\
            & SwiGLU down        & \HybridPolarGradM & $0.02$     & $0.95$        & $0.001$ \\
            & embedding          & \AdamW            & $0.10$     & $(0.9, 0.95)$ & $0$ \\
            & head               & \AdamW            & $0.001$    & $(0.9, 0.95)$ & $0$ \\
        \bottomrule
    \end{tabular}
\end{table}

\newpage
\subsection{Gemma 3 1B-Style Pre-Training}
\label{subsec:gemma3_add_expt}
In terms of wall-clock training time, for (a), configurations (i)--(iii) take 3.749 hours, 4.050 hours and 3.682 hours respectively, whereas for (b), configurations (i)--(iii) take 4.216 hours, 4.566 hours and 4.191 hours respectively. 

\begin{table}[H]
   	\centering
   	\caption{Training configurations for Gemma 3 1B-style pre-training.}
   	\label{table:config_gemma3}
    \small
   	\begin{tabular}{lc}
        \toprule
        Model & Gemma 3 1B on FineWeb-Edu-10B \\
        \midrule
        Context length & 1024 tokens \\
        Per-device batch size &  18 sequences \\
        Training steps & 50,000 \\
        Training tokens & 7,372,800,000 \\        
        Validation steps & 71 \\
        Validation tokens & 10,469,376 \\
        Precision & \texttt{bfloat16} \\         
        Data-parallel size & 8 (Nvidia H200) \\   
        \bottomrule
   	\end{tabular}
\end{table}

\begin{table}[H]
    \centering
    \caption{Optimizer hyperparameters for Gemma 3 1B-style pre-training.}
    \label{table:optim_hyperparams_gemma3}
    \scriptsize
    \begin{tabular}{cccccc}
        \toprule
        Configuration & Parameter type & Optimizer & $\gamma_0$ & momentum $\beta$ & weight decay $\lambda$ \\
        \midrule
        \multirow{5}{*}{\shortstack{(a)(i) \AdamW + \Muon \\ + \Muon \\ + \RowNormM}}
            & scalar / vector    & \AdamW            & $0.05$        & $(0.9, 0.95)$ & $0.001$ \\
            & linear / attention & \Muon             & $0.02$        & $0.95$        & $0.001$ \\
            & SwiGLU MLP         & \Muon             & $0.02$        & $0.95$        & $0.001$ \\
            & embedding          & \RowNormM         & $0.01$      & $0.95$        & $0$ \\
            & head               & \RowNormM         & $0.01$      & $0.95$        & $0$ \\
        \midrule
        \multirow{5}{*}{\shortstack{(a)(ii) \AdamW + \Muon \\ + \Muon \\ + \HybridPolarGradM \\ (row-norm/right-spectral; $\alpha=1$)}}
            & scalar / vector    & \AdamW            & $0.05$        & $(0.9, 0.95)$ & $0.001$ \\
            & linear / attention & \Muon             & $0.02$        & $0.95$        & $0.001$ \\
            & SwiGLU MLP         & \Muon             & $0.02$        & $0.95$        & $0.001$ \\
            & embedding          & \HybridPolarGradM & $0.00005$      & $0.95$        & $0$ \\
            & head               & \HybridPolarGradM & $0.00005$      & $0.95$        & $0$ \\
        \midrule
        \multirow{5}{*}{\shortstack{(a)(iii) \AdamW + \Muon \\ + \Muon + \AdamW}}
            & scalar / vector    & \AdamW            & $0.05$        & $(0.9, 0.95)$ & $0.001$ \\
            & linear / attention & \Muon             & $0.02$        & $0.95$        & $0.001$ \\
            & SwiGLU MLP         & \Muon             & $0.02$        & $0.95$        & $0.001$ \\
            & embedding          & \AdamW            & $0.001$      & $(0.9, 0.95)$ & $0$ \\
            & head               & \AdamW            & $0.001$      & $(0.9, 0.95)$ & $0$ \\
        \midrule
        \multirow{6}{*}{\shortstack{(b)(i) \AdamW + \Muon \\ + \HybridPolarGradM \\ (row-norm/right-spectral; $\alpha=0$) \\ + \RowNormM}}
            & scalar / vector    & \AdamW            & $0.05$        & $(0.9, 0.95)$ & $0.001$ \\
            & linear / attention & \Muon             & $0.02$        & $0.95$        & $0.001$ \\
            & SwiGLU gate/up     & \HybridPolarGradM & $0.02\sqrt{6}$ & $0.95$       & $0.001$ \\
            & SwiGLU down        & \HybridPolarGradM & $0.02$        & $0.95$        & $0.001$ \\
            & embedding          & \RowNormM         & $0.005$      & $0.95$        & $0$ \\
            & head               & \RowNormM         & $0.005$      & $0.95$        & $0$ \\
        \midrule
        \multirow{6}{*}{\shortstack{(b)(ii) \AdamW + \Muon \\ + \HybridPolarGradM \\ + \HybridPolarGradM \\ (row-norm/right-spectral) \\ (MLP: $\alpha=0$; \\ embedding/head: $\alpha=1$)}}
            & scalar / vector    & \AdamW            & $0.05$        & $(0.9, 0.95)$ & $0.001$ \\
            & linear / attention & \Muon             & $0.02$        & $0.95$        & $0.001$ \\
            & SwiGLU gate/up     & \HybridPolarGradM & $0.02\sqrt{6}$ & $0.95$       & $0.001$ \\
            & SwiGLU down        & \HybridPolarGradM & $0.02$        & $0.95$        & $0.001$ \\
            & embedding          & \HybridPolarGradM & $0.00005$       & $0.95$        & $0$ \\
            & head               & \HybridPolarGradM & $0.00005$       & $0.95$        & $0$ \\
        \midrule
        \multirow{6}{*}{\shortstack{(b)(iii) \AdamW + \Muon \\ + \HybridPolarGradM \\ (row-norm/right-spectral; $\alpha=0$) \\ + \AdamW}}
            & scalar / vector    & \AdamW            & $0.05$        & $(0.9, 0.95)$ & $0.001$ \\
            & linear / attention & \Muon             & $0.02$        & $0.95$        & $0.001$ \\
            & SwiGLU gate/up     & \HybridPolarGradM & $0.02\sqrt{6}$ & $0.95$       & $0.001$ \\
            & SwiGLU down        & \HybridPolarGradM & $0.02$        & $0.95$        & $0.001$ \\
            & embedding          & \AdamW            & $0.001$      & $(0.9, 0.95)$ & $0$ \\
            & head               & \AdamW            & $0.001$      & $(0.9, 0.95)$ & $0$ \\
        \bottomrule
    \end{tabular}
\end{table}

\subsubsection{Gemma 3 1B-Style Pre-Training Learning Rate Sweep}
For this pre-training experiment, we also perform a base learning rate sweep for the embedding and LM head matrices, keeping the learning rates for scalars/vector and matrices fixed. For simplicity, we keep both base learning rates for the embedding and LM head matrices to be equal, although more delicate tuning is possible. 

\begin{table}[h!]
    \centering
    \caption{Base learning-rate sweep for the input embedding and LM head matrices in Gemma 3 1B-style pre-training. We sweep only     $\gamma_{0,\mathrm{emb}}=\gamma_{0,\mathrm{head}}$. In setting (a), SwiGLU MLP projection matrices use \Muon; in setting (b), they use \HybridPolarGradM with a row-norm/right-spectral composition.}
    \label{table:gemma3_lr_sweep}
    \scriptsize
    \begin{tabular}{ccccc}
        \toprule
        Setting &
        Embedding/LM head optimizer &
        $\gamma_{0,\mathrm{emb}}=\gamma_{0,\mathrm{head}}$ &
        Final validation loss & Training time (hr) \\
        \midrule
        \multirow{12}{*}{\shortstack{(a) SwiGLU MLP:\\ \Muon}}
        & \multirow{4}{*}{\RowNormM}
            & $0.001$  & $2.7567$ & $3.721$ \\
        &   & $0.0025$ & $2.7483$ & $3.708$ \\
        &   & $0.005$  & $2.7453$ & $3.709$ \\
        &   & $0.01$   & $\mathbf{2.7434}$ & $3.749$ \\
        \cmidrule(lr){2-5}
        & \multirow{4}{*}{\shortstack{\HybridPolarGradM\\ (row-norm/right-spectral, $\alpha=1$)}}
          & $0.00005$   & $\mathbf{2.7492}$ & $4.050$ \\
        &    & $0.0001$  & $2.7496$ & $4.080$ \\
        &   & $0.00025$ & $2.7531$ & $4.128$ \\
        &   & $0.0005$  & $2.7713$ & $4.086$ \\
        \cmidrule(lr){2-5}
        & \multirow{4}{*}{\AdamW}
           & $0.00025$  & $2.7855$ & $3.690$ \\
        &   & $0.0005$ & $2.7781$ & $3.692$ \\
        &   & $0.001$  & $\mathbf{2.7761}$ & $3.682$ \\
        &   & $0.0025$  & $2.7776$ & $3.654$ \\
        \midrule
        \multirow{12}{*}{\shortstack{(b) SwiGLU MLP:\\ \HybridPolarGradM}}
        & \multirow{4}{*}{\RowNormM}
            & $0.001$  & $2.7467$ & $4.220$ \\
        &   & $0.0025$ &  $2.7385$ & $4.164$ \\
        &   & $0.005$  & $\mathbf{2.7251}$ & $4.216$ \\
        &   & $0.01$   & $2.7251$ & $4.267$ \\
        \cmidrule(lr){2-5}
        & \multirow{4}{*}{\shortstack{\HybridPolarGradM\\ (row-norm/right-spectral, $\alpha=1$)}}
          & $0.00005$   & $\mathbf{2.7160}$ & $4.566$ \\
        &    & $0.0001$  & $2.7199$ & $4.616$ \\
        &   & $0.00025$ & $2.7252$ & $4.620$ \\
        &   & $0.0005$  & $2.7362$ & $4.589$ \\        
        \cmidrule(lr){2-5}
        & \multirow{4}{*}{\AdamW}
           & $0.00025$  & $2.7790$ & $4.152$  \\
        &   & $0.0005$ & $2.7654$ & $4.164$ \\
        &   & $0.001$  & $\mathbf{2.7562}$ & $4.191$ \\
        &   & $0.0025$  & $2.7603$ & $4.029$ \\        
        \bottomrule
    \end{tabular}
\end{table}

We observe from \Cref{table:gemma3_lr_sweep} that the validation loss gaps between configuration (iii) and configurations (i)--(ii) remain substantial across the swept base learning rates. As shown in \Cref{fig:gemma3_2}, the separation between the \AdamW embedding/LM head baselines and the symmetry-compatible alternatives is not explained by a single learning-rate choice. Across the swept values of $\gamma_{0,\mathrm{emb}}=\gamma_{0,\mathrm{head}}$, the \AdamW curves make comparable or slightly faster initial progress, but remain consistently above the \RowNormM and \HybridPolarGradM curves later in training. This qualitative ordering holds both when the SwiGLU MLP projection matrices use \Muon and when they use \HybridPolarGradM. Thus, the improvement from symmetry-compatible vocabulary-indexed updates is robust to reasonable base learning-rate variation and does not depend on simultaneously replacing \Muon on the SwiGLU MLP projections. 

\begin{figure}[h!]
    \centering
    \begin{subfigure}[t]{\textwidth}
        \centering
        \includegraphics[width=\textwidth]{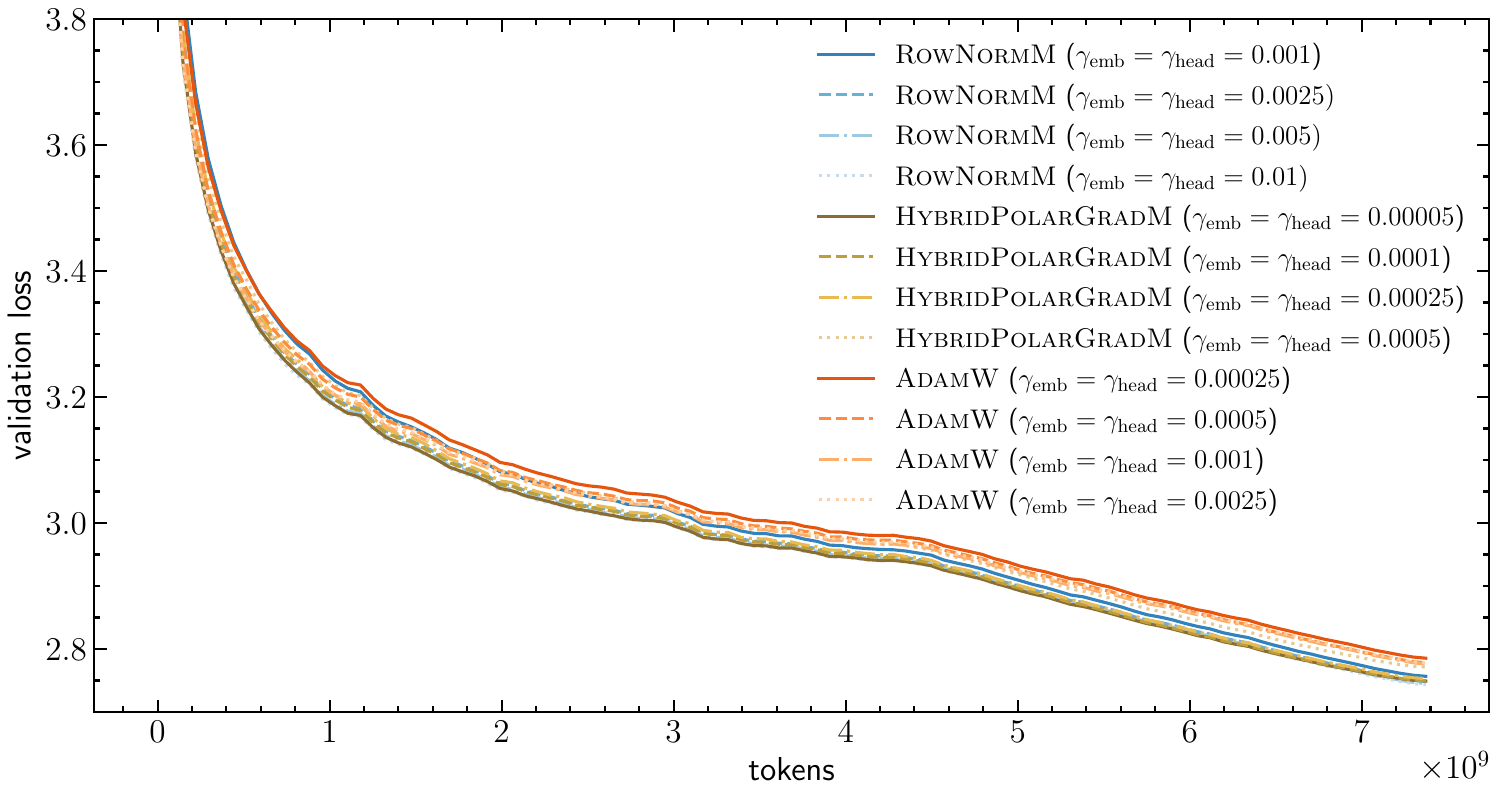}
        \caption{SwiGLU MLP projection matrices use \Muon.}
        \label{fig:gemma3_2_muon_swiglu}
    \end{subfigure}%
    \vspace*{2.5mm}
    \begin{subfigure}[t]{\textwidth}
        \centering
        \includegraphics[width=\textwidth]{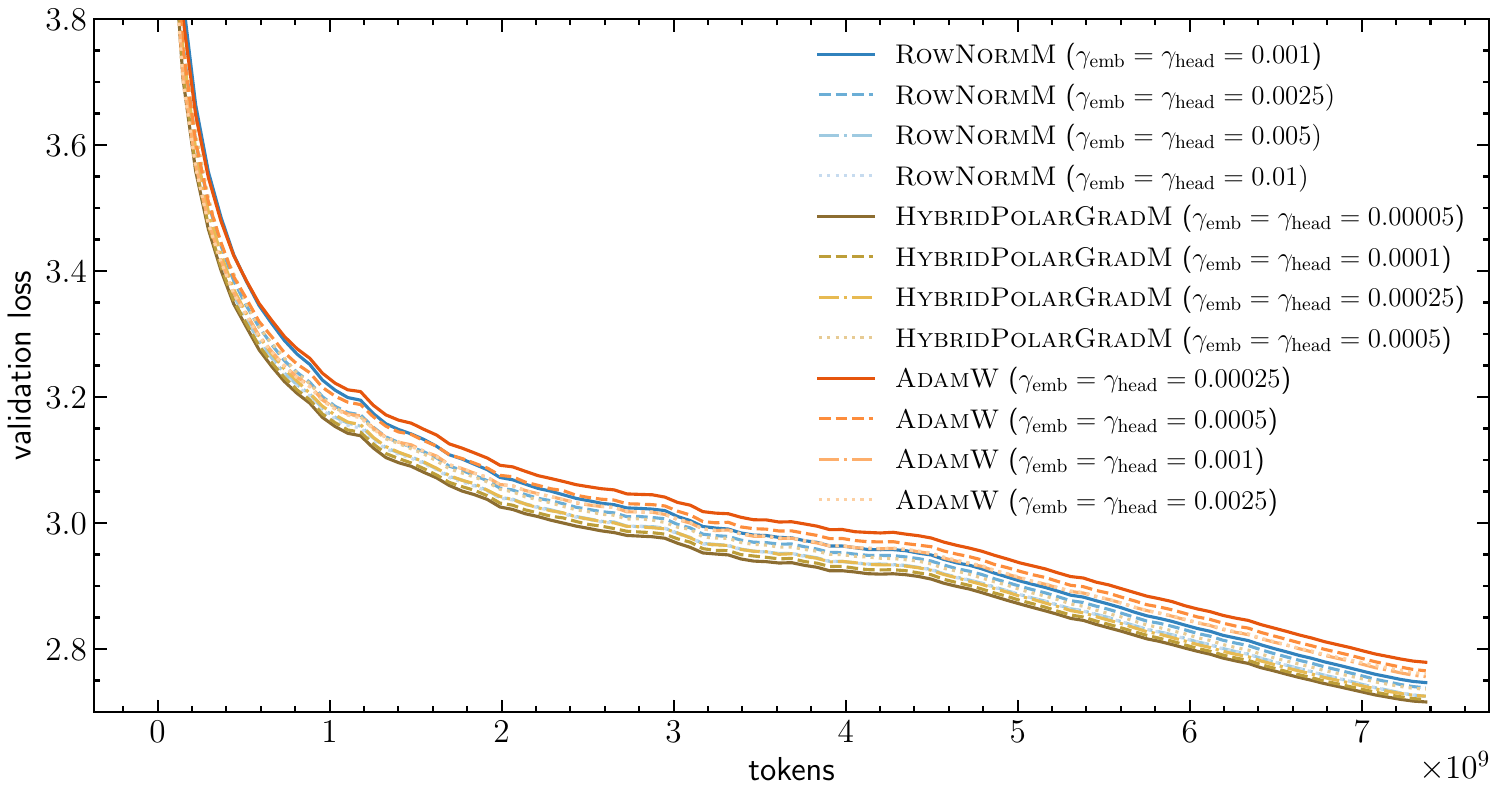}
        \caption{SwiGLU MLP projection matrices use \HybridPolarGradM with a row-norm/right-spectral composition.}
        \label{fig:gemma3_2_hybrid_swiglu}
    \end{subfigure}
    \caption{Validation loss curves for the Gemma 3 1B-style embedding/LM head learning-rate sweep. The swept learning rate is $\gamma_{0,\mathrm{emb}}=\gamma_{0,\mathrm{head}}$. The two panels differ only in the optimizer assigned to the SwiGLU MLP projection matrices: \Muon in \Cref{fig:gemma3_2_muon_swiglu} and \HybridPolarGradM in \Cref{fig:gemma3_2_hybrid_swiglu}. Across both settings, \RowNormM and \HybridPolarGradM remain consistently better than \AdamW for the input embedding and LM head matrices.}
    \label{fig:gemma3_2}
\end{figure}

\subsubsection{Gemma 3 1B-Style Pre-Training Across Random Seeds}
\label{subsubsec:gemma3_seeds}
To assess the robustness of the observed optimizer ordering, we repeat the Gemma 3 1B-style pre-training experiment with two additional random seeds. As in \Cref{subsec:gemma3}, we consider the same three optimizer assignments for the input embedding and LM head matrices: (i) \RowNormM, (ii) \HybridPolarGradM, and (iii) \AdamW. In all runs in this subsection, the SwiGLU MLP projection matrices use \HybridPolarGradM with a row-norm/right-spectral composition.

\begin{figure}[h!]
    \centering
    \begin{subfigure}[t]{\textwidth}
        \centering
        \includegraphics[width=\textwidth]{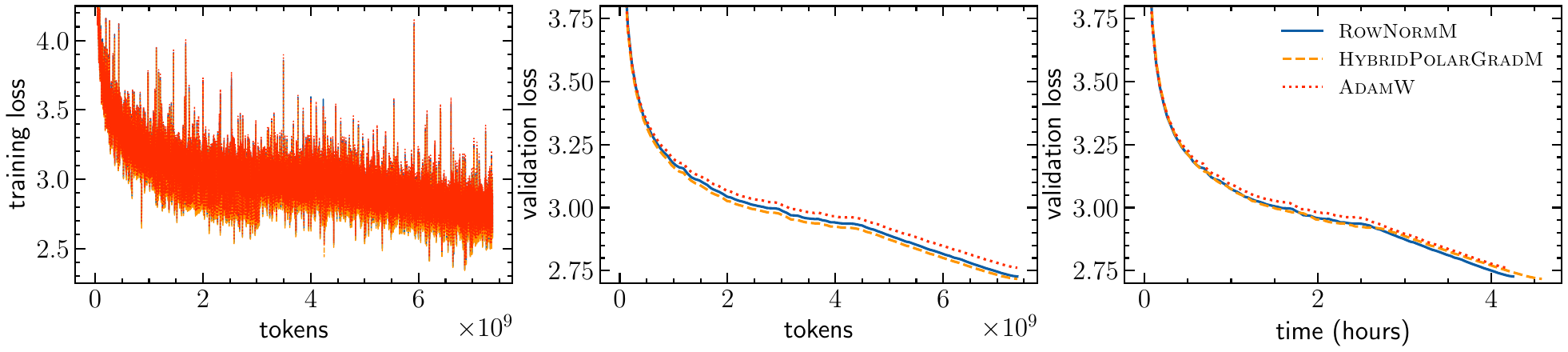}
        \caption{Second random seed.}
        \label{fig:gemma3_hybrid_mlp_142}
    \end{subfigure}%
    \vspace*{2.5mm}
    \begin{subfigure}[t]{\textwidth}
        \centering
        \includegraphics[width=\textwidth]{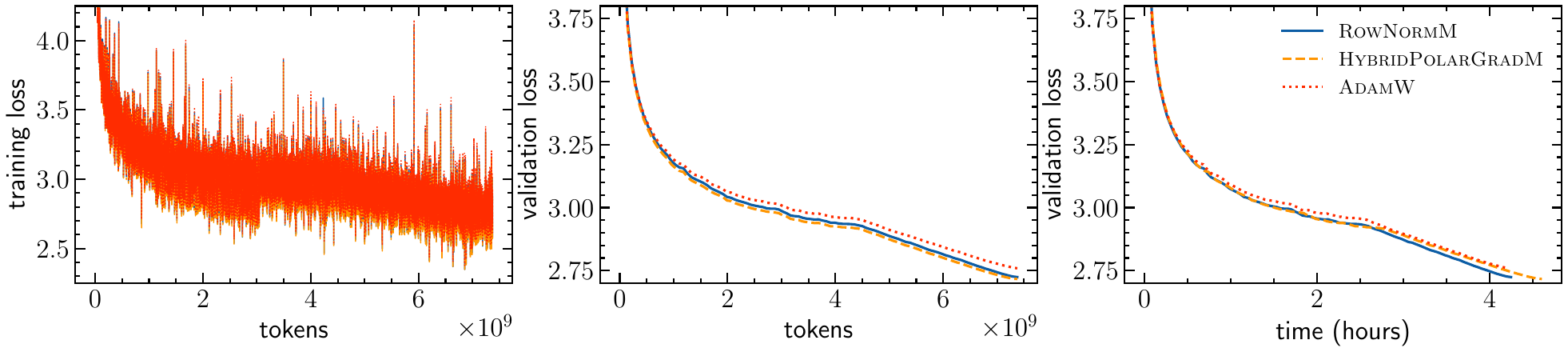}
        \caption{Third random seed.}
        \label{fig:gemma3_hybrid_mlp_242}
    \end{subfigure}
    \caption{Training and validation losses for Gemma 3 1B-style pre-training under two additional random seeds. In each subfigure, the three configurations differ only in the optimizer assigned to the input embedding and LM head matrices: \RowNormM, \HybridPolarGradM, or \AdamW. The SwiGLU MLP projection matrices use \HybridPolarGradM in all runs.}
    \label{fig:gemma3_seeds}
\end{figure}

\vspace*{-2.5mm}
\begin{table}[h!]
    \centering
    \caption{Final validation losses across three random seeds for Gemma 3 1B-style pre-training. The optimizer assignment varies only for the input embedding and LM head matrices; SwiGLU MLP projection matrices use \HybridPolarGradM throughout.}
    \label{tab:gemma3_seeds}
    \small
    \begin{tabular}{@{}lcccc@{}}
        \toprule
        Embedding/LM head optimizer & Seed 1 & Seed 2 & Seed 3 & Mean loss $\pm$ sample std \\
        \midrule
        \RowNormM & $2.7251$ & $2.7267$ & $2.7243$ & $2.7254 \pm 0.0012$ \\
        \HybridPolarGradM & $2.7160$ & $2.7172$ & $2.7166$ & $\mathbf{2.7166 \pm 0.0006}$ \\
        \AdamW & $2.7562$ & $2.7612$ & $2.7591$ & $2.7588 \pm 0.0025$ \\
        \bottomrule
    \end{tabular}
\end{table}

As shown in \Cref{tab:gemma3_seeds}, the qualitative ordering is stable across random seeds. Both symmetry-compatible optimizer assignments, \RowNormM and \HybridPolarGradM, consistently outperform the \AdamW baseline on the vocabulary-indexed matrices. \HybridPolarGradM achieves the lowest mean final validation loss, improving over \AdamW by approximately $0.0422$, while \RowNormM improves over \AdamW by approximately $0.0335$. These results indicate that the gains observed in the main Gemma 3 1B-style experiment are not an artifact of a single random initialization.

\subsubsection{Vocabulary-Logit Growth and Final Logit Soft-Capping}
\label{subsubsec:gemma_logit_softcapping}
We additionally study whether symmetry-compatible LM head updates affect the growth of final vocabulary logits in dense pre-training. This diagnostic is motivated by prior work on training instabilities in large Transformer models. \citet{chowdhery2022palm} introduce an auxiliary output z-loss, proportional to $\log^2 Z$, to keep the softmax normalizer $\log Z$ close to zero and improve stability. \citet{wortsman2024small} study small-scale proxies for large-scale Transformer instabilities, including attention-logit growth and divergence between output logits and log probabilities. Gemma 2 \citep{team2024gemma2} instead uses explicit logit soft-capping in both attention layers and the final vocabulary layer, applying $z \leftarrow c\tanh(z/c)$ with soft-cap values $c=50$ for attention logits and $c=30$ for final logits.

Our goal is to test whether optimizer geometry provides a complementary mechanism for controlling final vocabulary logits. For an LM head $W_{\rmout}\in\RR^{v\times d}$, the final logits are $z_t=W_{\rmout}h_t\in\RR^v$. Since the softmax distribution is invariant under shared shifts $z_t\mapsto z_t+c\One_v$, the LM head has a natural shared-logit-shift quotient geometry. The projected LM head updates proposed in this paper remove this symmetry-redundant direction from the actual parameter update, and may therefore reduce unnecessary or unstable growth of the final vocabulary logits.

We run Gemma 3 1B-style pre-training while varying the LM head optimizer and the final-logit soft-capping setting. Attention-logit soft-capping is disabled in all runs, so the diagnostic isolates the final vocabulary logits. In addition to validation cross-entropy, we log final-logit diagnostics on sampled validation token positions. Let $\calS$ denote the sampled positions. For each $t\in\calS$, let 
\[z_t = (z_{t,i})_{1\le i \le v} =W_{\rmout}h_t\in\RR^v, \qquad\oz_t\coloneqq \frac{1}{v}\sum_{i=1}^v z_{t,i}, \qquad \tz_t\coloneqq z_t-\oz_t\One_v. \]
We report the raw final-logit RMS, centered final-logit RMS, and maximum log-sum-exp given by
\[
    \left(\frac{1}{|\calS|}\sum_{t\in \calS}\frac{\euclidnorm{z_t}^2}{v}\right)^{\negthickspace\negthinspace\half},
    \qquad
    \left(\frac{1}{|\calS|}\sum_{t\in \calS}\frac{\euclidnorm{\tz_t}^2}{v}\right)^{\negthickspace\negthinspace\half},
    \qquad
    \max_{t\in \calS}\log\sum_{i=1}^{v}\exp(z_{t,i}).
\]
The raw RMS measures the overall logit scale, including the shared-logit-shift component. The centered RMS removes this shared component and measures logit growth on the LM head quotient space. The maximum log-sum-exp is a spike-sensitive diagnostic for unusually large final vocabulary logits. Under distributed training, this maximum is reduced across devices and therefore corresponds to the maximum over the union of sampled positions across ranks.

This experiment tests whether projected symmetry-compatible LM head updates provide an optimizer-side form of logit control. The quotient-geometry prediction is not that such updates shrink all logit variation, but that they suppress drift in the shared-logit-shift direction while preserving the centered logits that determine the softmax distribution. Thus, a positive outcome is that, without final logit soft-capping, projected \RowNormM or \HybridPolarGradM LM head updates reduce raw logit RMS and log-sum-exp growth while maintaining comparable centered-logit diagnostics and comparable or improved validation cross-entropy relative to the corresponding \AdamW LM head baseline.

In the figures below, \RowNormM and \HybridPolarGradM are configurations (b)(i) and (b)(ii) in \Cref{table:optim_hyperparams_gemma3}, respectively, where the same optimizer is used for both the embedding and LM head matrices. \RowNormM + \AdamW and \HybridPolarGradM + \AdamW use the first optimizer for the embedding and \AdamW for the LM head. All other optimizer hyperparameters are kept as in the original Gemma 3 1B-style setting given in \Cref{table:optim_hyperparams_gemma3}.

\paragraph{Without final logit soft-capping.}
We first consider Gemma 3 1B-style pre-training without final logit soft-capping. As shown in \Cref{fig:gemma3_logit_no_softcapping}, projected \RowNormM and \HybridPolarGradM LM head updates achieve slightly lower validation loss than the corresponding configurations using \AdamW on the LM head. The \AdamW LM head variants exhibit substantially larger raw logit RMS, with a pronounced late-stage surge for \RowNormM + \AdamW. The projected updates keep the raw logit RMS much smaller and avoid a comparable surge. Maximum log-sum-exp follows a different pattern: projected \HybridPolarGradM exhibits the largest late-stage values despite maintaining the smallest raw logit RMS.

The centered diagnostics indicate that the surge in \RowNormM + \AdamW is driven mainly by the shared-logit-shift component: its centered logit RMS increases only modestly and remains below that of projected \RowNormM. Moreover, raw and centered logit RMS nearly coincide for both projected methods, whereas the \AdamW LM head variants retain a substantial gap. The \HybridPolarGradM + \AdamW configuration also shows a late-stage rise in centered logit RMS. These observations support the quotient-geometry interpretation: projected LM-head updates suppress drift in a symmetry-redundant direction while slightly improving validation performance, without necessarily reducing extreme-value diagnostics such as maximum log-sum-exp.

\begin{figure}[h!]
    \centering
    \begin{subfigure}[t]{\textwidth}
        \centering
        \includegraphics[width=\textwidth]{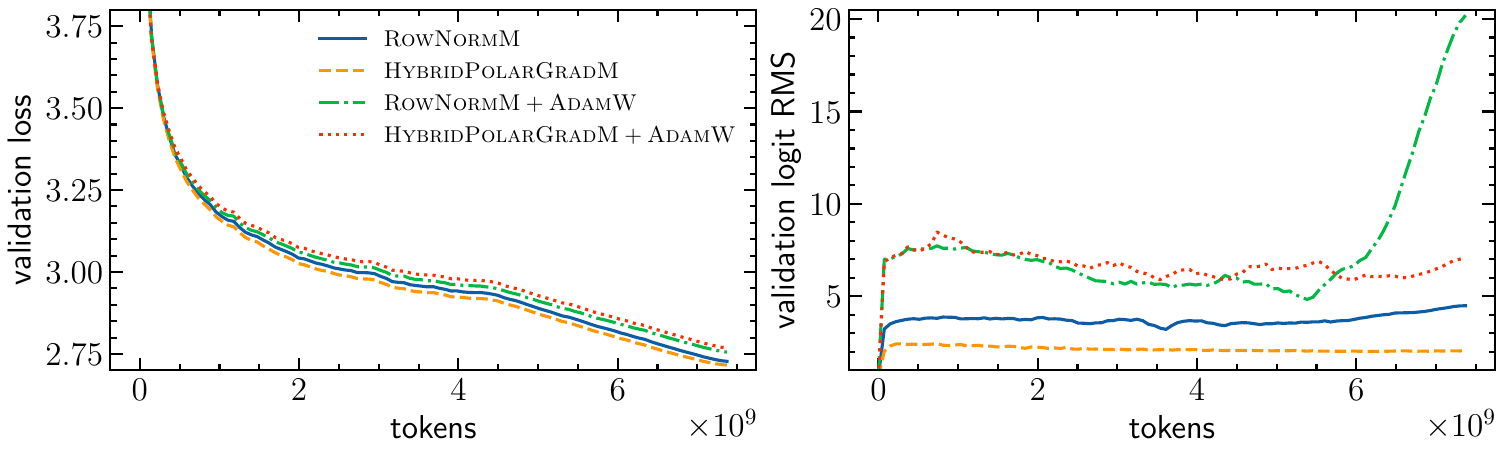}
        \caption{Validation cross-entropy and raw final-logit RMS.}
        \label{fig:gemma3_no_logit_1}
    \end{subfigure}%
    \vspace*{2.5mm}
    \begin{subfigure}[t]{\textwidth}
        \centering
        \includegraphics[width=\textwidth]{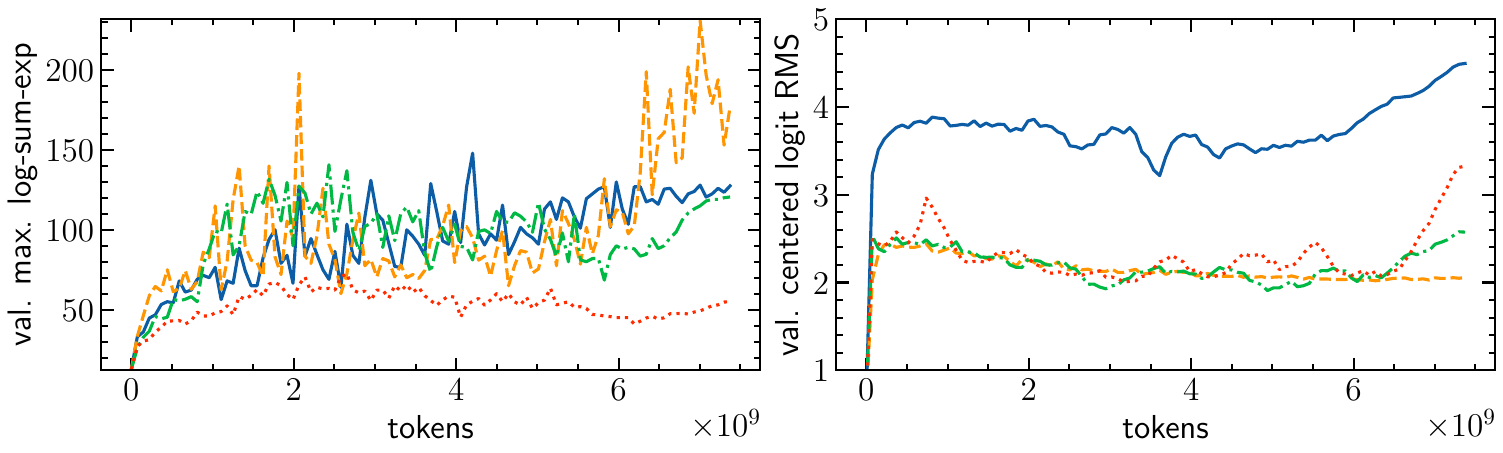}
        \caption{Validation maximum log-sum-exp and centered final-logit RMS.}
        \label{fig:gemma3_no_logit_2}
    \end{subfigure}
    \caption{Validation loss and final vocabulary-logit diagnostics for Gemma 3 1B-style pre-training without final logit soft-capping. The \AdamW LM head variants exhibit larger raw logit RMS and, for \RowNormM+\AdamW, a late-stage increase in maximum log-sum-exp, while the centered logit RMS remains comparable. This indicates that the growth is mainly in the shared-logit-shift component.}
    \label{fig:gemma3_logit_no_softcapping}
\end{figure}

\paragraph{With final logit soft-capping.}
We next repeat the comparison with final logit soft-capping enabled. As shown in \Cref{fig:gemma3_logit_softcapping}, the \AdamW LM head variants no longer exhibit the pronounced late-stage raw logit RMS growth observed without soft-capping. Their raw logit RMS generally decreases after the initial stage, with intermediate fluctuations, while maximum log-sum-exp remains in a comparable range across methods. The projected \RowNormM and \HybridPolarGradM updates maintain larger but steadier raw logit RMS for most of training, consistent with a stabilizing effect of soft-capping.

Nevertheless, the projected methods still achieve lower validation loss than the corresponding \AdamW LM head variants. Their larger centered logit RMS indicates that this benefit does not simply come from shrinking all logit variation. Raw and centered logit RMS nearly coincide for the projected methods, whereas the \AdamW variants retain an appreciable gap. Thus, soft-capping and projected updates exhibit complementary empirical benefits: soft-capping bounds the effective final logits, while the projected configurations maintain smaller shared-shift components and achieve lower validation loss. However, a shared shift of the pre-capping logits is not generally an exact symmetry after nonlinear soft-capping, so the quotient-geometry interpretation requires qualification.

\begin{figure}[h!]
    \centering
    \begin{subfigure}[t]{\textwidth}
        \centering
        \includegraphics[width=\textwidth]{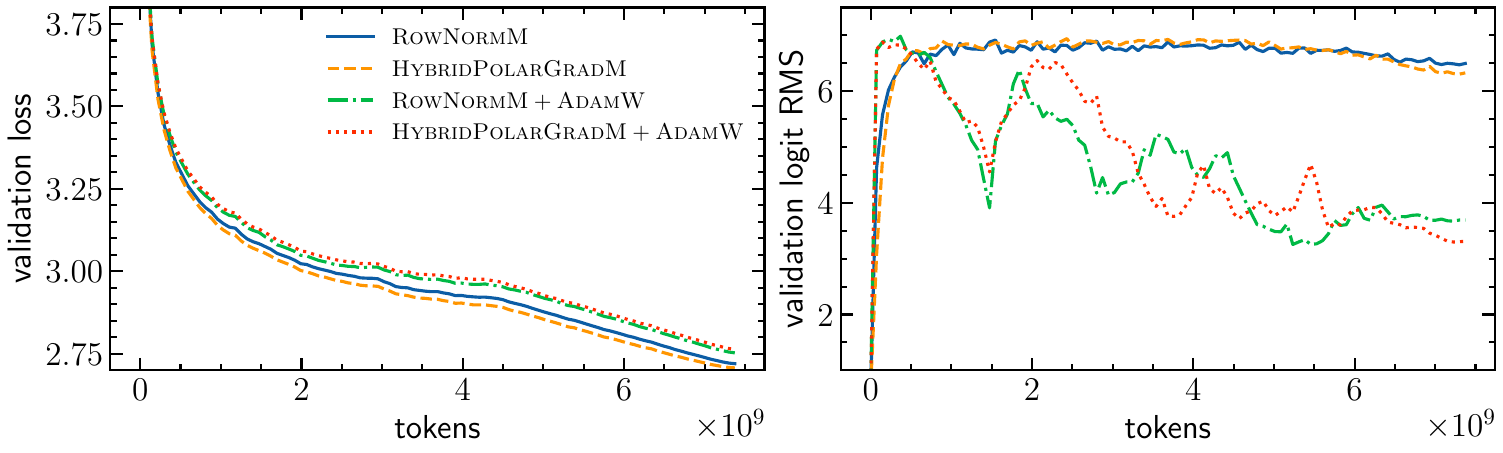}
        \caption{Validation cross-entropy and raw final-logit RMS.}
        \label{fig:gemma3_logit_1}
    \end{subfigure}%
    \vspace*{2.5mm}
    \begin{subfigure}[t]{\textwidth}
        \centering
        \includegraphics[width=\textwidth]{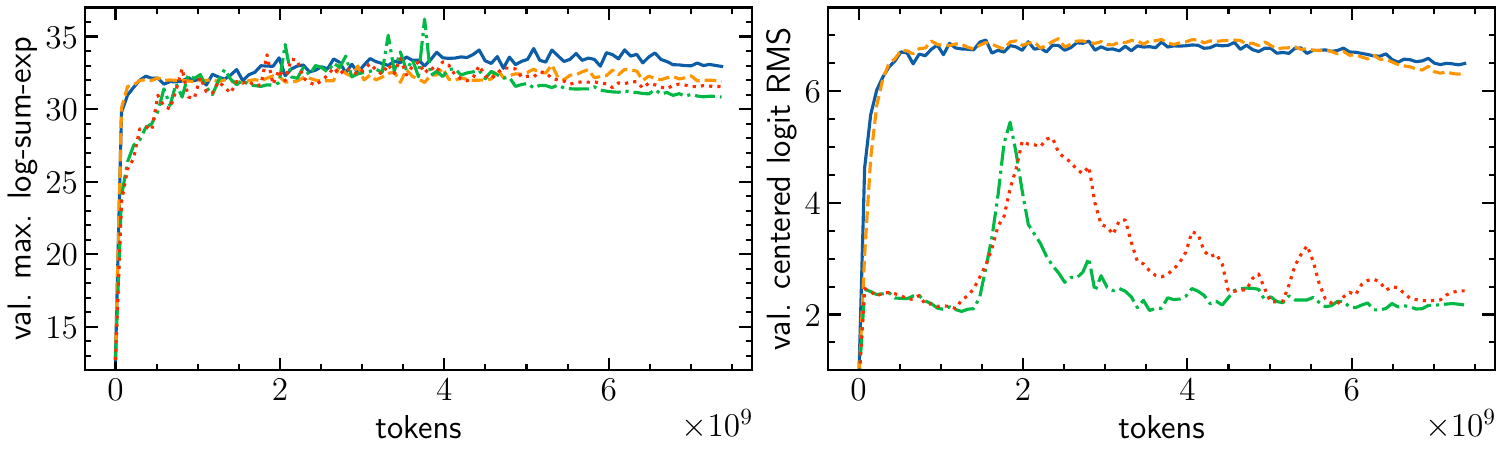}
        \caption{Validation maximum log-sum-exp and centered final-logit RMS.}
        \label{fig:gemma3_logit_2}
    \end{subfigure}
    \caption{Validation loss and final vocabulary-logit diagnostics for Gemma 3 1B-style pre-training with final logit soft-capping enabled. Soft-capping suppresses the raw logit-scale growth of the \AdamW LM head variants, while projected \RowNormM and \HybridPolarGradM LM head updates still achieve lower validation loss and stronger centered logit variation.}
    \label{fig:gemma3_logit_softcapping}
\end{figure}

We give all the final validation loss and logit statistics in \Cref{tab:gemma3-final-logit-statistics}. 
\begin{table}[t]
\centering
\caption{Final validation loss and logit statistics for Gemma 3 1B-style pre-training.}
\label{tab:gemma3-final-logit-statistics}
\footnotesize
\setlength{\tabcolsep}{5pt}
\begin{tabular}{llcccc}
\toprule
Setting & Optimizer & \shortstack{Validation\\loss} & \shortstack{Raw logit\\RMS} & \shortstack{Centered logit\\RMS} & \shortstack{Maximum\\log-sum-exp} \\
\midrule
\multirow{4}{*}{No logit soft-capping}
& \RowNormM & $2.7275$ & $4.4944$ & $4.4944$ & $127.50$ \\
& \HybridPolarGradM & $\mathbf{2.7161}$ & $2.0556$ & $2.0556$ & $177.00$ \\
& \RowNormM $+$ \AdamW & $2.7556$ & $20.1984$ & $2.5737$ & $120.57$ \\
& \HybridPolarGradM $+$ \AdamW & $2.7685$ & $7.0571$ & $3.3251$ & $55.0087$ \\
\midrule
\multirow{4}{*}{Logit soft-capping}
& \RowNormM & $2.7202$ & $6.4975$ & $6.4971$ & $32.9408$ \\
& \HybridPolarGradM & $\mathbf{2.7074}$ & $6.3287$ & $6.3285$ & $31.9113$ \\
& \RowNormM $+$ \AdamW & $2.7530$ & $3.6942$ & $2.1700$ & $30.8319$ \\
& \HybridPolarGradM $+$ \AdamW & $2.7636$ & $3.3124$ & $2.4248$ & $31.5483$ \\
\bottomrule
\end{tabular}
\end{table}

\newpage
\subsubsection{Separation of the Embedding and LM-Head Contributions}
\label{subsubsec:embedding_lm_head_contributions}
To separate the embedding and LM-head contributions, we evaluate mixed optimizer assignments for Gemma 3 1B-style pre-training with untied matrices; see \Cref{tab:embedding_lm_head_contributions}. Starting from the hyperparameters of the \RowNormM/\RowNormM and \HybridPolarGradM/\HybridPolarGradM configurations, we replace either optimizer with \AdamW, retaining the original learning rates without retuning. All other training settings remain fixed.

For \RowNormM, the embedding-only assignment achieves a validation loss of $2.7497$, improving over the \AdamW/\AdamW reference of $2.7562$ by $0.0065$. The LM-head-only assignment achieves $2.7275$, a larger improvement of $0.0287$. Using \RowNormM on both matrices reduces the loss to $2.7251$, giving an overall gain of $0.0311$ and an additional embedding contribution of $0.0024$ relative to the LM-head-only assignment.

For \HybridPolarGradM, the embedding-only assignment increases the loss to $2.8131$, whereas the LM-head-only assignment reduces it to $2.7216$, a gain of $0.0346$ over the reference. The fully hybrid assignment achieves the lowest loss of $2.7160$, improving over the reference by $0.0402$ and over the LM-head-only assignment by $0.0056$.

Under these training recipes, the larger gains arise from changing the LM head optimizer, while the embedding contribution depends on the LM head assignment. In particular, hybrid embedding updates help when paired with a hybrid LM head but perform poorly with an \AdamW LM head. These findings are conditional on the inherited hyperparameters, since the mixed assignments were not separately retuned. 

\begin{table}[h]
\centering
\caption{Embedding and LM-head optimizer contributions for Gemma 3 1B-style pre-training. Gains denote reductions in validation loss. The LM-head gain compares against an \AdamW LM head with the same embedding optimizer.}
\label{tab:embedding_lm_head_contributions}
\footnotesize
\setlength{\tabcolsep}{5pt}
\begin{tabular}{llrrr}
\toprule
Embedding optimizer & LM-head optimizer & \shortstack{Validation\\loss} & \shortstack{Gain over\\\AdamW/\AdamW} & \shortstack{LM-head gain at\\fixed embedding} \\
\midrule
\AdamW & \AdamW & $2.7562$ & N/A & N/A \\
\RowNormM & \AdamW & $2.7497$ & $0.0065$ & N/A \\
\AdamW & \RowNormM & $2.7275$ & $0.0287$ & $0.0287$ \\
\RowNormM & \RowNormM & $2.7251$ & $0.0311$ & $0.0246$ \\
\HybridPolarGradM & \AdamW & $2.8131$ & $-0.0569$ & N/A \\
\AdamW & \HybridPolarGradM & $2.7216$ & $0.0346$ & $0.0346$ \\
\HybridPolarGradM & \HybridPolarGradM & $2.7160$ & $0.0402$ & $0.0971$ \\
\bottomrule
\end{tabular}
\end{table}

\newpage
\subsection{\OLMoE-1B-7B-Style Pre-Training}
In terms of wall-clock training time, configurations (i)--(iv) take 9.442 hours, 9.587 hours, 9.467 hours and 9.535 hours respectively. 
\begin{table}[H]
   	\centering
   	\caption{Training configurations for \OLMoE-1B-7B-style pre-training.}
   	\label{table:config_olmoe}
    \small
   	\begin{tabular}{lc}
        \toprule
        Model & \OLMoE-1B-7B on FineWeb-Edu-10B \\
        \midrule
        Context length & 1024 tokens \\
        Per-device batch size & 20 sequences \\
        Training steps & 30,000 \\
        Training tokens & 4,915,200,000 \\        
        Validation steps & 64 \\
        Validation tokens & 10,485,760 \\
        Precision & \texttt{bfloat16} \\         
        Data-parallel size & 8 (Nvidia H200) \\   
        \bottomrule
   	\end{tabular}
\end{table}

\begin{table}[H]
    \centering
    \caption{Optimizer hyperparameters for \OLMoE-1B-7B-style pre-training.}
    \label{table:optim_hyperparams_olmoe}
    \scriptsize
    \begin{tabular}{cccccc}
        \toprule
        Configuration & Parameter type & Optimizer & $\gamma_0$ & momentum $\beta$ & weight decay $\lambda$ \\
        \midrule
        \multirow{5}{*}{\shortstack{(i) \AdamW + \Muon \\+ \RowNormM}}
            & scalar / vector     & \AdamW             & $0.01$ & $(0.9, 0.95)$  & $0.001$ \\
            & matrix     & \Muon              & $0.005$ & $0.95$ & $0.001$ \\
            & embedding  & \RowNormM   & $0.5$ & $0.95$  & $0$ \\
            & head    & \RowNormM   & $0.005$ & $0.95$  & $0$ \\
            & router & \RowNormM      & $0.00075$ & $0.95$    & $0$ \\
        \midrule
        \multirow{5}{*}{\shortstack{(ii) \AdamW + \Muon \\ + \RowNormM \\ + \LeftPolarGradM \\ ($\alpha=1$)}}
            & scalar / vector     & \AdamW             & $0.01$ & $(0.9, 0.95)$  & $0.001$ \\
            & matrix     & \Muon              & $0.005$ & $0.95$ & $0.001$ \\
            & embedding  & \RowNormM   & $0.5$ & $0.95$  & $0$ \\
            & head    & \RowNormM   & $0.005$ & $0.95$  & $0$ \\
            & router    & \LeftPolarGradM    & $0.00075$ & $0.95$  & $0$ \\ 
        \midrule
        \multirow{5}{*}{\shortstack{(iii) \AdamW + \Muon \\ + \RowNormM \\ + \AdamW}}
            & scalar / vector     & \AdamW             & $0.01$ & $(0.9, 0.95)$  & $0.001$ \\
            & matrix     & \Muon              & $0.005$ & $0.95$ & $0.001$ \\
            & embedding  & \RowNormM   & $0.5$ & $0.95$  & $0$ \\
            & head    & \RowNormM   & $0.005$ & $0.95$  & $0$ \\
            & router    & \AdamW     & $0.00075$ & $0.95$  & $0$ \\
        \midrule
        \multirow{5}{*}{\shortstack{(iv) \AdamW + \Muon \\ + \AdamW}}
            & scalar / vector     & \AdamW             & $0.01$ & $(0.9, 0.95)$   & $0.001$ \\
            & matrix     & \Muon              & $0.005$ & $0.95$ & $0.001$ \\
            & embedding  & \AdamW             & $0.0005$ & $(0.9, 0.95)$   & $0$ \\
            & head    & \AdamW             & $0.0005$ & $(0.9, 0.95)$ & $0$ \\
            & router & \AdamW             & $0.00075$ & $(0.9, 0.95)$  & $0$ \\        
        \bottomrule
    \end{tabular}
\end{table}

\newpage
\subsection{Downsized gpt-oss Pre-Training}
In terms of wall-clock training time, configurations (i)--(iv) take 16.40 hours, 16.54 hours, 16.96 hours and 16.88 hours respectively. 
\begin{table}[H]
   	\centering
   	\caption{Training configurations for downsized gpt-oss pre-training.}
   	\label{table:config_gpt_oss}
    \small
   	\begin{tabular}{lc}
        \toprule
        Model & Downsized gpt-oss on FineWeb-Edu-10B \\
        \midrule
        Context length & 1024 tokens \\
        Per-device batch size & 8 sequences \\
        Training steps & 60,000 \\
        Training tokens & 3,932,160,000 \\        
        Validation steps & 160 \\
        Validation tokens & 10,485,760 \\
        Precision & \texttt{bfloat16} \\         
        Data-parallel size & 8 (Nvidia H200) \\   
        \bottomrule
   	\end{tabular}
\end{table}

\begin{table}[H]
    \centering
    \caption{Optimizer hyperparameters for downsized gpt-oss pre-training.}
    \label{table:optim_hyperparams_gpt_oss}
    \scriptsize
    \begin{tabular}{cccccc}
        \toprule
        Configuration & Parameter type & Optimizer & $\gamma_0$ & momentum $\beta$ & weight decay $\lambda$ \\
        \midrule
        \multirow{5}{*}{\shortstack{(i) \AdamW + \Muon \\+ \RowNormM}}
            & scalar / vector     & \AdamW             & $0.005$ & $(0.9, 0.95)$  & $0.001$ \\
            & matrix     & \Muon              & $0.001$ & $0.95$ & $0.001$ \\
            & embedding  & \RowNormM   & $0.1$ & $0.95$  & $0$ \\
            & head    & \RowNormM   & $0.001$ & $0.95$  & $0$ \\
            & router & \RowNormM      & $0.00075$ & $0.95$    & $0$ \\
        \midrule
        \multirow{5}{*}{\shortstack{(ii) \AdamW + \Muon \\ + \RowNormM \\ + \LeftPolarGradM \\ ($\alpha=1$)}}
            & scalar / vector     & \AdamW             & $0.005$ & $(0.9, 0.95)$  & $0.001$ \\
            & matrix     & \Muon              & $0.001$ & $0.95$ & $0.001$ \\
            & embedding  & \RowNormM   & $0.1$ & $0.95$  & $0$ \\
            & head    & \RowNormM   & $0.001$ & $0.95$  & $0$ \\
            & router    & \LeftPolarGradM    & $0.00075$ & $0.95$  & $0$ \\ 
        \midrule
        \multirow{5}{*}{\shortstack{(iii) \AdamW + \Muon \\ + \RowNormM \\ + \AdamW}}
            & scalar / vector     & \AdamW             & $0.005$ & $(0.9, 0.95)$  & $0.001$ \\
            & matrix     & \Muon              & $0.001$ & $0.95$ & $0.001$ \\
            & embedding  & \RowNormM   & $0.1$ & $0.95$  & $0$ \\
            & head    & \RowNormM   & $0.001$ & $0.95$  & $0$ \\
            & router    & \AdamW     & $0.00075$ & $0.95$  & $0$ \\
        \midrule
        \multirow{5}{*}{\shortstack{(iv) \AdamW + \Muon \\ + \AdamW}}
            & scalar / vector     & \AdamW             & $0.005$ & $(0.9, 0.95)$   & $0.001$ \\
            & matrix     & \Muon              & $0.001$ & $0.95$ & $0.001$ \\
            & embedding  & \AdamW             & $0.0005$ & $(0.9, 0.95)$   & $0$ \\
            & head    & \AdamW             & $0.0005$ & $(0.9, 0.95)$ & $0$ \\
            & router & \AdamW             & $0.00075$ & $(0.9, 0.95)$  & $0$ \\        
        \bottomrule
    \end{tabular}
\end{table}

\subsubsection{Load-Balancing Loss and Router z-Loss Diagnostics}
\label{subsubsec:gpt_oss_router_diagnostics}
Following the diagnostic protocol of \citep{muennighoff2024olmoe}, we evaluate whether symmetry-compatible router updates improve expert load balancing and router stability. For each \MoE layer, we log the fraction of routed assignments sent to each expert, the total router probability mass assigned to each expert, the auxiliary load-balancing loss \citep{shazeer2017outrageously}, and the router z-loss \citep{zoph2022st}. We also track load entropy, load coefficient of variation, dead-expert fraction, maximum expert load, and router-logit scale. These diagnostics allow us to distinguish improvements in language model validation loss from improvements in routing dynamics.

Let $E$ denote the number of experts and let $k$ denote the number of activated experts per token. We measure the router load-balancing loss by $\calL_{\mathrm{LB}} = E\sum_{i=1}^{E} f_i P_i$, where $f_i$ is the fraction of routed top-$k$ assignments sent to expert $i$, and $P_i$ is the average router probability mass assigned to expert $i$. A perfectly balanced router has $\calL_{\mathrm{LB}}$ close to $1$. We also measure the router z-loss $\calL_{\mathrm{z}} = \frac{1}{N}\sum_{t=1}^{N} \left( \log\sum_{i=1}^{E}\exp (z_{t,i}) \right)^{\negthickspace2}$, where $z_{t,i}$ is the router logit for token $t$ and expert $i$ and $N$ is the batch size. This term penalizes excessive router-logit scale. 

We compare \RowNormM, \LeftPolarGradM, uncentered \LeftPolarGradM, \HybridPolarGradM, and \AdamW on router matrices while keeping all non-router optimizer assignments fixed: \AdamW for scalars and vectors, \Muon for linear, attention, and SwiGLU MLP matrices, and \RowNormM for embedding and LM head matrices. The uncentered \LeftPolarGradM baseline applies a Muon-style left-polar update directly to the router momentum without removing the shared-row component, and therefore ablates the centering operation induced by shared-logit-shift invariance. We keep the settings from \Cref{table:config_gpt_oss,table:optim_hyperparams_gpt_oss}. The hyperparameters of uncentered \LeftPolarGradM are set to be the same as those of \RowNormM and \LeftPolarGradM, while \HybridPolarGradM uses a smaller learning rate 0.00005.

\paragraph{Training objective without auxiliary losses.}
We first evaluate router dynamics when neither $\calL_{\mathrm{LB}}$ nor $\calL_{\mathrm{z}}$ is included in the training objective. As shown in \Cref{fig:gpt_oss_router_load_balancing,fig:gpt_oss_router_z_losses}, the choice of router optimizer has a large effect on both expert load balancing and router-logit scale. Compared with \AdamW, all four alternative router updates achieve substantially lower measured load-balancing loss and router z-loss on both training and validation batches. The centered and uncentered \LeftPolarGradM variants give the strongest load-balancing performance, while \HybridPolarGradM achieves the lowest router z-loss. On the load-balancing diagnostics, \HybridPolarGradM improves over \RowNormM but remains less effective than the purely left-polar variants.

The expert-assignment summaries in \Cref{fig:gpt_oss_router_expert_assignment} show a similar pattern. \AdamW has lower load entropy, higher load coefficient of variation, a larger dead-expert fraction, and higher maximum expert load. In contrast, both \LeftPolarGradM variants achieve the most balanced expert assignment and report zero dead-expert fractions at the final evaluation. \HybridPolarGradM is intermediate between the left-polar variants and \RowNormM, with a dead-expert fraction of $0.0052$, compared with $0.0156$ for \RowNormM and $0.2396$ for \AdamW. These results suggest that matrix-aware router updates can improve load balancing in this setting, even without explicitly optimizing a load-balancing auxiliary objective.

Validation cross-entropy also improves over the \AdamW router baseline. \HybridPolarGradM achieves the lowest final loss of $2.8606$, followed by \RowNormM at $2.8631$. Centered and uncentered \LeftPolarGradM achieve $2.8670$ and $2.8666$, respectively, compared with $2.8858$ for \AdamW. Thus, the strongest load balancing does not coincide with the lowest validation loss. The centered and uncentered \LeftPolarGradM variants behave similarly, suggesting that their observed gains arise mainly from the left-polar matrix geometry, while centering provides the symmetry-compatible quotient-space formulation induced by shared-logit-shift invariance.

\paragraph{Training objective with auxiliary losses.}
We next include $\calL_{\mathrm{LB}}$ and $\calL_{\mathrm{z}}$ in the training objective with weights $0.01$ and $0.001$, respectively, following the choices of similar work \citep{muennighoff2024olmoe,xue2024openmoe,shen2024jetmoe}. In this setting, we report validation cross-entropy as the primary language modeling metric, since the total training objective additionally contains auxiliary router penalties. As shown in \Cref{fig:gpt_oss_router_load_balancing_aux,fig:gpt_oss_router_z_losses_aux}, auxiliary losses improve routing diagnostics for all methods: the measured load-balancing loss moves closer to its balanced reference value of $1$, and the router z-loss decreases substantially.

Nevertheless, optimizer geometry remains important. Compared with \AdamW, all four alternative router updates yield lower measured load-balancing loss, lower router z-loss, higher load entropy, lower load coefficient of variation, and lower maximum expert load; see \Cref{fig:gpt_oss_router_expert_assignment_aux}. All methods, including \AdamW, report zero dead-expert fractions at the final evaluation. The four alternatives are closely clustered, with \HybridPolarGradM achieving the lowest validation load-balancing loss and router z-loss, at $1.0260$ and $0.1376$, respectively. Uncentered \LeftPolarGradM has the highest load entropy and the lowest load coefficient of variation and maximum expert load.

\RowNormM achieves the lowest final validation cross-entropy of $2.8605$, closely followed by \HybridPolarGradM at $2.8612$. Centered and uncentered \LeftPolarGradM achieve $2.8659$ and $2.8658$, respectively, compared with $2.8693$ for \AdamW. Thus, auxiliary router losses and matrix-aware router updates provide complementary improvements in routing diagnostics. These improvements do not uniformly translate into lower validation cross-entropy, however: \HybridPolarGradM changes slightly from $2.8606$ without auxiliary losses to $2.8612$ with them.

\begin{figure}[h!]
    \centering
    \begin{subfigure}[t]{\textwidth}
        \centering
        \includegraphics[width=\textwidth]{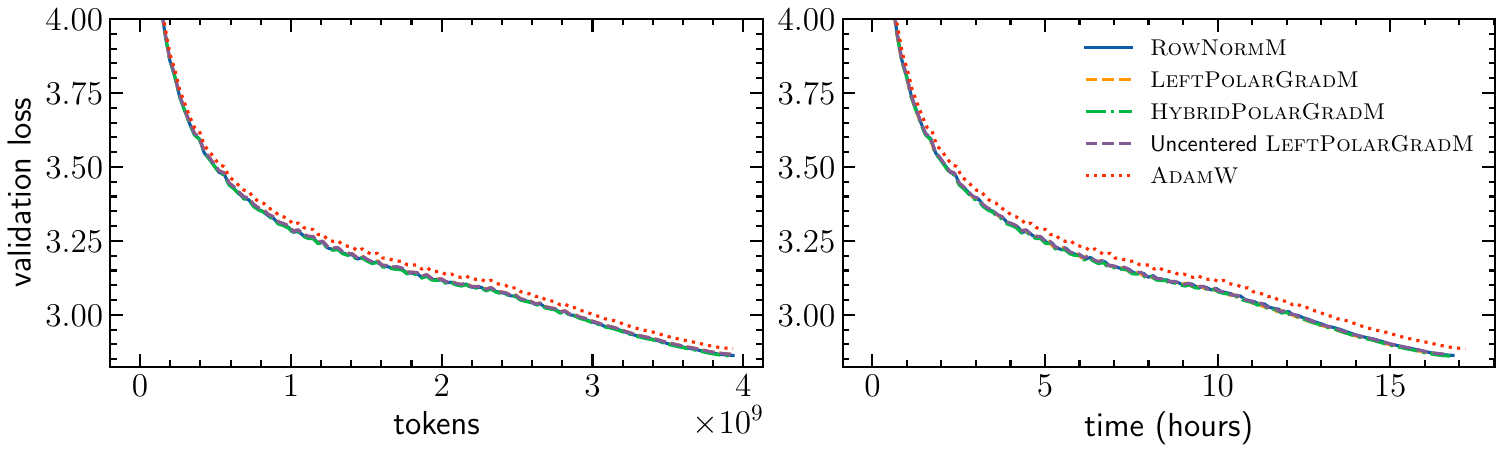}
        \caption{Validation cross-entropy loss.}
        \label{fig:gpt_oss_router_val_losses}
    \end{subfigure}%
    \vspace*{2.5mm}
    \begin{subfigure}[t]{\textwidth}
        \centering
        \includegraphics[width=\textwidth]{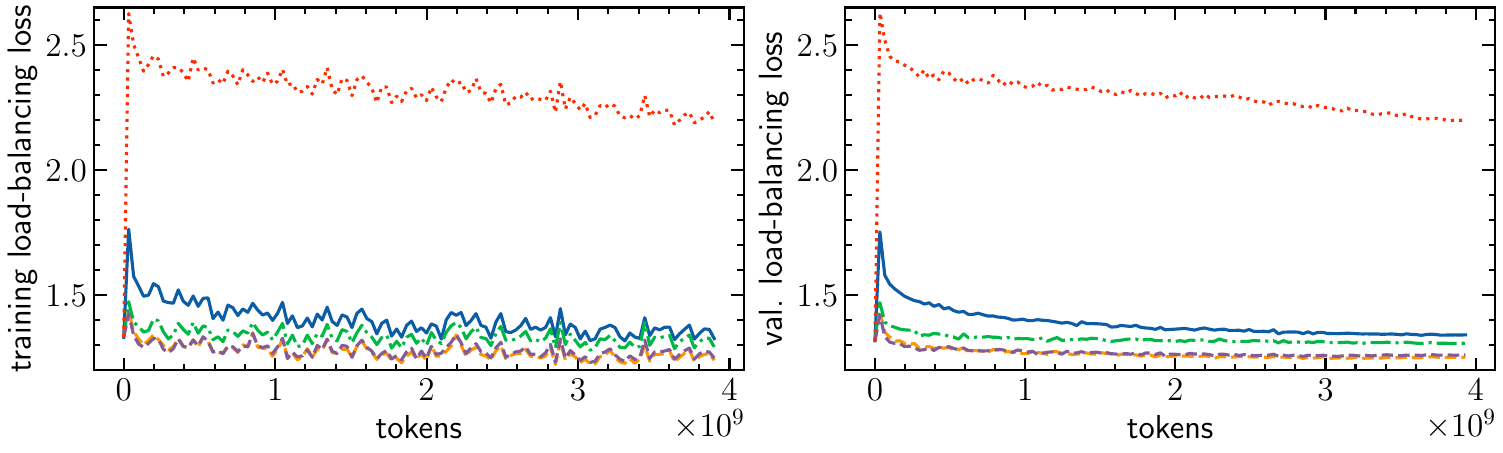}
        \caption{Training and validation load-balancing loss.}
        \label{fig:gpt_oss_router_load_balancing_losses}
    \end{subfigure}%
    \caption{Validation loss and measured load-balancing loss for downsized gpt-oss pre-training. The configurations differ only in the optimizer assigned to the \MoE router matrices. Lower load-balancing loss indicates more balanced expert usage.}
    \label{fig:gpt_oss_router_load_balancing}
\end{figure}

\begin{figure}[h!]
    \centering
    \includegraphics[width=\textwidth]{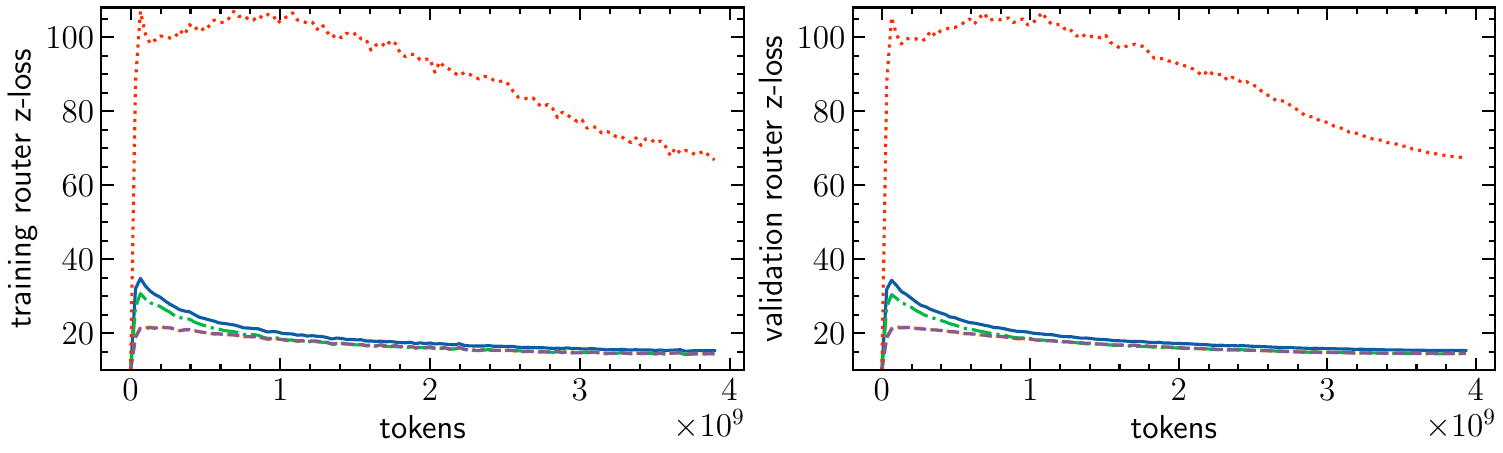}
    \caption{Training and validation router z-loss for downsized gpt-oss pre-training. Lower router z-loss indicates better control of router-logit scale.}
    \label{fig:gpt_oss_router_z_losses}
\end{figure}

\begin{figure}[h!]
    \centering
    \begin{subfigure}[t]{\textwidth}
        \centering
        \includegraphics[width=\textwidth]{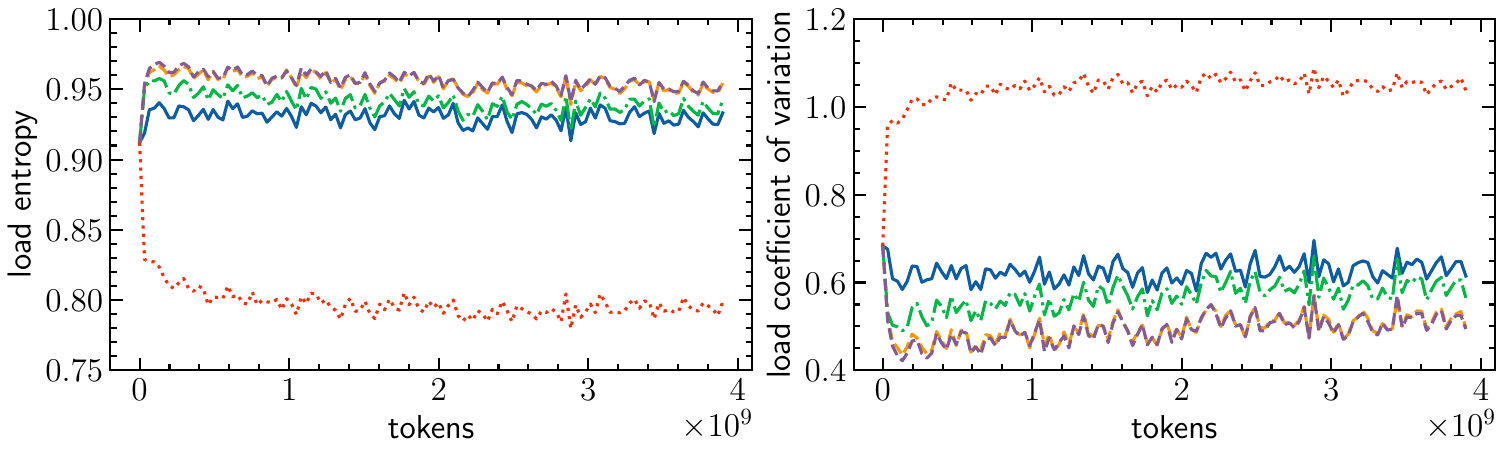}
        \caption{Load entropy and load coefficient of variation. Higher entropy and lower coefficient of variation indicate more balanced expert assignment.}
        \label{fig:gpt_oss_load_metrics}
    \end{subfigure}%
    \vspace*{2.5mm}
    \begin{subfigure}[t]{\textwidth}
        \centering
        \includegraphics[width=\textwidth]{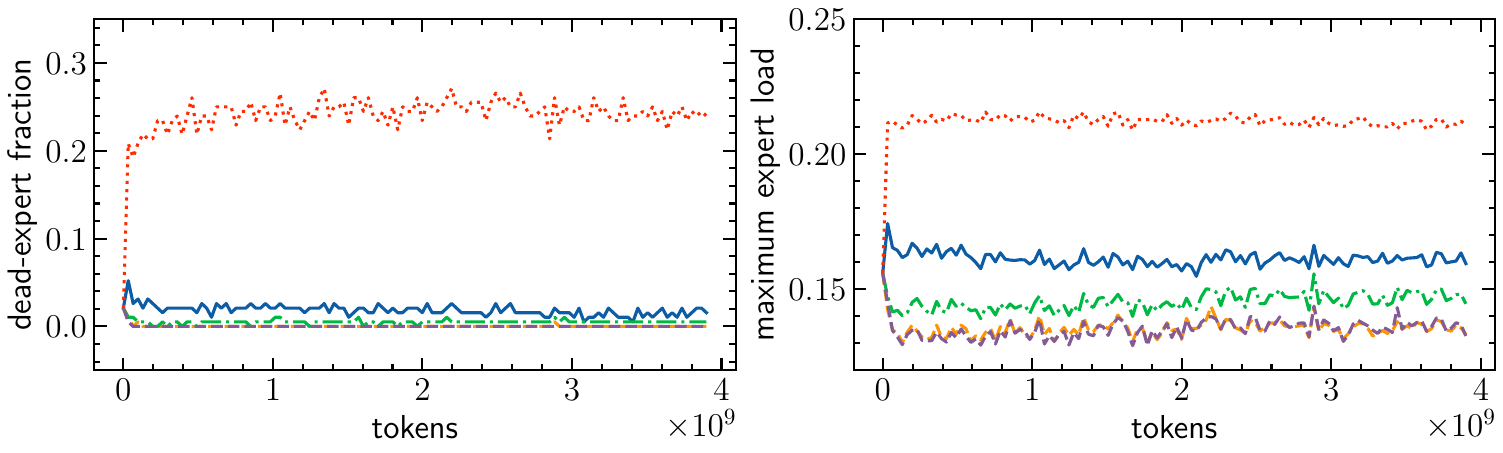}
        \caption{Dead-expert fraction and maximum expert load. Lower values indicate fewer unused experts and less expert dominance.}
        \label{fig:gpt_oss_expert_metrics}
    \end{subfigure}%
    \caption{Expert-assignment balance summaries for downsized gpt-oss pre-training. These diagnostics directly measure whether routed assignments are evenly distributed across experts.}
    \label{fig:gpt_oss_router_expert_assignment}
\end{figure}

\begin{figure}[h!]
    \centering
    \begin{subfigure}[t]{\textwidth}
        \centering
        \includegraphics[width=\textwidth]{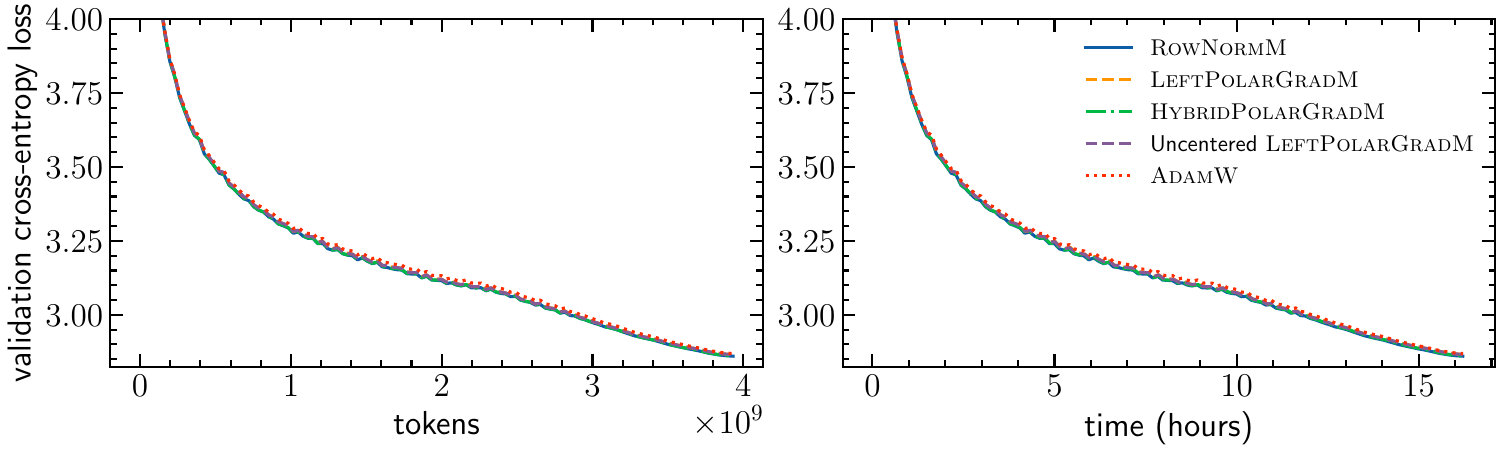}
        \caption{Validation cross-entropy loss.}
        \label{fig:gpt_oss_router_val_losses_aux}
    \end{subfigure}%
    \vspace*{2.5mm}
    \begin{subfigure}[t]{\textwidth}
        \centering
        \includegraphics[width=\textwidth]{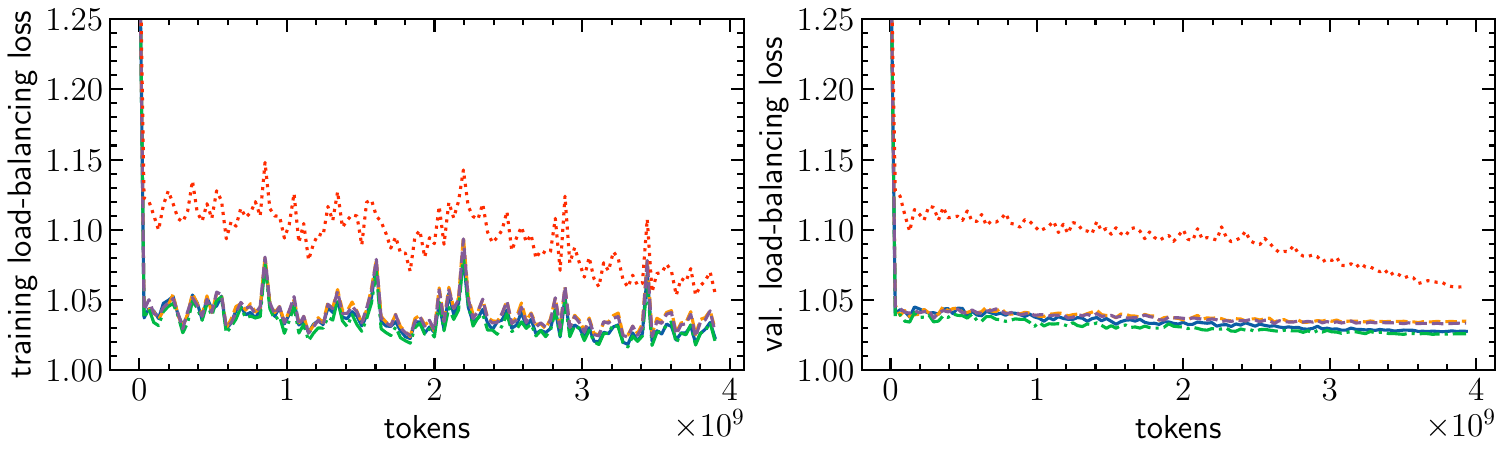}
        \caption{Training and validation load-balancing loss.}
        \label{fig:gpt_oss_router_load_balancing_losses_aux}
    \end{subfigure}%
    \caption{Validation loss and measured load-balancing loss for downsized gpt-oss pre-training with auxiliary router losses enabled. We report validation cross-entropy as the primary language modeling metric because the total training objective also includes auxiliary router losses.}
    \label{fig:gpt_oss_router_load_balancing_aux}
\end{figure}

\begin{figure}[h!]
    \centering
    \includegraphics[width=\textwidth]{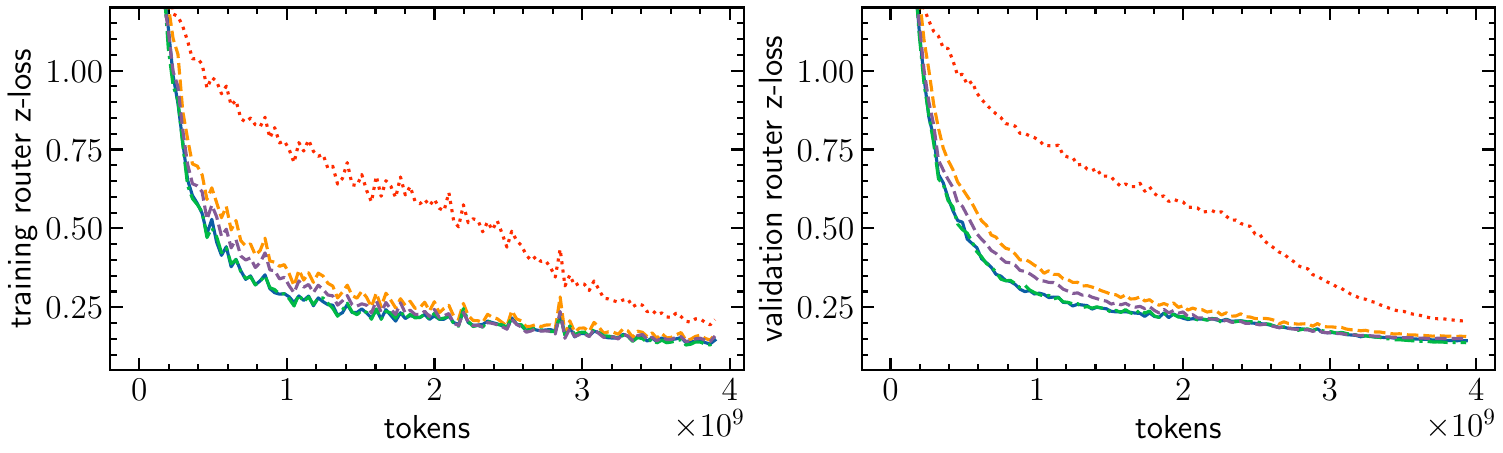}
    \caption{Training and validation router z-loss for downsized gpt-oss pre-training with auxiliary router losses enabled. Lower router z-loss indicates better control of router-logit scale.}
    \label{fig:gpt_oss_router_z_losses_aux}
\end{figure}

\begin{figure}[h!]
    \centering
    \begin{subfigure}[t]{\textwidth}
        \centering
        \includegraphics[width=\textwidth]{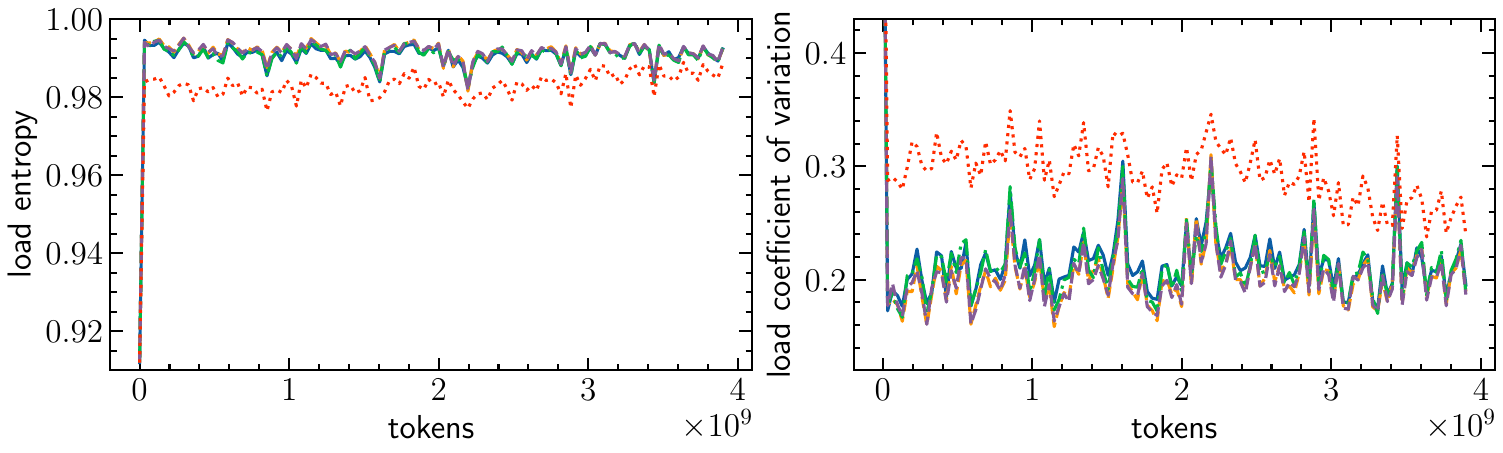}
        \caption{Load entropy and load coefficient of variation. Higher entropy and lower coefficient of variation indicate more balanced expert assignment.}
        \label{fig:gpt_oss_load_metrics_aux}
    \end{subfigure}%
    \vspace*{2.5mm}
    \begin{subfigure}[t]{\textwidth}
        \centering
        \includegraphics[width=\textwidth]{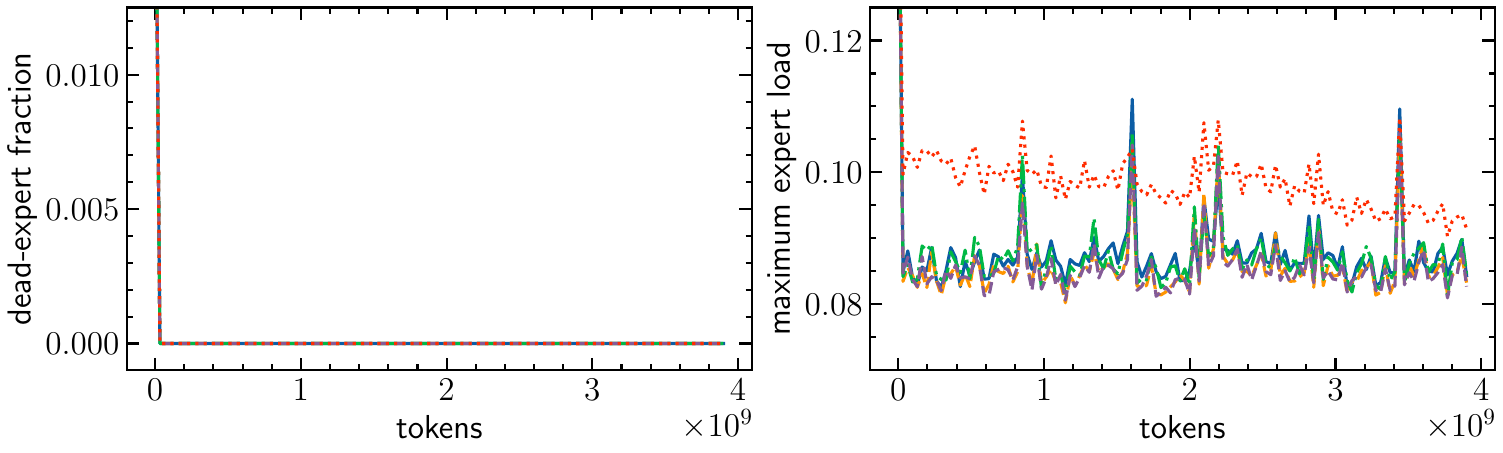}
        \caption{Dead-expert fraction and maximum expert load. Lower values indicate fewer unused experts and less expert dominance.}
        \label{fig:gpt_oss_expert_metrics_aux}
    \end{subfigure}%
    \caption{Expert-assignment balance summaries for downsized gpt-oss pre-training with auxiliary router losses enabled. These diagnostics directly measure whether routed assignments are evenly distributed across experts.}
    \label{fig:gpt_oss_router_expert_assignment_aux}
\end{figure}

\clearpage
\newpage
We give all the final validation loss and router statistics in \Cref{tab:gpt-oss-final-router-statistics}. 
\begin{table}[h!]
\centering
\caption{Final validation loss and router statistics for downsized gpt-oss pre-training, with and without the auxiliary load-balancing and router z-losses. Higher load entropy indicates more balanced expert utilization, whereas lower values are preferable for the other router statistics. The best value within each setting is shown in bold.}
\label{tab:gpt-oss-final-router-statistics}
\vspace*{2.5mm}
\small
\begin{tabular}{lccccc}
\toprule
\multirow{2}{*}{Optimizer} & \multirow{2}{*}{Validation loss} & \multicolumn{2}{c}{Load-balancing loss} & \multicolumn{2}{c}{Router z-loss} \\
\cmidrule(lr){3-4}\cmidrule(lr){5-6}
& & Train & Validation & Train & Validation \\
\midrule
\multicolumn{6}{l}{\textit{Without auxiliary losses}} \\
\addlinespace[1mm]
\RowNormM & $2.8631$ & $1.3269$ & $1.3418$ & $15.3222$ & $15.3382$ \\
\LeftPolarGradM & $2.8670$ & $\mathbf{1.2401}$ & $\mathbf{1.2527}$ & $14.6206$ & $14.5886$ \\
\HybridPolarGradM & $\mathbf{2.8606}$ & $1.2934$ & $1.3082$ & $\mathbf{14.4418}$ & $\mathbf{14.4332}$ \\
Uncentered \LeftPolarGradM & $2.8666$ & $1.2493$ & $1.2614$ & $14.6633$ & $14.5773$ \\
\AdamW & $2.8858$ & $2.1903$ & $2.2006$ & $66.8603$ & $67.5009$ \\
\midrule
\multicolumn{6}{l}{\textit{With auxiliary losses}} \\
\addlinespace[1mm]
\RowNormM & $\mathbf{2.8605}$ & $1.0234$ & $1.0278$ & $0.1457$ & $0.1438$ \\
\LeftPolarGradM & $2.8659$ & $1.0298$ & $1.0349$ & $0.1645$ & $0.1572$ \\
\HybridPolarGradM & $2.8612$ & $\mathbf{1.0210}$ & $\mathbf{1.0260}$ & $\mathbf{0.1421}$ & $\mathbf{0.1376}$ \\
Uncentered \LeftPolarGradM & $2.8658$ & $1.0274$ & $1.0336$ & $0.1587$ & $0.1507$ \\
\AdamW & $2.8693$ & $1.0553$ & $1.0596$ & $0.2107$ & $0.2064$ \\
\bottomrule
\end{tabular}

\vspace{3mm}

\begin{tabular}{lcccc}
\toprule
Optimizer & Load entropy & Load CV & Dead-expert frac. & Max.~expert load \\
\midrule
\multicolumn{5}{l}{\textit{Without auxiliary losses}} \\
\addlinespace[1mm]
\RowNormM & $0.9332$ & $0.6148$ & $0.0156$ & $0.1595$ \\
\LeftPolarGradM & $0.9545$ & $0.4971$ & $\mathbf{0.0000}$ & $0.1334$ \\
\HybridPolarGradM & $0.9424$ & $0.5650$ & $0.0052$ & $0.1446$ \\
Uncentered \LeftPolarGradM & $\mathbf{0.9549}$ & $\mathbf{0.4940}$ & $\mathbf{0.0000}$ & $\mathbf{0.1327}$ \\
\AdamW & $0.7987$ & $1.0373$ & $0.2396$ & $0.2110$ \\
\midrule
\multicolumn{5}{l}{\textit{With auxiliary losses}} \\
\addlinespace[1mm]
\RowNormM & $0.9923$ & $0.1954$ & $\mathbf{0.0000}$ & $0.0844$ \\
\LeftPolarGradM & $0.9924$ & $0.1895$ & $\mathbf{0.0000}$ & $0.0833$ \\
\HybridPolarGradM & $0.9926$ & $0.1883$ & $\mathbf{0.0000}$ & $0.0846$ \\
Uncentered \LeftPolarGradM & $\mathbf{0.9927}$ & $\mathbf{0.1868}$ & $\mathbf{0.0000}$ & $\mathbf{0.0827}$ \\
\AdamW & $0.9885$ & $0.2398$ & $\mathbf{0.0000}$ & $0.0915$ \\
\bottomrule
\end{tabular}
\end{table}

\end{document}